\documentclass[
  11pt,
  letterpaper,
  onecolumn,
  oneside]{article}

\usepackage[margin=1 in,letterpaper]{geometry}
\usepackage[singlespacing]{setspace}
\usepackage{fancyhdr}
\usepackage[inline]{enumitem}
\usepackage{newclude}
\usepackage{caption}
\usepackage{subcaption}

\usepackage[utf8]{inputenc}
\usepackage[T1]{fontenc}
\usepackage{lmodern}
\usepackage{appendix}
\usepackage{comment}
\usepackage{graphicx}
\graphicspath{ {images/} }
\usepackage{hyperref}
\usepackage{titlesec}
\newcommand{\periodafter}[1]{#1.}
\titleformat{\subsection}[runin]
  {\normalfont\bfseries}{\thesubsection}{0.5em}{\periodafter}
\titleformat{\subsubsection}[runin]
  {\normalfont\itshape}{\thesubsubsection}{0.5em}{\periodafter}

\newenvironment*{alphaenumerate}
{\begin{enumerate}[label={\alph*)}, noitemsep, nolistsep, align=left, leftmargin=*]}
{\end{enumerate}}
\newenvironment*{bulletlist}
{\begin{itemize}[noitemsep, nolistsep, align=left, leftmargin=*]}
{\end{itemize}}

\usepackage{stmaryrd}
\SetSymbolFont{stmry}{bold}{U}{stmry}{m}{n} 
\usepackage{amssymb}
\usepackage{amsthm}
\usepackage{mathtools}
\usepackage{bm}
\usepackage{mathrsfs}
\usepackage{bbm}

\usepackage{xparse}

\theoremstyle{definition}
\newtheorem{definition}{Definition}[section]
\newtheorem{remark}[definition]{Remark}

\theoremstyle{plain}
\newtheorem{theorem}[definition]{Theorem}
\newtheorem{proposition}[definition]{Proposition}
\newtheorem{lemma}[definition]{Lemma}
\newtheorem{corollary}[definition]{Corollary}

\newtheorem{observation}[definition]{Observation}

\newcommand*{\dequal}{\overset{\mathrm{d}}{=}}
\newcommand*{\setle}{\subseteq}
\newcommand*{\rarr}{\rightarrow}
\newcommand*{\lrarr}{\leftrightarrow}
\newcommand*{\darr}{\downarrow}
\newcommand*{\uarr}{\uparrow}
\newcommand*{\dotcup}{\,\dot\cup\,}
\newcommand*{\bigdotcup}{\dot\bigcup}

\newcommand*{\setb}{\mathcal B}

\newcommand*{\seti}{\mathcal I}
\newcommand*{\setp}{\mathcal P}
\newcommand*{\setq}{\mathcal Q}
\NewDocumentCommand{\mcla}{O{} O{}}{\mathcal A_{#1}^{#2}}
\NewDocumentCommand{\mclb}{O{} O{}}{\mathcal B_{#1}^{#2}}
\NewDocumentCommand{\mclc}{O{} O{}}{\mathcal C_{#1}^{#2}}
\NewDocumentCommand{\mcld}{O{} O{}}{\mathcal D_{#1}^{#2}}
\NewDocumentCommand{\mcle}{O{} O{}}{\mathcal E_{#1}^{#2}}
\NewDocumentCommand{\mclf}{O{} O{}}{\mathcal F_{#1}^{#2}}
\NewDocumentCommand{\mclh}{O{} O{}}{\mathcal H_{#1}^{#2}}
\NewDocumentCommand{\mcli}{O{} O{}}{\mathcal I_{#1}^{#2}}
\NewDocumentCommand{\mclk}{O{} O{}}{\mathcal K_{#1}^{#2}}
\NewDocumentCommand{\mclm}{O{} O{}}{\mathcal M_{#1}^{#2}}
\NewDocumentCommand{\mcln}{O{} O{}}{\mathcal N_{#1}^{#2}}
\NewDocumentCommand{\mclo}{O{} O{}}{\mathcal O_{#1}^{#2}}
\NewDocumentCommand{\mclp}{O{} O{}}{\mathcal P_{#1}^{#2}}
\NewDocumentCommand{\mclr}{O{} O{}}{\mathcal R_{#1}^{#2}}
\NewDocumentCommand{\mcls}{O{} O{}}{\mathcal S_{#1}^{#2}}
\NewDocumentCommand{\mclt}{O{} O{}}{\mathcal T_{#1}^{#2}}
\NewDocumentCommand{\mclv}{O{} O{}}{\mathcal V_{#1}^{#2}}
\NewDocumentCommand{\mclw}{O{} O{}}{\mathcal W_{#1}^{#2}}
\NewDocumentCommand{\mclx}{O{} O{}}{\mathcal X_{#1}^{#2}}
\NewDocumentCommand{\bmcla}{}{\bm{\mathcal A}}
\NewDocumentCommand{\bmcld}{}{\bm{\mathcal D}}

\NewDocumentCommand{\bmclr}{}{\bm{\mathcal R}}
\NewDocumentCommand{\bmclv}{}{\bm{\mathcal V}}
\NewDocumentCommand{\mfkg}{}{\mathfrak g}
\NewDocumentCommand{\mfkp}{}{\mathfrak p}
\NewDocumentCommand{\mfkA}{}{\mathfrak A}
\NewDocumentCommand{\mfkP}{}{\mathfrak P}
\NewDocumentCommand{\mfkPr}{}{\mathfrak P_{\mathrm{r}}}
\NewDocumentCommand{\mfkPc}{}{\mathfrak P_{\mathrm{c}}}
\NewDocumentCommand{\mrma}{}{\mathrm{a}}
\NewDocumentCommand{\mrmb}{}{\mathrm{b}}
\NewDocumentCommand{\mrmc}{}{\mathrm{c}}
\NewDocumentCommand{\mrmd}{}{\mathrm{d}}
\NewDocumentCommand{\mrme}{}{\mathrm{e}}
\NewDocumentCommand{\mrmf}{}{\mathrm{f}}
\NewDocumentCommand{\mrmg}{}{\mathrm{g}}

\NewDocumentCommand{\mrml}{}{\mathrm{l}}
\NewDocumentCommand{\mrmm}{}{\mathrm{m}}

\NewDocumentCommand{\mrmp}{}{\mathrm{p}}
\NewDocumentCommand{\mrmr}{}{\mathrm{r}}
\NewDocumentCommand{\mrms}{}{\mathrm{s}}
\NewDocumentCommand{\mrmt}{}{\mathrm{t}}
\NewDocumentCommand{\mrmu}{}{\mathrm{u}}
\NewDocumentCommand{\mrmv}{}{\mathrm{v}}
\NewDocumentCommand{\mrmw}{}{\mathrm{w}}
\NewDocumentCommand{\mrmtv}{}{\mathrm{TV}}

\NewDocumentCommand{\xlnx}{}{\Lambda}
\NewDocumentCommand{\sigmaR}{}{\bm\sigma}
\NewDocumentCommand{\sigmaIID}{O{}}{\bm\sigma^{*#1}}
\NewDocumentCommand{\sigmaNIS}{}{\hat{\bm\sigma}}
\NewDocumentCommand{\sigmaG}{}{\bm\sigma}
\NewDocumentCommand{\sigmaRG}{O{}}{\bm\sigma_{\mathrm{g}#1}}
\NewDocumentCommand{\sigmaRgG}{O{}}{\bm\sigma_{\Gamma #1}}
\NewDocumentCommand{\sigmaP}{}{\check\sigma}
\NewDocumentCommand{\sigmaPR}{}{\check{\bm\sigma}}
\NewDocumentCommand{\tauG}{O{}}{\tau_{\mathrm{g}#1}}
\NewDocumentCommand{\tauR}{O{} O{}}{\bm\tau}
\NewDocumentCommand{\tauTSa}{O{} O{}}{\bm\tau_{\circ #1}^{*#2}}
\NewDocumentCommand{\tauTSM}{O{} O{}}{\bm\tau_{#1}^{*#2}}
\NewDocumentCommand{\bmone}{}{\mathbbm 1}
\NewDocumentCommand{\ZG}{O{}}{Z_{\mathrm{g}#1}}
\NewDocumentCommand{\ZM}{O{}}{\overline Z_{\mathrm{m}#1}}
\NewDocumentCommand{\ZF}{O{}}{Z_{\mathrm{f}#1}}
\NewDocumentCommand{\ZV}{O{}}{Z_{\mathrm{v}#1}}
\NewDocumentCommand{\ZFa}{O{}}{\overline Z_{\mathrm{f}#1}}
\NewDocumentCommand{\ZFabu}{}{\xi}
\NewDocumentCommand{\ZFM}{O{}}{Z_{\mathrm{fm}#1}}
\NewDocumentCommand{\phiG}{O{}}{\phi_{\mathrm{g}#1}}

\NewDocumentCommand{\phiM}{O{}}{\overline\phi_{\mathrm{m}#1}}
\NewDocumentCommand{\phiTSM}{O{}O{}}{\phi_{#1}^{* #2}}
\NewDocumentCommand{\phiTSMe}{O{}O{}}{\bar\phi_{#1}^{* #2}}
\NewDocumentCommand{\phiTSMIID}{O{}O{}}{\phi_{\mathrm{m}^*#1}^{* #2}}
\NewDocumentCommand{\phiTSMNIS}{O{}O{}}{\phi_{\hat{\mathrm{m}}#1}^{* #2}}
\NewDocumentCommand{\phiTSYM}{O{}O{}}{\phi_{#1}^{\star #2}}
\NewDocumentCommand{\phiTSYMR}{O{}O{}}{\bm\phi_{#1}^{\star #2}}
\NewDocumentCommand{\phiTSYMS}{O{}O{}}{\phi_{\mathrm{m}^\circ#1}^{\star #2}}
\NewDocumentCommand{\phiTSP}{O{}O{}}{\phi_{\mathrm{p}#1}^{* #2}}
\NewDocumentCommand{\phiTSPIID}{O{}O{}}{\phi_{\mathrm{p}^*#1}^{* #2}}
\NewDocumentCommand{\phiTSPNIS}{O{}O{}}{\phi_{\hat{\mathrm{p}}#1}^{* #2}}
\NewDocumentCommand{\phia}{O{}}{\phi_{\mathrm{a}#1}}
\NewDocumentCommand{\phiq}{O{}}{\phi_{\mathrm{q}#1}}
\NewDocumentCommand{\phiqbl}{O{}}{\phi_{\mathrm{q}\downarrow #1}}
\NewDocumentCommand{\phiqbu}{O{}}{\phi_{\mathrm{q}\uparrow #1}}
\NewDocumentCommand{\phiI}{O{}}{\phi^{\lrarr #1}}
\NewDocumentCommand{\psiG}{O{}}{\psi_{\mathrm{g}#1}}
\NewDocumentCommand{\psiM}{O{}}{\overline\psi_{\mathrm{m}#1}}
\NewDocumentCommand{\psiR}{}{\bm\psi}
\NewDocumentCommand{\psiRa}{O{}}{\bm\psi_{\circ #1}}
\NewDocumentCommand{\psiTS}{}{\bm\psi^*}
\NewDocumentCommand{\psiae}{O{}}{\overline\psi_{\circ #1}}
\NewDocumentCommand{\psibl}{}{\psi_{\downarrow}}
\NewDocumentCommand{\psibu}{}{\psi_{\uparrow}}
\NewDocumentCommand{\psiI}{O{}}{\psi^{\leftrightarrow #1}}
\NewDocumentCommand{\psiItilde}{O{}}{\tilde\psi^{\leftrightarrow #1}}
\NewDocumentCommand{\psiIR}{O{}}{\bm\psi^{\leftrightarrow #1}}
\NewDocumentCommand{\psiIRa}{O{} O{}}{\bm\psi_{\circ #1}^{\leftrightarrow #2}}
\NewDocumentCommand{\psiTSM}{O{} O{}}{\bm\psi_{#1}^{* #2}}
\NewDocumentCommand{\psiTSa}{O{} O{}}{\bm\psi_{\circ#1}^{* #2}}
\NewDocumentCommand{\psiTSYa}{O{} O{}}{\bm\psi_{\circ #1}^{\star #2}}
\NewDocumentCommand{\psiTSYM}{O{} O{}}{\bm\psi_{#1}^{\star #2}}
\NewDocumentCommand{\psiITSa}{O{} O{}}{\bm\psi_{\circ #1}^{*\leftrightarrow #2}}
\NewDocumentCommand{\psiITS}{O{} O{}}{\bm\psi_{#1}^{*\leftrightarrow #2}}
\NewDocumentCommand{\psiITSM}{O{} O{}}{\bm\psi_{#1}^{*\leftrightarrow #2}}
\NewDocumentCommand{\psiITSMtilde}{O{} O{}}{\tilde{\bm\psi}_{\mathrm{m}#1}^{*\leftrightarrow #2}}
\NewDocumentCommand{\psiITSYa}{O{} O{}}{\bm\psi_{\circ #1}^{\star\leftrightarrow #2}}
\NewDocumentCommand{\psiITSYM}{O{} O{}}{\bm\psi_{\mathrm{m}#1}^{\star\leftrightarrow #2}}
\NewDocumentCommand{\psiWgG}{O{}}{\overline\psi_{\mathrm{w}|\mathrm{g}#1}}
\NewDocumentCommand{\lawG}{O{}}{\mu_{\mathrm{g}#1}}
\NewDocumentCommand{\GR}{}{\bm G}
\NewDocumentCommand{\GRM}{O{}}{\bm G_{#1}}
\NewDocumentCommand{\GTSM}{O{} O{}}{\bm G_{#1}^{*#2}}
\NewDocumentCommand{\GTSYM}{O{} O{}}{\bm G_{#1}^{\star #2}}
\NewDocumentCommand{\GTS}{}{\bm G^*}
\NewDocumentCommand{\GTSNIS}{}{\hat{\bm G}}
\NewDocumentCommand{\wR}{}{\bm w}
\NewDocumentCommand{\wRa}{O{}}{\bm w_{\circ #1}}
\NewDocumentCommand{\wTSM}{O{} O{}}{\bm w_{#1}^{* #2}}
\NewDocumentCommand{\wTS}{O{}}{\bm w^{* #1}}
\NewDocumentCommand{\wTSa}{O{} O{}}{\bm w_{\circ #1}^{* #2}}
\NewDocumentCommand{\wTSYa}{O{} O{}}{\bm w_{\circ #1}^{\star #2}}
\NewDocumentCommand{\wTSYM}{O{} O{}}{\bm w_{#1}^{\star #2}}
\NewDocumentCommand{\vTSM}{O{} O{}}{\bm v_{#1}^{* #2}}
\NewDocumentCommand{\vTSa}{O{} O{}}{\bm v_{\circ #1}^{* #2}}
\NewDocumentCommand{\vTSYa}{O{} O{}}{\bm v_{\circ #1}^{\star #2}}
\NewDocumentCommand{\vTSYM}{O{} O{}}{\bm v_{#1}^{\star #2}}
\NewDocumentCommand{\vR}{}{\bm v}
\NewDocumentCommand{\vRa}{O{}}{\bm v_{\circ #1}}
\NewDocumentCommand{\vTS}{}{\bm v^*}
\NewDocumentCommand{\domPsi}{}{\mathcal D_{\Psi}}
\NewDocumentCommand{\domPsiI}{O{}}{\mathcal D_{\Psi}^{\lrarr #1}}
\NewDocumentCommand{\domG}{}{\mathcal G}
\NewDocumentCommand{\domC}{O{}}{\mathcal D_{\Gamma #1}}
\NewDocumentCommand{\gammaa}{}{\overline\gamma}
\NewDocumentCommand{\gammaaG}{O{}}{\overline\gamma_{\mrmg #1}}
\NewDocumentCommand{\gammaaR}{}{\overline{\bm\gamma}}

\NewDocumentCommand{\gammaN}{O{}}{\gamma_{\mathrm{n}#1}}
\NewDocumentCommand{\gammaR}{}{\bm\gamma}
\NewDocumentCommand{\gammaIID}{O{}}{\bm\gamma^{*#1}}
\NewDocumentCommand{\gammaNIS}{}{\hat{\bm\gamma}}
\NewDocumentCommand{\alphaM}{O{}}{\alpha_{\mathrm{m}#1}}
\NewDocumentCommand{\alphaR}{}{\bm\alpha}
\NewDocumentCommand{\lawpsiS}{}{p}
\NewDocumentCommand{\lawpsi}{O{}}{\mu_{\Psi #1}}
\NewDocumentCommand{\setP}{}{\mathcal U}
\NewDocumentCommand{\setPtilde}{}{\tilde{\mathcal U}}
\NewDocumentCommand{\setPR}{}{\bm{\mathcal U}}

\NewDocumentCommand{\setPRT}{O{}}{\bm{\mathcal U}_{\mathrm{t}#1}}
\NewDocumentCommand{\thetaP}{}{\theta}
\NewDocumentCommand{\thetaPR}{}{\bm\theta}
\NewDocumentCommand{\ThetaP}{O{}}{\Theta^{\darr #1}}

\NewDocumentCommand{\ThetaPbl}{O{}}{\Theta_{\downarrow #1}}

\NewDocumentCommand{\indPRTa}{O{}}{\check{\bm u}_{\mathrm{t}\circ #1}}
\NewDocumentCommand{\indPRT}{O{}}{\check{\bm u}_{\mathrm{t}#1}}
\NewDocumentCommand{\indPR}{}{\check{\bm u}}
\NewDocumentCommand{\sucPT}{O{}}{\check p_{\mathrm{t}#1}}
\NewDocumentCommand{\sucD}{O{}}{p_{\mathrm{d}#1}}
\NewDocumentCommand{\tI}{O{}}{t^{\leftrightarrow #1}}
\NewDocumentCommand{\mI}{O{}}{m^{\leftrightarrow #1}}
\NewDocumentCommand{\mItilde}{O{}}{\tilde m^{\leftrightarrow #1}}
\NewDocumentCommand{\mIR}{O{}}{\bm m^{\leftrightarrow #1}}
\NewDocumentCommand{\mIRtilde}{O{}}{\tilde{\bm m}^{\leftrightarrow #1}}
\NewDocumentCommand{\mIRa}{O{}O{}}{\bm m_{\circ #1}^{\leftrightarrow #2}}
\NewDocumentCommand{\facsI}{O{}}{{\mathcal A}^{\leftrightarrow #1}}
\NewDocumentCommand{\facsV}{O{}}{{\mathcal A}_{\mathrm{v}#1}}
\NewDocumentCommand{\facsVRgD}{O{}}{\bm{\mathcal A}_{\mathrm{d}#1}}
\NewDocumentCommand{\wiresV}{O{}}{{\mathcal H}_{\mathrm{v}#1}}
\NewDocumentCommand{\lawYgC}{O{}}{\mu_{\mathrm{T}|\Gamma #1}}
\DeclareMathOperator{\unif}{u}
\DeclareMathOperator{\Po}{Po}
\DeclareMathOperator{\Bin}{Bin}
\NewDocumentCommand{\opm}{O{}}{\mu_{\bullet #1}}
\NewDocumentCommand{\mR}{}{\bm m}
\NewDocumentCommand{\mbu}{O{}}{m_{\uparrow #1}}
\NewDocumentCommand{\me}{}{\overline m}
\NewDocumentCommand{\hR}{}{\bm h}
\NewDocumentCommand{\hTSM}{O{} O{}}{\bm h_{#1}^{* #2}}
\NewDocumentCommand{\dR}{}{\bm d}
\NewDocumentCommand{\dcond}{O{}}{d_{\mathrm{cond}#1}}
\NewDocumentCommand{\degae}{}{\bar d}
\NewDocumentCommand{\degaR}{}{\bar{\bm d}}
\NewDocumentCommand{\degaIR}{}{\bar{\bm d}^\lrarr}
\NewDocumentCommand{\degabu}{}{d_{\uparrow}}
\NewDocumentCommand{\degF}{O{}}{d_{\mathrm{f}#1}}
\NewDocumentCommand{\degFR}{O{}}{\bm d^*_{\mathrm{f}#1}}
\NewDocumentCommand{\degH}{O{}}{d_{\mathrm{w}#1}}
\NewDocumentCommand{\epsm}{}{\varepsilon_{\mathrm{m}}}
\NewDocumentCommand{\deltam}{}{\delta_{\mathrm{m}}}
\NewDocumentCommand{\reals}{}{\mathbb R}
\NewDocumentCommand{\ints}{}{\mathbb Z}
\NewDocumentCommand{\expe}{}{\mathbb E}
\DeclareMathOperator{\Var}{Var}
\NewDocumentCommand{\prob}{}{\mathbb P}
\NewDocumentCommand{\eps}{}{\varepsilon}
\NewDocumentCommand{\mutinf}{}{I}
\NewDocumentCommand{\DKL}{}{D}
\NewDocumentCommand{\nablaI}{}{\nabla}
\NewDocumentCommand{\nablaIbl}{O{}}{\nabla_{\downarrow #1}}
\NewDocumentCommand{\bethe}{}{B}
\NewDocumentCommand{\bethebu}{O{}}{B_{\uparrow #1}}
\NewDocumentCommand{\ballSC}{O{}}{\mathcal S_{\circ #1}}

\NewDocumentCommand{\distW}{}{\mathrm{d}_{\mathrm{w}}}
\NewDocumentCommand{\distG}{}{\mathrm{d}_{\mathrm{g}}}
\NewDocumentCommand{\setES}{O{}}{\mathcal E_{\mathrm{s}#1}}
\NewDocumentCommand{\epsES}{O{}}{\eps_{\mathrm{s}#1}}
\NewDocumentCommand{\ellES}{O{}}{\ell_{\mathrm{s}#1}}
\NewDocumentCommand{\piR}{O{}}{\bm\pi}
\NewDocumentCommand{\piG}{O{}}{\pi_{\mathrm{g}#1}}
\NewDocumentCommand{\piGC}{O{}}{\check\pi_{\mathrm{g}#1}}
\NewDocumentCommand{\piGR}{O{}}{\hat\pi_{\mathrm{g}#1}}

\title{Mutual Information, Information-Theoretic Thresholds and the Condensation Phenomenon at Positive Temperature}
\author{Konstantinos Panagiotou\thanks{The research leading to these results has received funding from the European Research Council, ERC Grant Agreement 772606–PTRCSP.} \and Matija Pasch\footnotemark[1]}

\begin{document}

\maketitle

\begin{abstract}
There is a vast body of recent literature on
the reliability of communication through noisy channels,
the recovery of community structures in the stochastic block model,
the limiting behavior of the free entropy in spin glasses
and the solution space structure of constraint satisfaction problems.
At first glance, these topics ranging across several disciplines might seem unrelated. However, taking a closer look, structural similarities can be easily identified.

Factor graphs exploit these similarities to model the aforementioned objects and concepts in a unified manner.
In this contribution we discuss the asymptotic average case behavior of several quantities, where the average is taken over sparse Erdős–Rényi type (hyper-) graphs with positive weights, under certain assumptions.
For one, we establish the limit of the mutual information, which is used in coding theory to measure the reliability of communication.
We also determine the limit of the relative entropy, which can be used to decide if weak recovery is possible in the stochastic block model.
Further, we prove the conjectured limit of the quenched free entropy over the planted ensemble, which we use to obtain the preceding limits.
Finally, we describe the asymptotic behavior of the quenched free entropy (over the null model) in terms of the limiting relative entropy.
\end{abstract}
\newpage
\section{Introduction}
Consider the following two prototypical experiments for a given number $q>1$ of colors, a low temperature $T\in\reals_{>0}$ and a large number $n>q$ of vertices.
In the \emph{null model} and in each step we draw an edge $(i,j)\in[n]^2$ between two vertices $i,j\in[n]=\{1,\dots,n\}$ uniformly at random.
In the \emph{teacher-student model} we (the teacher) first draw a coloring $\sigma\in[q]^n$ of the vertices uniformly at random, and then iteratively draw edges $(i,j)\in[n]^2$ 
proportional to their weight $w(i,j)=\exp(-\beta\bmone\{\sigma(i)=\sigma(j)\})$, i.e.~with probability $p(i,j)=w(i,j)/\ZFa[,\sigma]$, where $\ZFa[,\sigma]=\sum_{i,j}w(i,j)$ and $\beta=1/T$ is the so-called \emph{inverse temperature}.
In words, we prefer dichromatic edges to monochromatic edges with respect to the \emph{ground truth} $\sigma$.
Notice that this preference is quantified, i.e.~the penalty $1-e^{-\beta}$ is a function of the inverse temperature.

\vspace{1mm}
\noindent{\bf Community Detection.}
Now, assume that we (the student) are shown a multi-graph $(i_a,j_a)_{a\in[m]}$ by the teacher obtained from one of the two experiments after $m\ge 0$ steps, but we are not told from \emph{which} experiment. Is it then possible to make an educated guess?
For $m=0$ the graph is empty, so the answer certainly is no.
For very, very large $m$ the frequencies $\hat p(i,j)=\frac{1}{m}|\{a\in[m]:(i_a,j_a)=(i,j)\}|$ are jointly very close to the uniform distribution in the null model and close to $p$ in the teacher-student model (with high probability), so the answer is yes, unless $\sigma$ is monochromatic.

This decision problem, which is a version of the well-known \emph{stochastic block model}, has important applications and it has been extensively studied, see~\cite{abbe2017} for an overview. From today's viewpoint we know that the problem undergoes a \emph{sharp phase transition} regarding our ability to tell the two settings apart: in~\cite{coja2018} it was shown that there is a `magic' ratio $d^*\in\reals_{\ge 0}$ such that when the average degree $2m/n$ approaches $d\le d^*$, the answer is (typically) no, and the answer is yes if $d>d^*$.
The quantity $d^*$ is known as the \emph{information-theoretic threshold for weak recovery}.

\vspace{1mm}
\noindent{\bf Spin Glasses and Constraint Satisfaction Problems.}
The previous example can also be encountered in various other settings in different forms and in disguise. To wit, in the context of spin glass theory it is a version of the infamous Potts model; in the setting of constraint satisfaction problems it is a version of graph coloring.
Let $\bm G\in([n]^2)^m$ be the null model and $\bm G^*(\sigma)\in([n]^2)^m$ be the teacher-student model on $m$ edges and for a given ground truth $\sigma$ above.
For $G=(i_a,j_a)_{a\in[m]}$ let $H_G(\sigma)=\sum_a\bmone\{\sigma(i_a)=\sigma(j_a)\}$ denote the \emph{Hamiltonian}.
As suggested by Observation \ref{obs_psiexpenullprod} below, the law of $\bm G^*(\sigma)$ is also given by
\begin{align*}
\prob[\bm G^*(\sigma)=G]=\frac{\exp(-\beta H_\sigma(G))}{\expe[\exp(-\beta H_\sigma(\GR))]}.
\end{align*}
The \emph{partition function} $Z(G)=\sum_\sigma\exp(-\beta H_G(\sigma))>0$ and the \emph{free entropy (density)} $\phi(G)=n^{-1}\ln(Z(G))$ play a central role, e.g.~$Z(G)$ is the number of valid colorings of $G$ for $\beta\rarr\infty$ and hence the free entropy embodies the exponential rate at which we do (not) encounter valid colorings.

The fundamental question is, of course, how $Z(G)$ and $\phi(G)$ behave \emph{typically}, that is typical with respect to $\GR$.
Assume that the average degree $2m/n$ tends to $d\in\reals_{\ge 0}$ for growing $n$, as before.
It turns out that $\phi(\GR)$ concentrates around the \emph{quenched free entropy} $\phiq(m)=\expe[\phi(\GR)]$, see also Proposition \ref{proposition_phi_concon} below, and that the limit $\phiq[,\infty](d)=\lim_{n\rarr\infty}\phiq(m)$ exists.
So, $Z(\GR)$ typically grows/decays exponentially and $\phiq[,\infty](d)$ is a (logarithmic) first order approximation for $Z(\GR)$.

This gives a first answer to our main question: the asymptotics of $Z(G)$ are governed by $\phiq[,\infty](d)$.
But there is a catch:
$\phiq(m)$ is very hard to control.
The go-to approximation is the first moment bound provided by Jensen's inequality, the easy to handle \emph{annealed free entropy} $\phia(m)=n^{-1}\ln(\expe[Z(\GR)])\ge\phiq(m)$.
But this begs another question -- when is
$\phia(m)$ a good estimate?
In particular, when does $\phia[,\infty](d)=\lim_{n\rarr\infty}\phia(m)$ coincide with $\phiq[,\infty](d)$?

This question is answered in \cite[Theorem 2.7]{coja2018} for a large class of models. It is shown that there is a \emph{condensation threshold}, meaning $\phiq[,\infty](d)=\phia[,\infty](d)$ for $d\le d^*$ and $\phiq[,\infty](d)<\phia[,\infty](d)$ for $d>d^*$.
The threshold was obtained by 
verifying the physics prediction that
$\lim_{n\rarr\infty}\phiq^*(m)=\bethebu(d)$,
where $\phiq^*(m)=\expe[\phi(\GTS(\sigmaIID))]$ is the planted model quenched free entropy
over the uniformly random ground truth $\sigmaIID\in[q]^n$ and
$\bethebu(d)$ is the \emph{maximum Bethe free entropy}.
It was further shown that $\phiq^*(m)$ is also subject to a phase transition at $d^*$ in that $\phiq[,\infty](d)=\phia[,\infty](d)=\bethebu(d)$ for $d\le d^*$ and $\phiq[,\infty](d)<\phia[,\infty](d)<\bethebu(d)$ for $d>d^*$.
Moreover, this threshold is equal to the information-theoretic threshold for weak recovery.

\vspace{1mm}
\noindent{\bf Beyond the Example \& Our Contribution.}
The previous example generalizes to a large class of problems covering the general stochastic block model, $k$-spin models from physics, satisfiability problems from computer science and to noisy channels from coding theory as discussed below.
The threshold $d^*$ is further crucial to the design and choice of corresponding algorithms, e.g.~solvers, in that it determines the onset of long-range correlations, see for example Theorem 1.4 in \cite{coja2018b}.

Most previous research tackles specific problems and derives results case by case.
Contributions like \cite{coja2018,coja2021} that work towards a general theory usually rely on restrictive or hard to check assumptions, e.g. uniform ground truths, permutation invariance and convexity assumptions, or solving optimization problems over infinite dimensional spaces.
Our contribution is threefold.
\begin{bulletlist}
\item We need significantly simpler assumptions compared to previous works. To wit, we allow bounded (as opposed to finite) supports for the weight functions. More crucially, we entirely drop symmetry and concavity assumptions (denoted by {\bf SYM} and {\bf BAL} in~\cite{coja2018}); the strength of our approach lies in deriving all results for any maximizer, rather than assuming that the uniform distribution is a maximizer and restricting to this specific choice. Our approach allows to consider biased community structures, any discrete memoryless source for noisy channels, and~bounded $k$-spin models directly.
\item Our theorems (as those of \cite{coja2018}) rely on natural but quite technical conditions, and checking them is usually quite elaborate.
This requires us to solve an infinite dimensional optimization problem and a $(q-1)$-dimensional optimization problem (cf.~{\bf POS} and ~{\bf BAL} in \cite{coja2018}).
We present a large class of models that covers essentially \emph{all known working examples} and for which the former assumption holds, so only the simpler latter assumption needs to be verified.
\item Finally, we obtain new results in this context.
We bound the order of convergence to the limiting quantities, with the error bounds being uniform over the choice of the model, since a variation of the parameters is common practice, e.g.~of the weights through the inverse temperature $\beta$, of the average degree $d$ as above and of the ground truth distribution as explained below.
Further, next to deriving the threshold location, we also quantify the behavior as follows.
We show that the relative entropy converges to the Jensen gap $\bethebu(d)-\phia[,\infty](d)$, and bound the gap $\phia[,\infty](d)-\phiq[,\infty](d)$.
Further details can be found in Section~\ref{information_theoretic_threshold} and Section \ref{implications_extensions_related_work}.
\end{bulletlist}
\subsection{Factor Graphs}\label{factor_graphs}
A \emph{(factor) graph} $G=(v_a,\psi_a)_{a\in\mcla}$ is given by the variable nodes $[n]$, the factor nodes $\mcla$, the arity $k_a>0$ of each factor $a\in\mcla$, its ordered neighborhood $v_a\in[n]^{k_a}$ and its weight function $\psi_a:[q]^{k_a}\rightarrow\reals_{\ge 0}$ over the colors $[q]$.
Notice that $G$ can be understood as a labeled bipartite graph with labeled edges.
For a law $\gamma$ on $[q]$ let $\sigmaIID=\sigmaIID_{\gamma,n}\in[q]^n$ be i.i.d.~with law $\gamma$.
Then, writing $\sigma_{v}=(\sigma(v_h))_{h\in[k]}\in[q]^k$ for $v\in[n]^k$, the \emph{weight} $\psi_G(\sigma)$ and the \emph{Gibbs measure} $\mu_{\gamma,G}(\sigma)$ of an assignment $\sigma\in[q]^n$ to the variables are given by
\begin{align}\label{weightG}
\psi_G(\sigma)=\prod_{a\in\mcla}\psi_a(\sigma_{v_a}),\,~~
\mu_{\gamma,G}(\sigma)=\frac{\prob[\sigmaIID=\sigma]\psi_G(\sigma)}{Z_\gamma(G)},\,~~
Z_\gamma(G)=\expe[\psi_G(\sigmaIID)],
\end{align}
where $\mu_{\gamma,G}$ and $\phi_{\gamma}(G)=\frac{1}{n}\ln(Z_\gamma(G))$ are only defined for $Z_\gamma(G)>0$.
\begin{remark}\label{remark_external_fields}
For the uniform distribution $\unif([q])$ on $[q]$,
the partition function of $G$ is
$q^nZ_{\unif([q])}(G)=\sum_{\sigma\in[q]^n}\psi_G(\sigma)$, and
the free entropy (density) is $\ln(q)+\phi_{\unif([q])}(G)$.
As for $\mu_{\gamma,G}$, we identify a law $\gamma$ on $[q]$ with its probability mass function $\gamma:[q]\rightarrow\reals_{\ge 0}$ and notice that $Z_{\gamma}(G)=\sum_{\sigma}\prod_{i\in[n]}\gamma(\sigma(i))\psi_G(\sigma)$ is the partition function of the graph $G'$ obtained from $G$ by attaching the unary factors $(i,\gamma)$ to the variables $i\in[n]$, known as external fields.
In this sense $Z_{\gamma}(G)$ is the \emph{partition function} and $\phi_{\gamma}(G)$ is the \emph{free entropy} of $G$.
\end{remark}
\subsection{Random Factor Graphs}\label{random_factor_graphs}
As in the introduction, we first consider the null model given by
\begin{bulletlist}
\item the number $q\in\ints_{\ge 1}$ of colors,
\item the arity or factor degree $k\in\ints_{\ge 1}$,
\item the random weight $\psiR:[q]^k\rightarrow[\psibl,\psibu]$ with law $\lawpsiS$
and bounds $\psibl\in(0,1/q)$, $\psibu=\psibl^{-1}$,
\item the ground truth distribution $\gamma^*$ on $[q]$ with $\gamma^*\ge\psibl$ (componentwise),
\item the expected variable degree $d\in[0,\degabu]$ with bound $\degabu\in\reals_{>0}$, and
\item a number $n\in\ints_{\ge 1}$ of variables.
\end{bulletlist}
For a number $m\in\ints_{\ge 0}$ of factors let $\GR_{n,m,p}\dequal(\unif([n]^k)\otimes p)^{\otimes m}$, where $p_1\otimes p_2$ denotes the product measure of $p_1$ and $p_2$, $p^{\otimes n}$ is the $n$-fold product of $p$, $\bm x\dequal\mu$ means that $\bm x$ has the law $\mu$, and $\bm x\dequal\bm y$ that $\bm x$ and $\bm y$ have the same law.
The teacher-student model $\GTS_{n,m,p}(\sigma)$ for a fixed ground truth $\sigma\in[q]^n$ is given by the Radon-Nikodym derivative $G\mapsto\psi_G(\sigma)/\expe[\psi_{\GR}(\sigma)]$ with respect to $\GR_m$.

Let $(\sigmaIID_{\gamma^*,n},\mR_{d,n})\dequal\gamma^{*\otimes n}\otimes\Po(dn/k)$ be the random ground truth (as in the initial example) and a random number of factors such that the average degree of a variable is $d=km/n$, where $\Po(\lambda)$ denotes the Poisson distribution with parameter $\lambda$.
Notice that there is no conceptual motivation to consider the laws $\gamma^{*\otimes n}$ and $\Po(dn/k)$, they are chosen for technical reasons.
Details can be found in Section \ref{implications_extensions_related_work}.
We suppress dependencies that are clear from the given context.

\vspace{1mm}
\noindent{\bf Planting in the Erdős–Rényi Model.}
With the definition of the random graphs in place, we want to embed our initial example into the general framework.
However, notice that the specific definition of the teacher-student model, given by i.i.d.~choices, does not match the general definition, given by the Radon-Nikodym derivative. In order to resolve this issue we need some notation.

Let $\mclp([q])\setle\reals^q$ be the set of probability measures on $[q]$. For an assignment $\sigma\in[q]^n$ let $\gammaN[,\sigma]\in\mclp([q])$ be the relative color frequencies of $\sigma$, i.e.~$\gammaN[,\sigma](z)=\frac{1}{n}|\{i\in[n]:\sigma(i)=z\}|$ for $z\in[q]$.
Finally, for $(\psiR,\sigmaR)\dequal p\otimes\gamma^{\otimes k}$ let
\begin{align}\label{ZFaDef}
\ZFa[,p]:\mclp([q])\rightarrow[\psibl,\psibu],\,
\gamma\mapsto\expe[\psiR(\sigmaR)].
\end{align}
\begin{observation}\label{obs_psiexpenullprod}
We have $\GTS_m(\sigma)\dequal(\vTS_{\sigma},\psiTS_{\sigma})^{\otimes m}$, where $(\vTS_{\sigma},\psiTS_{\sigma})$ is given by the Radon-Nikodym derivative $(v,\psi)\mapsto\psi(\sigma_v)/\ZFa(\gammaN[,\sigma])$ with respect to $\unif([n]^k)\otimes p$.
\end{observation}
\begin{proof}
Notice that $\psi_G$ factorizes, the wires-weight pairs of the null model are i.i.d.~and hence also the expectation factorizes with $\expe[\psi_{\GR(m)}(\sigma)]=\zeta_\sigma^m$,where
\begin{align*}
\zeta_\sigma=\sum_v\frac{1}{n^k}\expe[\psiR(\sigma_v)]
=\sum_\tau\expe[\psiR(\tau)]\frac{\sum_v\bmone\{\sigma_v=\tau\}}{n^k}
=\ZFa(\gammaN[,\sigma]).
\end{align*}
\end{proof}
\subsection{Examples}\label{examples}
Observation \ref{obs_psiexpenullprod} recovers our example for $k=2$, $\psiR(\sigma)=\exp(-\beta\bmone\{\sigma_1=\sigma_2\})$, $\sigma\in[q]^2$, and $\gamma^*=\unif([q])$.
Next, we introduce a slightly more sophisticated version of the stochastic block model and examples from the other disciplines.

\vspace{1mm}
\noindent{\bf Composing Stochastic Block Models.}
As before, we restrict to $k=2$ and $q>1$.
In the initial example we distinguished two cases -- the vertices belong to the same community, i.e.~$\sigma(i)=\sigma(j)$, or they don't.
Now, we compose two of these models as follows.
For this purpose let $\mclc[1]$, $\mclc[2]\setle[q]$
be a non-trivial partition of $[q]$ into two types of communities and $c_1^=,c_1^{\neq},c_2^=,c_2^{\neq},c_{\leftrightarrow}\in\reals_{>0}$ be such that $c_1^=\le c_1^{\neq}$, $c_2^=\le c_2^{\neq}$ and $c_1^{\neq},c_2^{\neq}\le c_\leftrightarrow$.
For $i\in[2]$ and $\sigma\in\mclc[i][2]$ let $\psiR(\sigma)=c_i^{=}$ if $\sigma(1)=\sigma(2)$ and $\psiR(\sigma)=c_i^{\neq}$ otherwise.
For $\sigma\in[q]^2\setminus(\mclc[1][2]\cup\mclc[2][2])$ let $\psiR(\sigma)=c_\leftrightarrow$.
In words, we always prefer to mix, in terms of communities and types.

\vspace{1mm}
\noindent{\bf The $\bm k$-spin model.}
Consider $q=2$ over spins $\{-1,1\}$ and an inverse temperature $\beta\in\reals_{>0}$.
Further, let $\bm J\in[-J,J]$, $J\in\reals_{>0}$, be such that $\bm J\dequal-\bm J$.
For $\sigma\in[q]^k$ we consider the weight $\bm\psi(\sigma)=\exp(-\beta\bm J\prod_h\sigma(h))$.
The null model equipped with these weights is the $k$-spin model.

\vspace{1mm}
\noindent{\bf NAE-SAT.}
Consider $q=2$ over $\{0,1\}$ and let $\eps\in(0,1)$.
Further, let $\bm x\in\{0,1\}^k$ be uniform, $\bm x'=(1-\bm x(h))_{h\in[k]}$, $\bm\psi(x)=\eps$ for $x\in\{\bm x,\bm x'\}$ and $\bm\psi(x)=1$ for $x\in\{0,1\}^k\setminus\{\bm x,\bm x'\}$.
The null model equipped with these weights is the not-all-equal $k$-satisfiability problem with soft constraints.
The solvers mentioned in the introduction try to find solutions $x\in\mcls$ (ground states, to be precise, since we consider soft constraints) in the solution space $\mcls=\{x'\in[q]^n:\psi_G(x)=1\}$. Clearly, the structure of $\mcls$ (e.g.~connectivity) is crucial for the efficiency of such solvers.

\vspace{1mm}
\noindent{\bf Graphical Channels.}
Consider the following noisy channel from \cite{abbe2015}.
Fix a number $q'\in\ints_{>0}$ of possible output values.
For a given input $y\in[q]^k$ to the channel fix a fully supported (conditional) distribution $\nu_y\in\mclp([q'])$ on the possible outputs for the received input $y$, to model the noisy channel.
For an \emph{input message} $x\in[q]^n$ and a (multi-) graph $v\in([n]^k)^m$ let $y_{v,x}=(x_{v_a})_{a\in[m]}$ be the corresponding \emph{codeword}.
For a given codeword $y\in([q]^k)^m$ let the output $\bm z(y)\in[q']^m$ of the channel have the law $\bigotimes_{a\in[m]}\nu_{y_a}$, i.e.~the $m$ elements of $y$ are transmitted independently through identical channels (respectively sequentially through the same memoryless channel).
We consider communication through this noisy channel using a random code, given by the uniformly random graph $\vR\in([n]^k)^m$.
The \emph{graphical channel} with graph $v$ and kernel $(\nu_y)_y$ is the map $[q]^n\times[q']^m\rarr[0,1]$, $(x,z)\mapsto\prob[\bm z(y_{v,x})=z]$.

It is certainly not intuitive to consider transmissions $[q]^k\rarr[q']$, since we usually think of the communicated signal to be preserved or altered -- but preserving the signal is impossible in this model!
Hence, let us motivate this construction with the following, well-known example.
In the binary case $q=2$ over $\{0,1\}$ and given that all coordinates of $\vR_a$ are distinct for all $a\in[m]$\footnote{This amounts to considering a simple $k$-uniform hypergraph with labeled hyperedges, cf.~Section \ref{implications_extensions_related_work}.}, let $\bm M=(\bmone\{\exists h\in[k]\,\vR_{a,h}=i\})_{a\in[m],i\in[n]}\in\{0,1\}^{m\times n}$.
For an input message $x\in\{0,1\}^n$ let $\bm y(x)=y_{\bm v,x}$, and let $\bm y^\circ(x)=\bm Mx=(\sum_{h=1}^k\bm y_{a,h}(x))_{a\in[m]}$ be the codeword, given by the sum of the components of $\bm y_a(x)$ modulo $2$.
Now, we consider the communication through a binary symmetric channel, i.e.~for a fixed error probability $\eta\in(0,1/2)$ and for an input $y\in\{0,1\}$ we let $\tilde\nu_y\in\mclp(\{0,1\})$ be given by $\tilde\nu_y(y)=1-\eta$ and $\tilde\nu_y(1-y)=\eta$. As above, the channel output for a given codeword $y^\circ\in\{0,1\}^m$ is given by independent transmissions $\bm z^\circ(y^\circ)\dequal\bigotimes_{a\in[m]}\tilde\nu_{y^\circ_a}$.
This is an LDGM code with codeword $\bm y^\circ(x)$ and a binary symmetric channel.
Now, we split off the last deterministic encoding step, the sum modulo 2, and consider it as part of the transmission as follows.
The codeword for an input $x\in[q]^n$ is $\bm y(x)$.
For an input $y\in\{0,1\}^k$ and using $y^\circ=\sum_hy_h\in\{0,1\}$ let the channel output $\nu_y\in\mclp(\{0,1\})$ be given by $\nu_y(y^\circ)=1-\eta$ and $\nu_y(1-y^\circ)=\eta$.
The output is $\bm z(y)\dequal\bigotimes_{a\in[m]}\nu_{y_a}$ as before.
Clearly, this graphical channel and the LDGM encoded communication through the binary symmetric channel are equivalent, and only differ in the modularization of the communication process.

In our context, it is irrelevant if we consider $\bm y^\circ(x)$ or $\bm y(x)$ to be the codeword in the example, since this does not change the channel capacity\footnote{This refers to the capacity of the graphical channel with inputs in $[q]^n$ and outputs in $[q']^m$.}, and it does not affect the conditional mutual information
$\mutinf(\bm x,\bm z(y_{\bm v,\bm x})|\bm v)=\sum_vn^{-km}\mutinf(\bm x,\bm z(y_{v,\bm x}))$ of the input and the output given the code,
\begin{align*}
\mutinf(\bm x,\bm z(y_{v,\bm x}))=\sum_{x,z}\prob[\bm x=x,\bm z(y_{v,x})=z]\ln\left(\frac{\prob[\bm x=x,\bm z(y_{v,x})=z]}{\prob[\bm x=x]\sum_{x'}\prob[\bm x=x',\bm z(y_{v,x'})=z]}\right),
\end{align*}
for any choice of $\bm x\in[q]^n$\footnote{For the LDGM code example it is crucial that the code is determined by $\vR$ since the last encoding step, the sums modulo 2, to obtain $\bm y^\circ(x)$ is deterministic.}.
We consider a discrete memoryless source $\bm x\dequal\gamma^{*\otimes n}$ for $\gamma^*\in\mclp([q])$.

Let $c=\sup_{\alpha}\mutinf(\bm y^\circ_\alpha,\bm z^\circ(\bm y^\circ_\alpha))$ be the channel capacity for $(\nu_y)_y$, where $\bm y^\circ_\alpha\dequal\alpha$ for $\alpha\in\mclp([q]^k)$ and $\bm z^\circ(y)\dequal\nu_y$ for $y\in[q]^k$.
Recall from the proof of Observation \ref{obs_psiexpenullprod} that $y_{\bm v,x}\dequal(\gammaN[,x]^{\otimes k})^{\otimes m}$, and that $\gammaN[,\bm x]$ concentrates around $\gamma^*$ as $n$ tends to infinity.
This suggests that 
$\frac{1}{m}\mutinf(\bm x,\bm z(y_{\bm v,\bm x})|\bm v)$ can asymptotically only attain $c$ if
$\mutinf(\bm y^\circ_\alpha,\bm z^\circ(\bm y^\circ_\alpha))=c$ for $\alpha=\gamma^{*\otimes k}$.
Notice that this holds for the LDGM code with the binary symmetric channel in \cite{coja2018}, since the capacity of the channel is attained at $\unif([q])$, i.e.~consistent with the distribution of the discrete memoryless source.

Now, we finally build the connection to the factor graphs.
For $y\in[q]^k$ and $z\in[q']$ let $\psi_z(y)=\nu_y(z)/p^*(z)$ for any fully supported law $p^*$ on $[q']$ and let $p$ be given by $p(\psi_z)=p^*(z)$ (We assume that the $\psi_z$ are distinct and obtain equal $\psi_z$'s as a limiting case over the variation of $p^*$).
\begin{observation}\label{obs_grchmutinf}
We have $\expe[\psiR]\equiv 1$ and $\mutinf(\bm x,\bm z(y_{\bm v,\bm x})|\bm v)=\mutinf(\sigmaIID,\GTS_m(\sigmaIID))$.
\end{observation}
\begin{proof}
For $\sigma\in[q]^k$ we have
$\expe[\psiR(\sigma)]=\sum_{z}p^*(z)\psi_z(\sigma)=\sum_z\nu_{\sigma}(z)=1$, so
$\ZFa\equiv 1$.
Now, for $x\in[q]^n$, $v\in([n]^k)^m$, $z\in[q']^m$ with $\psi=(\psi_{z_a})_{a\in[m]}$ and $y=(x_{v_a})_{a\in[m]}$ we have
\begin{align*}
\prob[\GTS_m(x)=(v,\psi)]
=\frac{1}{n^{km}}\prod_{a\in[m]}\frac{p^*(z_a)\nu_{y_a}(z_a)}{p^*(z_a)}
=\prob[\bm v=v,\bm z(y_{\bm v,x})=z].
\end{align*}
This shows that $(\sigmaIID,\GTS_m(\sigmaIID))\dequal(\bm x,\bm v,\bm z(y_{\bm v,\bm x}))$ by identifying $z$ with $\psi_z$, so
\begin{align*}
\mutinf(\sigmaIID,\GTS_m(\sigmaIID))
=\mutinf(\bm x,(\bm v,\bm z(y_{\bm x,\bm v})))
=\mutinf(\bm x,\bm z(y_{\bm x,\bm v})|\bm v),
\end{align*}
since $\bm v$ is independent of $\bm x$.
\end{proof}
\section{Free Entropies, Divergence and the Mutual Information}
In this section, we present the main results.
First, we discuss the required assumptions and present a class of models that satisfy these assumptions.
Then we turn to the limit of the quenched free entropy over the teacher-student model, the relative entropy, the quenched free entropy over the null model and the mutual information.
Finally, we discuss implications, generalizations, open problems and related work.
The proofs of the results will be discussed in Section \ref{proof_strategy}.
\subsection{Model Assumptions}\label{assumptions}
Recall from Section \ref{random_factor_graphs} that we restrict to $p\in\mclp([\psibl,\psibu]^{[q]^k})$ and $\gamma^*\ge\psibl$, the boundedness assumptions for the weight functions (and the ground truth distribution).
However, we need two additional assumptions, which weaken {\bf BAL} and {\bf POS} in \cite{coja2018} respectively.
Let $\psi:[q]^k\rightarrow\reals_{\ge 0}$, $h\in[k]$,
$\gamma\in\mclp([q])^k$ and $\gamma'\in(\mclp([q])^k)^2$.
Further, let $\sigmaR\dequal\bigotimes_{h'\in[k]}\gamma_{h'}$,
and let $\sigmaR'\dequal\bigotimes_{h'\in[k]}\sigmaR'(h')$ be given by
$\sigmaR'(h)\dequal\gamma'_{1,h}$ and
$\sigmaR'(h')\dequal\gamma'_{2,h'}$ for $h'\neq h$.
Recall $\ZFa[,p]$ from Equation (\ref{ZFaDef}), let $t\in\reals_{\ge 0}$ and further
\begin{align}\label{ZF_def}
\ZFM(\psi,h,\gamma')=\expe\left[\psi(\sigmaR')\right],\quad
\ZF(\psi,\gamma)=\expe\left[\psi(\sigmaR)\right],\quad
\ZFabu_p=\sup_{\gamma}\ZFa[,p](\gamma),\quad
\xlnx(t)=t\ln(t).
\end{align}
Notice that $\ZFa[,p](\gamma)$ is the expectation of $\ZF(\psi,\gamma)$ over $\psiR$ for constant $\gamma$, and that $\ZF(\psi,\gamma)$ is $\ZFM(\psi,h,\gamma')$ for $\gamma'_1=\gamma'_2=\gamma$.
For the weakened {\bf POS} assumption we randomize the arguments to $\ZFM$ and $\ZF$ as follows.
Let $\mclp[][2]([q])=\mclp(\mclp([q]))$ be the laws on the color distributions $\mclp([q])\setle\reals^q$.
For a pair $\pi\in(\mclp[][2]([q]))^2$ of such laws let $\gammaR_{\pi}\dequal\pi_1^{\otimes k}\otimes\pi_2^{\otimes k}$ be the product of the corresponding i.i.d.~color distributions.
Let $(\psiR,\hR,\gammaR_\pi)\dequal\psiR\otimes\hR\otimes\gammaR_\pi$ be independent with
$\psiR\dequal p$, $\bm h\dequal\unif([k])$ and
\begin{align}\label{nablaI_def}
\nablaI(\pi)=\expe\left[\xlnx\left(\ZF(\psiR,\gammaR_{\pi,1})\right)
+(k-1)\xlnx\left(\ZF(\psiR,\gammaR_{\pi,2})\right)
-k\xlnx\left(\ZFM(\psiR,\hR,\gammaR_{\pi})\right)
\right].
\end{align}
Now, with $\gammaR_\pi\dequal\pi$ for $\pi\in\mclp[][2]([q])$, and for given $\gamma\in\mclp([q])$ we consider the infimum
\begin{align}\label{setpi_def}
\nablaIbl(p,\gamma)=\inf_{\pi\in\mclp[*,\gamma][2]([q])}\nablaI(\pi),\quad
\mclp[*,\gamma][2]([q])=\left\{\pi\in\mclp[][2]([q]):\expe[\gammaR_\pi]=\gamma\right\}.
\end{align}
Finally, we are able to state the four assumptions, given by
\begin{align}
\mfkA=\left\{(p,\gamma^*)\in\mclp\left([\psibl,\psibu]^{[q]^k}\right)\times\mclp([q]):\gamma^*\ge\psibl,\ZFa[,p](\gamma^*)=\ZFabu_p,\nablaIbl(p,\gamma^*)\ge 0\right\}.
\end{align}
So, the first two assumptions are the boundedness assumptions for $(p,\gamma^*)$.
The third assumption $\ZFa[,p](\gamma^*)=\ZFabu_p$ that $\gamma^*$ is a maximizer of $\ZFa[,p]$ is the weakened\footnote{The assumption {\bf BAL} in \cite{coja2018} additionally requires $\gamma^*=\unif([q])$ and $\ZFa$ to be concave.} {\bf BAL} assumption from \cite{coja2018}.
Recalling the proof of Observation \ref{obs_psiexpenullprod} and the concentration of $\gammaN[,\sigmaIID]$, this means that $\sigmaIID$ should (typically asymptotically) maximize the expected weight.
The last assumption $\nablaIbl(p,\gamma^*)\ge 0$ is the weakened {\bf POS} assumption from \cite{coja2018} (cf.~\cite{coja2021}).
Some intuition for this assumption is provided below, for the example of graphical channels.

\vspace{1mm}
\noindent{\bf Valid Models.}
Instead of verifying the assumptions for the examples in Section \ref{examples} case by case, we consider the following class of models.
Let $(\bm a_i,\bm b_i,\bm\Delta_i)\in\reals_{>0}\times\reals\times\reals^{[q]^k}$ be such that $\expe[\bm b_i^{\ell}|\bm a_i]\ge 0$ for $\ell\in\ints_{\ge 3}$, such that $|\bm b_i\bm\Delta_i|<1$, such that $\bm b_i$ and $\bm\Delta_i$ are conditionally independent given $\bm a_i$, and such that $\psibl\le\psiR_i\le\psibu$, where $\psiR_i(\sigma)=\bm a_i(1-\bm b_i\bm\Delta_i(\sigma))$ for $\sigma\in[q]^k$ and $i\in\{-1,1\}$.

The subscript $i\in\{-1,1\}$ distinguishes two types of models.
For $i=-1$ let $\bm b_{-1}$ be such that $\expe[\bm b_{-1}^{2\ell+1}|\bm a]=0$ for $\ell\in\ints_{>0}$ and $\bm\Delta_{-1}(\sigma)=\prod_h\bm f_{-1,h}(\sigma(h))$ for $\sigma\in[q]^k$, where $\bm f_{-1}=(\bm f_{-1,h})_h\in(\reals^q)^k$ is conditionally i.i.d.~given $\bm a_{-1}$.
Most notably, in this model the conditionally i.i.d.~(normalized) ``penalties'' $\bm f_{-1}$ may take negative values, while the condition on $\bm b_{-1}$ ensures a certain symmetry.

For $i=1$ let $\bm\Delta_1(\sigma)=\sum_{i=1}^\infty\prod_h\bm f_{1,h,i}(\sigma_h)$ for $\sigma\in[q]^k$, where $\bm f_1=(\bm f_{1,h})_h\in((\reals_{\ge 0}^q)^{\ints_{>0}})^k$ is conditionally i.i.d. given $\bm a_1$.
So, in this case the penalties have to be non-negative, but we do not require the symmetry condition and most notably we may sum over penalties, allowing arbitrary dependencies among the summands, while still enforcing independence of the coordinates.

Now, for $i\in\{-1,1\}$ let $p\in\mclp[i]$ if there exists $(\bm a_i,\bm b_i,\bm\Delta_i)$ as above such that $p$ is the law of $\psiR_i$.
For the sake of \emph{simplicity} the \emph{countable} convex combinations are the valid models
\begin{align*}
\mclp=\left\{\sum_i\alpha(i)p(i):\alpha\in\mclp(\ints_{>0}),\,p\in(\mclp[-1]\cup\mclp[1])^{\ints_{>0}}\right\}.
\end{align*}
\begin{proposition}\label{prop_validmodels}
We have $\{(p,\gamma^*)\in\mclp\times\mclp([q]):\gamma^*\ge\psibl,\ZFa[,p](\gamma^*)=\ZFabu_p\}\setle\mfkA$.
\end{proposition}
Proposition \ref{prop_validmodels} suggests that for $p\in\mclp$ it is sufficient to check if $\gamma^*$ is a maximizer of $\ZFa[,p]$.

\vspace{1mm}
\noindent{\bf NAE-SAT.}
Let $f_{\circ,x^*}(x)=\bmone\{x=x^*\}$ for $x,x^*\in\{0,1\}$.
For $x^*\in\{0,1\}^k$ let $x^*_1=x^*$ and $x^*_2=(1-x^*(h))_h$.
Further, let $f_{x^*}\in((\reals_{\ge 0}^{\{0,1\}})^{\ints_{>0}})^k$ be given by $f_{x^*,i,h}=f_{\circ,x^*_i(h)}$ for $i\in[2]$, $h\in[k]$ and $f_{x^*,i,h}\equiv 0$ otherwise.
Finally, let $b=1-\eps$ and $a=1$.
Then the weight is given by $\psiR(x)=a(1-b\bm\Delta(x))$, $\bm\Delta(x)=\sum_i\prod_h\bm f_{i,h}(x_h)$, $x\in\{0,1\}^k$, where $\bm f=f_{\bm x}$ is i.i.d.~since $\bm x=\unif_{\{0,1\}}^{\otimes k}$ is.
Proposition \ref{prop_validmodels} and $\expe[\psiR]\equiv 1-2^{1-k}(1-\eps)$
yield $(p,\gamma^*)\in\mfkA$ for any $\gamma^*\in\mclp([q])$ with $\gamma^*\ge\psibl$.

\vspace{1mm}
\noindent{\bf The $\bm k$-spin Model.}
For the $k$-spin model we have $\bm\psi(\sigma)=a(\bm J)(1-b(\bm J)\prod_h\sigma(h))$ with $a(\bm J)=(e^{\beta\bm J}+e^{-\beta\bm J})/2$ and $b(\bm J)=(e^{\beta\bm J}-e^{-\beta\bm J})/(e^{\beta\bm J}+e^{-\beta\bm J})$ (Example 1 in \cite{panchenko2004}).
Since $a:\reals\rightarrow\reals_{\ge 1}$ is even and a bijection on $\reals_{\ge 0}$, further $\bm J\dequal-\bm J$ is symmetric and $b$ is odd we have $b(\bm J)\dequal\bm s f(a(\bm J))$ for $f=b\circ a^{-1}$, with $a^{-1}$ being the inverse of $a$ on $\reals_{\ge 0}$, and $\bm s\in\{-1,1\}$ uniform and independent of $a(\bm J)$.
Proposition \ref{prop_validmodels} and $\expe[\psiR]\equiv 1$ suggest that
$(p,\gamma^*)\in\mfkA$ for any $\gamma^*\in\mclp([q])$ with $\gamma^*\ge\psibl$.

\vspace{1mm}
\noindent{\bf The Stochastic Block Model.}
Let $a_{i,0}=c_{\leftrightarrow}-c_i^{\neq}\ge 0$, $a_{i,\sigma^*}=c_i^{\neq}-c_i^=\ge 0$,
$f_{i,0}(\sigma)=\bmone\{\sigma\in\mclc[i]\}$ and $f_{i,\sigma^*}(\sigma)=\bmone\{\sigma\in\mclc[i],\sigma=\sigma^*\}$ for $\sigma,\sigma^*\in[q]$ and $i\in[2]$. Then we have
\begin{align*}
\psi(\sigma)=c_{\leftrightarrow}-\sum_{i=1}^2\sum_{\sigma^*=0}^qa_{i,\sigma^*}\prod_{h=1}^2f_{i,\sigma^*}(\sigma(h)).
\end{align*}
Proposition \ref{prop_validmodels} applies with $a=c_{\leftrightarrow}$, $b=1$ and the factors scaled with $(a_{i,\sigma^*}/c_{\leftrightarrow})^{1/k}$ to take the leading coefficients into account.
For the maximizer of $\ZFa[,p]$ we consider fixed pushforwards $x=\gamma(\mclc[1])\in[0,1]$, $\gamma(\mclc[2])=1-x$. Then we can (locally) maximize over the conditional laws independently, which amounts to the maximization of a standard stochastic block model (\cite{coja2018}),
yielding the uniform distributions $\unif_{\mclc[1]},\unif_{\mclc[2]}$ as the unique local maximizers, unique unless $x\in\{0,1\}$ or $a_{i,1}=0$ for some $i\in\{1,2\}$.
Formally, using $q_i=|\mclc[i]|$, $i\in[2]$, we have $\ZFabu_p=\sup_xf(x)$ with
\begin{align*}
f(x)=c_\leftrightarrow-b_1x^2-b_2(1-x)^2,\,
b_1=\left(a_{1,0}+\frac{a_{1,1}}{q_1}\right),\,
b_2=\left(a_{2,0}+\frac{a_{2,1}}{q_2}\right).
\end{align*}
Unless $\psi\equiv c_{\leftrightarrow}$, the function $f$ has the unique maximizer $x=b_2/(b_1+b_2)$.

\vspace{1mm}
\noindent{\bf Graphical Channels.}
Recall the notions from Section \ref{examples}.
First, recall from Observation \ref{obs_grchmutinf} that $\expe[\psiR]\equiv 1$ and hence $\ZFa[,p](\gamma^*)=\ZFabu_p$ for all $\gamma^*$.
On the other hand, it is model-dependent if $\nablaIbl(p,\gamma)\ge 0$ holds.
However, we observe that it is invariant to the choice of $p^*$.
\begin{observation}\label{obs_graphical_channelspos}
We have $\nablaI_{p,\gamma^*}(\pi)=\nablaI_{p_\circ,\gamma^*}(\pi)$ for $\pi\in\mclp[*,\gamma^*][2]([q])^2$, where $p_\circ$ is induced by $p^*=\unif([q'])$.
\end{observation}
\begin{proof}
Recall that for $y\in[q]^k$ and $z\in[q']$ we have $\psi_z(y)=\nu_y(z)/p^*(z)$.
Hence, for $\gamma\in(\mclp([q])^k)^2$ and $h\in[k]$ we have
$\ZFM(\psi_z,h,\gamma)=\prob[\bm z^\circ(\bm y^\circ_{\alpha})=z]/p^*(z)$, where $\alpha=\alpha_{h,\gamma}$ is given by $\bm y^\circ_{\alpha}\dequal\bigotimes_{h'}\bm y^\circ_{\alpha}(h')$, $\bm y^\circ_{\alpha}(h)\dequal\gamma_{1,h}$ and $\bm y^\circ_{\alpha}(h')\dequal\gamma_{2,h'}$ for $h'\in[k]\setminus\{h\}$.
But then $p^*$ cancels out when we take the expectation over $\psiR$, so
\begin{align*}
\expe[\xlnx(\ZFM(\psiR,h,\gamma))]
=\sum_z\prob[\bm z^\circ(\bm y^\circ_{\alpha})=z]\ln\left(\frac{\prob[\bm z^\circ(\bm y^\circ_{\alpha})=z]}{p^*(z)}\right)
=\DKL\left(\bm z^\circ(\bm y^\circ_{\alpha})\middle\|\,p^*\right).
\end{align*}
We obtain the results for $\ZF$ as special cases and hence a form of $\nablaI$ in terms of relative entropies (which are additive for independent random variables).
Now, the assertion follows from the decomposition of the relative entropy into the cross entropy and the entropy, i.e.~
$\DKL(\bm z^\circ(\bm y^\circ_{\alpha})\|p^*)=H(\bm z^\circ(\bm y^\circ_{\alpha})\|p^*)-H(\bm z^\circ(\bm y^\circ_{\alpha}))$, where
\begin{align*}
H(\bm z^\circ(\bm y^\circ_{\alpha})\|p^*)
&=-\sum_z\prob[\bm z^\circ(\bm y^\circ_{\alpha})=z]\ln(p^*(z)),\\
H(\bm z^\circ(\bm y^\circ_{\alpha}))
&=-\sum_z\prob[\bm z^\circ(\bm y^\circ_{\alpha})=z]\ln(\prob[\bm z^\circ(\bm y^\circ_{\alpha})=z]).
\end{align*}
Next, we take the expectation with respect to $\gammaR_\pi$.
Since the cross entropy is linear in the first component, using independence and linearity of the probability in $\alpha$ yields
$\expe[\expe[H(\bm z^\circ(\bm y^\circ_{\alpha_{h,\gammaR_\pi}})\|p^*)|\gammaR_\pi]]
=H(\bm z^\circ(\bm y^\circ_{\alpha^*})\|p^*)$, where $\alpha^*=\gamma^{*\otimes k}$.
Using that $\ZF$ is a special case of $\ZFM$ we obtain the corresponding results.
Finally, we notice that the expected cross entropies cancel out since they do not depend on $\pi$, and hence
\begin{align*}
\nablaI_{p,\gamma^*}(\pi)
=\expe\left[\expe\left[
kH(\bm z^\circ(\bm y^\circ_{\bm\alpha'}))
-(k-1)H(\bm z^\circ(\bm y^\circ_{\bm\alpha_2}))
-H(\bm z^\circ(\bm y^\circ_{\bm\alpha_1}))\middle|\gammaR_\pi,\hR\right]\right]
\end{align*}
does not depend on $p^*$, where $\bm\alpha'=\alpha_{\hR,\gammaR_\pi}$ and
$\bm\alpha_i\dequal\bigotimes_{h\in[k]}\gammaR_{\pi,i,h}$ for $i\in[2]$.
\end{proof}
\subsection{The Quenched Free Entropy of the Planted Ensemble}
\label{bethe_main}
Our first main result yields the limit of the teacher-student model quenched free entropy.
Recall $\xlnx$, $\ZF$ from Equation (\ref{ZF_def}), and $\mclp[][2]([q])$, $\mclp[*,\gamma^*][2]([q])$ from Equation (\ref{setpi_def}).
For $d'\in\ints_{\ge 0}$ and $(\psi,h,\gamma)\in(\reals_{>0}^{[q]^k}\times[k]\times\mclp([q])^k)^{\ints_{>0}}$ let
\begin{align*}
\ZV[,\gamma^*](d',\psi,h,\gamma)=\sum_{\sigma\in[q]}\gamma^*(\sigma)\prod_{a\in[d']}\left(\sum_{\tau\in[q]^k}\bmone\left\{\tau\left(h_a\right)=\sigma\right\}\psi(\tau)\prod_{h'\neq h}\gamma_{a,h'}(\tau(h'))\right).
\end{align*}
Let $\pi\in\mclp[][2]([q])$, $(\dR,\bm\psi,\bm h,\bm\gamma)\dequal\Po(d)\otimes(p\otimes\unif([k])\otimes\pi^{\otimes k})^{\otimes\ints_{>0}}$, $(\psiR_\circ,\gammaR_\circ)\dequal p\otimes\pi^{\otimes k}$, and let
\begin{align*}
\bethe_{p,\gamma^*,d}(\pi)
=\expe\left[\frac{1}{\ZFabu_p^{\dR}}\xlnx\left(\ZV[,\gamma^*](\dR,\psiR,\hR,\gammaR)\right)\right]
-\frac{d(k-1)}{k\ZFabu_p}\expe\left[\xlnx\left(\ZF(\psiR_{\circ},\gammaR_\circ)\right)\right]
\end{align*}
be the (limiting) Bethe free entropy (density for $\GTS_{\mR}(\sigmaIID)$).
Further, we denote the supremum over $\mclp[*,\gamma^*][2]([q])$ with $\bethebu[,p,\gamma^*](d)=\sup_{\pi\in\mclp[*,\gamma^*][2]([q])}\bethe_{p,\gamma^*,d}(\pi)$.
Finally, let $f(n)=\mclo(g(n))$ if there exists $c(\mfkg)\in\reals_{>0}$ such that $|f(n)|\le cg(n)$ for all $n\in\ints_{>0}$, where $\mfkg=(q,k,\psibl,\degabu)$.
\begin{theorem}\label{thm_bethe}
There exists $\rho(\mfkg)\in\reals_{>0}$ such that for $(p,\gamma^*)\in\mfkA$ we have
\begin{align*}
\expe[\phi_{\gamma^*}(\GTS_{\mR}(\sigmaIID))]=\bethebu(d)+\mclo(n^{-\rho}).
\end{align*}
\end{theorem}
\subsection{The Information-Theoretic Threshold}\label{information_theoretic_threshold}
The second main result addresses the relative entropy of $(\sigmaIID,\GTS(\sigmaIID))$ with respect to $(\sigmaG_{\gamma^*,\GR},\GR)$, where $\sigmaG_{\gamma^*,G}\dequal\mu_{\gamma^*,G}$ are the Gibbs spins from Equation (\ref{weightG}), further $\GR=\GR_{\mR}$ and $\GTS=\GTS_{\mR}$.
If $\bm a$ has a Radon-Nikodym derivative $r$ with respect to $\bm b$, let $\DKL(\bm a\|\bm b)=\expe[\ln(r(\bm a))]$ and $\DKL(\bm a\|\bm b)=\infty$ otherwise.
Further, let $\phia(d)=\phia[,p](d)=\frac{d}{k}\ln(\ZFabu_p)$.
\begin{theorem}\label{thm_infth}
With $\rho$ from Theorem \ref{thm_bethe} and for $(p,\gamma^*)\in\mfkA$ we have
\begin{align*}
\frac{1}{n}\DKL(\sigmaIID,\GTS(\sigmaIID)\|\sigmaG_{\gamma^*,\GR},\GR)=\bethebu(d)-\phia(d)+\mclo(n^{-\rho}).
\end{align*}
\end{theorem}
The Gibbs spins $\sigmaG_{\gamma^*,n}$ are the standard Gibbs spins for the graph with external fields. This reweighting is required for Theorem \ref{thm_infth} to be reasonable\footnote{This can e.g.~be seen by considering the trivial infinite temperature model $\bm\psi\equiv 1$.}.

Since the relative entropy is non-negative, Theorem \ref{thm_infth} suggests that $\bethebu(d)\ge\phia(d)$.
We will also see that $\phia(d)=\lim_{n\rightarrow\infty}\frac{1}{n}\ln(\expe[Z_{\gamma^*}(\GR)])$ is the annealed free entropy limit, and that $\frac{1}{n}\expe[\xlnx(Z_{\gamma^*}(\GR))]/\expe[Z_{\gamma^*}(\GR)]\rarr\bethebu(d)$, so $\bethebu(d)-\phia(d)$ is a Jensen gap.
Intuitively, Theorem \ref{thm_infth} states that the teacher-student and the null model are indistinguishable in the replica symmetric regime $\mfkPr=\{(p,\gamma^*,d)\in\mfkP:\bethebu(d)=\phia(d)\}$, $\mfkP=\mfkA\times[0,\degabu]$, while they are distinguishable in the condensation regime
$\mfkPc=\mfkP\setminus\mfkPr=\{(p,\gamma^*,d)\in\mfkP:\bethebu(d)>\phia(d)\}$.
\subsection{The Condensation Threshold}\label{condensation_threshold}
We confirmed that the replica symmetric and the condensation regime indeed govern the behavior of the relative entropy.
Next, we ensure that the quenched free entropy for the null model indeed behaves as expected. For this purpose let $\GR=\GR_{\mR}$ and
\begin{align*}
\phiqbu[,p,\gamma^*](d)=\limsup_{n\rightarrow\infty}\expe[\phi_{\gamma^*}(\GR)],\quad
\phiqbl[,p,\gamma^*](d)=\liminf_{n\rightarrow\infty}\expe[\phi_{\gamma^*}(\GR)].
\end{align*}
Let $\phiqbu(d)=\phiqbu[,p,\gamma^*](d)$ and $\phiqbl(d)=\phiqbl[,p,\gamma^*](d)$.
\begin{theorem}\label{thm_cond}
Recall $\rho$ from Theorem \ref{thm_bethe}.
\begin{alphaenumerate}
\item\label{thm_cond_r}
We have $\expe[\phi_{\gamma^*}(\GR)]=\phia(d)+\mclo(n^{-\rho/2})$ for $(p,\gamma^*,d)\in\mfkPr$.
\item\label{thm_cond_c}
There exists $c(\mfkg)\in\reals_{>0}$ such that for $(p,\gamma^*,d)\in\mfkP$ we have
\begin{align*}
\phia(d)-\phiqbu(d)\ge c\sup_{d'\in[0,d]}(\bethebu(d')-\phiqbl(d'))^2.
\end{align*}
\end{alphaenumerate}
\end{theorem}
Recall from Theorem \ref{thm_infth} that $\delta^*(d)=\bethebu(d)-\phia(d)\ge 0$.
Theorem \ref{thm_cond}\ref{thm_cond_c} implies that $\delta_{\uparrow}(d)=\phia(d)-\phiqbu(d)\ge 0$, so in particular $\delta_\downarrow(d)=\phia(d)-\phiqbl(d)\ge\delta_\uparrow(d)\ge 0$.
Now, looking at Theorem \ref{thm_cond}\ref{thm_cond_c} through the eyes of the annealed free entropy gives
\begin{align*}
\delta_\uparrow(d)\ge c\sup_{d'\in[0,d]}(\delta^*(d')+\delta_\downarrow(d'))^2
\ge c\sup_{d'\in[0,d]}(\delta^*(d')+\delta_\uparrow(d'))^2.
\end{align*}
With $\dcond=\dcond(p,\gamma^*)=\inf\{d\in\reals_{>0}:\delta^*(d)>0\}\in[0,\infty]$ Theorem \ref{thm_cond}\ref{thm_cond_r} suggests that $\phiqbl(d)=\phiqbu(d)=\phia(d)=\bethebu(d)$ for $d\in[0,\dcond)$, since $\delta^*(d)=0$, and in particular $(p,\gamma^*,d)\in\mfkPr$.
For $d\in(\dcond,\infty)$ there exists $d'\in[\dcond,d)$ such that $\delta^*(d')>0$, so Theorem \ref{thm_cond}\ref{thm_cond_c} suggests that $\delta_\uparrow(d)>0$.
But then Theorem \ref{thm_cond}\ref{thm_cond_r} requires that $\delta^*(d)>0$ and thereby $(p,\gamma^*,d)\in\mfkPc$.
In a nutshell, the regimes are intervals and $\dcond(p,\gamma^*)$ is a threshold, the condensation threshold by Theorem \ref{thm_cond} and the information-theoretic threshold by Theorem \ref{thm_infth}. We will see that $\delta^*(0)=0$ and $\delta^*$ is continuous, so $(p,\gamma^*,\dcond)\in\mfkPr$.
Theorem \ref{thm_cond}\ref{thm_cond_c} further allows to establish upper bounds for $\phiqbu(d)$, the simplest by considering $d'=d$ and solving the quadratic inequality, i.e.
\begin{align*}
\delta_\uparrow(d)\ge\tilde\delta-\sqrt{\tilde\delta^2-\delta^*(d)},\quad
\tilde\delta=\frac{1}{2c}-\delta^*(d),
\end{align*}
where the above is well-defined since $c$ is such that $\delta^*(d)\le 1/(4c)$.
\subsection{The Mutual Information}
We turn to the fourth and last main result, regarding the limit of the mutual information. In general the mutual information is given by
$\mutinf(\bm a,\bm b)=\DKL(\bm a,\bm b\|\bm a\otimes\bm b)$, which is consistent with the definition in the introduction.
\begin{theorem}\label{thm_mutinf}
With $\rho$ from Theorem \ref{thm_bethe} and for $(p,\gamma^*,d)\in\mfkP$ we have
\begin{align*}
\frac{1}{n}\mutinf\left(\sigmaIID,\GTS_{\mR}(\sigmaIID)\right)=\frac{d}{k\ZFabu}\expe\left[\xlnx\left(\psiR(\sigmaR)\right)\right]-\bethebu(d)+\mclo(n^{-\rho}),\,
(\psiR,\sigmaR)\dequal p\otimes\gamma^{*\otimes k}.
\end{align*}
\end{theorem}
Based on the discussion of graphical channels, where the normalization with $m$ is canonical, and on Theorem \ref{thm_infth}, the threshold behavior of the mutual information is most apparent if for $d>0$ we write the limit as
\begin{align*}
\lim_{n\rightarrow\infty}\frac{1}{n}\mutinf\left(\sigmaIID,\GTS_{\mR}(\sigmaIID)\right)
=\frac{d}{k}\left[\frac{1}{\ZFabu}\left(\expe\left[\xlnx\left(\psiR(\sigmaR)\right)\right]-
\xlnx(\ZFabu)\right)-\frac{k}{d}(\bethebu(d)-\phia(d))\right].
\end{align*}
\subsection{Implications, Extensions and Related Work}
\label{implications_extensions_related_work}
For the sake of brevity we did not present our main results in their full generality.
In the following, we discuss the actual scope and strength of our results, further implications and related work.

\vspace{1mm}
\noindent{\bf Simplifying Assumptions.}
The quantities in the main results are scaled with $1/n$ and can be written as expectations over $\mR$.
The conditional expectations given $\mR$ rescaled with $1/\mR$ are bounded and Lipschitz in $\mR$.
Hence, the main results hold for more general factor counts $(\bm m^*_n)_n$.
To be specific, let $\eps_i:\ints_{>0}\rightarrow\reals_{>0}$ with $\lim_{n\rightarrow\infty}\eps_i(n)=0$ for $i\in[2]$ and $d^*=\limsup_{n\rightarrow\infty}\expe[\mR^*_n]$.
Then the results hold for $\mR$ replaced by $\mR^*_n$ and $d$ replaced by $d^*$ if $\prob[|\dR^*_n-d^*|>\eps_1(n)]\le\eps_2(n)$, $\expe[\bmone\{|\dR^*_n-d^*|>\eps_1(n)\}\dR^*_n]\le\eps_2(n)$ for all $n\in\ints_{>0}$ and $d^*\le\degabu$, where $\dR^*_n=k\mR^*_n/n$.

Similarly, we only presented the case allowing parallel edges and identical factors.
However, we also establish the results for simple hypergraphs (both with labeled and unlabeled hyperedges).

\vspace{1mm}
\noindent{\bf External Fields.}
As mentioned in Remark \ref{remark_external_fields}, we recover the standard partition function and Gibbs measure for graphs with normalized external fields $\gamma^*$.
Fix $c\in\reals_{>0}$, let $\eta^*=c\gamma^*$ and consider the graph obtained by attaching the external field $\eta^*$ to each variable. This yields a perfect coupling of the null models, without and with $\eta^*$, and the teacher-student models (obtained by reweighting with $\psi_{\GR_m}(\sigma)$), since attaching $\eta^*$
adds a fixed factor $\prod_i\eta^*(\sigma(i))$ to $\psi_{\GR_m}(\sigma)$, since $\sigma$ is fixed.
Hence, mutual informations and relative entropies coincide.
Further, taking the standard partition function adds a constant $\ln(\|\eta^*\|_1)=\ln(c)$ to the discussed free entropy densities.

\vspace{1mm}
\noindent{\bf Modes of Convergence.}
The uniform treatment of models allows us to seamlessly move back and forth between the finite size objects and the limiting objects, and to understand the behavior under a variation of the model parameters, say the weight $\psiR$ through the inverse temperature $\beta$, the average degree as in Theorem \ref{thm_cond} or the ground truth distribution.

Having clarified the benefits, this also holds with respect to $\mR^*_n$ as introduced above, i.e.~the results hold for $\mfkg=(q,k,p,\gamma^*,\degabu,\eps_1,\eps_2)$ and $\mR^*$ such that the inequalities with respect to $\eps_{1,2}$ hold.
This further extends to factor graphs with external fields.

Moreover, as mentioned above, the results are phrased in terms of expectations over $\mR^*$. Instead, we may consider convergence in probability (and high probability events) for the conditional expectations given $\mR^*$. Due to the aforementioned properties of the conditional expectations, our results extend in this way to the conditional expectations (this mode is stated in \cite{coja2021}).

\vspace{1mm}
\noindent{\bf Model-Specific Results.}
While the uniform treatment of models is highly desirable, it might be misleading at times.
For example, in the discussion of Theorem \ref{thm_cond} we ignored the restriction $d\le\degabu$. The reason is that for any fixed choice of $(p,\gamma^*)\in\mfkA$ we can choose $\degabu$ arbitrarily large and in particular larger than $\dcond$, if $\dcond$ is finite.
Hence, only the constant $c$ in Theorem \ref{thm_cond}\ref{thm_cond_c} depends on $\degabu$ (while the uniform bound in Theorem \ref{thm_cond}\ref{thm_cond_r} only has to hold up to $\dcond$).

More importantly, being a continuous function on a compact set, $\ZFa$ does attain its maximum $\ZFabu$ at some $\gamma^*\in\mclp([q])$.
Assume without loss of generality that the support of $\gamma^*$ is $[q']$.
Obviously, we have $\min_{z\in[q']}\gamma^*(z)>0$.
So, consider the model on $[q']$ given by $\gamma'\in\mclp([q'])$, $\gamma'=\gamma^*$, and $\psiR':[q']^k\rightarrow[\psibl,\psibu]$, $\sigma\mapsto\psiR(\sigma)$, with law $p'$. Now, we have $\ZF[,p'](\gamma')=\ZF[,p](\gamma^*)=\ZFabu_{p}$ and hence $\ZF[,p'](\gamma')=\ZFabu_{p'}$. Further, if we have $\nablaIbl(p,\gamma)\ge 0$, then we have $\nablaIbl(p',\gamma')\ge 0$ because the Gibbs marginals $\gammaR_\pi$ are absolutely continuous with respect to $\gamma^*$ for $\pi\in\mclp[*,\gamma^*][2]([q])$, i.e.~$\gammaR_{\pi,i,h}(z)=0$ for $z\in[q]\setminus[q']$, $i\in[2]$, $h\in[k]$.
Hence, our results do apply to $(p',\gamma')$ if they would apply to $(p,\gamma^*)$ without the restriction $\gamma^*\ge\psibl$.
Moreover, since $\gamma^*$ restricts the ground truths to $[q']^n$ and $Z_{\gamma^*}(G)$ restricts the considered assignments to $[q']^n$, the results for $(p',\gamma')$ are exactly the same as they would be for $(p,\gamma^*)$.
In a nutshell, the restriction to $\gamma^*\ge\psibl$ is only relevant for the uniform convergence, our results apply to any maximizer $\gamma^*\in\mclp([q])$ of $\ZFa[,p]$, if $\nablaIbl(p,\gamma^*)\ge 0$. In particular, this explains why we cover the case $q=1$.

\vspace{1mm}
\noindent{\bf The Planted Model.}
Since we discuss the quenched free entropy density with respect to the planted, reweighted, model in Theorem \ref{thm_bethe}, the related notions are also reweighted, in particular $\bethe$ and $\nablaI$.
This is the exact reason for the appearance of $\xlnx(\cdot)$ (and also explains why we work conditional to $\bm a$ in the definition of the valid models).
With respect to both, it is immediate that $\expe[\ZF(\psiR,\gammaR_\pi)]=\ZFabu$ under our assumptions, hence $(\psi,\gamma)\mapsto\ZF(\psi,\gamma)/\ZFabu$ is a Radon-Nikodym derivative. The corresponding observation holds for $\ZFM$, and even
$(\psi,h,\gamma)\mapsto\ZV(d,\psi,h,\gamma)/\ZFabu^d$ is a Radon-Nikodym derivative.
This allows to transition to the reweighted measures, where we lose independence, but e.g.~recover a well-known form of the Bethe free entropy.

\vspace{1mm}
\noindent{\bf Graphical Channels.}
Recall the discussion of graphical channels in Section \ref{assumptions} and let $d>0$.
As explained above, we may also consider any fixed sequence $m_n\rightarrow d$ and obtain the same limit for the mutual information. With Observation \ref{obs_grchmutinf} and Theorem \ref{thm_mutinf} we hence have
\begin{align*}
\lim_{n\rightarrow\infty}\frac{1}{m}\mutinf(\bm x,\bm z(y_{\bm v,\bm x})|\bm v)
&=\expe\left[\xlnx\left(\psiR(\sigmaR)\right)\right]
-\frac{k}{d}\bethebu(d),\\
\expe[\xlnx(\psiR(\sigmaR))]
&=H(\bm z^\circ(\sigmaR)\|p^*)-\expe[\expe[H(\bm z^\circ(\sigmaR))|\sigmaR]],
\end{align*}
if $\nablaIbl(p_\circ,\gamma^*)\ge 0$ using Observation \ref{obs_graphical_channelspos}
and the notions in the proof.
This shows that $\expe[\xlnx(\psiR(\sigmaR))]$ 
attains its minimum $\mutinf(\sigmaR,\bm z^\circ(\sigmaR))$
at the unique minimizer $p^*=p^*_\circ$, where $p^*_\circ$ is the law of $\bm z^\circ(\sigmaR)$.
This shows that $(p,\gamma^*,d)\in\mfkPc$ for $p^*\neq p^*_\circ$ (since the limit does not depend on $p^*$, so $\bethebu(d)>0$).
For $p^*=p^*_\circ$ we have
\begin{align*}
\lim_{n\rightarrow\infty}\frac{1}{m}\mutinf(\bm x,\bm z(y_{\bm v,\bm x})|\bm v)
=\mutinf(\sigmaR,\bm z^\circ(\sigmaR))-\bethebu(d),
\end{align*}
so the limit attains
$\mutinf(\sigmaR,\bm z^\circ(\sigmaR))$ if and only if $(p,\gamma^*,d)\in\mfkPr$.
In particular, recalling the discussion in Section \ref{examples}, if the channel capacity $c$ is attained at $\gamma^{*\otimes k}$ for some $\gamma^*\in\mclp([q])$ and $\nablaIbl(p_\circ,\gamma^*)\ge 0$, then
$\lim_{n\rightarrow\infty}\frac{1}{m}\mutinf(\bm x,\bm z(y_{\bm v,\bm x})|\bm v)=c$ for $d\le\dcond(p,\gamma^*)$.

\vspace{1mm}
\noindent{\bf Related Work.}
We recommend \cite{abbe2017,moore2017} for an excellent introduction to community detection and a survey of results.
The graphical channels are introduced and discussed in \cite{abbe2015}.
For an excellent introduction to factor graph models, related quantities and results in the context of spin glasses, coding theory and complexity theory we highly recommend \cite{mezard2009}.

The sparse Erdős–Rényi type model discussed in this contribution has received considerable attention, in particular when it comes to specific problems.
The following, narrow selection of references is closely related to our proofs and results.
The limit of the null model quenched free entropy was discussed for $q=2$ and permutation invariant weights in \cite{talagrand2001},
and based on the interpolation method in \cite{franz2003, panchenko2004}.
The latter presents a class of models reminiscent of the valid models in this work.
However, although related, the convexity assumptions required for the discussion of the null model and the planted model using the interpolation method differ due to the reweighting.
The case with symmetric independent factors for valid models seems to be due to Maneva (cf.~\cite{montanari2008}).
We derived the valid models in this work based on the examples in \cite{coja2018} and \cite{coja2021}, in particular on the stochastic block model, NAE-SAT and the $k$-spin model.
Next to the presented models, further examples for valid models include positive temperature $k$-SAT (\cite{franz2003,panchenko2004,coja2018,montanari2008}), XOR-SAT (\cite{franz2003}) and hypergraph coloring (\cite{coja2018}).

Not only  the concept of graphical channels, also the main results in \cite{abbe2015} are closely related to Theorem \ref{thm_bethe} and Theorem \ref{thm_mutinf}.
As discussed in Observation \ref{obs_grchmutinf}, the model in \cite{abbe2015} satisfies $\expe[\psiR]\equiv 1$, hence $\ZFa(\unif([q]))=\ZFabu$ trivially holds, while Hypothesis H is closely related to $\inf_{\gamma^*}\nablaIbl(p,\gamma^*)\ge 0$.
As indicated above, only the case $\gamma^*=\unif([q])$ is discussed in \cite{abbe2015}. However, most notably only the existence of a limit is established.
On the other hand, sub-additivity of the free entropy is established and the weights may vanish.

The results in \cite{coja2018} establish not only the existence, but the exact values of the limits in a more general setting.
The assumption {\bf POS} in \cite{coja2020} is closely related to Hypothesis H in \cite{abbe2015}, both due to the interpolation method, which is used in the former case to establish Proposition 3.7 and in the latter to obtain sub-additivity of the free entropy. As mentioned above, {\bf BAL} in \cite{coja2018} (minus concativity) holds for the models in \cite{abbe2015}. Finally, both \cite{abbe2015} and \cite{coja2018} focus on weight distributions $p$ with finite support and $\gamma^*=\unif([q])$.
The results in \cite{coja2021} extend parts of \cite{coja2018} to more general degree distributions, under a more restrictive {\bf SYM} assumption and a weaker {\bf POS} assumption. The extension of our results to convergence in probability as stated in \cite{coja2021} was discussed above.

This work directly extends Theorem 2.2, Theorem 2.6 and Theorem 2.7 in \cite{coja2018} as follows. We consider the weaker {\bf POS} assumption from \cite{coja2021} (yielding new insights as demonstrated in Observation \ref{obs_graphical_channelspos}), drop the assumption {\bf SYM}, extend to arbitrary bounded weights and ground truth distributions $\gamma^*\in\mclp([q])$, weaken {\bf BAL} correspondingly, and in particular $\ZFa$ is not required to be concave.
Under these weaker assumptions we derive stronger results, namely the uniform treatment in all four theorems, also for the explicit limit in Theorem \ref{thm_infth} and the limit in Theorem \ref{thm_cond}\ref{thm_cond_r}, and the explicit bound in Theorem \ref{thm_cond}\ref{thm_cond_c}.

For example, we extend the results for the stochastic block model to a more general version, as illustrated above.
Also the $k$-spin model (cf.~\cite{coja2021}) is now covered for fairly general distributions (still not the Gaussian, though) without requiring additional arguments, as illustrated.
For this example, but also for $k$-SAT, NAE-SAT and LDGM codes (cf.~\cite{coja2018})
we extend the results from $\gamma^*=\unif([q])$ to $\gamma^*\in\mclp([q])$, i.e.~from the free entropy without external fields to the free entropy with (fixed and equal) external fields, and correspondingly for the mutual information, as illustrated.

We implement these significant advances using the same approach as in \cite{coja2018,coja2021}, composed of mutual contiguity with the Nishimori ground truth, concentration, the Aizenman-Sims-Starr scheme and the interpolation method.
Hence, we do not present a new proof technique on a high level,
but rather point out the potential of the existing, ingenious approach.
Since this work is closely related to \cite{coja2018}, we deliberately reuse the notation and the form of presentation for easier comparison.

\vspace{1mm}
\noindent{\bf Open Problems.}
The current discussion adds degrees of freedom to $\dcond$ by introducing a variation of the ground truth distribution $\gamma^*$.
Depending on the problem, there may be little room for the choice of $\gamma^*$.
For other problems however, say lower bounds for satisfiability thresholds, a variation of $\gamma^*$ may even be desirable.
In this case, understanding the behavior of the threshold and the limiting quantities under a variation of $\gamma^*$ may lead to valuable theoretical insights.

Further relaxing the assumptions given by $\mfkA$, i.e.~considering pairs $(p,\gamma^*)$ that violate $\ZFa[,p](\gamma^*)=\ZFabu_p$ or $\nablaIbl(p,\gamma^*)\ge 0$, 
might facilitate to extend the results to ferromagnetic problems like the associative stochastic block model.
In particular, we believe that weakening $\ZFa[,p](\gamma^*)=\ZFabu_p$ is possible without relying on new techniques.
On the other hand, developing new proof techniques to verify $\nablaIbl(p,\gamma^*)\ge 0$ might extend the results beyond the valid models presented here, e.g.~to positive temperature occupation problems. For example, adapting and extending the proof of Lemma 6.15 in \cite{abbe2015} would certainly be helpful.

By carefully working through the proofs, the results can be extended to more general ground truths $\sigmaIID$ and weights $\psiR$.
An unrestricted choice of ground truth certainly is useful in the discussion of the channel capacity of graphical channels or community detection, while an extension to weights $\psiR:[q]^k\rightarrow\reals_{>0}$ satisfying mild assumptions (like Equation (2.1) in \cite{coja2018b}) would be highly desirable to cover a popular branch of spin glasses.
Establishing that the limit of the quenched free entropy for $(p,\gamma^*)\in\mfkA$ exists would be helpful, not only in the context of Theorem \ref{thm_cond}.
Finally, we believe that the results of this contribution can be extended to more general degree distributions, similar to the extension \cite{coja2021} of \cite{coja2018}. This extension is highly desirable for graphical channels application-wise.

\section{Outline of the Proof}\label{proof_strategy}
We briefly recall the approach mentioned in Section \ref{implications_extensions_related_work} and assume that $(p,\gamma^*,d)\in\mfkP$, without further mention.
The bounds in the following results only depend on $\mfkg=(q,k,\psibl,\degabu)$.

We obtain Proposition \ref{prop_validmodels} using a Taylor series expansion as in \cite{coja2021}. Then we show that the resulting contributions are non-negative, yielding {\bf POS} in \cite{coja2018}.
The proofs of all main results rely on the properties of the Nishimori ground truth $\sigmaNIS_{p,\gamma^*,n,m}\in[q]^n$. We present the details in Section \ref{ps_nishimori}.
Next, we derive two crucial properties of the free entropies in the main results: concentration and Lipschitz continuity of the conditional expectations.
Details can be found in Section \ref{ps_phi_concon}.

The proof of Theorem \ref{thm_bethe} relies on the pinning lemma discussed in Section \ref{ps_pinning}.
In Section \ref{ps_bethe} we explain its application and the steps required to obtain Theorem \ref{thm_bethe}.
Theorem \ref{thm_infth} follows from Theorem \ref{thm_bethe} using the properties of $\sigmaNIS_{m}$.
The result is immediate for $\DKL(\GTS_{\mR}(\sigmaNIS_{\mR})\|\GR_{\mR})$, only the discussion of the conditional relative entropy requires some care.
The proof of Theorem \ref{thm_cond} is also rather short, but relies on two clever ideas.
Compared to the preceding two results, for the proof of Theorem \ref{thm_mutinf} it is rather cumbersome to decompose the mutual information into the ground truth weight and the free entropy, and to derive the asymptotics of the former using $\sigmaNIS_m$.
\subsection{The Nishimori Ground Truth}\label{ps_nishimori}
As explained in Remark \ref{remark_external_fields} and Section \ref{information_theoretic_threshold}, we need to consider Gibbs measures $\mu_{\gamma^*,G}$ that are consistent with $\sigmaIID$.
In order to control both $\mu_{\gamma^*,G}$ and $\sigmaIID$ given $\GTS_m(\sigmaIID)$,
we recover the Bayes optimal case and hence ensure that the Nishimori condition holds (Section 1.2.2 in \cite{zdeborova2016}), by introducing the Nishimori ground truth $\sigmaNIS_m$, given by the Radon-Nikodym derivative
\begin{align*}
\hat r_{p,\gamma^*,n,m}:[q]^n\rightarrow\reals_{>0},\,
\sigma\mapsto\frac{\expe[\psi_{\GR(m)}(\sigma)]}{\expe[Z_{\gamma^*}(\GR(m)]}
\end{align*}
with respect to $\sigmaIID$.
Let $\mbu=2\degabu n/k$, and let
$\|\bm a-\bm b\|_{\mrmtv}=\sup_{\mcle}|\mu(\mcle)-\nu(\mcle)|$
be the total variation distance of $\bm a$, $\bm b$ with laws $\mu$, $\nu$ respectively.
\begin{proposition}\label{proposition_mutual_contiguity}
Let $m\in\ints_{\ge 0}$ with $m\le\mbu$.
\begin{alphaenumerate}
\item\label{proposition_mutual_contiguity_rnup}
There exists $c\in\reals_{>0}$ with $\hat r_m\le c$.
\item\label{proposition_mutual_contiguity_rndown}
There exists $c\in\reals_{>0}$ such that $\hat r_m(\sigma)\ge\exp(-c\|\gammaN[,\sigma]-\gamma^*\|_\mrmtv^2n)$.
\item\label{proposition_mutual_contiguity_rncond}
The ground truths $\sigmaIID$ given $\gammaN[,\sigmaIID]$ and $\sigmaNIS_m$ given $\gammaN[,\sigmaNIS_m]$ have the same law.
\item\label{proposition_mutual_contiguity_nishi}
We have $(\sigmaNIS_m,\GTS_m(\sigmaNIS_m))\dequal(\sigmaR_{\gamma^*,\GTS(m,\sigmaNIS(m))},\GTS_m(\sigmaNIS_m))$.
\end{alphaenumerate}
\end{proposition}
\subsection{Concentration and Continuity}\label{ps_phi_concon}
Before we turn to the proof of Theorem \ref{thm_bethe}, we establish concentration (self-averaging) and continuity of the free entropies.
\begin{proposition}\label{proposition_phi_concon}
Let $m\in\ints_{\ge 0}$, $\sigma\in[q]^n$ and $\GR^\circ\in\{\GR_m,\GTSM[m](\sigma),\GTSM[m](\sigmaIID),\GTSM[m](\sigmaNIS_m)\}$.
\begin{alphaenumerate}
\item\label{proposition_phi_concon_bu}
There exists $c\in\reals_{>0}$ such that $|\phi_{\gamma^*}(\GR^\circ)|\le ckm/n$ almost surely.
\item\label{proposition_phi_concon_conc}
There exists $c\in\reals_{>0}^2$ such that
$\prob[|\phi_{\gamma^*}(\GR^\circ)-\expe[\phi_{\gamma^*}(\GR^\circ)]|\ge r]
\le c_2e^{-c_1r^2n}$ for $m\le\mbu$ and $r\in\reals_{\ge 0}$.
\item\label{proposition_phi_concon_cont}
For $\gammaN[,\sigma]\ge\psibl/2$ and $m\le\mbu$, $\sigma'\in[q]^n$, $m'\in\ints_{\ge 0}$ and $m^\circ\in\ints_{\ge 0}^2$ we have
\begin{align*}
|\expe[\phi_{\gamma^*}(\GTS_m(\sigma)]-\expe[\phi_{\gamma^*}(\GTS_{m'}(\sigma')]|
&\le L\left(\|\gammaN[,\sigma]-\gammaN[,\sigma']\|_\mrmtv+
\left|\frac{km}{n}-\frac{km'}{n}\right|\right),\\
|\expe[\phi_{\gamma^*}(\GR(m^\circ_1)]-\expe[\phi_{\gamma^*}(\GR(m^\circ_2)]|
&\le L\left|\frac{km^\circ_1}{n}-\frac{km^\circ_2}{n}\right|.
\end{align*}
\end{alphaenumerate}
\end{proposition}
Proposition \ref{proposition_phi_concon}\ref{proposition_phi_concon_bu} with
$\expe[\bmone\{|\dR^*_n-d^*|>\eps_1(n)\}\dR^*_n]\le\eps_2(n)$ from Section \ref{implications_extensions_related_work} suggests that we can restrict the expectations accordingly, whereas Proposition \ref{proposition_phi_concon}\ref{proposition_phi_concon_cont} ensures that on the interval $|\dR^*_n-d^*|\le\eps_1(n)$ the conditional expectations asymptotically coincide.
Regarding the ground truths, we recall that $\prob[\|\gammaN[,\sigmaIID]-\gamma^*\|_\mrmtv\ge r]\le c'\exp(-cr^2n)$ for suitable $c$, $c'$, and hence we can restrict to converging color frequencies (since the free entropies are uniformly bounded for $m\le\mbu$). Proposition \ref{proposition_phi_concon}\ref{proposition_phi_concon_cont} then ensures that the conditional expected free entropies asymptotically coincide.
\subsection{The Pinning Lemma}\label{ps_pinning}
Proposition \ref{proposition_mutual_contiguity} and Proposition \ref{proposition_phi_concon} establish that the quenched free entropy (densities) for $\sigmaIID$, $\sigmaNIS_{m}$ asymptotically coincide. The following pinning lemma illustrates \emph{why} working with $\sigmaNIS_{m}$ is desirable.
Recall that the product measure $\alpha=\bigotimes_{a,h}\gamma_{a,h}$ of the marginals $\gamma$ is used in the definition of $\ZF$, $\ZV$ for the Bethe free entropy.
The law $\alpha$ corresponds to the joint distribution of $\mu_{\gamma^*,\GTS_{\mR}(\sigmaIID)}$ on a random number of variables in the finite size case.
One of the main obstacles is to show that these joint distributions indeed asymptotically factorize, and this is exactly where the pinning lemma comes into play.
For $\sigmaR\in[q]^n$ with law $\mu$, $\ell\in\ints_{>0}$ and $v\in[n]^\ell$ let $\mu|_v$ be the law of $\sigmaR_v\in[q]^\ell$. For $\ell=1$ we use the shorthand $\mu|_{v(1)}=\mu|_v$. 
Further, let
$\iota_\circ(\mu,v)=\DKL(\mu|_v\|\bigotimes_{h}\mu|_{v(h)})$ and $\iota_\ell(\mu)=\expe[\iota_\circ(\mu,\vR)]$ with $\bm v\dequal\unif([n]^{\ell})$.
For $\sigmaP\in[q]^n$ and $\setP\setle[n]$ let $[\mu]^\downarrow_{\setP,\sigmaP}\in\mclp([q]^n)$ be the law of $\sigmaR|(\sigmaR(i))_{i\in\setP}=(\sigmaP(i))_{i\in\setP}$, if this is defined.
The next result generalizes Lemma 3.5 in \cite{coja2018}, corresponding to $\ell=2$, and states a stronger version that addresses the conditional relative entropy directly.
\begin{lemma}\label{lemma_pinning}
For $n\in\ints_{>0}$, $\mu\in\mclp([q]^n)$, $\ell\in\ints_{>0}$ and $\ThetaP\in(0,n]$ the following holds.
Let
$\thetaPR\dequal\unif([0,\ThetaP])$,
further i.i.d.~Bernoulli $\indPR\in\{0,1\}^n$ with success probability $\thetaPR/n\in[0,1]$,
$\setPR=\indPR^{-1}(1)$ and $\sigmaR\dequal\mu$ with $(\sigmaR,\setPR)\dequal\sigmaR\otimes\setPR$.
Then we have $\expe[\iota([\mu]^\downarrow_{\setPR,\sigmaR})]\le\binom{\ell}{2}\ln(q)/\ThetaP$.
\end{lemma}
\subsection{The Planted Model Quenched Free Entropy}\label{ps_bethe}
We use the interpolation method to obtain the lower bound
in Theorem \ref{thm_bethe}.
\begin{proposition}\label{proposition_int}
We have $\expe[\phi_{\gamma^*}(\GTS_{\mR}(\sigmaIID))]\ge\bethebu(d)-\mclo(n^{-1/4})$.
\end{proposition}
We use the Aizenman-Sims-Starr scheme to obtain the upper bound in Theorem \ref{thm_bethe}.
\begin{proposition}\label{proposition_ass}
There exists $\rho\in\reals_{>0}$ such that $\expe[\phi_{\gamma^*}(\GTS_{\mR}(\sigmaIID))]\le\bethebu(d)+\mclo(n^{-\rho})$.
\end{proposition}
Both methods require that $\mR\dequal\Po(dn/k)$ is Poisson distributed.
Further, we need the pinning lemma in both cases.
For this purpose we decorate the graphs with an additional type of factors, say pins, to turn $\mu_{\gamma^*,G}$ into $[\mu_{\gamma^*,G}]^\downarrow_{\setP,\sigmaP}$, which then ensures asymptotic independence.
Clearly, we have to ensure that the effect of this pinning procedure on the quenched free entropy is asymptotically negligible.

For the interpolation method we need yet another type of factors, say interpolators, to model the transition between the decoupled model underlying the Bethe free entropy and the graph. The derivative of this transition is closely related to $\nablaI_{p,\gamma^*}$, and it is its non-negativity that establishes the lower bound in the end.

The Aizenman-Sims-Starr scheme only relies on the standard factors and pins (and external fields).
Here, we make use of the fact that the quenched free entropy \emph{density} can be rewritten as the average difference of the quenched free entropies for $n+1$ and $n$. Hence, deriving the limit of the difference yields the limit of the quenched free entropy density.
The implications of this method are significantly stronger than Proposition \ref{proposition_ass} suggests.
In fact, we show that the quenched free entropy density \emph{converges to} the expected Bethe free entropy density, with the expectation taken over certain Gibbs marginal distributions $\bm\pi\in\mclp[*,\gamma^*][2]([q])$.

So, in a nutshell, the first part of the proof clarifies how exactly we can utilize the pinning lemma and justifies the application.
In the second part we implement the interpolation method using the fully decorated graphs, followed by the implementation of the Aizenman-Sims-Starr scheme with the slightly simpler graphs in the third part.
\section{Preliminaries}
In the remainder of this contribution we provide the proofs for all statements that have not yet been established.
Next to the obvious claims, we stated that $\gammaN[,\sigmaIID]$ concentrates around $\gamma^*$ while discussing the graphical channels in Section \ref{examples} and in Section \ref{assumptions}, which holds by Observation \ref{obs_gtiid}.
The limit $\phia$ in Section \ref{information_theoretic_threshold} is determined in Observation \ref{obs_phia}.
In Section \ref{condensation_threshold} we claimed that $\delta^*(0)=0$ and that $\delta^*$ is continuous in $d$, which is covered in Section \ref{adddisc_lipschitz_bounds}.
We turn to the claims in Section \ref{implications_extensions_related_work}.
We establish the main results by considering graphs with (normalized) external fields and general factor counts $\mR^*$, and establish the corresponding results explicitly.
The discussion of the special cases $q=1$, $k\le 1$, of simple hypergraphs, convergence in probability, more general external fields, the reweighted laws for the planted model and the reduction of Lemma \ref{lemma_pinning} to Lemma 3.5 in \cite{coja2018} can be found in the Section \ref{additional_discussion}.
So, for now we restrict to $q,k\ge 2$.

We explain the structure of the following discussion in Section \ref{proof_roadmap}, then we introduce the required notions and results from the literature in Section \ref{notions_notation}, and conclude with the proof of Proposition \ref{prop_validmodels} in Section \ref{proof_prop_validmodels}.
\subsection{A Roadmap to the Proofs}\label{proof_roadmap}
In Section \ref{preparations} we derive basic results that are required for the proofs of all main results.
Specifically, in Section \ref{decorated_factor_graphs} we introduce decorated graphs and establish basic properties.
In Section \ref{nishimori_ground_truth} we discuss basic properties of the Nishimori ground truth $\sigmaNIS$, including the proof of Proposition \ref{proposition_mutual_contiguity}.
In Section \ref{phi_concon} we establish boundedness, continuity and concentration for the free entropy, including the proof of Proposition \ref{proposition_phi_concon}.

Section \ref{thm_bethe_proof} is devoted to the proof of Theorem \ref{thm_bethe}.
Specifically, in Section \ref{pinning} we discuss the Gibbs measures of decorated graphs, establish the pinning lemma \ref{lemma_pinning} and apply it to the graphs, establish a result for reweighted marginal distributions of general measures and apply it to the graphs, and finally discuss projections of $\mclp[][2]([q])$ onto $\mclp[*][2]([q])$.
In Section \ref{interpolation} we then turn to the interpolation method including the proof of Proposition \ref{proposition_int}.
The discussion in Section \ref{ass} addresses the Aizenman-Sims-Starr scheme including the proof of Proposition \ref{proposition_ass}, where we also establish Theorem \ref{thm_bethe}.

In Section \ref{main_proofs} we present the proofs of the remaining main results.
We establish Theorem \ref{thm_infth} in Section \ref{main_proofs_dkl}, followed by the proof of Theorem \ref{thm_cond} in Section \ref{main_proofs_condth}, and conclude the proof of the main results with Theorem \ref{thm_mutinf} in Section \ref{main_proofs_mutinf}.

As mentioned above, the special cases $q=1$ and $k=1$ are discussed in Section \ref{additional_discussion}, where we also formalize the discussion in Section \ref{implications_extensions_related_work}.
\subsection{Notions, Notation and Results from the Literature}\label{notions_notation}
We consider a sufficiently rich probability space $(\Omega,\mclf,\prob)$.
Random quantities are denoted in bold, e.g.~$\bm a:\Omega\rightarrow\mcla$.
The $(\bm a,\bm b)$-derivative is the Radon-Nikodym derivative of $\bm a$ with respect to $\bm b$.
We use the Poisson distribution $\Po(\lambda)$, the binomial distribution $\Bin(n,p)$, the uniform distribution $\unif(\mcls))$ and the one-point mass $\opm[,\mcls,s]\in\mclp(\mcls)$ on $s\in\mcls$.
Further, we use $\mcla\dotcup\mclb$ for the disjoint union, $[n]=\ints\cap[1,n]$, $2^{\mcls}$ for the power set of $\mcls$, $\binom{\mcla}{a}\setle 2^{\mcls}$ for the $a$-subsets $\mclb\setle\mcla$ with $|\mclb|=a$, $\mcla[][\mclb]$ for the maps $f:\mclb\rightarrow\mcla$, and $[n]_m=\{x\in[n]^m:\forall i\in[n]\,|x^{-1}(i)|\le 1\}$ for the injections.
We consider spaces equipped with their canonical structure unless mentioned otherwise, mostly subspaces of $\reals^a$.
We use $\le$ for componentwise inequalities and $\equiv$ for componentwise equality.
We use $\dequal$ for equality in distribution.
A space $\mcla$ is a copy of a space $\mclb$ if it carries the same structure under some bijection, in which case we identify $\mcla$ with $\mclb$.
We identify $\mu\in\mclp(\ints)$ with its probability mass function $\mu:\ints\rightarrow[0,1]$.
We further use similar identifications to focus on the relevant arguments while avoiding technical routine discussions.

We (partially) suppress dependencies for brevity, e.g.~$f_a(x)=f_a=f$.
Clearly, this leads to ambiguities, e.g.~$\GTS$ may refer to $\GTS(\mR,\sigmaIID)$, $\GTS(m,\sigma)$ or any other combination.
Hence, when we omit a dependency, the dependency is the same quantity as in the definition and thereby uniquely identified.
Further, we keep the notation consistent to earn this degree of flexibility.
Finally, we may use $f_{a,x}=f_a(x)=f(a,x)$ interchangeably, for readability or to indicate the distinction between variables and parameters.
Similarly, we use mixed notation for random quantities $\bm x$ and their laws $\bm x\dequal\mu$, e.g.~$\DKL(\bm x_1\|\bm x_2)=\DKL(\mu_1\|\mu_2)$.

We extend $(\sigma_i)_{i\in[n]}\in[q]^n$ to maps, i.e.~for $v\in[n]^k$ let $\sigma_v=(\sigma_{v(h)})_h$ ,
and let $\mu|_v$ be the law of $\sigmaR_{v}$ with $\sigmaR\dequal\mu$, as in Section \ref{proof_strategy}.
If $v$ is the enumeration of $\mclv\setle[n]$, i.e.~the unique strictly increasing map $v:[|\mclv|]\rarr\mclv$,
we use the shorthands $\sigma_{\mclv}=\sigma_v$ and $\mu|_{\mclv}=\mu|_v$ and in particular
$\sigma_i=\sigma_{\{i\}}$, $\mu|_i=\mu_{\{i\}}$.
Further, let $\mu|_*=\sum_i\frac{1}{n}\mu|_i\in\mclp([q])$ be the law of $\sigmaR_{\bm i}\in[q]$, with $(\sigmaR,\bm i)\dequal\mu\otimes\unif([n])$.
We denote the total variation distance by $\|\mu-\nu\|_\mrmtv=\sup_{\mcle}|\mu(\mcle)-\nu(\mcle)|$ and let $\Gamma(\mu_1,\mu_2)\setle\mclp(\mclx[1]\times\mclx[2])$ denote the couplings of $\mu_1\in\mclp(\mclx[1])$ and $\mu_2\in\mclp(\mclx[2])$, i.e.~for $\nu\in\Gamma(\mu_1,\mu_2)$ we have $\nu|_1=\mu_1$ and $\nu|_2=\mu_2$.
\begin{observation}\label{obs_tv}
Notice that the following holds.
\begin{alphaenumerate}
\item\label{obs_tv_norm}
For $\mu\in\mclp([q])^2$ we have $\|\mu_1-\mu_2\|_\mrmtv=\frac{1}{2}\|\mu_1-\mu_2\|_1$.
\item\label{obs_tv_prod}
For $\mu\in(\mclp([q])^n)^2$ we have $\|\bigotimes_i\mu_{1,i}-\bigotimes_i\mu_{2,i}\|_\mrmtv\le\sum_i\|\mu_{1,i}-\mu_{2,i}\|_\mrmtv$.
\item\label{obs_tv_cond}
For $\bm x,\bm x'\in[q]$ and $\bm y(x)\in[q']$, $x\in[q]$, we have $\|(\bm x,\bm y(\bm x))-(\bm x',\bm y(\bm x'))\|_\mrmtv=\|\bm x-\bm x'\|_\mrmtv$.
\item\label{obs_tv_lipschitz}
For $\ell\in\ints_{>0}^2$ with $\ell_1\le\ell_2$, further $v_i\in[n]^{\ell(i)}$, $i\in[2]$, with $v_1=v_{2,[\ell(1)]}$, and $\mu\in\mclp([q]^n)^2$ we have
$\|\mu_1|_{v_1}-\mu_2|_{v_1}\|_\mrmtv\le\|\mu_1|_{v_2}-\mu_2|_{v_2}\|_\mrmtv$
and $\|\mu_1|_*-\mu_2|_*\|_\mrmtv\le\expe[\|\mu_1|_{\bm i}-\mu_2|_{\bm i}\|_\mrmtv]$.
\item\label{obs_tv_coupling}
For a coupling $\bm y$ of $\bm x_1$, $\bm x_2\in[q]$ we have $\|\bm x_1-\bm x_2\|_\mrmtv\le\prob[\bm y_1\neq\bm y_2]$ and there exists $\bm y\in\Gamma(\bm x_1,\bm x_2)$ with 
$\|\bm x_1-\bm x_2\|_\mrmtv=\prob[\bm y_1\neq\bm y_2]$.
\item\label{obs_tv_pinsker}
For $\bm x_1$, $\bm x_2\in[q]$ we have $\|\bm x_1-\bm x_2\|_\mrmtv\le\sqrt{\frac{1}{2}\DKL(\bm x_1\|\bm x_2)}$.
\end{alphaenumerate}
\end{observation}
\begin{proof}
Part \ref{obs_tv}\ref{obs_tv_norm} can be found on page 153 in \cite{janson2000}, Part \ref{obs_tv}\ref{obs_tv_coupling} on page 10 in \cite{villani2009}, Part \ref{obs_tv}\ref{obs_tv_pinsker} is Pinsker's inequality, e.g.~Equation (2.8) in \cite{coja2018}.
For Part \ref{obs_tv}\ref{obs_tv_prod} we have
\begin{align*}
2\|\mu_1\otimes\mu_2-\nu_1\otimes\nu_2\|_\mrmtv
&=\sum_{x}|\mu_1(x_1)\mu_2(x_2)-\nu_1(x_1)\nu_2(x_2)|\\
&\le\sum_{x}\mu_1(x_1)|\mu_2(x_2)-\nu_2(x_2)|
+\sum_{x}\nu_2(x_2)|\mu_1(x_1)-\nu_1(x_1)|,
\end{align*}
so the assertion holds for $n=2$. The general case follows by induction analogous to the above.
Part \ref{obs_tv}\ref{obs_tv_cond} follows similarly, using  
Part \ref{obs_tv}\ref{obs_tv_norm}, distributivity and normalization of the conditional laws.
For Part \ref{obs_tv}\ref{obs_tv_lipschitz} notice that
\begin{align*}
\|f(\bm x_1)-f(\bm x_2)\|_\mrmtv\le\|\bm x_1-\bm x_2\|_\mrmtv
\end{align*}
holds in general and specifically for restrictions.
The second part of the assertion follows from Part \ref{obs_tv}\ref{obs_tv_norm} and the triangle inequality.
\end{proof}
We use the uniform distribution $\unif(\mcls)$, the Binomial distribution $\Bin(n,p)$ and the Poisson distribution $\Po(\lambda)$.
\begin{observation}\label{obs_poisson}
Let $\me\in\reals_{\ge 0}$ and $\mR\dequal\Po(\me)$.
\begin{alphaenumerate}
\item\label{obs_poisson_expe}
We have $\prob[\mR=m]m=\me\prob[\mR=m-1]$.
\item\label{obs_poisson_bin}
Let $p\in[0,1]$, $\mR_1\dequal\Po(p\me)$, $\mR_2\dequal\Po((1-p)\me)$ and $\bm n(m)\dequal\Bin(m,p)$, then we have $(\mR_1+\mR_2,\mR_1)\dequal(\mR,\bm n(\mR))$.
\item\label{obs_poisson_dkl}
For $\me'\in\reals_{>0}$, $\mR'\dequal\Po(\me')$ we have $\DKL(\mR\|\mR')=\me'-\me+\me\ln(\me/\me')$.
\item\label{obs_poisson_prob}
There exist $c,c'\in\reals_{>0}$ with $\prob[|\mR-\me|\ge r]\le c'\exp(-\frac{cr^2}{\me+ r})$ for $r\in\reals_{\ge 0}$.
\end{alphaenumerate}
\end{observation}
\begin{proof}
The first three parts can be easily verified directly, the last part follows
from Theorem 2.1 with Remark 2.6 in \cite{janson2000} and $c=1/2$, $c'=2$, where we notice that
for $\me=r=0$ the exponent is $0$.
\end{proof}
For $x\in\reals$ let $\lceil x\rceil x=\min\ints_{\ge x}$ and $\lfloor x\rfloor=\max\ints_{\le x}$.
\subsection{Proof of Proposition \ref{prop_validmodels}}\label{proof_prop_validmodels}
Let $\gamma^*\in\mclp([q])$ and $\pi\in(\mclp[*,\gamma^*][2]([q]))^2$.
Let $p\in\mclp[-1]\cup\mclp[1]$ and $\psiR\dequal p$ with $\psiR(\sigma)=\bm a(1-\bm b\bm\Delta(\sigma))$ for $\sigma\in[q]^k$.
Using $\psiR^\circ(\sigma)=1-\bm b\bm\Delta(\sigma)$, linearity of $\ZF$, $\ZFM$ in $\psi$ and $\xlnx(at)=a\xlnx(t)+t\xlnx(a)$ yields
\begin{align*}
\nablaI(\pi)&=\expe\left[\bm a\left(\xlnx\left(\ZF(\psiR^\circ,\gammaR_{\pi,1})\right)
+(k-1)\xlnx\left(\ZF(\psiR^\circ,\gammaR_{\pi,2})\right)
-k\xlnx\left(\ZFM(\psiR^\circ,\hR,\gammaR_{\pi})\right)
\right)\right]+\nabla_0,\\
\nabla_0&=\expe\left[\ln(\bm a)\left(\ZF(\psiR^\circ,\gammaR_{\pi,1})
+(k-1)\ZF(\psiR^\circ,\gammaR_{\pi,2})
-k\ZFM(\psiR^\circ,\hR,\gammaR_{\pi})\right)\right]=0,
\end{align*}
by taking the expectation over $\gammaR_\pi$ for $\nabla_0$ using independence and then using $\pi\in(\mclp[*,\gamma^*][2]([q]))^2$.
Using the Taylor series expansion $\xlnx(1-t)=-t+\sum_{\ell\ge 2}\frac{1}{\ell(\ell-1)}t^\ell$ for $t<1$ and $|\bm b\bm\Delta|<1$ yields
\begin{align*}
\nablaI(\pi)&=\sum_{\ell\ge 2}\expe\left[\frac{\bm a\bm b^\ell}{\ell(\ell-1)}\left(
\ZF(\bm\Delta,\gammaR_{\pi,1})^\ell+(k-1)\ZF(\bm\Delta,\gammaR_{\pi,2})^\ell
-k\ZFM(\bm\Delta,\hR,\gammaR_{\pi})^\ell
\right)\right]+\nabla_0,\\
\nabla_0&=\expe\left[\bm a\bm b\left(k\ZFM(\bm\Delta,\hR,\gammaR_{\pi})-(k-1)\ZF(\bm\Delta,\gammaR_{\pi,2})-\ZF(\bm\Delta,\gammaR_{\pi,1})\right)\right]=0,
\end{align*}
as before. Using conditional independence of $\bm b$ and $\bm\Delta$, we notice that $\nablaI(\pi)\ge 0$ if for each $\ell\ge 2$ we have
$\expe[\bm b^{\ell}|\bm a]\bm P(\ell)\ge 0$, where
\begin{align*}
\bm P(\ell)=\expe[\ZF(\bm\Delta,\gammaR_{\pi,1})^\ell+(k-1)\ZF(\bm\Delta,\gammaR_{\pi,2})^\ell
-k\ZFM(\bm\Delta,\hR,\gammaR_{\pi})^\ell|\bm a].
\end{align*}
Hence, for $p\in\mclp[-1]$ we can restrict to $\ell=2\ell'$, $\ell'\in\ints_{>0}$.
Let $\bm\Delta(\sigma)=\prod_h\bm f_{-1,h}(\sigma(h))$ and $\bm s_{i,h}=\sum_\sigma\bm f_{-1,h}(\sigma)\gammaR_{\pi,i,h}(\sigma)$, then distributivity and conditional independence yields
\begin{align*}
\bm P(\ell)
=\prod_{h}\bar{\bm s}_{1,h}+(k-1)\prod_{h}\bar{\bm s}_{2,h}-k\bar{\bm s}_{1,h}\prod_{h'\neq h}\bar{\bm s}_{2,h'},
\end{align*}
where $\bar{\bm s}_{i,h}=\expe[\bm s_{i,h}^\ell|\bm a]\ge 0$ since $\ell=2\ell'$. Now, recall that $(\gammaR_{\pi,i,h},\bm f_{-1,h})\dequal(\gammaR_{\pi,i,1},\bm f_{-1,1})$ given $\bm a$ have the same law and hence $\bar{\bm s}_{i,h}=\bar{\bm s}_{i,1}$.
This shows that $\nablaI(\pi)\ge 0$ for $p\in\mclp[-1]$ since
\begin{align*}
x^k+(k-1)y^k-kxy^{k-1}=(x-y)^2\sum_{r=0}^{k-2}(k-1-r)x^ry^{k-2-r}\ge 0
\end{align*}
for $x,y\ge 0$.
Let $\bm\Delta(\sigma)=\sum_{i\ge 1}\prod_h\bm f_{1,h,i}(\sigma_h)$ for $p\in\mclp[1]$ and $\bm s_{i,j,h}=\sum_\sigma\bm f_{1,h,i}(\sigma)\gammaR_{\pi,j,h}(\sigma)$, then we have
\begin{align*}
\bm P(\ell)
&=\expe\left[\left(\sum_i\prod_h\bm s_{i,1,h}\right)^\ell+(k-1)\left(\sum_i\prod_h\bm s_{i,2,h}\right)^\ell-k\left(\sum_i\bm s_{i,1,\hR}\prod_{h'\neq\hR}\bm s_{i,2,h'}\right)^\ell\middle|\bm a\right]\\
&=\sum_{i\in\ints_{>0}^\ell}(\bm x_1(i)+(k-1)\bm x_2(i)-\sum_h\bm x_{3,h}(i)),
\end{align*}
where $\bm x_j(i)=\expe[\prod_{m=1}^\ell\prod_{h}\bm s_{i(m),j,h}|\bm a]=\bar{\bm s}_j(i)^k$, $\bar{\bm s}_j(i)=\expe[\prod_m\bm s_{i(m),j,1}|\bm a]$, for $j\in[2]$ since the factors are conditionally independent in $h$ and have the same law, which also shows that
$\bm x_{3,h}(i)=\expe[\prod_{m=1}^\ell(\bm s_{i(m),1,h}\prod_{h'\neq h}\bm s_{i(m),2,h'})|\bm a]
=\bar{\bm s}_1(i)\bar{\bm s}_2(i)^{k-1}$ and hence that $\nablaI(\pi)\ge 0$ using the previous observations and $\bar{\bm s}_j(i)\ge 0$, since $\bm f\ge 0$.
Now, the assertion follows from the observation that $\nablaI_{\bar p,\gamma^*}(\pi)=\sum_i\alpha(i)\nablaI_{p(i),\gamma^*}(\pi)$ for $\bar p=\sum_i\alpha(i)p(i)$.
\section{Preparations}\label{preparations}
Recall $q,k\in\ints_{\ge 2}$, $\psibl\in(0,1/q)$, $\psibu=1/\psibl$ and $\degabu\in\reals_{>0}$ from Section \ref{random_factor_graphs}.
Further, let $\deltam,\epsm:\ints_{>0}\rightarrow\reals_{>0}$ with
$\lim_{n\rightarrow\infty}\deltam(n)=0$, $\lim_{n\rightarrow\infty}\epsm(n)=0$ denote the bounds that we will use for $\mR^*$ from Section \ref{implications_extensions_related_work}.
Hence, the global parameters are $\mfkg=(q,k,\psibl,\degabu,\deltam,\epsm)$.
We keep $\mfkg$ fixed throughout the remainder and do not track dependencies on $\mfkg$.
However, we occasionally write $c_\mfkg$ to stress that $c$ only depends on $\mfkg$.

Without loss of generality we may assume that $\psibl$ is arbitrarily small and that $\degabu$ is arbitrarily large since this only increases the set of model parameters.
We further assume without loss of generality that $\epsm$ and $\deltam$ are non-increasing.
After we restricted to the $\deltam$-ball around $d$, the largest average degree to be considered is $d+\deltam(n)\le\degabu+\deltam(1)$.
Without loss of generality we take $\deltam(1)=\degabu$, so $\mbu[,n]=2\degabu n/k$ from Section \ref{ps_nishimori} is the desired maximal factor count.

We use the Landau notation as discussed in Section \ref{proof_strategy}, i.e.~only with respect to the asymptotics of the number $n$ of variables and for functions depending on $n$ and $\mfkg$ only.
\subsection{Decorated Factor Graphs}\label{decorated_factor_graphs}
Let $\domPsi=[\psibl,\psibu]^{[q]^k}$ be the domain of the weights and let $\domG_{n,m}=([n]^k\times\domPsi)^m$ be the domain of the graphs from Section \ref{random_factor_graphs}.
\subsubsection{Random Decorated Graphs}\label{random_decorated_graphs}
The decorated factor graphs are given by
\begin{bulletlist}
\item a weight function $\psiRa\in\domPsi$ with law $\lawpsi$ and expectation $\psiae=\expe[\psiRa]$,
\item a ground truth distribution $\gamma^*\in\mclp([q])$ with $\gamma^*\ge\psibl$,
\item an average degree $\degae\in[0,\degabu]$ such that $(\lawpsi,\gamma^*,\degae)\in\mfkP=\mfkA\times[0,\degabu]$,
\item a Gibbs marginal distribution $\pi\in\mclp[*,\gamma^*][2]([q])$,
\item an interpolation time $\tI\in[0,1]$,
\item a pinning bound $\ThetaP\in\reals_{\ge 0}$,
\item a number $n\in\ints_{>\ThetaP}$ of variables,
\item a number $m\in\ints_{\ge 0}$ of factors,
\item a ground truth $\sigma\in[q]^n$,
\item interpolator counts $\mI\in\ints_{\ge 0}^n$ with $\facsI_{\mI}=\{i,h):i\in[n],h\in[\mI_i]\}$,
\item pins $\setP\setle[n]$ and
\item a pinning assignment $\sigmaP\in[q]^n$,
\end{bulletlist}
which we will keep fixed throughout the remainder.
For $G=(v,\psi)\in\domG$
let $[G]^{\Gamma}_{\gamma^*}=G'=(v'_a,\psi'_a)_{a\in\mcla}$ be given by $\mcla=[m]\dotcup[n]$,
$G'_{[m]}=G$ and $(v'_a,\psi'_a)=(a,\gamma^*)$ for $a\in[n]$, i.e.~we attach the unary weight $\gamma^*$ to each variable.
Similarly, for $\psiI\in\domPsiI[\facsI]$, $\domPsiI=[\psibl,\psibu]^q$, let $[G]^\lrarr_{\mI,\psiI}=G'=(v'_a,\psi'_a)_{a\in\mcla}$ be given by $\mcla=[m]\dotcup\facsI$, $G'_{[m]}=G$ and $(v'_a,\psi'_a)=(i,\psiI_{i,h})$ for $a=(i,h)\in\facsI$, i.e.~to each variable $i\in[n]$ we attach $\mI_i$ unary interpolation weights $\psiI_{i,h}$, $h\in[\mI_i]$.
Finally, the pinned graph is $[G]^\darr_{\setP,\sigmaP}=G'=(v'_a,\psi'_a)_{a\in\mcla}$ given by 
$\mcla=[m]\dotcup\setP$, $G'_{[m]}=G$ and the unary wires-weight pairs $(i,\opm[,[q],\sigmaP(i)])$ for $i\in\setP$.

For $G'=[G]^\Gamma$ we let $[G']^\lrarr$ be the graph obtained from $G'$ by attaching interpolators analogously to the above, and also define other combinations analogously.
Further, we define the combined operators analogously, e.g.~$[G]^{\Gamma\downarrow}$ attaches external fields and pins.

For the interpolation weight in the null model let $(\psiRa,\hR,\gammaR)\dequal\lawpsi\otimes\unif([k])\otimes\pi^{\otimes k}$ and
\begin{align}\label{psiIR_def}
\psiIRa[,\lawpsi,\gamma^*,\pi]:[q]\rightarrow[\psibl,\psibu],\,
\sigma\mapsto\sum_{\tau\in[q]^k}\bmone\{\tau_{\hR}=\sigma\}\psiRa(\tau)\prod_{h\neq\hR}\gammaR(\tau_h).
\end{align}
With $\mfkp=(\lawpsi,\gamma^*,\pi,n,m,\mI,\setP)$ the null model $\GR_{\mfkp}=[\wR]^{\Gamma\lrarr\darr}_{\psiIR,\sigmaPR}$ is given by
$(\wR,\psiIR,\sigmaPR)\dequal\wR\otimes\psiIR\otimes\sigmaPR$,
$\wR_{\lawpsi,n,m}=(\vR_{n,m},\psiR_{\lawpsi,m})\dequal\wRa^{\otimes m}$,
$\wRa[,\lawpsi,n]=(\vRa[,n],\psiRa[,\lawpsi])\dequal\unif([n]^k)\otimes\lawpsi$,
$\psiIR_{\lawpsi,\gamma^*,\pi,n,\mI}\dequal\psiIRa[][\otimes\facsI]$ and
$\sigmaPR_n\dequal\unif([q]^n)$.
The standard weight, Gibbs measure, partition function and free entropy density of $G=[(v,\psi)]^{\Gamma\lrarr\darr}$ are
\begin{align*}
\psiG[,G](\sigma)&=\gamma^{*\otimes n}(\sigma)\bmone\{\sigma_{\setP}=\sigmaP_{\setP}\}
\prod_{a\in[m]}\psi_a(\sigma_{v(a)})\prod_{(i,h)\in\facsI}\psiI_{i,h}(\sigma_i),\\
\lawG[,G](\sigma)&=\frac{\psiG[,G](\sigma)}{\ZG(G)},\,
\sigmaRG[,G]\dequal\lawG[,G],\,
\ZG(G)=\sum_{\sigma\in[q]^n}\psiG[,G](\sigma),\,
\phiG(G)=\frac{1}{n}\ln(\ZG(G)).
\end{align*}
Let $\psiM[,\mfkp](\sigma)=\expe[\psiG[,\GR](\sigma)]$, $\ZM[,\mfkp]=\expe[\ZG(\GR)]$ and $\phiM[,\mfkp]=\expe[\phiG(\GR)]$ be the expectations, and
$\GTS_{\mfkp}(\sigma)$ the teacher-student model given by the $(\GTS(\sigma),\GR)$-derivative $G\mapsto\psiG[,G](\sigma)/\psiM(\sigma)$.

Now, let $\thetaPR_{\ThetaP}\dequal\unif([0,\ThetaP])$, for $\thetaP\in\reals_{\ge 0}$ let
$\indPRTa[,\theta,n]\in\{0,1\}$ be Bernoulli with success probability $\sucPT[,\theta,n]=\theta/n\in[0,1]$, $\indPRT[,\theta,n]\dequal\indPRTa^{\otimes n}$ and $\setPR_{\ThetaP,n}=\indPRT[,\thetaPR]^{-1}(1)$.
Let $\mR_{\degae,\tI,n}\dequal\Po(\tI\degae n/k)$, $\mIRa[,\degae,\tI]\dequal\Po((1-\tI)\degae)$, $\mIR_{\degae,\tI,n}\dequal\mIRa[][\otimes n]$, and let the joint distribution be given by $(\mR,\mIR,\setPR)\dequal\mR\otimes\mIR\otimes\setPR$.
Let $\degaR_{\degae,\tI,n}=k\mR/n$ be the average degree (with respect to the standard factors) and $\degaIR_{\degae,\tI,n}=\|\mIR\|_1/n$ be the average degree with respect to the interpolators.
Finally, we consider $\sigmaIID_{\gamma^*,n}$ to be independent of $(\mR,\mIR,\setPR)$.
\begin{remark}
Clearly, the expected average degree $\degae$  corresponds to the expected degree $d$ in Section \ref{random_factor_graphs} and to the limit $d^*$ for $\mR^*$ in Section \ref{implications_extensions_related_work}, hence we let $\degae=d=d^*$ in the remainder without further mention.
Similarly, the law $\lawpsi$ is clearly $p$ from Section \ref{random_factor_graphs}, so we let $p=\lawpsi$ in the remainder without further mention.
\end{remark}
\subsubsection{Factor Assignment Distribution}\label{factor_assignment_distribution}
Recall $\ZFa[,\lawpsi]$ from Equation (\ref{ZFaDef}) and $\ZFabu_{\lawpsi}$ from Equation (\ref{ZF_def}).
For $\gamma\in\mclp([q])$ and $\tau\in[q]^k$ let $\lawYgC[,\lawpsi,\gamma]\in\mclp([q]^k)$ be given by
\begin{align}\label{lawYgC_def}
\lawYgC[,\gamma](\tau)=\frac{1}{\ZFa(\gamma)}\psiae(\tau)\prod_{h\in[k]}\gamma(\tau_h).
\end{align}
We will see that $\lawYgC$ is the law of the assignment to a factor induced by $\sigma$ under $\GTS(\sigma)$.
Further, it is clearly closely related to $\psiIRa$.
\begin{observation}\label{obs_fad}
Let $\gamma\in\mclp([q])$ and notice that the following holds.
\begin{alphaenumerate}
\item\label{obs_fad_bounds}
We have $\psibl\le\psiRa,\psiae,\ZFa,\ZFabu\le\psibu$.
\item\label{obs_fad_poly}
The map $\ZFa$ is a $q$-variate polynomial of degree $k$ on a compact set and hence attains $\ZFabu$.
\item\label{obs_fad_lipschitz}
There exists $L_\mfkg\in\reals_{>0}$ such that $\ZFa$ is $L$-Lipschitz.
\item\label{obs_fad_maxbound}
There exists $c_\mfkg\in\reals_{>0}$ such that $\ZFa(\gamma)\ge\ZFabu-c\|\gamma-\gamma^*\|_\mrmtv^2$.
\item\label{obs_fad_lipschitz_lawYgC}
There exists $L_\mfkg\in\reals_{>0}$ such that $\lawYgC:\mclp([q])\rarr\mclp([q]^k)$ is $L$-Lipschitz.
\item\label{obs_fad_rnbound}
There exists $c_\mfkg\in\reals_{>0}$ such that $c\le r\le c^{-1}$ for the $(\lawYgC[,\gamma],\gamma^{\otimes k})$-derivative $r_\gamma$.
\item\label{obs_fad_rnlipschitz}
There exists $L_\mfkg\in\reals_{>0}$ such that $r:\mclp([q])\rarr\reals_{>0}^{[q]^k}$, $\gamma\mapsto r_\gamma$, is $L$-Lipschitz.
\item\label{obs_fad_rnambound}
There exists $c_\mfkg\in\reals_{>0}$ such that $c\le r_*\le c^{-1}$ for the $(\lawYgC[,\gamma]|_*,\gamma)$-derivative $r_{*,\gamma}$.
\item\label{obs_fad_rnamlipschitz}
There exists $L_\mfkg\in\reals_{>0}$ such that $r_*:\mclp([q])\rarr\reals_{>0}^{q}$, $\gamma\mapsto r_{*,\gamma}$, is $L$-Lipschitz.
\item\label{obs_fad_amstar}
We have $\lawYgC[,\gamma^*]|_*=\gamma^*$ and hence $r_{*,\gamma^*}\equiv 1$.
\end{alphaenumerate}
\end{observation}
\begin{proof}
Recall that $\psiRa\in\domPsi$, $\psiae=\expe[\psiRa]$, $\ZFa(\gamma)=\expe[\psiae(\sigmaIID_{\gamma,k})]$ and $\ZFabu=\sup_\gamma\ZFa(\gamma)$ for Part \ref{obs_fad}\ref{obs_fad_bounds}.
For Part \ref{obs_fad}\ref{obs_fad_poly} notice that $\ZFa$ is the restriction of $f:\reals^q\rarr\reals$, $x\mapsto\sum_\tau\psiae(\tau)\prod_hx_{\tau(h)}$ to $\mclp([q])\setle\reals^q$.
Since we need the derivatives anyway, notice that the $\tau$-th partial derivative $f_\tau(x)$ of $f$ for $\tau\in[q]$ at $x\in\reals^q$, using the product rule, is given by
\begin{align}\label{obs_fad_proof_gradient}
f_\tau(x)=\sum_h\sum_{\tau'}\psiae(\tau')\bmone\{\tau'_h=\tau\}\prod_{h'\neq h}x_{\tau'(h')},
\end{align}
so $k\psibl\le f_\tau(\gamma)\le k\psibu$ for $\gamma\in\mclp([q])$.
Now, the (one-dimensional) fundamental theorem of calculus ensures that
\begin{align*}
|\ZFa(\gamma_1)-\ZFa(\gamma_2)|\le k\psibu\|\gamma_1-\gamma_2\|_1=2k\psibu\|\gamma_1-\gamma_2\|_\mrmtv
\end{align*}
for $\gamma\in\mclp([q])^2$. For Part \ref{obs_fad}\ref{obs_fad_maxbound} we compute the Hessian
\begin{align*}
H_{x,\tau}=\sum_{h\in[k]_2}\sum_{\tau'}\psiae(\tau')\bmone\{\tau'_{h}=\tau\}\prod_{h'\not\in\{h_1,h_2\}}x_{h'},\,\tau\in[q]^2.
\end{align*}
This yields $k(k-1)\psibl\le H_{\gamma}\le k(k-1)\psibu$.
So, using that $\ZFa(\gamma^*)=\ZFabu$ is the maximum for $\gamma^*$ in the interior, i.e.~the first derivative vanishes, yields the assertion using the first order Taylor approximation with the Lagrange form of the remainder and $c=2k(k-1)\psibu$.

For Part \ref{obs_fad}\ref{obs_fad_lipschitz_lawYgC}
we use the triangle inequality, boundedness and Lipschitz continuity of $\ZFa$ and Observation \ref{obs_tv} to obtain
\begin{align*}
\|\lawYgC[,\gamma_1]-\lawYgC[,\gamma_2]\|_\mrmtv
\le\frac{1}{2}\left|1-\frac{\ZFa(\gamma_1)}{\ZFa(\gamma_2)}\right|
+\psibu^2\|\gamma_1^{\otimes k}-\gamma_1^{\otimes k}\|_\mrmtv
\le 2k\psibu^2\|\gamma_1-\gamma_2\|_\mrmtv.
\end{align*}
Part \ref{obs_fad}\ref{obs_fad_rnbound} follows from Part \ref{obs_fad}\ref{obs_fad_bounds} with $c=\psibu^2$, Part \ref{obs_fad}\ref{obs_fad_rnlipschitz} from Part \ref{obs_fad}\ref{obs_fad_bounds} and Part \ref{obs_fad}\ref{obs_fad_lipschitz}
since
\begin{align*}
\|r_{\gamma_2}-r_{\gamma_1}\|_1=\sum_\tau\frac{\psiae(\tau)}{\ZFa(\gamma_1)\ZFa(\gamma_2)}|\ZFa(\gamma_1)-\ZFa(\gamma_2)|\le 2kq^k\psibu^4\|\gamma_2-\gamma_1\|_\mrmtv,\,\gamma\in\mclp([q]).
\end{align*}
Part \ref{obs_fad}\ref{obs_fad_rnambound} follows from Part \ref{obs_fad}\ref{obs_fad_rnbound} and $\gamma^{\otimes k}|_*=\gamma$.
For Part \ref{obs_fad}\ref{obs_fad_rnamlipschitz} we use the triangle inequality and Part \ref{obs_fad}\ref{obs_fad_bounds} to get
\begin{align*}
\|r_{*,\gamma_2}-r_{*,\gamma_1}\|_1\le\sum_{\tau,h}\frac{\psibu}{k}\sum_{\tau'\in[q]^{k}}\bmone\{\tau'_h=\tau\}\left|
\frac{\gamma_2^{\otimes[k]\setminus\{h\}}\left(\tau'_{[k]\setminus\{h\}}\right)}{\ZFa(\gamma_2)}-\frac{\gamma_1^{\otimes[k]\setminus\{h\}}\left(\tau'_{[k]\setminus\{h\}}\right)}{\ZFa(\gamma_1)}\right|.
\end{align*}
Relabeling, the triangle inequality and Observation \ref{obs_tv}\ref{obs_tv_prod} yield
\begin{align*}
\|r_{*,\gamma_2}-r_{*,\gamma_1}\|_1\le q\psibu\left(\psibu(k-1)\|\gamma_2-\gamma_1\|_\mrmtv+\psibu^2\left|\ZFa(\gamma_1)-\ZFa(\gamma_2)\right|\right),
\end{align*}
so Part \ref{obs_fad}\ref{obs_fad_lipschitz} completes the proof.
Finally, for Part \ref{obs_fad}\ref{obs_fad_amstar} we recall the partial derivative $f_\tau(\gamma^*)=k\ZFa(\gamma^*)\lawYgC[,\gamma^*]|_*(\tau)/\gamma^*(\tau)=k\ZFa(\gamma^*)r_{*,\gamma^*}(\tau)$ from Equation (\ref{obs_fad_proof_gradient}) and that $\gamma^*$ is a maximizer, so the derivatives in the directions $\opm[,[q],\tau]-\opm[,[q],q]$, $\tau\in[q-1]$, vanish and hence $r_{*,\gamma^*}(\tau)=r_{*,\gamma^*}(q)$ for all $\tau$, which completes the proof.
\end{proof}
\subsubsection{Expectations and Bounds}\label{GRM_bounds_expe}
We derive naive bounds and compute the expectations, which ensures that the teacher-student model is well-defined since $\psiM>0$ and $\phiG$, $\lawG$ are well-defined since $\ZG>0$.
\begin{observation}\label{obs_GRM_expebounds}
Let $M=m+\|\mI\|_1$.
\begin{alphaenumerate}
\item\label{obs_GRM_expeboundsI}
We have $\expe[\psiIRa]\equiv\ZFabu$ and $\psibl\le\psiIRa\le\psibu$.
\item\label{obs_GRM_expeboundsPsiG}
We have $\psibl^{M}\bmone\{\sigma_{\setP}=\sigmaP_{\setP}\}\gamma^{*\otimes n}(\sigma)
\le\psiG[,G](\sigma)\le\psibu^{M}\bmone\{\sigma_{\setP}=\sigmaP_{\setP}\}\gamma^{*\otimes n}(\sigma)$ for $G\in[\domG]^{\Gamma\lrarr\darr}$.
\item\label{obs_GRM_expeboundsZG}
We have $\psibl^{M}\gamma^{*\otimes\setP}(\sigmaP_{\setP})\le\ZG(G)\le\psibu^{M}\gamma^{*\otimes\setP}(\sigmaP_{\setP})$ for $G\in[\domG]^{\Gamma\lrarr\darr}$.
\item\label{obs_GRM_expeboundsLawG}
We have $\psibl^{2M}\prob[\sigmaIID=\sigma|\sigmaIID_\setP=\sigmaP_\setP]
\le\lawG[,G](\sigma)\le\psibu^{2M}\prob[\sigmaIID=\sigma|\sigmaIID_\setP=\sigmaP_\setP]$ for $G\in[\domG]^{\Gamma\lrarr\darr}$.
\item\label{obs_GRM_expeboundsPsiM}
We have $\psiM(\sigma)=q^{-|\setP|}\ZFabu^{\|\mI\|_1}\gamma^{*\otimes n}(\sigma)\ZFa(\gammaN[,\sigma])^m$.
\item\label{obs_GRM_expeboundsZM}
We have $\ZM=q^{-|\setP|}\ZFabu^{\|\mI\|_1}\expe[\ZFa(\gammaN[,\sigmaIID])^m]$.
\end{alphaenumerate}
\end{observation}
\begin{proof}
For Part \ref{obs_GRM_expebounds}\ref{obs_GRM_expeboundsI} we use independence, further $\pi\in\mclp[*,\gamma^*][2]([q])$ for the expectations, $\ZFa(\gamma^*)=\ZFabu$ for the normalization and Observation \ref{obs_fad}\ref{obs_fad_amstar} to obtain
\begin{align*}
\expe[\psiIRa(\tau)]=\sum_h\frac{1}{k}\sum_{\tau'}\bmone\{\tau'_h=\tau\}\psiae(\tau')\prod_{h'\neq h}\gamma^*(\tau'_{h'})
=\ZFabu r_{*,\gamma^*}(\tau)=\ZFabu,\tau\in[q],
\end{align*}
while $\psiIRa\in\domPsiI$ is immediate from Observation \ref{obs_fad}\ref{obs_fad_bounds}.
Part \ref{obs_GRM_expebounds}\ref{obs_GRM_expeboundsPsiG} follows with 
Part \ref{obs_GRM_expebounds}\ref{obs_GRM_expeboundsI} and
Observation \ref{obs_fad}\ref{obs_fad_bounds},
Part \ref{obs_GRM_expebounds}\ref{obs_GRM_expeboundsZG}
follows with Part \ref{obs_GRM_expebounds}\ref{obs_GRM_expeboundsPsiG},
Part \ref{obs_GRM_expebounds}\ref{obs_GRM_expeboundsLawG}
follows with Part \ref{obs_GRM_expebounds}\ref{obs_GRM_expeboundsPsiG} and
Part \ref{obs_GRM_expebounds}\ref{obs_GRM_expeboundsZG}.
Part \ref{obs_GRM_expebounds}\ref{obs_GRM_expeboundsPsiM} follows with independence,
Part \ref{obs_GRM_expebounds}\ref{obs_GRM_expeboundsI} and
the proof of Observation \ref{obs_psiexpenullprod} for the standard factors.
Part \ref{obs_GRM_expebounds}\ref{obs_GRM_expeboundsZM} follows from
Part \ref{obs_GRM_expebounds}\ref{obs_GRM_expeboundsPsiM} by summing over $\sigma$.
\end{proof}
Next, we provide bounds for $\degaR$ corresponding to the bounds for $\mR^*$ in Section \ref{implications_extensions_related_work}, a corollary to Observation \ref{obs_poisson}.
\begin{corollary}\label{cor_dega}
Let $r\in\reals_{\ge 0}$, $\mclb[][\circ]=(\tI\degae-r,\tI\degae+r)$ and
$\mclb[][\lrarr]=((1-\tI)\degae-r,(1-\tI)\degae+r)$.
We have $\|\mIR\|_1\dequal\Po((1-\tI)\degae n)$.
Further, there exists $c_\mfkg\in\reals_{>0}^2$ such that
\begin{align*}
\prob[\degaR\not\in\mclb[][\circ]],\,
\prob[\degaIR\not\in\mclb[][\lrarr]],\,
\expe[\bmone\{\degaR\not\in\mclb[][\circ]\}\degaR],\,
\expe[\bmone\{\degaIR\not\in\mclb[][\lrarr]\}\degaIR]\le c_2\exp\left(-\frac{c_1r^2n}{1+r}\right).
\end{align*}
\end{corollary}
\begin{proof}
We have $\|\mIR\|_1\dequal\Po((1-\tI)\degae n)$ by Observation \ref{obs_poisson}\ref{obs_poisson_bin}.
With $c\in\reals_{>0}^2$ from Observation \ref{obs_poisson}\ref{obs_poisson_prob} we have
\begin{align*}
\prob[|\degaR-\tI\degae|\ge r]=\prob\left[\left|\mR-\frac{\tI\degae n}{k}\right|\ge\frac{rn}{k}\right]\le c_2\exp\left(-\frac{c_1nr^2}{k(\tI\degae+r)}\right)
\end{align*}
and hence the first part follows with $c_1/(k\degabu)$, using $\tI\le 1$ and $\degabu\ge 1$.
For the third part fix $\rho\in\reals_{>0}$ large.
For $r\le\rho/\sqrt{n}$, $c_1>0$ notice that $\exp(-c_1r^2n/(1+r))\ge e^{-c_1\rho^2}$, so for $c_2\ge\degabu e^{c_1\rho^2}$ we have
\begin{align*}
\expe[\bmone\{|\degaR-\tI\degae|\ge r\}\degaR]
\le\expe[\degaR]=\tI\degae\le\degabu\le c_2\exp\left(-\frac{c_1r^2n}{1+r}\right).
\end{align*}
So, let $r\ge\rho/\sqrt{n}$. Using Observation \ref{obs_poisson}\ref{obs_poisson_expe} and the triangle inequality yields
\begin{align*}
\expe[\bmone\{|\degaR-\tI\degae|\ge r\}\degaR]
&=\tI\degae\prob\left[\left|\mR+1-\frac{\tI\degae n}{k}\right|\ge\frac{rn}{k}\right]
\le\degabu\prob\left[\left|\mR-\frac{\tI\degae n}{k}\right|\ge\frac{rn-k}{k}\right]
\end{align*}
and $rn-k\ge\frac{1}{2}rn+\frac{1}{2}\rho-k\ge\frac{1}{2}rn$ using $n\ge 1$, so the first part completes the proof, since the results for $\degaIR$ follow analogously.
\end{proof}
\subsubsection{Independent Factors}\label{GTSM_iid}
We extend Observation \ref{obs_psiexpenullprod} to decorated graphs.
Let $\gamma=\gammaN[,\sigma]$.
Further, let the teacher-student model wires-weight pair $\wTSa[,\lawpsi,n,\sigma]$ be given by the $(\wTSa,\wRa)$-derivative $(v,\psi)\mapsto\psi(\sigma_v)/\ZFa(\gamma)$.
For $\tau\in[q]$ the interpolation weight $\psiITSa[,\lawpsi,\gamma^*,\pi,\tau]$ is given by the $(\psiITSa,\psiIRa)$-derivative $\psi\mapsto\psi(\tau)/\ZFabu$.
Finally, let 
\begin{align*}
(\wTSM[\lawpsi,n,m,\sigma],\psiITSM[\lawpsi,\gamma^*,\pi,n,\mI,\sigma])\dequal\wTSa[][\otimes m]\otimes\bigotimes_{i\in[n]}\psiITSa[,\sigma(i)][\otimes\mI_i].
\end{align*}
\begin{observation}\label{obs_TSM_iid}
We have $\GTS(\sigma)\dequal[\wTSM[\sigma]]^{\Gamma\lrarr\darr}_{\bm a}$ with $\bm a=(\psiITSM[\sigma],\setP,\sigma)$.
\end{observation}
\begin{proof}
Let $\gamma=\gammaN[,\sigma]$ and
$r(v,\psi)=\psi(\sigma_v)/\ZFa(\gamma)$, $r^{\lrarr}_\tau(\psi)=\psi(\tau)/\ZFabu$ denote the derivatives. Observation \ref{obs_GRM_expebounds}\ref{obs_GRM_expeboundsPsiM} shows that $\gamma^{*\otimes n}(\sigma)$ cancels out in the $(\GTSM(\sigma),\GRM)$-derivative and further
\begin{align*}
\prob[\GTSM(\sigma)\in\mcle]=\expe\left[q^{|\setP|}\bmone\{\sigma_{\setP}=\sigmaPR_{\setP}\}\prod_{a}r(\wR_a)\prod_{i,h}r^{\lrarr}_{\sigma(i)}(\psiIR_{i,h})\bmone\left\{[\wR]^{\Gamma\lrarr\darr}_{\psiIR,\setP,\sigmaPR}\in\mcle\right\}\right].
\end{align*}
Recall that $[\cdot]^\darr_{\setP,\sigmaP}$ only depends on $\sigmaP$ through the values $\sigmaP_{\setP}$ to be pinned, so on the event $\sigma_{\setP}=\sigmaPR_{\setP}$ we have 
$[\wR]^{\Gamma\lrarr\darr}_{\psiIR,\setP,\sigmaPR}=[\wR]^{\Gamma\lrarr\darr}_{\psiIR,\setP,\sigma}$. After this substitution we can take the expectation over $\sigmaPR$ due to independence, i.e.~$\expe[q^{|\setP|}\bmone\{\sigma_{\setP}=\sigmaPR_{\setP}\}]=1$.
This completes the proof, due to independence.
\end{proof}
\begin{remark}\label{remark_stgonly}
Observation \ref{obs_TSM_iid} allows to discuss the standard graph $\wTSM$, the interpolators $\psiITSM[\sigma]$ and the pins separately in most situations.
We will make use of this convenient feature to reduce the (notational) complexity and increase the transparency.
For example, in the following we will discuss the law of the standard graph, further notions and properties. This discussion directly applies to $\wTSM$ and in this sense to $\GTSM(\sigma)$.
\end{remark}
\subsubsection{Factor Side Assignments}\label{GTSYM}
For $G=(v,\psi)\in\domG$ let $\tauG[,G,\sigma]=(\sigma_{v(a)})_{a\in[m]}$ be the assignment to the factors induced by $\sigma$ under $G$, and let $\tauTSM[\lawpsi,n,m,\sigma]=\tauG[,\GTSM(\sigma),\sigma]=\tauG[,\wTSM(\sigma),\sigma]$ be the induced ground truth factor assignment.
Using $\gamma=\gammaN[,\sigma]$, let $\domC[,\gamma]=\gamma^{-1}(\reals_{>0})\setle[q]$ and notice that $\tauG[,G,\sigma]\in(\domC[,\gamma]^k)^m$.

On the other hand, let $\tauTSa[,\lawpsi,\gamma]\dequal\lawYgC[,\gamma]$ with $\lawYgC[,\gamma]$ from Section \ref{factor_assignment_distribution} and notice that the support of $\tauTSa[,\gamma]$ is $\domC[,\gamma]^k$ by Observation \ref{obs_fad}\ref{obs_fad_rnbound}.
For $\tau\in\domC[,\gamma]^k$ let
$\wTSYa[,\lawpsi,n,\sigma,\tau]=(\vTSYa,\psiTSYa)\dequal\vTSYa\otimes\psiTSYa$, where
$\vTSYa[,n,\sigma,\tau]\dequal\bigotimes_h\unif(\sigma^{-1}(\tau_h))$ and
$\psiTSYa[,\lawpsi,\tau]$ is given by the $(\psiTSYa,\psiRa)$-derivative $\psi\mapsto\psi(\tau)/\psiae(\tau)$.
Finally, for $\tau\in(\domC[,\gamma]^k)^m$ let
$\wTSYM[\lawpsi,n,m](\sigma,\tau)
\dequal\bigotimes_{a\in[m]}\wTSYa[,\tau(a)]$.
\begin{observation}\label{obs_TSYM}
We have $(\tauTSM[\sigma],\wTSM(\sigma))\dequal(\bm\tau,\wTSYM(\sigma,\bm\tau))$, $\bm\tau\dequal\tauTSa[,\gamma][\otimes m]$, $\gamma=\gammaN[,\sigma]$.
\end{observation}
\begin{proof}
Due to independence we may restrict to $(\tauTSa,\wTSYa[,\tauTSa])$ and $(\sigma_{\vTSa},\wTSa)$, but then the assertion holds since
\begin{align*}
\prob[(\tauTSa,\wTSYa(\tauTSa))\in\mcle]
&=\sum_{\tau,v}\frac{\bmone\{\sigma_{v}=\tau\}}{\ZFa(\gamma)n^k}\expe\left[\psiRa(\tau)\bmone\{(\tau,v,\psiRa)\in\mcle\}\right]
=\prob[(\sigma_{\vTSa},\wTSa)\in\mcle].
\end{align*}
\end{proof}
In words, the law of $\tauTSM(\sigma)$ factorizes and $\wTSM(\sigma)$ conditional to $\tauTSM(\sigma)=\tau$ is given by $\wTSYM(\sigma,\tau)$.
Let $\GTSYM(\sigma,\tau)$ be given by the law of $\GTSM(\sigma)|\tauTSM[\sigma]=\tau$.
\subsubsection{Variable Degrees}\label{var_degrees}
For $\tau\in([q]^k)^m$ and $\tau'\in[q]^k$ let
\begin{align}\label{alphaM_def}
\alphaM[,\tau](\tau')=\frac{|\tau^{-1}(\tau')|}{m}=\frac{1}{m}\left|\left\{a\in[m]:\tau_a=\tau'\right\}\right|.
\end{align}
Notice that $\alphaM$ is not well-defined for $m=0$, but $m\alphaM$ is.
We reserve the preimage notation $\tau^{-1}(\tau')$ for $\tau'\in[q]^k$ (as opposed to $\tau'\in[q]$).
For $G=(v,\psi)\in\domG$ let
$\facsV[,G](i)=\{a\in[m]:i\in v_a([k])\}$
be the (factor) neighborhood of $i\in[n]$ and
$\degF[,G](i)=|\facsV[,G](i)|$ the (factor) degree.

Similarly, let 
$\wiresV[,G](i)=\{(a,h):v_a(h)=i\}$ be the (wire) neighborhood of $i$ and
$\degH[,G](i)=|\wiresV[,G](i)|$ the (wire) degree.
Finally, let $\sucD[,\lawpsi,n,\sigma](\sigma_i)=\prob[i\in\vTSa[,\sigma]([k])]$.
\begin{observation}\label{obs_degsucprob}
Let $i\in[n]$, $\gamma=\gammaN[,\sigma]$, $\mu=\lawYgC[,\gamma]$ and $\eta=\expe[|\vTSa[,\sigma][-1](i)|]$.
\begin{alphaenumerate}
\item\label{obs_degsucprob_supp}
We have $\sucD=1$ if $n=1$ and $\sucD\in(0,1)$ otherwise.
\item\label{obs_degsucprob_expe}
We have $\eta=\frac{k\mu|_*(\sigma_i)}{n\gamma(\sigma_i)}$.
\item\label{obs_degsucprob_bounds}
There exists $c_\mfkg\in\reals_{>0}$ with
$\frac{1}{cn}\le\sucD(\sigma_i)\le\eta\le\frac{c}{n}$ and
$\sucD(\sigma_i)\ge\eta-\frac{c}{n^2}$.
\end{alphaenumerate}
\end{observation}
\begin{proof}
With Observation \ref{obs_TSYM} we obtain $\sucD$ and $\eta$, i.e.
\begin{align*}
\sucD&=\prob[i\in\vTSYa[,\tauTSa]([k])]
=\prob[\exists h\in[k]\,\tauTSa[,h]=\sigma_i,\vTSYa[,\tauTSa](h)=i]
=\frac{\prob[\tauTSa\not\in([q]\setminus\{\sigma_i\})^k]}{n\gamma(\sigma_i)},\\
\eta &=\sum_h\prob[\vTSa[,h]=i]
=\sum_h\prob[\vTSYa[,\tauTSa,h]=i]
=\frac{k\mu|_*(\sigma_i)}{n\gamma(\sigma_i)}.
\end{align*}
With the union bound we have $\sucD(\sigma_i)\le\eta$ and with $c$ from Observation \ref{obs_fad}\ref{obs_fad_rnambound} further $\eta\le kc/n$.
For $n=1$ we clearly have $\sucD=1$, $\eta=k$ and hence both lower bounds hold for $c\ge k-1$.
For $n>1$ we have $n-1\ge n/2$ and hence
\begin{align*}
\sucD
&=\expe\left[\frac{\psiRa(\sigma_{\vRa})}{\ZFa(\gamma)}\bmone\left\{i\in\vRa([k])\right\}\right]
\ge\psibl^2\left(1-\frac{(n-1)^k}{n^k}\right)
=\frac{\psibl^2}{n^k}\sum_{\ell=0}^{k-1}\binom{k}{\ell}(n-1)^\ell\\
&\ge\frac{k\psibl^2(n-1)^{k-1}}{n^k}\ge\frac{k\psibl^2}{2^{k-1}n}.
\end{align*}
This also shows that $\sucD\in(0,1)$ for $n>1$.
The upper bound on the derivative gives
\begin{align*}
\eta-\sucD(\sigma_i)
&\le\psibu^2\expe\left[|\vRa^{-1}(i)|-\bmone\{i\in\vRa([k])\}\right]
=\frac{\psibu^2}{n^k}\left(kn^{k-1}-(n^k-(n-1)^k)\right)\\
&=\frac{\psibu^2}{n^k}\left(k\sum_{\ell=0}^{k-1}\binom{k-1}{\ell}(n-1)^\ell-\sum_{\ell=0}^{k-1}\binom{k}{\ell}(n-1)^\ell\right)\\
&=\frac{\psibu^2}{n^k}\sum_{\ell=0}^{k-1}(k-\ell-1)\binom{k}{\ell}(n-1)^\ell
\le\frac{\psibu^2}{n^2}\sum_{\ell=0}^{k-2}(k-\ell-1)\binom{k}{\ell}.
\end{align*}
\end{proof}
Next, we apply the bounds for the success probabilities to the degrees.
Let
$\bm d^\star_{\mrmw,n,m,\sigma,\tau}(i)\dequal\Bin(km\alphaM[,\tau]|_*(\sigma_i),\frac{1}{n\gammaN[,\sigma](\sigma_i)})$ and $\degFR[,\lawpsi,n,m,\sigma](i)=\Bin(m,\sucD(\sigma_i))$.
Notice that both degrees only depend on $i$ through $\sigma_i$.
\begin{observation}\label{obs_degrees}
Notice that the following holds for $i\in[n]$.
\begin{alphaenumerate}
\item\label{obs_degrees_wire}
We have $\degH[,\wTSYM(\sigma,\tau)](i)
\dequal\bm d^\star_\mrmw(i)$ and
$\expe[\degH[,\wTSM(\sigma)](i)]=m\expe[|\vTSa[,\sigma][-1](i)|]$.
\item\label{obs_degrees_factor}
We have $\degF[,\wTSM(\sigma)](i)\dequal\degFR(i)$ and
$\expe[\degF[,\wTSM(\sigma)](i)]=m\sucD[,\sigma](\sigma_i)$.
\end{alphaenumerate}
\end{observation}
\begin{proof}
Notice that
\begin{align*}
km\alphaM[,\tau]|_*(\sigma_i)
&=\sum_h\sum_{\tau'}\bmone\{\tau'_h=\sigma_i\}\sum_{a}\bmone\{\tau_a=\tau'\}=|\{(a,h):\tau_{a,h}=\sigma_i\}|.
\end{align*}
Hence, Part \ref{obs_degrees}\ref{obs_degrees_wire} and Part \ref{obs_degrees}\ref{obs_degrees_factor} follow from Observation \ref{obs_TSYM} and Observation \ref{obs_TSM_iid}.
\end{proof}
By an abuse of notation we use the shorthands
$\bm d^\star_\mrmw(i)=\degH[,\wTSYM(\sigma,\tau)](i)$ and 
$\degFR(i)=\degF[,\wTSM(\sigma)](i)$.
Combining these observations does not only yield uniform bounds (also in the choice of $\sigma$!), we also obtain the law of the degrees and bounds under the Poisson number of factors, and uniform Lipschitz continuity of the degree in $\sigma$.
\begin{corollary}\label{cor_degbounds}
Let $i\in[n]$, $\gamma=\gammaN[,\sigma]$, $\mu=\lawYgC[,\gamma]$,
$m\alpha=m\alphaM[,\tau]$ and $\me=\tI\degae n/k$.
\begin{alphaenumerate}
\item\label{cor_degbounds_m}
We have $\bm d^*_\mrmf\le\bm d^*_\mrmw$.
Further, there exists $c_\mfkg\in\reals_{>0}$ such that
$\frac{km}{cn}\le\expe[\bm d^*_\mrmf]\le\expe[\bm d^*_\mrmw]\le\frac{ckm}{n}$ and $\expe[\bm d^*_\mrmw]-\expe[\bm d^*_\mrmf]\le\frac{ckm}{n^2}$.
\item\label{cor_degbounds_po}
We have $(\bm d^*_{\mrmf,\mR},\mR-\bm d^*_{\mrmf,\mR})\dequal\Po(\sucD\me,(1-\sucD)\me)$.
Further, there exists $c_\mfkg\in\reals_{>0}$ such that
$\frac{\tI\degae}{c}\le\expe[\bm d^*_{\mrmf,\mR}]\le\expe[\bm d^*_{\mrmw,\mR}]\le c\tI\degae$ and $\expe[\bm d^*_{\mrmw,\mR}]-\expe[\bm d^*_{\mrmf,\mR}]\le\frac{c\tI\degae}{n}$.
\item\label{cor_degbounds_conc}
Using $\delta=\|\gammaN[,\sigma]-\gamma^*\|_\mrmtv$, there exists $c_\mfkg\in\reals_{>0}$ such that
\begin{align*}
\left|\expe[\bm d^*_{\mrmw,m}]-\frac{km}{n}\right|&\le\frac{ckm}{n}\delta,\,
\left|\expe[\bm d^*_{\mrmw,\mR}]-\tI\degae\right|\le c\tI\degae\delta,\\
\left|\expe[\bm d^*_{\mrmf,m}]-\frac{km}{n}\right|&\le\frac{ckm}{n}\left(\delta+\frac{1}{n}\right),\,
\left|\expe[\bm d^*_{\mrmf,\mR}]-\tI\degae\right|\le c\tI\degae\left(\delta+\frac{1}{n}\right).
\end{align*}
\end{alphaenumerate}
\end{corollary}
\begin{proof}
Using $\wTSM=(\vTSM,\psiTSM)$ and $\eta=\expe[|\vTSa[][-1](i)|]$, for Part \ref{cor_degbounds}\ref{cor_degbounds_m} we have
\begin{align*}
\bm d^*_\mrmf=\sum_{a\in[m]}\bmone\{\exists h\in[k]\,\vTSM[a,h]=i\}
\le\sum_{a\in[m]}\left|\vTSM[a][-1](i)\right|,
=\bm d^*_\mrmw,
\end{align*}
further we have $\expe[\bm d^*_\mrmf]=m\sucD$ and $\expe[\bm d^*_\mrmw]=m\eta$ from Observation \ref{obs_degrees}, hence the bounds follow from Observation \ref{obs_degsucprob} by rescaling with $k\ge 2$.
The law in Part \ref{cor_degbounds}\ref{cor_degbounds_po} follows from Observation \ref{obs_degrees}\ref{obs_degrees_factor} and Observation \ref{obs_poisson}\ref{obs_poisson_bin}. The bounds are obtained by taking expectations in Part \ref{cor_degbounds}\ref{cor_degbounds_m}.
For Part \ref{cor_degbounds}\ref{cor_degbounds_conc} we use Observation \ref{obs_degsucprob}\ref{obs_degsucprob_expe}, Observation \ref{obs_fad}\ref{obs_fad_amstar} and Observation \ref{obs_fad}\ref{obs_fad_rnamlipschitz} to obtain
\begin{align*}
\left|\expe[\bm d^*_{\mrmw,m}]-\frac{km}{n}\right|
=\frac{km}{n}\left|\frac{\lawYgC[,\gammaN[,\sigma]]|_*(\sigma_i)}{\gammaN[,\sigma](\sigma_i)}-\frac{\lawYgC[,\gamma^*]|_*(\sigma_i)}{\gamma^*(\sigma_i)}\right|
\le\frac{ckm}{n}\delta.
\end{align*}
The remainder is now immediate from Jensen's inequality and Part \ref{cor_degbounds}\ref{cor_degbounds_m}.
\end{proof}
\subsubsection{Neighborhood Decomposition}\label{known_neighborhoods}
Let $\mcld[-]=([n]\setminus\{i\})^k$ be the factor neighborhoods excluding $i$ and $\mcld[+]=[n]^k\setminus\mcld[-]$ the neighborhoods covering $i$.
For $n>1$ let $\wR_{-\circ,\lawpsi,n,i}=(\vR_{-\circ,n,i},\psiR_{-\circ,\lawpsi})\dequal\unif(\mcld[-])\otimes\lawpsi$ and let $\wR^*_{-\circ,\lawpsi,n,i,\sigma}$ be given by the $(\wR^*_{-\circ},\wR_{-\circ})$-derivative $(v,\psi)\mapsto\psi(\sigma_v)/\expe[\psiR_{-\circ}(\sigma_{\vR_{-\circ}})]$.
Let $\wR_{+\circ,\lawpsi,n,i}=(\vR_{+\circ,n,i},\psiR_{+\circ,\lawpsi})\dequal\unif(\mcld[+])\otimes\lawpsi$ and let $\wR^*_{+\circ,\lawpsi,n,i,\sigma}$ be given by the $(\wR^*_{+\circ},\wR_{+\circ})$-derivative $(v,\psi)\mapsto\psi(\sigma_v)/\expe[\psiR_{+\circ}(\sigma_{\vR_{+\circ}})]$.
For given $d\in\ints\cap[0,m]$ let
\begin{align*}
(\wTS_{-,\lawpsi,n,m-d,i,\sigma},\wTS_{+,\lawpsi,n,d,i,\sigma})
\dequal\wTS[\otimes(m-d)]_{-\circ}\otimes\wTS[\otimes d]_{+\circ}.
\end{align*}
For given $\mcla\in\binom{[m]}{d}$
let $\alpha_+:[d]\rarr\mcla$, $\alpha_-:[m-d]\rarr[m]\setminus\mcla$ be the enumerations, and let $\wTS_{\mrma,\mcla}=(\wTS_{\mrma,\mcla}(a))_{a\in[m]}$ be given by $\wTS_{\mrma,\mcla}(a)=\wTS_{+}(\alpha_+^{-1}(a))$ for $a\in\mcla$ and $\wTS_{\mrma,\mcla}(a)=\wTS_{-}(\alpha_-^{-1}(a))$ for $a\in[m]\setminus\mcla$.
Finally, let
$\facsVRgD[,m,d]=\unif(\binom{[m]}{d})$ and
$\wTS_{\mrmd,\lawpsi,n,m,i,d,\sigma}=\wTS_{\mrma,\facsVRgD}$.
\begin{observation}\label{obs_known}
We have $\wTS_{\mrmd}(i,\degFR(i),\sigma)\dequal\wTSM(\sigma)$.
\end{observation}
\begin{proof}
With $\wTSM=(\vTSM,\psiTSM)$, $\bmcla^*=\{a\in[m]:i\in\vTSM[a]([k])\}$ and
$\bm b^*=(\bmone\{a\in\bmcla^*\})_a$
we have $\bm b^*\dequal\bm b^{*\otimes m}_\circ$, where $\bm b^*_\circ\in\{0,1\}$ is given by the success probability $\sucD(\sigma_i)$.
So, for $b\in\{0,1\}^m$ using $p_a=\prob[\bm b^*_\circ=b_a]$, $\wR=(\vR,\psiR)$ from Section \ref{random_decorated_graphs} and
$\bm b=(\bmone\{i\in\vR_a([k])\})_a$ we have
\begin{align*}
\prob[\wTSM\in\mcle|\bmcla=\mcla]
=\expe\left[\prod_a\frac{\psiR_a(\sigma_{\vR(a)})\bmone\{\bm b_a=b_a\}}{\ZFa(\gammaN[,\sigma])p_a}\bmone\{\wR\in\mcle\}\right].
\end{align*}
Now, the $(\wR_{-\circ},\wRa)$-derivative is $(v,\psi)\mapsto\bmone\{i\not\in v([k])\}/\prob[i\not\in\vRa([k])]$, so the $(\wR^*_{-\circ},\wRa)$-derivative is
$(v,\psi)\mapsto\bmone\{i\not\in v([k])\}\psi(\sigma_v)/\expe[\bmone\{i\not\in\vRa([k])\}\psiRa(\sigma_{\vRa})]$. On the other hand, for any $a\in b^{-1}(0)$ we have
\begin{align*}
\ZFa(\gammaN[,\sigma])p_a
=\ZFa(\gammaN[,\sigma])\expe\left[\frac{\psiRa(\sigma_{\vRa})}{\ZFa(\gammaN[,\sigma])}\bmone\{i\not\in\vRa([k])\}\right]
=\expe[\bmone\{i\not\in\vRa([k])\}\psiRa(\sigma_{\vRa})].
\end{align*}
Similarly, the $(\wR_{+\circ},\wRa)$-derivative is $(v,\psi)\mapsto\bmone\{i\in v([k])\}\psi(\sigma_v)/\expe[\bmone\{i\in\vRa([k])\}\psiRa(\sigma_{\vRa})]$ and
$\ZFa(\gammaN[,\sigma])p_a
=\expe[\bmone\{i\in\vRa([k])\}\psiRa(\sigma_{\vRa})]$ for $a\in b^{-1}(1)$.
This shows that $\wTSM|\bmcla^*=\mcla$ and $\wTSM[\mrma,\mcla]$ have the same law.

Next, notice that $\wTSM\dequal\wTSM\circ\alpha$ for any permutation $\alpha\in[m]_m$ of the factors, which yields $\bmcla^*\dequal\alpha(\bmcla^*)$, hence that $\bmcla^*||\bmcla^*|=d$ is uniform and thereby has the same law as $\facsVRgD$.
This shows that $\wTSM||\bmcla^*|=d$ has the same law as $\wTSM[\mrmd]$ and thereby completes the proof.
\end{proof}
\begin{remark}\label{remark_known_noi}
Notice that $\wTSa[,-\circ]$ does not depend on $\sigma_i$ due to the definition of $\vR_{-\circ}$.
\end{remark}
\subsubsection{Standard Graphs}\label{standard_graphs}
In this section we discuss the relation of the decorated factor graphs from Section \ref{random_decorated_graphs} and the standard factor graphs from Section \ref{random_factor_graphs}.
For this purpose let $\GR_\mrmd$, $\GTS_\mrmd$ denote the decorated graphs and $\GR_\mrms$, $\GTS_\mrms$ the standard graphs.
\begin{observation}\label{obs_standard_graphs}
Assume that $\ThetaP=0$ and $\tI=1$.
Let
$\GRM[\mrmd]=\GRM[\mrmd,\mR,\mIR,\setPR]$,
$\GTSM[\mrmd]=\GTSM[\mrmd,\mR,\mIR,\setPR]$,
$\GRM[\mrms]=\GRM[\mrms,\mR]$ and
$\GTSM[\mrms]=\GTSM[\mrms,\mR]$.
Then we have $\GRM[\mrmd]\dequal[\GR_{\mrms}]^\Gamma$, 
$\GTSM[\mrmd](\sigma)\dequal[\GTSM[\mrms](\sigma)]^\Gamma$,
$\psiG[,\GRM[\mrmd]](\sigma)\dequal\gamma^{*\otimes n}(\sigma)\psi_{\GRM[\mrms]}(\sigma)$,
$\ZG(\GRM[\mrmd])\dequal Z_{\gamma^*}(\GR_{\mrms})$,
$\ZG(\GTSM[\mrmd](\sigma))\dequal Z_{\gamma^*}(\GTSM[\mrms](\sigma))$ and
$\lawG[,\GRM[\mrmd]]\dequal \mu_{\gamma^*,\GRM[\mrms]}$,
$\lawG[,\GTSM[\mrmd](\sigma)]\dequal \mu_{\gamma^*,\GTSM[\mrms](\sigma)}$.
\end{observation}
\begin{proof}
Notice that $\mIR\equiv 0$ and $\setPR=\emptyset$.
Further, we have $\GRM[\mrmd]\dequal[\GRM[\mrms]]^\Gamma$ by definition, hence
$\psiG[,\GRM[\mrmd]](\sigma)\dequal\gamma^{*\otimes n}(\sigma)\psi_{\GRM[\mrms]}(\sigma)$ and $\expe[\psiG[,\GRM[\mrmd]](\sigma)|\mR]=\gamma^{*\otimes n}(\sigma)\expe[\psi_{\GRM[\mrms]}(\sigma)|\mR]$, so
as in the proof of Observation \ref{obs_TSM_iid} $\gamma^{*\otimes n}(\sigma)$ cancels out in the $(\GTSM(\sigma),\GRM)$-derivative and thereby $\GTSM[\mrmd](\sigma)\dequal[\GTSM[\mrms](\sigma)]^\Gamma$.
Finally, notice that $Z_{\gamma^*}(G)=\sum_\sigma\gamma^{*\otimes n}(\sigma)\psi_G(\sigma)=\sum_\sigma\psi_{[G]^\Gamma}(\sigma)=\ZG([G]^\Gamma)$.
The remainder follows analogously.
\end{proof}
Notice that the result for $\phiG$, $\phi$ is implied.
\subsection{The Nishimori Ground Truth}\label{nishimori_ground_truth}
In this section we discuss the Nishimori ground truth $\sigmaNIS$ from Section \ref{ps_nishimori} and its relation to $\sigmaIID$.
In Section \ref{gtnis_decorated_graphs} we show that $\sigmaNIS$ satisfies the Nishimori condition for the decorated graph and prove Proposition \ref{proposition_mutual_contiguity}\ref{proposition_mutual_contiguity_nishi}.
In Section \ref{ground_truths} we discuss the color frequencies
$\gammaIID_{\gamma^*,n}=\gammaN[,\sigmaIID]$, $\gammaNIS_{\lawpsi,\gamma^*,n,m}=\gammaN[,\sigmaNIS]$ and the conditional laws $\sigmaIID|\gammaIID$, $\sigmaNIS|\gammaNIS$, including the proof of Proposition \ref{proposition_mutual_contiguity}\ref{proposition_mutual_contiguity_rncond}, the upper bound in 
Proposition \ref{proposition_mutual_contiguity}\ref{proposition_mutual_contiguity_rnup} and the lower bound in Proposition \ref{proposition_mutual_contiguity}\ref{proposition_mutual_contiguity_rndown}.
Finally, in Section \ref{gtnis_coupling} we bound the total variation distance of Nishimori ground truths for different values of $m$.
\subsubsection{Decorated Graphs}\label{gtnis_decorated_graphs}
Recall the $(\sigmaNIS,\sigmaIID)$-derivative $\hat r_{\lawpsi,\gamma^*,n,m}$ from Section \ref{ps_nishimori}.
\begin{observation}\label{obs_nishicond}
Notice that the following holds.
\begin{alphaenumerate}
\item\label{obs_nishicond_inv}
We have $\prob[\sigmaNIS=\sigma]=\psiM(\sigma)/\ZM$ and $\hat r(\sigma)=\ZFa(\gammaN[,\sigma])^m/\expe[\ZFa(\gammaN[,\sigmaIID])^m]$.
\item\label{obs_nishicond_joint}
The Radon-Nikodym derivative of $(\sigmaNIS,\GTSM(\sigmaNIS))$ with respect to $\sigmaIID\otimes\GRM$ is
\begin{align*}
(\sigma,G)\mapsto\frac{\psiG[,G](\sigma)}{\gamma^{*\otimes n}(\sigma)\ZM}.
\end{align*}
\item\label{obs_nishicond_GNIS}
The $(\GTSM(\sigmaNIS),\GRM)$-derivative is $G\mapsto\ZG(G)/\ZM$.
\item\label{obs_nishicond_dec}
We have $(\sigmaNIS,\GTSM(\sigmaNIS))\dequal(\sigmaRG[,\GTSM(\sigmaNIS)],\GTSM(\sigmaNIS))$.
\end{alphaenumerate}
\end{observation}
\begin{proof}
With $\GRM[\mrms]$ denoting the standard graph and using
Observation \ref{obs_standard_graphs} we have
\begin{align*}
\prob[\sigmaNIS=\sigma]
=\frac{\gamma^{*\otimes n}(\sigma)\expe[\psi_{\GRM[\mrms]}(\sigma)]}{\expe[Z_{\gamma^*}(\GRM[\mrms])]}
=\frac{\expe[\psiG[,[\GRM[\mrms]]^\Gamma](\sigma)]}{\expe[\ZG([\GRM[\mrms]]^\Gamma)]},
\end{align*}
i.e.~the ratio of the expectations for the decorated factor graph without interpolators and pins.
But with Observation \ref{obs_GRM_expebounds}\ref{obs_GRM_expeboundsPsiM} and
Observation \ref{obs_GRM_expebounds}\ref{obs_GRM_expeboundsZM} this gives
\begin{align*}
\prob[\sigmaNIS=\sigma]
=\frac{\gamma^{*\otimes n}(\sigma)\ZFa(\gammaN[,\sigma])^m}{\expe[\ZFa(\gammaN[,\sigmaIID])^m]}
=\frac{\psiM(\sigma)}{\ZM}.
\end{align*}
Using Part \ref{obs_nishicond}\ref{obs_nishicond_inv}, $(\sigmaIID,\GRM)\dequal\sigmaIID\otimes\GRM$ and for an event $\mcle$ we have
\begin{align*}
\prob[(\sigmaNIS,\GTSM(\sigmaNIS))\in\mcle]&=\expe\left[\frac{\psiM(\sigmaIID)}{\gamma^{*\otimes n}(\sigmaIID)\ZM}\frac{\psiG[,\GRM](\sigmaIID)}{\psiM(\sigmaIID)}\bmone\{(\sigmaIID,\GRM)\in\mcle\}\right].
\end{align*}
This shows that the $(\GTSM(\sigmaNIS),\GRM)$-derivative is $\sum_\sigma\gamma^{*\otimes n}(\sigma)\frac{\psiG[,G](\sigma)}{\gamma^{*\otimes n}(\sigma)\ZM}=\frac{\ZG(G)}{\ZM}$.
This also shows Part \ref{obs_nishicond}\ref{obs_nishicond_dec} since the joint derivative is the product of the individual derivatives. 
\end{proof}
This establishes Proposition \ref{proposition_mutual_contiguity}\ref{proposition_mutual_contiguity_nishi}
with Observation \ref{obs_standard_graphs} and Observation \ref{obs_nishicond}\ref{obs_nishicond_dec} since $\lawG[,[G]^\Gamma]=\mu_{\gamma^*,G}$ and $G\mapsto[G]^\Gamma$ is a bijection.
\begin{remark}\label{remark_nishixlnx}
Observation \ref{obs_nishicond}\ref{obs_nishicond_GNIS} establishes that $\expe[\phiG(\GTSM(\sigmaNIS))]=\frac{1}{n}\expe[\xlnx(\ZG(\GR))]/\expe[Z(\GR)]$.
\end{remark}
\subsubsection{Ground Truths}\label{ground_truths}
Since both ground truths $\sigmaIID$, $\sigmaNIS$ are invariant to decorations we assume that $\mI\equiv 0$ and $\setP=\emptyset$ in this section.
Using $\gamma=\gammaN[,\sigma]$ let $\mcls=\{\sigma'\in[q]^n:\gammaN[,\sigma']=\gamma\}$ and $\sigmaRgG[,n,\gamma]\dequal\unif(\mcls)$.
\begin{observation}\label{obs_gtiid}
Notice that the following holds.
\begin{alphaenumerate}
\item\label{obs_gtiid_cond}
We have $(\gammaIID,\sigmaIID)\dequal(\gammaIID,\sigmaRgG[,\gammaIID])$.
\item\label{obs_gtiid_prob}
There exist $c_\mfkg\in\reals_{>0}^2$ with $\prob[\|\gammaIID-\gamma^*\|_\mrmtv\ge r]\le c_2e^{-c_1r^2n}$ for $r\in\reals_{\ge 0}$.
\item\label{obs_gtiid_expe}
There exists $c\in\reals_{>0}$ such that $\expe[\|\gammaIID-\gamma^*\|_\mrmtv]\le c/\sqrt{n}$.
\item\label{obs_gtiid_var}
There exists $c\in\reals_{>0}$ such that $\expe[\|\gammaIID-\gamma^*\|_\mrmtv^2]\le c/n$.
\end{alphaenumerate}
\end{observation}
\begin{proof}
The first part is clear, further the union bound with Hoeffding's inequality yields
\begin{align*}
\prob[\|\gammaIID-\gamma^*\|_\mrmtv\ge r]
\le\prob[\exists\tau\in[q]|\gammaIID(\tau)-\gamma^*(\tau)|\ge 2r/q]
\le 2q\exp\left(-\frac{8}{q^2}r^2n\right).
\end{align*}
For the next part we have $\expe[\|\gammaIID-\gamma^*\|_\mrmtv]\le\int_0^\infty c_2e^{-c_1r^2n}\mathrm{d}r=\frac{c_2\sqrt{\pi}}{2\sqrt{c_1n}}$.
Similarly, for the last part we have
$\expe[\|\gammaIID-\gamma^*\|_\mrmtv^2]\le\int_0^\infty c_2e^{-c_1rn}\mathrm{d}r=\frac{c_2}{c_1n}$.
\end{proof}
Next, we establish the bounds for $\psiM$ and $\ZM$.
\begin{lemma}\label{lemma_firstmom}
Notice that the following holds for $m\le\mbu$.
\begin{alphaenumerate}
\item\label{lemma_firstmomZ}
There exists $c_\mfkg\in\reals_{>0}$ such that $c\ZFabu^m\le\ZM\le\ZFabu^m$.
\item\label{lemma_firstmomPsi}
There exists $c_\mfkg\in\reals_{>0}$ with
$\exp(-c\|\gammaN[,\sigma]-\gamma^*\|_\mrmtv^2n)\ZFabu^m\gamma^{*\otimes n}(\sigma)\le\psiM(\sigma)\le\ZFabu^m\gamma^{*\otimes n}(\sigma)$.
\end{alphaenumerate}
\end{lemma}
\begin{proof}
Observation \ref{obs_GRM_expebounds}\ref{obs_GRM_expeboundsPsiM} yields $\psiM(\sigma)=\gamma^{*\otimes n}(\sigma)\ZFa(\gamma)^m\le\gamma^{*\otimes n}(\sigma)\ZFabu^m$, so $\ZM\le\ZFabu^m$, where $\gamma=\gammaN[,\sigma]$.
Next, with $\tilde c$ from Observation \ref{obs_fad}\ref{obs_fad_maxbound}
let $c_\mfkg=8\degabu\tilde c\xlnx(\psibu)/k$
and $\delta=\|\gammaN[,\sigma]-\gamma^*\|_\mrmtv$.
For $\delta^2\ge 1/(2\tilde c\psibu)$ we have
\begin{align*}
\frac{\psiM(\sigma)}{\gamma^{*\otimes n}(\sigma)\ZFabu^m}
\ge\left(\frac{\psibl}{\psibu}\right)^{\mbu}
=\psibu^{-4\degabu n/k}=\exp\left(-\frac{c}{2\tilde c\psibu}n\right)\ge e^{-c\delta^2n}.
\end{align*}
For $\delta^2\le 1/(2\tilde c\psibu)$ we have
$\frac{\ZFa(\gamma)}{\ZFabu}\ge 1-\frac{\tilde c\delta^2}{\psibl}\ge\frac{1}{2}$.
For $t\in\reals_{\ge 1/2}$ we have $\ln(t)\ge(2t-3)(1-t)$
and thereby
\begin{align*}
\ln\left(\frac{\ZFa(\gamma)}{\ZFabu}\right)
\ge-\left(1+\frac{2\tilde c\delta^2}{\psibl}\right)\frac{\tilde c\delta^2}{\psibl}
\ge-\frac{2\tilde c\delta^2}{\psibl},\,
\frac{\psiM(\sigma)}{\ZFabu^m\gamma^{*\otimes n}(\sigma)}
\ge\exp\left(-\frac{2\mbu\tilde c\delta^2}{\psibl}\right)
\ge e^{-c\delta^2n}.
\end{align*}
So, by Jensen's inequality and with $\tilde c$ from Observation \ref{obs_gtiid}\ref{obs_gtiid_var} we have
\begin{align*}
\frac{\ZM}{\ZFabu^m}\ge\expe\left[\exp\left(-c\|\gammaIID-\gamma^*\|_\mrmtv^2n\right)\right]
\ge\exp\left(-c\expe[\|\gammaIID-\gamma^*\|_\mrmtv^2]n\right)
\ge\exp(-c\tilde c).
\end{align*}
\end{proof}
Now, we prove the remainder of Proposition \ref{proposition_mutual_contiguity} and translate Observation \ref{obs_gtiid} to $\sigmaNIS$.
\begin{corollary}\label{cor_mutcont}
Notice that the following holds for $m\le\mbu$.
\begin{alphaenumerate}
\item\label{cor_mutcont_rnbu}
There exists $c_\mfkg\in\reals_{>0}$ such that $\hat r\le c$.
\item\label{cor_mutcont_rnbl}
There exists $c_\mfkg\in\reals_{>0}$ such that $\hat r(\sigma)\ge\exp(-c\|\gammaN[,\sigma]-\gamma^*\|_\mrmtv^2n)$.
\item\label{cor_mutcont_cond}
We have $(\gammaNIS,\sigmaNIS)\dequal(\gammaNIS,\sigmaRgG[,\gammaNIS])$.
\item\label{cor_mutcont_prob}
There exist $c_\mfkg\in\reals_{>0}^2$ with $\prob[\|\gammaNIS-\gamma^*\|_\mrmtv\ge r]\le c_2e^{-c_1r^2n}$ for $r\in\reals_{\ge 0}$.
\item\label{cor_mutcont_expe}
There exists $c\in\reals_{>0}$ such that $\expe[\|\gammaNIS-\gamma^*\|_\mrmtv]\le c/\sqrt{n}$.
\item\label{cor_mutcond_var}
There exists $c\in\reals_{>0}$ such that $\expe[\|\gammaNIS-\gamma^*\|_\mrmtv^2]\le c/n$.
\end{alphaenumerate}
\end{corollary}
\begin{proof}
With Observation \ref{obs_nishicond}\ref{obs_nishicond_inv}, $\tilde c$ from Lemma \ref{lemma_firstmom}\ref{lemma_firstmomZ} and Lemma \ref{lemma_firstmom}\ref{lemma_firstmomPsi} we have $\hat r\le \tilde c^{-1}$.
With $c$ from Lemma \ref{lemma_firstmom}\ref{lemma_firstmomPsi} and Lemma 
\ref{lemma_firstmom}\ref{lemma_firstmomZ} we have $\hat r(\sigma)\ge\exp(-c\|\gammaN[,\sigma]-\gamma^*\|_\mrmtv^2n)$.
The next part is immediate from the result for $\hat r$ in Observation \ref{obs_nishicond}\ref{obs_nishicond_inv}, which also shows that $\psiM$ is invariant to permutations of $\sigma$.
The last parts follow from Part \ref{cor_mutcont}\ref{cor_mutcont_rnbu} applied to Observation \ref{obs_gtiid}.
\end{proof}
\subsubsection{Coupling Nishimori Ground Truths}\label{gtnis_coupling}
Since $\sigmaNIS$ is invariant to decorations we assume that $\mI\equiv 0$ and $\setP=\emptyset$.
In this section we derive a bound for $\|\sigmaNIS_{m+1}-\sigmaNIS_m\|_\mrmtv$, which then extends to any $\sigmaNIS_m$, $\sigmaNIS_{\tilde m}$ using the triangle inequality.
\begin{observation}\label{obs_niscoupling}
Notice that the following holds for $m\le\mbu$.
\begin{alphaenumerate}
\item\label{obs_niscoupling_rn}
There exists $c_\mfkg\in\reals_{>0}$ with $1-c\|\gammaN[,\sigma]-\gamma^*\|_\mrmtv^2\le r(\sigma)\le 1+\frac{c}{n}$ for the $(\sigmaNIS_{m+1},\sigmaNIS_m)$-derivative $r$.
\item\label{obs_niscoupling_tv}
There exists $c_\mfkg\in\reals_{>0}$ such that $\|\sigmaNIS_{m+1}-\sigmaNIS_m\|_\mrmtv\le c/n$.
\end{alphaenumerate}
\end{observation}
\begin{proof}
With Observation \ref{obs_GRM_expebounds}\ref{obs_GRM_expeboundsPsiM} and $c'$ from Observation \ref{obs_fad}\ref{obs_fad_maxbound} we obtain
\begin{align*}
\frac{\psiM[,m+1](\sigma)}{\psiM[,m](\sigma)}=\ZFa(\gammaN[,\sigma])\in[\ZFabu-c'\|\gammaN[,\sigma]-\gamma^*\|_\mrmtv^2,\ZFabu].
\end{align*}
This equality and $c''$ from Corollary \ref{cor_mutcont}\ref{cor_mutcond_var} further yield
\begin{align*}
\frac{\ZM[,m+1]}{\ZM[,m]}=\expe\left[\ZFa(\gammaNIS_m)\right]
\ge\ZFabu-c'\expe\left[\|\gammaNIS-\gamma^*\|_\mrmtv^2\right]
\ge\ZFabu-\frac{c'c''}{n}
\end{align*}
and the upper bound $\ZFabu$.
Hence, we have
\begin{align*}
\frac{\prob[\sigmaNIS_{m+1}=\sigma]}{\prob[\sigmaNIS_{m}=\sigma]}
\ge 1-\frac{c'}{\ZFabu}\|\gammaN[,\sigma]-\gamma^*\|_\mrmtv^2
\ge 1-\frac{c'}{\psibu}\|\gammaN[,\sigma]-\gamma^*\|_\mrmtv^2.
\end{align*}
For $n\le n_\circ$ with $n_\circ=2c'c''/\ZFabu\le 2c'c''\psibu$ we use Observation \ref{obs_nishicond}\ref{obs_nishicond_inv} to obtain
\begin{align*}
\frac{\prob[\sigmaNIS_{m+1}=\sigma]}{\prob[\sigmaNIS_{m}=\sigma]}
\le\psibu^{2\mbu}
\le 1+\frac{n_\circ\psibu^{4\degabu n_\circ/k}}{n}
\le 1+\frac{2c'c''\exp(1+\frac{8}{k}c'c''\xlnx(\psibu)\degabu)}{n}.
\end{align*}
For $n\ge n_\circ$ we use the bounds above and $1/(1-t)\le1+2t$ for $t\in[0,1/2]$ to obtain
\begin{align*}
\frac{\prob[\sigmaNIS_{m+1}=\sigma]}{\prob[\sigmaNIS_{m}=\sigma]}
\le\frac{1}{1-\frac{c'c''}{\ZFabu n}}
\le 1+\frac{2c'c''}{\ZFabu n}\le 1+\frac{2c'c''\psibu}{n},
\end{align*}
which completes the proof of Part \ref{obs_niscoupling}\ref{obs_niscoupling_rn}.
Combining this result with Observation \ref{obs_tv}\ref{obs_tv_norm} gives
\begin{align*}
\|\sigmaNIS_{m+1}-\sigmaNIS_m\|_\mrmtv
=\frac{1}{2}\expe\left[\left|r(\sigmaNIS_m)-1\right|\right]
\le\frac{c}{2}\left(\expe\left[\|\gammaNIS-\gamma^*\|_\mrmtv^2\right]+\frac{1}{n}\right),
\end{align*}
which completes the proof using Corollary \ref{cor_mutcont}\ref{cor_mutcond_var}.
\end{proof}
This completes the discussion of the Nishimori ground truth $\sigmaNIS$.
\subsubsection{Ground Truth Given the Graph}\label{ground_truth_given_graph}
In this section we consider arbitrary choices of $\mI$ and $\setP$.
Due to the Nishimori condition \ref{obs_nishicond}\ref{obs_nishicond_dec} the Nishimori ground truth $\sigmaNIS$ conditional to $\GTSM(\sigmaNIS)$ has the same distribution as the Gibbs spins $\sigmaG$. Hence, we only need to discuss the kernel for $\sigmaIID$ given $\GTSM(\sigmaIID)$.
For this purpose let $r_{\mrmg,\sigma}(G)=\psiG[,G](\sigma)/\psiM(\sigma)$ be the $(\GTSM(\sigma),\GRM)$-derivative, $r_\mrmg^*(G)=\expe[r_{\mrmg,\sigmaIID}(G)]$ and for $G\in[\domG]^{\Gamma\lrarr\darr}$ let $\sigmaIID_{\mrmg,G}\in[q]^n$ be given by the $(\sigmaIID_{\mrmg,G},\sigmaIID)$-derivative $r_{\mrms,G}(\sigma)=r_{\mrmg,\sigma}(G)/r_\mrmg^*(G)$.
\begin{observation}\label{obs_condiid}
Let $\GTSM=\GTSM(\sigmaIID)$ and $M=m+\|\mI\|_1$.
\begin{alphaenumerate}
\item\label{obs_condiid_bound}
The $(\GTSM,\GRM)$-derivative is $r_\mrmg^*$ with $\psibl^{2M}(q\psibl)^{|\setP|}\le r_\mrmg^*\le\psibu^{2M}q^{|\setP|}$.
\item\label{obs_condiid_law}
We have
$(\sigmaIID,\GTSM)\dequal(\sigmaIID_{\mrmg,\GTSM},\GTSM)$.
\end{alphaenumerate}
\end{observation}
\begin{proof}
For $G=[(v,\psi)]^{\Gamma\lrarr\darr}$ and using Observation \ref{obs_GRM_expebounds} we have
\begin{align*}
r_\mrmg^*(G)
=\expe\left[\prod_a\frac{\psi_a(\sigmaIID_{v(a)})}{\ZFa(\gammaIID)}
\prod_{(i,h)\in\facsI}\frac{\psiI_{i,h}(\sigma_i)}{\ZFabu}
\prod_{i\in\setP}\frac{\bmone\{\sigmaIID_i=\sigmaP_i\}}{q^{-1}}\right]
\ge\psibl^{2M}(q\psibl)^{|\setP|}
\end{align*}
using $\gamma^{*\otimes\setP}(\sigmaP_\setP)\ge\psibl^{|\setP|}$
and the upper bound follows analogously with
$\gamma^{*\otimes\setP}(\sigmaP_\setP)\le 1$.
For the second part with $(\sigmaIID,\GRM)\dequal\sigmaIID\otimes\GRM$ we have
\begin{align*}
\prob\left[(\sigmaIID_{\mrmg,\GTSM},\GTSM)\in\mcle\right]
=\expe\left[r^*_\mrmg(\GRM)r_{\mrms,\GRM}(\sigmaIID)\bmone\left\{(\sigmaIID,\GRM)\in\mcle\right\}\right]
=\prob\left[(\sigmaIID,\GTSM)\in\mcle\right].
\end{align*}
\end{proof}
\subsubsection{Gibbs Spins}
In this section we consider arbitrary choices of $\mI$ and $\setP$.
Due to the Nishimori condition \ref{obs_nishicond}\ref{obs_nishicond_dec} we have $\sigmaNIS\dequal\sigmaRG[,\GTSM(\sigmaNIS)]$. Hence, we only need to discuss
$\sigmaRG[,\GTSM(\sigmaIID)]$.
\begin{observation}\label{obs_gibbsTSIID}
Let $\gammaR=\gammaN[,\sigmaR]$ with $\sigmaR=\sigmaRG[,\GTSM(\sigmaIID)]$ and $m\le\mbu$.
\begin{alphaenumerate}
\item\label{obs_gibbsTSIID_prob}
There exists $c_\mfkg\in\reals_{>0}^2$ such that $\prob[\|\gammaR-\gamma^*\|_\mrmtv\ge r]\le c_2e^{-c_1r^2n}$.
\item\label{obs_gibbsTSIID_expe}
There exists $c_\mfkg\in\reals_{>0}$ such that
$\expe[\|\gammaR-\gamma^*\|_\mrmtv]\le c/\sqrt{n}$.
\item\label{obs_gibbsTSIID_var}
There exists $c_\mfkg\in\reals_{>0}$ such that
$\expe[\|\gammaR-\gamma^*\|_\mrmtv^2]\le c/n$.
\end{alphaenumerate}
\end{observation}
\begin{proof}
Let $c^*\in\reals_{>0}^2$ be from Observation \ref{obs_gtiid}\ref{obs_gtiid_prob}, $\hat c\in\reals_{>0}^2$ be from Corollary \ref{cor_mutcont}\ref{cor_mutcont_prob} and $c'\in\reals_{>0}$ from Corollary \ref{cor_mutcont}\ref{cor_mutcont_rnbl}.
With $r^*=\sqrt{\frac{\hat c_1}{2c'}}r$ we have
\begin{flalign*}
\prob\left[\|\gammaR-\gamma^*\|_\mrmtv\ge r\right]
&\le e^{c'r^{*2}n}\prob\left[\|\gammaNIS-\gamma^*\|_\mrmtv\ge r,\|\gammaNIS-\gamma^*\|_\mrmtv< r^*\right]+c^*_2e^{-c^*_1r^{*2}n}\\
&\le\exp\left(\frac{1}{2}\hat c_1r^2n\right)\hat c_2e^{-\hat c_1r^2n}+c^*_2\exp\left(-\frac{c^*_1\hat c_1}{2c'}r^2n\right),
\end{flalign*}
which completes the proof with $c_2=\hat c_2+c^*_2$ and $c_1=\min(\hat c_1/2,c^*_1\hat c_1/(2c'))$. The remainder is completely analogous to the proof of Observation \ref{obs_gtiid}.
\end{proof}
\begin{remark}
We could also discuss $\GTSM[\mIR,\setPR](\sigmaIID)$ here to obtain permutation invariance of the posterior. Further, the expected color frequencies of the Gibbs spins could also go here.
\end{remark}
\subsubsection{Relative Entropies}\label{relative_entropies}
In this section we compare the various assignments using relative entropies.
\begin{observation}\label{obs_DKLgt}
Let $\GTSM=\GTSM(\sigmaIID)$ and $m\le\mbu$.
\begin{alphaenumerate}
\item\label{obs_DKLgt_NISIID}
There exists $c_\mfkg\in\reals_{>0}$ with $\DKL(\sigmaNIS\|\sigmaIID)\le c$.
\item\label{obs_DKLgt_IIDNIS}
There exists $c_\mfkg\in\reals_{>0}$ with $\DKL(\sigmaIID\|\sigmaNIS)\le c$.
\item\label{obs_DKLgt_IIDCG}
There exists $c_\mfkg\in\reals_{>0}$ such that $\expe[\expe[\DKL(\sigmaIID_{\mrmg,\GTSM}\|\sigmaRG[,\GTSM])|\GTSM]]\le c$.
\end{alphaenumerate}
\end{observation}
\begin{proof}
With $c'$ from Corollary \ref{cor_mutcont}\ref{cor_mutcont_rnbu} we have
$\DKL(\sigmaNIS\|\sigmaIID)=\expe[\ln(\hat r(\sigmaNIS))]\le\ln(c')$.
With $c'$ from Corollary \ref{cor_mutcont}\ref{cor_mutcont_rnbl} and $c''$ from Observation \ref{obs_gtiid}\ref{obs_gtiid_var} we have 
\begin{align*}
\DKL(\sigmaIID\|\sigmaNIS)=\expe\left[\ln\left(\hat r(\sigmaIID)^{-1}\right)\right]
\le c'n\expe\left[\|\gammaIID-\gamma^*\|_\mrmtv^2\right]
\le c'c''.
\end{align*}
Using the definitions, the $(\sigmaRG[,\GTSM],\sigmaIID_{\mrmg,\GTSM})$-derivative $r_{\mrms,G}$ can be composed of
\begin{align*}
r_{\mrms,G}(\sigma)
=\frac{\psiM(\sigma)r_\mrmg^*(G)}{\psiG[,G](\sigma)}
\cdot\frac{\psiG[,G](\sigma)}{\gamma^{*\otimes n}(\sigma)\ZG(G)}
=\expe\left[\frac{\hat r(\sigma)}{\hat r(\sigmaRG[,G])}\right].
\end{align*}
Now, with $c'$ from Corollary \ref{cor_mutcont}\ref{cor_mutcont_rnbu}, $c''$ from Corollary \ref{cor_mutcont}\ref{cor_mutcont_rnbl} and $\gammaR_{\mrmg,G}=\gammaN[,\sigmaRG[,G]]$ we have
\begin{align*}
r_{\mrms,G}(\sigma)\ge\exp\left(-c''\|\gammaN[,\sigma]-\gamma^*\|_\mrmtv^2 n\right)/c'.
\end{align*}
Hence, with the tower property, Observation \ref{obs_condiid} and $c'''$ from Observation \ref{obs_gtiid}\ref{obs_gtiid_var} we have
\begin{align*}
\expe[\expe[\DKL(\sigmaIID_{\mrmg,\GTSM}\|\sigmaRG[,\GTSM])|\GTSM]]
=\expe[-\ln(r_{\mrms,\GTSM}(\sigmaIID))]
\le\ln(c')+c''c'''.
\end{align*}
\end{proof}
\begin{remark}
Maybe we can also obtain the other relative entropy (but we don't need it).
\end{remark}
\subsection{Concentration and Continuity}\label{phi_concon}
In this section we prove Proposition \ref{proposition_phi_concon} and related results.
In Section \ref{phi_lipschitz} we establish boundedness and Lipschitz continuity of the free entropy on the factor graph level for general decorated graphs, yielding Proposition \ref{proposition_phi_concon}\ref{proposition_phi_concon_bu}.

In Section \ref{phiM_continuity} we establish Lipschitz continuity of $\expe[\phiG(\GRM)]$ for $\mI\equiv 0$ and $\setP=\emptyset$. Then we establish Lipschitz continuity for
$\expe[\phiG(\GTSYM[m,\mIR,\setPR](\sigma,\tau))]$ and
$\expe[\phiG(\GTSM[m,\mIR,\setPR](\sigma))]$ in Section \ref{phiTS_continuity}, yielding Proposition \ref{proposition_phi_concon}\ref{proposition_phi_concon_cont}.
In Section \ref{phiTSIIDNIS_asymptotics} and for $\mI\equiv 0$ and $\setP=\emptyset$ we show that
$\expe[\phiG(\GTSM[\mR](\sigmaNIS_{\mR}))]=\expe[\phiG(\GTSM[\mR](\sigmaIID))]+o(1)$ which explains why using $\sigmaNIS$ for Proposition \ref{proposition_int} is reasonable, and we further show that $\expe[\phiG(\GTSM[\mR^*](\sigmaIID))]=\expe[\phiG(\GTSM[\mR](\sigmaIID))]+o(1)$, which supports the corresponding claim in Section \ref{implications_extensions_related_work} regarding $\mR^*$ and Theorem \ref{thm_bethe}.
Based on these results we then establish concentration for $\mI\equiv 0$, $\setP=\emptyset$ in Section \ref{concentration}, yielding Proposition \ref{proposition_phi_concon}\ref{proposition_phi_concon_conc}.
\subsubsection{The Free Entropy}\label{phi_lipschitz}
Let $G=[w]^{\Gamma\lrarr\darr}_{\mI,\psiI,\setP,\sigma}$ with $w=(v,\psi)\in\domG_{n,m}$ and
$\tilde G=[\tilde w]^{\Gamma\lrarr\darr}_{\mItilde,\psiItilde,\setPtilde,\tilde\sigma}$
with $\tilde w=(\tilde v,\tilde \psi)\in\domG_{n,\tilde m}$.
Let $\mclv[1][\downarrow]=[n]\setminus(\setP\cup\setPtilde)$ be the unpinned variables, 
$\mclv[2][\downarrow]=\{i\in\setP\cap\setPtilde:\sigma_i=\tilde\sigma_i\}$ the variables pinned to the same value, and $\mclv[][\downarrow]=\mclv[1][\downarrow]\cup\mclv[2][\downarrow]$.
Further, let $m_\cap=\min(m,\tilde m)$, $\mI_{\cap}=(\min(\mI_i,\mItilde_i))_i$, $\facsI_\cap=\{(i,h):i\in[n],h\in[m_{\cap,i}]\}$ and
\begin{align*}
\mcla[=]&=\left\{a\in[m_\cap]:w_a=\tilde w_a,v_a([k])\setle\mclv[][\darr]\right\},\,
\facsI_{=}&=\left\{(i,h)\in\facsI_{\cap}:\psiI_{i,h}=\psiItilde_{i,h},i\in\mclv[][\darr]\right\}.
\end{align*}
Now, let $D=m-m_\cap+\sum_i(\mI_i-\mI_{\cap,i})$, $\tilde D=\tilde m-m_\cap+\sum_i(\mItilde_i-\mI_{\cap,i})$ be the excess factors, $D_\cap=m_\cap-|\mcla[=]|+|\facsI_\cap\setminus\facsI_=|$ the bad factors and let $\distG(G,\tilde G)=D+\tilde D+2D_\cap+n-|\mclv[][\darr]|$ be the distance of $G$ and $\tilde G$.
\begin{observation}\label{obs_phi_lipschitz}
There exists $c_\mfkg\in\reals_{>0}$ such that  $|\phiG(G)-\phiG(\tilde G)|\le\frac{c}{n}\distG(G,\tilde G)$ and $|\phiG(G)|\le\frac{c}{n}(m+\|\mI\|_1+|\setP|)$.
\end{observation}
\begin{proof}
Let $G$, $\tilde G$ be as in the definition of $\distG$.
First, we get rid of the excess factor and the bad factors, i.e.
\begin{align*}
\ZG(G)\ge\psibl^{D+D_\cap}\expe\left[\bmone\{\sigmaIID_{\setP\setminus\mclv[2][\darr]}=\sigma_{\setP\setminus\mclv[2][\darr]},\sigmaIID_{\mclv[2][\darr]}=\sigma_{\mclv[2][\darr]}\}\prod_{a\in\mcla[=]}\psi_a(\sigmaIID_{v_a})\prod_{(i,h)\in\facsI_=}\psiI_{i,h}(\sigmaIID_i)\right].
\end{align*}
Now, all but the first part of the indicator only depends on $\sigmaIID_{\mclv[][\darr]}$, so we can use independence, $\gamma^*\ge\psibl$ and further transition to $\tilde G$, i.e.
\begin{align*}
\ZG(G)\ge\psibl^{D+D_\cap+|\setP\setminus\mclv[2][\darr]|}\expe\left[\bmone\{\sigmaIID_{\mclv[2][\darr]}=\tilde\sigma_{\mclv[2][\darr]}\}\prod_{a\in\mcla[=]}\tilde\psi_a(\sigmaIID_{\tilde v_a})\prod_{(i,h)\in\facsI_=}\psiItilde_{i,h}(\sigmaIID_i)\right].
\end{align*}
This clearly gives $\phiG(G)\ge\frac{-\ln(\psibu)}{n}\distG(G,\tilde G)+\phiG(\tilde G)$, the upper bound follows analogously and hence the first part of the assertion holds with $c=\ln(\psibu)$. The second part holds due to Observation \ref{obs_GRM_expebounds}\ref{obs_GRM_expeboundsZG} and $\gamma^*\ge\psibl$.
\end{proof}
Proposition \ref{proposition_phi_concon}\ref{proposition_phi_concon_bu} follows from Observation \ref{obs_standard_graphs} applied to Observation \ref{obs_phi_lipschitz}.
\subsubsection{Continuity for the Null Model}\label{phiM_continuity}
In this section we establish Proposition \ref{proposition_phi_concon}\ref{proposition_phi_concon_cont} for the null model, implied by the following result for the decorated graph version.
Recall $\phiM=\expe[\phiG(\GRM[m])]$ from Section \ref{random_decorated_graphs}.
\begin{lemma}\label{lemma_contphiM}
There exists $L_\mfkg\in\reals_{>0}$ such that
$|\phiM(m_1)-\phiM(m_2)|\le L|\frac{km_1}{n}-\frac{km_2}{n}|$ for $m\in\ints_{\ge 0}^2$, $\mI\equiv 0$ and $\setP=\emptyset$.
\end{lemma}
\begin{proof}
For $\mI\equiv 0$ and $\setP=\emptyset$ we have $\distG(G,\tilde G)=m+\tilde m-2|\mcla[=]|$ in Section \ref{phi_lipschitz}.
Assume without loss of generality that $m_1\le m_2$ and consider the canonical coupling of $\GRM(m_1)=[\wR_{m_1}]^\Gamma$ and $\GRM(m_2)=[\wR_{m_2}]^\Gamma$, i.e.~$\wR_{m_1}=\wR_{m_2,[m_1]}$. Under this coupling we have $\mcla[=]=[m_1]$ and hence $\distG(\GRM(m_1),\GRM(m_2))=m_2-m_1$, so Jensen's inequality and Observation \ref{obs_phi_lipschitz} yield
\begin{align*}
|\phiM(m_1)-\phiM(m_2)|\le\expe\left[\left|\phiG(\GRM(m_1))-\phiG(\GRM(m_2))\right|\right]
\le\frac{c'}{n}(m_2-m_1)
=L\left|\frac{km_1}{n}-\frac{km_2}{n}\right|
\end{align*}
with $L=c'/k$, and thereby complete the proof.
\end{proof}
Proposition \ref{proposition_phi_concon}\ref{proposition_phi_concon_cont} for the null model follows from Observation \ref{obs_standard_graphs} applied to Lemma \ref{lemma_contphiM}.
\subsubsection{Continuity for the Teacher-Student Model}\label{phiTS_continuity}
In this section we establish a version of Proposition \ref{proposition_phi_concon}\ref{proposition_phi_concon_cont} for the expected free entropy $\phiTSYM(m,\sigma,\tau)=\expe[\phiG(\GTSYM[m,\mIR,\setPR](\sigma,\tau))]$ over the two-sided planted model and the more general decorated graphs.
The result for the teacher-student model then follows as a corollary.
\begin{lemma}\label{lemma_contphiTSYM}
Let $\gammaN[,\sigma]\ge\frac{1}{2}\psibl$, $m\le\mbu$, further $\tilde m\in\ints_{\ge 0}$, $\tilde\sigma\in[q]^n$ and $\tilde\tau\in(\domC[,\tilde\gamma]^k)^{\tilde m}$. There exists $L_\mfkg\in\reals_{>0}$ such that
\begin{align*}
\left|\phiTSYM(m,\sigma,\tau)-\phiTSYM(\tilde m,\tilde\sigma,\tilde\tau)\right|
\le\frac{L}{n}\left(\|n\gammaN[,\sigma]-n\gammaN[,\tilde\sigma]\|_1+\|m\alphaM[,\tau]-\tilde m\alphaM[,\tilde\tau]\|_1\right).
\end{align*}
\end{lemma}
\begin{proof}
Let $\gamma=\gammaN[,\sigma]$, $\tilde\gamma=\gammaN[,\tilde\sigma]$,
$\alpha=\alphaM[,\tau]$ and $\tilde\alpha=\alphaM[,\tilde\tau]$.
First, we show that
$\phiTSYM(m,\sigma,\tau)=\phiTSYM(\tilde m,\tilde\sigma,\tilde\tau)$ for the special case that 
$\tilde m=m$, $\tilde\gamma=\gamma$ and $\tilde\alpha=\alpha$, i.e.~there exist permutations $\nu\in[n]_n$ and $\mu\in[m]_m$ such that $\tilde\sigma\circ\nu=\sigma$ and $\tilde\tau\circ\mu=\tau$.
Similar to the proof of Observation \ref{obs_known} we consider a permutation $\mu$ of the factors, and moreover a permutation $\nu$ of the variables.
For given $(v,\psi)\in\domG$ let $f(v,\psi)=(\tilde v,\tilde\psi)\in\domG$ be given by $\tilde v_{\mu(a),h}=\nu(v_{a,h})$ and $\tilde\psi_{\mu(a)}=\psi_a$.
Notice that $\wTSYM(\tilde\sigma,\tilde\tau)\dequal f(\wTSYM(\sigma,\tau))$ since $f$ is a simple relabeling of variables and factors.
Further, let $f^\lrarr(\psiI)=\psiItilde$ with $\psiItilde_{\nu(i),h}=\psiI_{i,h}$ and notice that $\psiITSM[\tilde\sigma]\dequal f^{\lrarr}(\psiITSM[\sigma])$.
Finally, using $\mItilde\circ\nu=\mI$ and $\setPtilde=\nu(\setP)$ we have
\begin{align*}
\phiG\left([(v,\psi)]^{\Gamma\lrarr\darr}_{\mI,\psiI,\setP,\sigma}\right)
=\phiG\left([(\tilde v,\tilde\psi)]^{\Gamma\lrarr\darr}_{\mItilde,\psiItilde,\setPtilde,\tilde\sigma}\right),
\end{align*}
i.e.~the free entropy is invariant to a relabeling of factors and variables.
This shows that
\begin{align*}
\phiG\left(\GTSYM[m,\mItilde,\setPtilde](\tilde\sigma,\tilde\tau)\right)
\dequal\phiG\left(\left[f(\wTSYM(\sigma,\tau))\right]^{\Gamma\lrarr\darr}_{\mItilde,f^{\lrarr}(\psiITSM[\sigma]),\setPtilde,\tilde\sigma}\right)
=\phiG\left(\GTSYM[m,\mI,\setP](\sigma,\tau)\right).
\end{align*}
Since both $\mIR$ and $\setPR=\indPRT[,\thetaPR]^{-1}(1)$ are obtained from i.d.d.~random variables (given $\thetaPR$), we have
$\GTSYM[m,\mIR\circ\nu,\nu(\setPR)](\tilde\sigma,\tilde\tau)\dequal\GTSYM[m,\mIR,\setPR](\tilde\sigma,\tilde\tau)$ and thereby
$\phiTSYM(m,\sigma,\tau)=\phiTSYM(m,\tilde\sigma,\tilde\tau)$.
This completes the proof of the special case and in particular shows that $\phiTSYM(m,n\gamma,m\alpha)=\phiTSYM(m,\sigma,\tau)$ is well-defined.

Hence, for the general case we assume without loss of generality that $m\le\tilde m$
and that $\sigma$, $\tau$, $\tilde\sigma$ and $\tilde\tau$ are ordered as follows.
Let $n_{\cap\Gamma}=(\min(n\gamma(\tau'),n\tilde\gamma(\tau')))_{\tau'\in[q]}$ and
$n_\cap=\|n_{\cap\Gamma}\|_1$.
Analogously, let $m_{\cap\mathrm{A}}=(\min(m\alpha(\tau'),\tilde m\tilde\alpha(\tau')))_{\tau'\in[q]^k}$ and $m_\cap=\|m_{\cap\mathrm{A}}\|_1$.
We assume that
$\sigma_{[n_\cap]}=\tilde\sigma_{[n_\cap]}$ and $\tau_{[m_\cap]}=\tilde\tau_{[m_\cap]}$.

Next, we consider the following union. Let $n_\cup=n+(n-n_\cap)$ and $\sigma_\cup=(\sigma,\tilde\sigma_{[n]\setminus[n_\cap]})$.
Analogously, let $m_\cup=m+(\tilde m-m_\cap)$ and $\tau_\cup=(\tau,\tilde\tau_{[\tilde m]\setminus[m_\cap]})$.
Let $\nu_1:[n]\rarr[n]$ be the identity, $\nu_2:[n]\rarr[n_\cap]\cup([n_\cup]\setminus[n])$ the enumeration, $\mu_1:[m]\rarr[m]$ the identity and $\mu_2:[\tilde m]\rarr[m_\cap]\cup([m_\cup]\setminus[m])$ the enumeration.
The union graph $\GR_\cup=[\wR_\cup]^{\Gamma\lrarr\darr}_{\mIR_\cup,\psiITSM[\cup],\setPR_\cup,\sigma_\cup}$ is given by
\begin{align*}
(\wR_\cup,\mIR_\cup,\psiITSM[\cup],\setPR_\cup)\dequal
\wR_\cup\otimes(\mIR_\cup,\psiITSM[\cup])\otimes\setPR_\cup,
\end{align*}
where $\wR_\cup\dequal\wTSYM[m_\cup](\sigma_\cup,\tau_\cup)$ and the remainder is given as follows.
The interpolator counts $\mIR_\cup$ are given by $\mIR_{\cup,[n]}\dequal\mIR_n$ and $\mIR_{\cup}\circ\nu_2=\mIR_{\cup}\circ\nu_1$, i.e.~we copy the values to the remaining positions.
Given $\mIR_\cup$ we have $\psiITSM[\cup]\dequal\psiTSM[n_\cup,\sigma_\cup]$  for the interpolation weights.
Similarly, for the pins let $\indPRT[\cup,\thetaP]\in\{0,1\}^{n_\cup}$ be given by $\indPRT[\cup,\thetaP,[n]]\dequal\indPRT[,n,\thetaP]$ and $\indPRT[\cup,\thetaP]\circ\nu_2=\indPRT[\cup,\thetaP]\circ\nu_1$. Further, let $\setPR=\indPRT[\cup,\thetaPR_n]^{-1}(1)$ with $\thetaPR_n\dequal\unif([0,\ThetaP])$ from Section \ref{random_decorated_graphs}.
In words, we obtain $\mIR_\cup$ and $\setPR_\cup$ by choosing the correct distribution on $[n]$ and copying the values to the remainder (yielding the correct distribution there), and then take the law $\GTSYM(\sigma_\cup,\tau_\cup)$.

Given a graph $G_\cup=[(v_\cup,\psi_\cup)]^{\Gamma\lrarr\darr}_{\mI_\cup,\psiI_\cup,\setP_\cup,\sigma_\cup}$ from $\GR_\cup$ and $i\in[2]$, let $\GR_i(G_\cup)=[(\vR,\psi)]^{\Gamma\lrarr\darr}_{\mI,\psiI,\setP,\sigma}$ be given by $\mI=\mI_\cup\circ\nu_i$, $\psiI=\psiI_\cup\circ\nu_i$, $\setP=\nu_i^{-1}(\setP_\cup)$, $\psi=\psi_\cup\circ\mu_i$, $\vR(a,h)=\nu_i^{-1}(v_{\cup}(\mu(a),h))$ if $v_{\cup}(\mu(a),h)\in\nu_i([n])$ and otherwise $\vR(a,h)\dequal\unif(\mcls[i])$ independent of everything else, where $\mcls[i]=\nu_i^{-1}(\mcls[\cup])$ and $\mcls[\cup]=\sigma_\cup^{-1}(\tau_\cup(\mu_i(a),h))$.
Notice that $\tau_\cup(\mu_1(a),h)=\tau(a,h)$ and further $\mcls[1]=\sigma^{-1}(\tau_{a,h})$. 
Analogously, we obtain $\mcls[2]=\tilde\sigma^{-1}(\tilde\tau_{a,h})$.

Now, we claim that $\GR_1(\GR_\cup)\dequal\GTSYM[m,\mIR,\setPR](\sigma,\tau)$ and $\GR_2(\GR_\cup)\dequal\GTSYM[\tilde m,\mIR,\setPR](\tilde\sigma,\tilde\tau)$. Due to the absence of dependencies and by construction it is straightforward to see that the pinning indicators (sets), the interpolator counts, the interpolation weights and the standard weights have the correct distribution, which leaves us with the (standard) neighborhoods.
But using $\GR_\cup=[(\vR_\cup,\psiR_\cup)]^{\Gamma\lrarr\darr}$,
$\GR_1(\GR_\cup)=[(\vR,\psiR)]^{\Gamma\lrarr\darr}$, for $a\in[m]$, $h\in[k]$ and $i\in\sigma^{-1}(\tau_{a,h})$ we have
\begin{align*}
\prob[\vR(a,h)=i]&=\prob[\vR_{\cup}(a,h)=i]+\prob[\vR_\cup(a,h)\not\in\sigma^{-1}(\tau_{a,h}),\vR(a,h)=i]\\
&=\frac{1}{|\sigma_\cup^{-1}(\tau_{a,h})|}+\frac{|\sigma_\cup^{-1}(\tau_{a,h})|-|\sigma^{-1}(\tau_{a,h})|}{|\sigma_\cup^{-1}(\tau_{a,h})|}\cdot\frac{1}{|\sigma^{-1}(\tau_{a,h})|}
=\frac{1}{|\sigma^{-1}(\tau_{a,h})|},
\end{align*}
and thereby also $\vR(a,h)\dequal\unif(\sigma^{-1}(\tau_{a,h}))$ has the correct distribution.
This shows that $\GR_1(\GR_\cup)\dequal\GTSYM[m,\mIR,\setPR](\sigma,\tau)$, and we obtain 
$\GR_2(\GR_\cup)\dequal\GTSYM[\tilde m,\mIR,\setPR](\tilde\sigma,\tilde\tau)$ analogously.

In the next step we want to apply Observation \ref{obs_phi_lipschitz}, hence we have to bound $\distG(\GR_1,\GR_2)$ using $\GR_i=\GR_i(\GR_\cup)=[(\vR_i,\psiR_i)]^{\Gamma\lrarr\darr}_{\mIR_i,\psiTSM[i],\setPR_i,\sigma_\cup\circ\nu_i}$. By construction we have $\mIR=\mIR_1=\mIR_2$ and $\setPR=\setPR_1=\setPR_2$ (almost surely), so $\mclv[1][\darr]=[n]\setminus\setPR$, $\mclv[2][\darr]=\setPR\cap[n_\cap]$, $\mclv[][\darr]=([n]\setminus\setPR)\cup[n_\cap]$, $\min(m,\tilde m)=m$, $m^\lrarr_\cap=\mIR$, $\mcla[\cap][\lrarr]=\mcla[\mIR][\lrarr]$, $D=0$, $\tilde D=\tilde m-m$, $D_\cap=m-|\mcla[=]|+|\mcla[\cap][\lrarr]\setminus\mcla[=][\lrarr]|$ and
$\distG(\GR_1,\GR_2)=\tilde m-m+2D_\cap+|\setPR\setminus[n_\cap]|$.
Notice that $\bmcla_=^\circ\setle\mcla[=]$ with
\begin{align*}
\bmcla_=^\circ=\left\{a\in[m_\cap]:\vR_{\cup,a}([k])\setle[n_\cap]\right\}
\end{align*}
and $\{(i,h):i\in[n_\cap],h\in[\mIR_i]\}\setle\mcla[=][\lrarr]$ by construction, so
\begin{align*}
D_\cap\le m-m_\cap+|\{a\in[m_\cap]:\exists h\in[k]\,\vR_\cup(a,h)>n_\cap\}|+\sum_{i=n_\cap+1}^n\mIR_i.
\end{align*}
Hence, we can upper bound the number of factors by the number of wires, which is then the total degree of the variables $[n_\cup]\setminus[n_\cap]$ with respect to the factors $[m_\cap]$, i.e.
\begin{align*}
D_\cap&\le m-m_\cap+|\{(a,h)\in[m_\cap]\times[k]:\vR_\cup(a,h)>n_\cap\}|+\sum_{i=n_\cap+1}^n\mIR_i\\
&=m-m_\cap+\sum_{i=n_\cap+1}^{n_\cup}\bm d(i)+\sum_{i=n_\cap+1}^n\mIR_i,\\
\bm d(i)&=|\{(a,h)\in[m_\cap]\times[k]:\vR_\cup(a,h)=i\}|.
\end{align*}
Notice that $\bm d(i)$ is exactly the (wire) degree of $i\in[n_\cup]$ in $\wTSYM[m_\cap](\sigma_\cup,\tau_{\cup,[m_\cap]})$, so the discussion in Section \ref{var_degrees} applies.
Further, notice that
\begin{align*}
\distG(\GR_1,\GR_2)\le\tilde m+m-2m_\cap +\sum_{i=n_\cap+1}^{n_\cup}\bm d(i)+\sum_{i=n_\cap+1}^n\mIR_i.
\end{align*}
Now, taking the expectation, using the coupling, Jensen's inequality and $c$ from Observation \ref{obs_phi_lipschitz} yields
\begin{align*}
\left|\phiTSYM(m,\sigma,\tau)-\phiTSYM(\tilde m,\tilde\sigma,\tilde\tau)\right|
&\le\frac{c}{n}\left(\tilde m+m-2m_\cap +\sum_{i=n_\cap+1}^{n_\cup}\expe[\bm d(i)]+\sum_{i=n_\cap+1}^n\expe[\mIR_i]\right).
\end{align*}
By definition we have $\expe[\mIR_i]=(1-\tI)\degae\le\degabu$, and by Observation \ref{obs_degrees}\ref{obs_degrees_wire} we have $\bm d(i)\dequal\Bin(|\mclh|,1/|\sigma_\cup^{-1}(\sigma')|)$ for $i\in[n_\cup]\setminus[n_\cap]$, with $\sigma'=\sigma_\cup(i)$ and $\mclh=\{(a,h)\in[m_\cap]\times[k]:\tau_\cup(a,h)=\sigma'\}$ from the proof of Observation \ref{obs_degrees}\ref{obs_degrees_wire}. This gives
\begin{align*}
\expe[\bm d(i)]=\frac{|\mclh|}{|\sigma_\cup^{-1}(\sigma')|}
\le\frac{km_\cap}{|\sigma^{-1}(\sigma')|}
\le\frac{k\mbu}{n\psibl/2}
=4\degabu\psibu.
\end{align*}
Using $\|n\gamma-n\tilde\gamma\|_1=2(n-n_\cap)=n_\cup-n_\cap$ and $\|m\alpha-\tilde m\tilde\alpha\|_1=m+\tilde m-2m_\cap$ yields
\begin{align*}
\left|\phiTSYM(m,\sigma,\tau)-\phiTSYM(\tilde m,\tilde\sigma,\tilde\tau)\right|
&\le\frac{c}{n}\left(\|m\alpha-\tilde m\tilde\alpha\|_1+(4\degabu\psibu+\degabu)\|n\gamma-n\tilde\gamma\|_1\right),
\end{align*}
and completes the proof with $L=c\degabu(4\psibu+1)$.
\end{proof}
\begin{remark}
We briefly reflect the proof of Lemma \ref{lemma_contphiTSYM}.
\begin{alphaenumerate}
\item
Initially, we discussed permutations. Since we have $\mItilde\neq\mI$ and $\setPtilde\neq\setP$ in general, considering the random quantities $\mIR$, $\setPR$ is convenient.
\item 
In reference to the upcoming Aizenman-Sims-Starr scheme, notice that the coupling construction works because we consider fixed $\tI$, $\degae$, $\ThetaP$ and $n$, i.e.~we have the same type of decorations for $\GR_1(\GR_\cup)$, $\GR_2(\GR_\cup)$.
\end{alphaenumerate}
\end{remark}
Now, we obtain the result for $\phiTSM(m,\sigma)=\expe[\phiG(\GTSM[m,\mIR,\setPR](\sigma))]$ as a corollary.
\begin{corollary}\label{cor_contphiTSM}
Let $\gammaN[,\sigma]\ge\frac{1}{2}\psibl$, $m\le\mbu$, further $\tilde m\in\ints_{\ge 0}$ and $\tilde\sigma\in[q]^n$. There exists $L_\mfkg\in\reals_{>0}$ such that
\begin{align*}
\left|\phiTSM(m,\sigma)-\phiTSM(\tilde m,\tilde\sigma)\right|
\le L\left(\|\gammaN[,\sigma]-\gammaN[,\tilde\sigma]\|_1+\left|\frac{km}{n}-\frac{k\tilde m}{n}\right|\right).
\end{align*}
\end{corollary}
\begin{proof}
Let $\gamma=\gammaN[,\sigma]$, $\tilde\gamma=\gammaN[\tilde\sigma]$ and assume without loss of generality that $m\le\tilde m$.
Using the coupling lemma \ref{obs_tv}\ref{obs_tv_coupling}, fix a coupling $\mu$ of $\lawYgC[,\gamma]$ and $\lawYgC[\tilde\gamma]$ and let $\tauR\dequal\mu^{\otimes\tilde m}$.
With Observation \ref{obs_TSYM} we have
\begin{align*}
\GTSM[m,\mIR,\setPR](\sigma)\dequal\GTSYM[m,\mIR,\setPR](\sigma,\tauR_{1,[m]}),\,
\GTSM[\tilde m,\mIR,\setPR](\sigma)\dequal\GTSYM[\tilde m,\mIR,\setPR](\tilde\sigma,\tauR_{2}).
\end{align*}
Now, with the tower property of the expectation, Jensen's inequality and $L^\star$ from Lemma \ref{lemma_contphiTSYM} we have
\begin{align*}
\left|\phiTSM(m,\sigma)-\phiTSM(\tilde m,\tilde\sigma)\right|
\le\frac{L^\star}{n}\expe\left[\|n\gamma-n\tilde\gamma\|_1+\|m\alphaR-\tilde m\tilde\alphaR\|_1\right]
\end{align*}
with $\alphaR=\alphaM[,\tauR_{1,[m]}]$ and $\tilde\alphaR=\alphaM[,\tauR_{2}]$.
The triangle inequality gives
\begin{align*}
\|m\alphaR-\tilde m\tilde\alphaR\|_1
\le\sum_{\tau'}\left(\sum_{a\in[m]}\left|\bmone\{\tauR_{1,a}=\tau'\}-\bmone\{\tauR_{2,a}=\tau'\}\right|+\sum_{a=m+1}^{\tilde m}\bmone\{\tauR_{2,a}=\tau'\}\right)
\end{align*}
and hence $\expe[\|m\alphaR-\tilde m\tilde\alphaR\|_1]\le 2m\prob[\tauR_{1,1}\neq\tauR_{2,1}]+\tilde m-m=2m\|\lawYgC[,\gamma]-\lawYgC[,\tilde\gamma]\|_\mrmtv+\tilde m-m$, so with $L'$ from Observation \ref{obs_fad}\ref{obs_fad_lipschitz_lawYgC} we have
\begin{align*}
\left|\phiTSM(m,\sigma)-\phiTSM(\tilde m,\tilde\sigma)\right|
\le 2L^\star\|\gamma-\tilde\gamma\|_\mrmtv+\frac{2L'L^\star}{k}\frac{km}{n}\|\gamma-\tilde\gamma\|_1+\frac{L^\star}{k}\left(\frac{k\tilde m}{n}-\frac{km}{n}\right),
\end{align*}
so the assertion holds with $L=\frac{2L^\star}{k}(k+2L'\degabu)$.
\end{proof}
Proposition \ref{proposition_phi_concon}\ref{proposition_phi_concon_cont} for the teacher-student model follows from Observation \ref{obs_standard_graphs} applied to Lemma \ref{cor_contphiTSM} for $\tI=1$ and $\ThetaP=0$.
\subsubsection{Teacher-Student Model Asymptotics}\label{phiTSIIDNIS_asymptotics}
Throughout this section we assume that $\mI\equiv 0$ and $\setP=\emptyset$ for convenience.
We discuss the behavior of the expected free entropies under random factor counts and random ground truths.
For this purpose let $\Gamma^+=(\lceil n\gamma^*(\tau)\rceil)_\tau$, $\Gamma^-=(\lfloor n\gamma^*(\tau)\rfloor)_\tau$, further let $\Gamma\in\ints_{\ge 0}^q$ be such that $\Gamma^-\le\Gamma\le\Gamma^+$ and $\|\Gamma\|_1=n$, so for $\gamma^\circ=\frac{1}{n}\Gamma$ we have $\gamma^\circ\in\mclp([q])$ and $\|\gamma^\circ-\gamma^*\|_\infty\le 1/n$.
Let $\sigma^\circ\in[q]^n$ be the non-decreasing assignment with $\gammaN[,\sigma^\circ]=\gamma^\circ$.
Finally, let $m^\circ=\lfloor\degae n/k\rfloor$ and recall $\mR^*$, $\epsm$, $\deltam$ from the introduction to Section \ref{preparations}.
\begin{corollary}\label{cor_phiTSIIDNIS}
Let $m\le\mbu$, $\mI\equiv 0$, $\setP=\emptyset$ and $\phiTSM[m](\sigma)=\expe[\phiG(\GTSM(\sigma))]$.
\begin{alphaenumerate}
\item\label{cor_phiTSIIDNIS_m}
There exists $c_\mfkg\in\reals_{>0}$ such that $|\expe[\phiTSM[m](\sigmaIID)]-\phiTSM[m](\sigma^\circ)|\le c/\sqrt{n}$ and the same holds for $\sigmaIID$ replaced by $\sigmaNIS$.
\item\label{cor_phiTSIIDNIS_p}
We have $\expe[\phiTSM[\mR](\sigmaIID)]=\phiTSM[m^\circ](\sigma^\circ)+\mclo(\epsm+\deltam+n^{-1/2})$ and the same holds for $\sigmaIID$ replaced by $\sigmaNIS_{\mR}$.
Further, this statement also holds for $\mR$ replaced by $\mR^*$.
\end{alphaenumerate}
\end{corollary}
\begin{proof}
For $n\ge 2\psibu$ we have $\gamma^\circ\ge\psibl/2$.
Hence, with Jensen's inequality, $L$ from Corollary \ref{cor_contphiTSM} and $c^*$ from Observation \ref{obs_gtiid}\ref{obs_gtiid_expe} we have
\begin{align*}
|\expe[\phiTSM[m](\sigmaIID)]-\phiTSM[m](\sigma^\circ)|
\le L\expe[\|\gammaIID-\gamma^\circ\|_\mrmtv]
\le c'/\sqrt{n}
\end{align*}
with $c'=Lc^*$, and the same holds for $\sigmaIID$ replaced by $\sigmaNIS$ and $c'=L\hat c$ with $\hat c$ from Corollary \ref{cor_mutcont}\ref{cor_mutcont_expe}, so Part \ref{cor_phiTSIIDNIS}\ref{cor_phiTSIIDNIS_m} holds with $c=L\max(c^*,\hat c)=L\hat c$ for $n\ge 2\psibu$.
For $n\le 2\psibu$ we take $c'$ from Observation \ref{obs_phi_lipschitz} to obtain 
$|\expe[\phiTSM[m](\sigmaIID)]-\phiTSM[m](\sigma^\circ)|\le 2c'm/n\le 4c'\degabu/k\le c/\sqrt{n}$ with $c=\sqrt{2\psibu}4c'\degabu/k$.

For $\deltam$, $\epsm$ sufficiently large and using Corollary \ref{cor_dega} we may consider $\mR$ to be a special case of $\mR^*$.
With $c$ from Observation \ref{obs_phi_lipschitz} notice that
\begin{align*}
E&=|\expe[\phiTSM[\mR^*](\sigmaNIS_{\mR^*})]-\phiTSM[m^\circ](\sigma^\circ)|
\le\expe\left[\frac{c\mR^*}{n}\right]+\frac{cm^\circ}{n}
\le\frac{c}{k}\epsm+\frac{2c\degabu}{k}+\frac{c\degabu}{k}
\end{align*}
is uniformly bounded for all $n$.
For $n\ge 2\psibu$ recall that $\gamma^\circ\ge\psibl/2$ and $m^\circ\le\mbu$.
Using Jensen's inequality, $L$ as above, $\hat c$ from Corollary \ref{cor_mutcont}\ref{cor_mutcont_expe}, $d^\circ=km^\circ/n$ and the triangle inequality we obtain
$E\le LE_1+LE_2$ with
\begin{align*}
E_1&=\expe[\|\gammaNIS_{\mR^*}-\gamma^\circ\|_\mrmtv]
\le\expe[\|\gammaNIS_{\mR^*}-\gamma^*\|_\mrmtv]+\frac{q}{2n}\\
&\le\expe[\bmone\{|\bm d^*-\degae|\le\deltam\}\|\gammaNIS_{\mR^*}-\gamma^*\|_\mrmtv]
+\epsm+\frac{q}{2n}
\le\frac{\hat c}{\sqrt{n}}+\epsm+\frac{q}{2n},\\
E_2&=\expe[|\bm d^*-d^\circ|]
\le\expe[|\bm d^*-\degae|]+\frac{k}{n}\\
&\le\deltam+\expe[\bmone\{|\bm d^*-\degae|>\deltam\}\bm d^*]+\degae\prob[|\bm d^*-\degae|>\deltam]+\frac{k}{n}
\le\deltam+\epsm+\degabu\epsm+\frac{k}{n}.
\end{align*}
The result for $\sigmaIID$ follows analogously with $\hat c$ replaced by $c^*$ from Observation \ref{obs_gtiid}\ref{obs_gtiid_expe}.
\end{proof}
\subsubsection{Concentration}\label{concentration}
Throughout this section we assume that $\mI\equiv 0$ and $\setP=\emptyset$ for convenience.
First, we establish concentration for the models over i.i.d.~factors.
\begin{lemma}\label{lemma_mcdiarmid}
Let $\mI\equiv 0$, $\setP=\emptyset$ and $m\le\mbu$.
There exists $c_\mfkg\in\reals_{>0}^2$ such that
\begin{align*}
\prob\left[|\phiG(\GRM)-\expe[\phiG(\GRM)]|\ge r\right]\le c_2 e^{-c_1r^2n}
\end{align*}
for $r\in\reals_{\ge 0}$ and the same holds for $\GRM$ replaced by $\GTSYM(\sigma,\tau)$ and $\GTSM(\sigma)$.
\end{lemma}
\begin{proof}
Recall the proof of Lemma \ref{lemma_contphiM}. For $\tilde m=m$ in Section \ref{phi_lipschitz} we have 
\begin{align*}
\distG(G,\tilde G)=2m-2|\mcla[=]|=2|\{a\in[m]:(v_a,\psi_a)\neq(\tilde v_a,\tilde\psi_a)\}|.
\end{align*}
So, for $|\mcla[=]|=m-1$ and $c'$ from Observation \ref{obs_phi_lipschitz} we have $|\phiG(G)-\phiG(\tilde G)|\le\frac{2c'}{n}$.
Since $\phiG(\GRM)=\phiG([\wR]^\Gamma)$ is a function of $m$ i.i.d.~pairs McDiarmid's inequality yields
\begin{align*}
\prob\left[|\phiG(\GRM)-\expe[\phiG(\GRM)]|\ge r\right]
\le 2\exp\left(-\frac{2 r^2}{m\left(\frac{2c'}{n}\right)^2}\right)
\le c_2 e^{-c_1r^2n}
\end{align*}
with $c_2=2$ and $c_1=\frac{k}{4c'^2\degabu}$.
Using Observation \ref{obs_TSM_iid} and Observation \ref{obs_TSYM}, the proofs for $\GTSM(\sigma)$ and $\GTSYM(\sigma,\tau)$ are completely analogous, with the same constants.
\end{proof}
\begin{remark}
This proof extends to any fixed $\setP$ (and $\sigmaP$) since this determines the pinning weights due to fixed $\sigma$, and to not too large $\|\mI\|_1$ analogously to the standard factors.
\end{remark}
Next, we establish concentration for random ground truths.
\begin{lemma}\label{lemma_concphiTSMIIDNIS}
Let $\mI\equiv 0$, $\setP=\emptyset$ and $m\le\mbu$.
There exists $c_\mfkg\in\reals_{>0}^2$ such that
\begin{align*}
\prob\left[|\phiG(\GTSM(\sigmaIID))-\expe[\phiG(\GTSM(\sigmaIID))]|\ge r\right]\le c_2 e^{-c_1r^2n}
\end{align*}
for $r\in\reals_{\ge 0}$ and the same holds for $\sigmaIID$ replaced by $\sigmaNIS$.
\end{lemma}
\begin{proof}
Let $\phiTSM(\sigma)=\expe[\phiG(\GTSM(\sigma))]$ and $\phiTSMe=\expe[\phiTSM(\sigmaIID)]$.
With $c^\circ$ from Corollary \ref{cor_phiTSIIDNIS} let $\rho=3c^\circ$,
and with $L$ from Corollary \ref{cor_contphiTSM} let $n_\circ=\max(2\psibu,(3qL/\rho)^2)$.
In the following we consider the case
$n\ge n_\circ$ and $r\ge\rho/\sqrt{n}$, then the case
$n\le n_\circ$, and finally the case $r\le\rho/\sqrt{n}$.

For $n\ge n_\circ$ and $r\ge\rho/\sqrt{n}$ the following holds.
Using $\|\gamma^\circ-\gamma^*\|_\infty\le 1/n\le1/n^\circ$ we have $\gamma^\circ\ge\psibl/2$ and hence Corollary \ref{cor_contphiTSM} applies and yields $|\phiTSM(\sigmaIID)-\phiTSM(\sigma^\circ)|\le L\|\gammaIID-\gamma^\circ\|_\mrmtv$.
Notice that $\gamma^\circ$ is also sufficiently close to $\gamma^*$ relative to $r$, to be precise we have $\|\gamma^\circ-\gamma^*\|_\mrmtv\le\frac{q}{2n}\le\frac{q\rho}{6qL\sqrt{n}}\le\frac{r}{6L}$. The same holds for the expected free entropy, i.e.~$|\phiTSM(\sigma^\circ)-\phiTSMe|\le\frac{c^\circ}{\sqrt{n}}=\frac{\rho}{3\sqrt{n}}\le\frac{1}{3}r$.
So, using the triangle inequalities suggested by the above yields
\begin{align*}
|\phiG(\GTSM(\sigmaIID))-\phiTSMe|
\le|\phiG(\GTSM(\sigmaIID))-\phiTSM(\sigmaIID)|+L\left(\|\gammaIID-\gamma^*\|_\mrmtv+\frac{r}{6L}\right)+\frac{1}{3}r.
\end{align*}
On $|\phiG(\GTSM(\sigmaIID))-\phiTSMe|\ge r$ we have
$\|\gammaIID-\gamma^*\|_\mrmtv\ge r/(6L)$ or $|\phiG(\GTSM(\sigmaIID))-\phiTSM(\sigmaIID)|\ge r/3$, so with $c_\Gamma^*$ from Observation \ref{obs_gtiid}\ref{obs_gtiid_prob} and $c_{\mrmm}$ from Lemma \ref{lemma_mcdiarmid} we have
\begin{align*}
P&=\prob\left[|\phiG(\GTSM(\sigmaIID))-\expe[\phiG(\GTSM(\sigmaIID))]|\ge r\right]\\
&\le c_{\Gamma,2}\exp\left(-\frac{c_{\Gamma,1}}{36L^2}r^2n\right)+c_{\mrmm,2}\exp\left(-\frac{c_{\mrmm,1}}{9}r^2n\right)
\le c'_2e^{-c_1r^2n}
\end{align*}
with $c'_2=c_{\Gamma,2}+c_{\mrmm,2}$ and $c_1=\min(\frac{c_{\Gamma,1}}{36L^2},\frac{c_{\mrmm,1}}{9})$.
For $n\le n_\circ$ with $c_\uarr$ from Observation \ref{obs_phi_lipschitz} we have
$|\phiG(\GTSM(\sigmaIID))-\phiTSMe|\le c_\uarr\mbu/n=r_\uarr$ with $r_\uarr=2c_\uarr\degabu/k$.
For $r\le r_\uarr$ we have
\begin{align*}
P&\le 1=c_2''\exp\left(-c_1r_\uarr^2n_\circ\right)
\le c''_2e^{-c_1r^2n}
\end{align*}
with $c''_2=\exp(c_1r_\uarr^2n_\circ)$, but for $r>r_\uarr$ we have $P=0\le c''_2e^{-c_1r^2n}$.
For $r\le\rho/\sqrt{n}$ we have
\begin{align*}
P&\le 1=e^{c_1r^2n}e^{-c_1r^2n}\le c'''_2e^{-c_1r^2n}
\end{align*}
with $c'''_2=e^{c_1\rho^2}$. Choosing $c_2=\max(c'_2,c''_2,c'''_2)$ completes the proof, since $c^*_\Gamma$ replaced by $\hat c_\Gamma$ from Corollary \ref{cor_mutcont}\ref{cor_mutcont_prob} yields the analogous result for $\sigmaNIS$.
\end{proof}
Finally, Proposition \ref{proposition_phi_concon}\ref{proposition_phi_concon_conc} follows from Observation \ref{obs_standard_graphs} applied to Lemma \ref{lemma_mcdiarmid} and to Lemma \ref{lemma_concphiTSMIIDNIS}.
\section{The Planted Model Quenched Free Entropy}\label{thm_bethe_proof}
We turn to the proof of Theorem \ref{thm_bethe}.
In Section \ref{pinning} we prove Lemma \ref{lemma_pinning}, apply the results to $\GTSM(\sigmaNIS)$ and $\GTSM(\sigmaIID)$ and show that the effect of the pinning on the quenched free entropy density is asymptotically negligible.

In Section \ref{interpolation} we implement the interpolation method and show Proposition \ref{proposition_int}. Thereafter, we can discard the interpolators once and for all, restricting to $\tI=1$ and $\mI\equiv 0$.
In Section \ref{ass} we implement the Aizenman-Sims-Starr scheme for the simplified model.
Finally, in Section \ref{proof_thm_bethe} we complete the proof.
\subsection{Pinned Measures and Their Marginal Distributions}\label{pinning}
This section is composed of four parts. The first part is dedicated to the proof of Lemma \ref{lemma_pinning}, which is based on \cite{coja2018}, \cite{montanari2008a} and covered by Sections \ref{pinning_conditional_entropy} to \ref{pinning_lemma_proof}. Then we further discuss the pinning of Gibbs measures in the Sections \ref{pinning_gibbs} and \ref{pinning_qfed}. In the third part, Section \ref{pinning_marginal_distributions}, we discuss the marginal distributions of (pinned) measures and prove another proposition for general (pinned) measures.
In the last part, Sections \ref{pinning_gibbs_marginal_distributions} and \ref{pinning_limiting_marginal_distributions}, we apply this proposition to decorated graphs and discuss projections onto $\mclp[*][2]([q])$.

In Section \ref{pinning_conditional_entropy} we introduce the underlying model, the erasure channel, and the corresponding conditional entropy of the assignment.
In Section \ref{pinning_condentder} we take the derivative of the conditional entropy with respect to the pinning probability, yielding the crucial connection to the mutual information.
In Section \ref{pinning_mi} we introduce the generalized mutual information.
Finally, in Section \ref{pinning_lemma_proof} we complete the proof of Lemma \ref{lemma_pinning}.

In Section \ref{pinning_gibbs} we apply the results to the Gibbs measure $\lawG$ of the decorated graphs. In Section \ref{pinning_qfed} we argue that the impact on the quenched free entropy density by adding pins for $\ThetaP=\ThetaP(n)=o(n)$ is asymptotically negligible.

Next, we introduce empirical marginal distributions in Section \ref{pinning_marginal_distributions}. We further introduce a conditional and a reweighted version of the marginal distribution and show that these asymptotically coincide if the empirical color frequencies concentrate and the measure is $\eps$-symmetric, which in particular holds for pinned measures.

In Section \ref{pinning_gibbs_marginal_distributions} we show that the empirical color frequencies of the Gibbs spins concentrate and hence in particular the proposition for general measures applies to the Gibbs measure induced by the graph.
Finally, in Section \ref{pinning_limiting_marginal_distributions} we introduce a projection of $\mclp[][2]([q])$ onto $\mclp[*][2]([q])$, and then show that the distance of the Gibbs marginal distribution to its projection vanishes.
\subsubsection{The Erasure Channel, Conditional Entropy and Random Conditioning}\label{pinning_conditional_entropy}
For $(\bm x,\bm y,\bm z)\in[q]^3$, $q\in\ints_{\ge 2}$,
the cross entropy, the entropy and the relative entropy are
\begin{align*}
H(\bm x\|\bm y)=\sum_x-\prob[\bm x=x]\ln(\prob[\bm y=x]), H(\bm x)=H(\bm x\|\bm x), \DKL(\bm x\|\bm y)=H(\bm x\|\bm y)-H(\bm x)
\end{align*}
respectively.
Notice that the definition of the relative entropy is consistent with the general case from Section \ref{information_theoretic_threshold}, and in particular both the cross entropy and the relative entropy are finite if and only if $\bm x$ is absolutely continuous with respect to $\bm y$.
The conditional cross entropy, the conditional entropy and the conditional relative entropy are
\begin{align*}
H(\bm x\|\bm y|\bm z)=\expe[\expe[H(\bm x\|\bm y)|\bm z]],\,
H(\bm x|\bm z)=H(\bm x\|\bm x|\bm z),\,
\DKL(\bm x\|\bm y|\bm z)=\expe[\expe[\DKL(\bm x\|\bm y)|\bm z]].
\end{align*}
The conditional mutual information is $I(\bm x,\bm y|\bm z)=\DKL(\bm x,\bm y\|\bm x\otimes\bm y|\bm z)$, which is consistent with the definition of the conditional mutual information for the graphical channels in Section \ref{examples}.
For now, we focus on the following conditional entropy.

Let $n\in\ints_{>0}$, $\mu\in\mclp([q]^n)$ and $\bm x_\mu\dequal\mu$ a random vector of values. Further, let $p\in[0,1]^n$ and let $\bm r\in\{0,1\}^n$ be the revealment given by $\bm r\dequal\bigotimes_i\bm r_i$ and Bernoulli variables $\bm r_i$ with success probability $p_i$.
Using the joint distribution $(\bm x,\bm r)\dequal\bm x\otimes\bm r$ let $\bm\chi=(\bm r_i\bm x_i)_i\in[q]_\circ^n$ with $[q]_\circ=[q]\cup\{0\}$ be the partial observation.
This approach reflects \cite{montanari2008a}.

Fix values $x\in[q]^n$, revealments $r\in\{0,1\}^n$ and let $\chi=(r_ix_i)_i$ be the partial observation, and $\mclr=r^{-1}(1)=\chi^{-1}([q])$ the revealed coordinates.
Further, fix known coordinates $\mclk\setle[n]$, tested coordinates $\mclt\setle[n]$ and selected coordinates $\mcls\setle[n]$.
Now, let
\begin{align*}
\eta_{n,\mu,p}(\mcls,x_{\mclk},\chi_{\mclt})
&=H(\bm x_{\mcls}|\bm x_{\mclk}=x_{\mclk},\bm\chi_{\mclt}=\chi_{\mclt}),\\
\overline\eta_{n,\mu,p}(\mcls,\mclk,\mclt)
&=H(\bm x_{\mcls}|\bm x_{\mclk},\bm\chi_{\mclt})
=\expe\left[\eta(\mcls,\bm x_{\mclk},\bm\chi_{\mclt})\right]
\end{align*}
be the (pointwise) entropy and the conditional entropy respectively.
As already indicated by the definition of $\vR$ in Section \ref{ps_pinning} we consider selections with repetition.
Hence, we establish that the definition above is indeed sufficient for our purposes and further derive a few useful basic properties.
\begin{observation}\label{obs_pin_condentbasic}
Notice that the following holds.
\begin{alphaenumerate}
\item\label{obs_pin_condentbasic_triv}
We have
$\eta(\mcls,x_\emptyset,\chi_{\mclt})=H(\bm x_{\mcls}|\bm\chi_{\mclt}=\chi_{\mclt})$,
further $\eta(\mcls,x_{\mclk},\chi_\emptyset)
=H(\bm x_{\mcls}|\bm x_{\mclk}=x_{\mclk})$ and 
$\eta(\emptyset,\cdot,\cdot)=0$.
\item\label{obs_pin_condentbasic_rep}
Let $s,k,t\in\ints_{\ge 0}$, $\sigma\in[n]^s$, $\kappa\in[n]^k$ and $\tau\in[n]^t$ such that $\sigma([s])=\mcls$, $\kappa([k])=\mclk$ and $\tau([t])=\mclt$.
Then we have 
$H(\bm x_{\sigma}|\bm x_{\kappa}=x_{\kappa},\bm\chi_{\tau}=\chi_{\tau})
=\eta(\mcls,x_{\mclk},\chi_{\mclt})$.
\item\label{obs_pin_condentbasic_normpoint}
We have
$\eta(\mcls,x_{\mclk},\chi_{\mclt})
=\eta(\mcls\setminus\mclk,x_{\mclk},\chi_{\mclt\setminus\mclk})
=\eta(\mcls\setminus\mclk[][*],x_{\mclk[][*]},\chi_\emptyset)$ with $\mclk[][*]=\mclk\cup(\mclt\cap\mclr)$.
\item\label{obs_pin_condentbasic_normexpe}
For $\mcls=\mcls[1]\dotcup\mcls[2]$ we have
\begin{align*}
\overline\eta(\mcls,\mclk,\mclt)
=\overline\eta(\mcls[1]\setminus\mclk,\mclk,\mclt\setminus\mclk)
+\overline\eta(\mcls[2]\setminus\mclk,\mclk\cup\mcls[1],\mclt\setminus\mclk).
\end{align*}
\end{alphaenumerate}
\end{observation}
\begin{proof}
Recall well-known properties of the conditional entropy, in particular that $H(\bm a|\bm b)=0$ if and only if $\bm a$ is determined by $\bm b$, and the chain rule.
Further, notice that the conditional entropy is exclusively a function of the laws, and that
\begin{align*}
\prob[\bm a=a,\bm a=a,\bm b=b,\bm c_1=c_1|\bm b=b,\bm b=b,\bm c_2=c_2]=\prob[\bm a=a|\bm b=b]
\end{align*}
whenever $\bm c_1=c_1$, $\bm c_2=c_2$ almost surely.
This shows Part \ref{obs_pin_condentbasic}\ref{obs_pin_condentbasic_triv} and
Part \ref{obs_pin_condentbasic}\ref{obs_pin_condentbasic_rep}.
Further, notice that
\begin{align*}
\eta(\mcls,x_{\mclk},\chi_{\mclt})
&=H(\bm x_{\mcls}|\bm x_{\mclk}=x_{\mclk},\bm\chi_{\mclt}=\chi_{\mclt})\\
&=H(\bm x_{\mcls}|\bm x_{\mclk}=x_{\mclk},\bm x_{\mclt\cap\mclr}=x_{\mclt\cap\mclr},\bm r_{\mclt}=r_{\mclt})
=H(\bm x_{\mcls\setminus\mclk[][*]}|\bm x_{\mclk[][*]}=x_{\mclk[][*]})\\
&=\eta(\mcls\setminus\mclk[][*],x_{\mclk[][*]},\chi_\emptyset)
\end{align*}
using $(\bm x,\bm r)=\bm x\otimes\bm r$, so Part \ref{obs_pin_condentbasic}\ref{obs_pin_condentbasic_normpoint} holds since this also holds for $\mcls[][\circ]=\mcls\setminus\mclk$, $\mclt[][\circ]=\mclt\setminus\mclk$, and 
$\mclk[][*]=\mclk\cup(\mclt[][\circ]\cap\mclr)$ and
$\mcls\setminus\mclk[][*]=\mcls[][\circ]\setminus\mclk[][*]$.
With $\mcls[1][\circ]=\mcls[1]\setminus\mclk$, $\mcls[2][\circ]=\mcls[2]\setminus\mclk$,
$\mcls[][\circ]=\mcls[1][\circ]\dotcup\mcls[2][\circ]$, Part \ref{obs_pin_condentbasic}\ref{obs_pin_condentbasic_normpoint} and the chain rule for the conditional entropy we have
\begin{align*}
\overline\eta(\mcls,\mclk,\mclt)
&=\overline\eta(\mcls[][\circ],\mclk,\mclt[][\circ])
=H(\bm x_{\mcls[1][\circ]}|\bm x_{\mclk},\bm\chi_{\mclt[][\circ]})
+H(\bm x_{\mcls[2][\circ]}|\bm x_{\mclk\cup\mcls[1][\circ]},\bm\chi_{\mclt[][\circ]}),
\end{align*}
which completes the proof of Part \ref{obs_pin_condentbasic}\ref{obs_pin_condentbasic_normexpe}.
\end{proof}
Based on Observation \ref{obs_pin_condentbasic} we assume that $\mcls\cap\mclk=\emptyset$ and $\mclt\cap\mclk=\emptyset$.
Notice that Observation \ref{obs_pin_condentbasic}\ref{obs_pin_condentbasic_normpoint} 
using $\bmclr=\mclt\cap\bm r^{-1}(1)$ yields the minimal form
\begin{align*}
\overline\eta(\mcls,\mclk,\mclt)=\expe\left[\eta\left(\mcls\setminus\bmclr,\bm x_{\mclk\cup\bmclr},\bm\chi_\emptyset\right)\right]
=\expe\left[\expe\left[H(\bm x_{\mcls\setminus\bmclr}|\bm x_{\mclk\cup\bmclr})\middle|\bmclr\right]\right].
\end{align*}
This representation reflects the approach in \cite{coja2018}.
\subsubsection{The Conditional Entropy Derivative}\label{pinning_condentder}
Let $\mcls\cap\mclk=\emptyset$ and $\mclt\cap\mclk=\emptyset$ in this section.
Let $\frac{\partial}{\partial x_i}f(x)$ denote the $i$-th partial derivative of $f$ at $x$.
\begin{lemma}\label{lemma_pin_condentder}
For $i\in[n]$ we have $\frac{\partial}{\partial p_i}\overline\eta_p(\mcls,\mclk,\mclt)=-\bmone\{i\in\mclt\}I(\bm x_{\mcls},\bm x_{i}|\bm x_{\mclk},\bm\chi_{\mclt\setminus\{i\}})$.
\end{lemma}
\begin{proof}
With Observation \ref{obs_pin_condentbasic} and $\bmclr=\mclt\cap\bm r^{-1}(1)$ we have
\begin{align*}
\overline\eta(\mcls,\mclk,\mclt)
&=\expe\left[\expe\left[H(\bm x_{\mcls\setminus\bmclr}|\bm x_{\mclk\cup\bmclr})\middle|\bmclr\right]\right]\\
&=\sum_{r\in\{0,1\}^{\mclt}}\prod_{i\in\mclt}\prob[\bm r_i=r_i]H(\bm x_{\mcls\setminus r^{-1}(1)}|\bm x_{\mclk\cup r^{-1}(1)}).
\end{align*}
This shows that $\frac{\partial}{\partial p_i}\overline\eta(\mcls,\mclk,\mclt)=0$ for $i\in[n]\setminus\mclt$.
For $i\in\mclt$ let $\mclt[][\circ]=\mclt\setminus\{i\}$.
Then we have
\begin{align*}
\frac{\partial}{\partial p_i}\overline\eta(\mcls,\mclk,\mclt)
&=\sum_{r\in\{0,1\}^{\mclt}}\prob[\bm r_{\mclt[][\circ]}=r_{\mclt[][\circ]}]H(\bm x_{\mcls\setminus r^{-1}(1)}|\bm x_{\mclk\cup r^{-1}(1)})(r_i-(1-r_i))\\
&=\overline\eta(\mcls\setminus\{i\},\mclk\cup\{i\},\mclt[][\circ])
-\overline\eta(\mcls,\mclk,\mclt[][\circ])\\
&=\overline\eta(\mcls,\mclk\cup\{i\},\mclt[][\circ])
-\overline\eta(\mcls,\mclk,\mclt[][\circ])
=-I(\bm x_{\mcls},\bm x_{i}|\bm x_{\mclk},\bm\chi_{\mclt[][\circ]})\\
\end{align*}
since $I(\bm a,\bm b|\bm c)=H(\bm b|\bm c)-H(\bm b|\bm a,\bm c)$.
\end{proof}
\subsubsection{Mutual Information, Relative Entropy and the Product of the Marginals}
\label{pinning_mi}
The last sections were dedicated to the discussion of the conditional entropy.
Now, we turn to the following relative entropy.
For $(\bm a,\bm b)=((\bm a_h)_{h\in\mclh},\bm b)$ let
\begin{align*}
I(\bm a|\bm b)=\DKL\left(\bm a\middle\|\bigotimes_{h\in\mclh}\bm a_h\middle|\bm b\right)
=H\left(\bm a\middle\|\bigotimes_{h\in\mclh}\bm a_h\middle|\bm b\right)-H(\bm a|\bm b)
=\sum_{h\in\mclh}H(\bm a_h|\bm b)-H(\bm a|\bm b).
\end{align*}
Notice that for $\bm a=(\bm a_1,\bm a_2)$ this definition of $I(\bm a|\bm b)$ indeed coincides with the definition of the conditional mutual information $I(\bm a_1,\bm a_2|\bm b)$.
\begin{observation}\label{obs_pin_mi}
Let $(\bm a,\bm b)\in\mcla[][m]\times\mclb[][n]$ with $m,n\in\ints_{\ge 0}$ and $\mcla,\mclb\neq\emptyset$. Further, let $k,\ell\in\ints_{\ge 0}$, $v\in[m]^k$, $w\in[n]^{\ell}$, $\mclv=v([k])$ and $\mclw=w([\ell])$.
\begin{alphaenumerate}
\item\label{obs_pin_mi_norm}
For $\ell'\in\ints_{\ge 0}$, $w'\in[m]^{\ell'}$ with $w'([\ell'])\setle\mclw$ we have
\begin{align*}
I(\bm a_v,\bm b_{w'}|\bm b_{w})=I(\bm a_v|\bm b_{\mclw})
=\sum_{i\in[m]}|v^{-1}(i)|H(\bm a_i|\bm b_{\mclw})-H(\bm a_{\mclv}|\bm b_{\mclw}).
\end{align*}
\item\label{obs_pin_mi_partition}
For $j\in\ints_{\ge 0}$ and $\bigdotcup_{i\in[j]}\mclk[i]=[k]$ with $v'(i)=v_{\mclk[i]}$ we have
\begin{align*}
I(\bm a_v|\bm b_{w})=\sum_{i\in[j]}I(\bm a_{v'(i)}|\bm b_{\mclw})+I((\bm a_{v'(i)})_{i\in[j]}|\bm b_{\mclw}).
\end{align*}
\end{alphaenumerate}
\end{observation}
\begin{proof}
Part \ref{obs_pin_mi}\ref{obs_pin_mi_norm} is immediate from the properties of the conditional entropy since
\begin{align*}
I(\bm a_v,\bm b_{w'}|\bm b_{w})
&=\sum_{h\in[k]}H(\bm a_{v(h)}|\bm b_{w})+\sum_{h\in[\ell']}H(\bm b_{w'(h)}|\bm b_{w})
-H(\bm a_v,\bm b_{w'}|\bm b_w)\\
&=\sum_{i\in[m]}|v^{-1}(i)|H(\bm a_i|\bm b_{\mclw})-H(\bm a_{\mclv}|\bm b_{\mclw}).
\end{align*}
The second part is also immediate from the conditional entropy representation since
\begin{align*}
I(\bm a_v|\bm b_{w})
&=\sum_{h\in[k]}H(\bm a_{v(h)}|\bm b_{\mclw})-H(\bm a_v|\bm b_{\mclw})\\
&=\sum_{i\in[j]}I(\bm a_{v'(i)}|\bm b_{\mclw})+\sum_{i\in[j]}H(\bm a_{v'(i)}|\bm b_{\mclw})-H(\bm a_{\mclv}|\bm b_{\mclw})\\
&=\sum_{i\in[j]}I(\bm a_{v'(i)}|\bm b_{\mclw})+I((\bm a_{v'(i)})_{i\in[j]}|\bm b_{\mclw}).
\end{align*}
\end{proof}
As for the conditional entropy, Observation \ref{obs_pin_mi}\ref{obs_pin_mi_norm} yields a normalized form, and Observation \ref{obs_pin_mi}\ref{obs_pin_mi_partition} is a partitioning property of the mutual information.

Fix known coordinates $\mclk\setle[n]$, tested coordinates $\mclt\setle[n]$, further $s\in\ints_{\ge 0}$, a selection $\sigma\in[n]^s$ and let $\mcls=\sigma([s])$.
In the following we discuss the mutual information given by
\begin{align*}
\iota_{n,\mu,p}(\sigma,x_{\mclk},\chi_{\mclt})
&=I(\bm x_{\sigma}|\bm x_{\mclk}=x_{\mclk},\bm\chi_{\mclt}=\chi_{\mclt}),\\
\overline\iota_{n,\mu,p}(\sigma,\mclk,\mclt)
&=I(\bm x_{\sigma}|\bm x_{\mclk},\bm\chi_{\mclt})
=\expe[\iota(\sigma,\bm x_{\mclk},\bm \chi_{\mclt})].
\end{align*}
Observation \ref{obs_pin_mi}\ref{obs_pin_mi_norm} ensures that it is sufficient to consider sets $\mclk$, $\mclt$.
Next, we establish basic properties and build the connection to the conditional entropy.
\begin{observation}\label{obs_pin_mix}
Let $\mcld=\sigma^{-1}(\mcls\setminus\mclk)$, $\sigma^\circ=\sigma_{\mcld}$ and $\mclt[][\circ]=\mclt\setminus\mclk$.
\begin{alphaenumerate}
\item\label{obs_pin_mix_point}
Let $\mclk[][*]=\mclk\cup(\mclt\cap\mclr)$, $\mcld[][*]=\sigma^{-1}(\mcls\setminus\mclk[][*])$ and $\sigma^*=\sigma_{\mcld[][*]}$. Then we have
\begin{align*}
\iota(\sigma,x_{\mclk},\chi_{\mclt})
=\iota(\sigma^\circ,x_{\mclk},\chi_{\mclt[][\circ]})
=\iota(\sigma^\circ,x_{\mclk[][*]},\chi_\emptyset)
=\iota(\sigma^*,x_{\mclk[][*]},\chi_\emptyset).
\end{align*}
Further, we have $\iota(\sigma,x_{\mclk},\chi_{\mclt})=\sum_h\eta(\{\sigma(h)\},x_{\mclk},\chi_{\mclt})-\eta(\mcls,x_{\mclk},\chi_{\mclt})$.
\item\label{obs_pin_mix_expe}
We have $\overline\iota(\sigma,\mclk,\mclt)=\overline\iota(\sigma^\circ,\mclk,\mclt[][\circ])$ and $\overline\iota(\sigma,\mclk,\mclt)=\sum_h\overline\eta(\{\sigma(h)\},\mclk,\mclt)-\overline\eta(\mcls,\mclk,\mclt)$.
\end{alphaenumerate}
\end{observation}
\begin{proof}
Observation \ref{obs_pin_mi}\ref{obs_pin_mi_norm} implies 
$\iota(\sigma,x_{\mclk},\chi_{\mclt})=\sum_h\eta(\{\sigma(h)\},x_{\mclk},\chi_{\mclt})-\eta(\mcls,x_{\mclk},\chi_{\mclt})$, and further
Observation \ref{obs_pin_condentbasic}\ref{obs_pin_condentbasic_normpoint} yields
\begin{align*}
\iota(\sigma,x_{\mclk},\chi_{\mclt})
&=\sum_{i\in\mcls\setminus\mclk}|\sigma^{\circ-1}(i)|\eta(\{i\},x_{\mclk},\chi_{\mclt[][\circ]})-\eta(\mcls\setminus\mclk,x_{\mclk},\chi_{\mclt[][\circ]})\\
&=\sum_{i\in\mcls\setminus\mclk[][*]}|\sigma^{*-1}(i)|\eta(\{i\},x_{\mclk[][*]},\chi_{\emptyset})-\eta(\mcls\setminus\mclk[][*],x_{\mclk[][*]},\chi_{\emptyset}),
\end{align*}
which establishes the remainder of Part \ref{obs_pin_mix}\ref{obs_pin_mix_point}, and Part \ref{obs_pin_mix}\ref{obs_pin_mix_expe} follows by taking expectations.
\end{proof}
Based on Observation \ref{obs_pin_mix} we assume that $\mcls\cap\mclk=\emptyset$ and $\mclt\cap\mclk=\emptyset$.
Notice that Observation \ref{obs_pin_mix}\ref{obs_pin_mix_point} 
using $\bmclr=\mclt\cap\bm r^{-1}(1)$ yields the minimal form
\begin{align*}
\overline\iota(\sigma,\mclk,\mclt)
=\expe\left[\iota\left(\sigma,\bm x_{\mclk\cup\bmclr},\bm\chi_\emptyset\right)\right]
=\expe\left[\iota\left(\sigmaR^*,\bm x_{\mclk\cup\bmclr},\bm\chi_\emptyset\right)\right]
=\expe\left[\expe\left[I(\bm x_{\sigmaR^*}|\bm x_{\mclk\cup\bmclr})\middle|\bmclr\right]\right],
\end{align*}
where $\sigmaR^*=\sigma_{\bmcld}$ with $\bmcld=\sigma^{-1}(\mcls\setminus\bmclr)$.
This representation reflects the approach in \cite{coja2018}.
\begin{remark}
In the proof of Lemma \ref{lemma_pin_condentder} we have already seen a recursive structure.
Further, Observation \ref{obs_pin_mix}\ref{obs_pin_mix_expe} yields a representation of $\overline\iota$ as a linear combination of $\overline\eta$-terms, while Observation \ref{obs_pin_mi}\ref{obs_pin_mi_partition} applied to Lemma \ref{lemma_pin_condentder} yields a representation of the derivative as a linear combination of $\overline\iota$-terms (and hence $\overline\eta$-terms). This is one way to obtain all higher derivatives of $\overline\eta$ and $\overline\iota$.
\end{remark}
\subsubsection{The Pinning Lemma}\label{pinning_lemma_proof}
In this section we prove Lemma \ref{lemma_pinning}.
Recall the pinning operation $[\mu]^\darr_{\setP,\sigmaP}$ from Section \ref{ps_pinning} and let $\bm\mu_{n,\mu,p}=[\mu]^\darr_{\bm r^{-1}(1),\bm x}$.
\begin{lemma}\label{lemma_pin_iotapinmu}
We have $\overline\iota(\sigma,\emptyset,[n])=\expe[\expe[I(\bm x_{\bm\mu,\sigma})|\bm\mu]]$.
\end{lemma}
\begin{proof}
Notice that
\begin{align*}
\iota(\sigma,x_{\mclk},\chi_\emptyset)=I(\bm x_{\mu,\sigma}|\bm x_{\mclk}=x_{\mclk})
=I(\bm x_{[\mu]^\darr_{\mclk,x},\sigma}),
\end{align*}
since by definition $[\mu]^\darr_{\mclk,x}$ is the law of $\bm x_\mu|\bm x_{\mu,\mclk}=x_{\mclk}$.
Now, the assertion follows from Observation \ref{obs_pin_mix}\ref{obs_pin_mix_point} since
$\overline\iota(\sigma,\emptyset,[n])
=\expe[\iota(\sigma,\bm x_{\bmclr},\bm\chi_\emptyset)]=\expe[\expe[I(\bm x_{\bm\mu,\sigma})|\bm\mu]]$.
\end{proof}
Fix $\ell\in\ints_{>0}$ and $p\in[0,1]$.
Let $p_/=(p)_{i\in[n]}$, $\bm v_{n,\ell}\dequal\unif([n]^{\ell})$, $\eta^*_{n,\mu,\ell}(p)=\expe[\overline\eta_{p_/}(\bm v([\ell]),\emptyset,[n])]$,
$\iota^*_{n,\mu,\ell}(p)=\expe[\overline\iota_{p_/}(\bm v,\emptyset,[n])]$ and
\begin{align*}
\delta^*_{n,\mu,\ell}(p)
=\iota^*_{n,\mu,\ell+1}(p)-\iota^*_{n,\mu,\ell}(p)
\end{align*}
\begin{lemma}\label{lemma_pin_pinder}
Let $\ell\in\ints_{>0}$ and $\bm v=\bm v_{n,\ell+1}\in[n]^{\ell+1}$.
\begin{alphaenumerate}
\item\label{lemma_pin_pinder_delta}
We have $\delta^*(p)=I(\bm x_{\bm v_{[\ell]}},\bm x_{\mu,\bm v_{\ell+1}}|\bm\chi_{p_/},\bm v)$.
\item\label{lemma_pin_pinder_der}
We have $\frac{\partial}{\partial p}\eta^*(p)=-\frac{n}{1-p}\delta^*(p)$.
\end{alphaenumerate}
\end{lemma}
\begin{proof}
Notice that $\bm v_{[\ell]}\dequal\bm v_\circ$ with $\bm v_\circ=\bm v_{n,\ell}\in[n]^{\ell}$, let $\bm\chi=\bm\chi_{p_/}$ and $\bm a=(\bm\chi_{p_/},\bm v)$.
With Observation \ref{obs_pin_mi}\ref{obs_pin_mi_partition} we have
\begin{align*}
\delta^*(p)
&=I(\bm x_{\bm v}|\bm a)-I(\bm x_{\bm v_{[\ell]}}|\bm a)
=I(\bm x_{\bm v}|\bm a)-I(\bm x_{\bm v_{[\ell]}}|\bm a)
-I(\bm x_{\bm v_{\ell+1}}|\bm a)
=I(\bm x_{\bm v_{[\ell]}},\bm x_{\mu,\bm v_{\ell+1}}|\bm a),
\end{align*}
since $I(\bm b|\bm a)=\DKL(\bm b\|\bm b|\bm a)=0$ for $\bm b\in\mclb[][1]$.
For Part \ref{lemma_pin_pinder}\ref{lemma_pin_pinder_der} we use Lemma \ref{lemma_pin_condentder} and the chain rule to obtain
\begin{align*}
\frac{\partial}{\partial p}\eta^*(p)
&=\sum_{i\in[n]}\expe\left[\frac{\partial}{\partial p_i}\overline\eta_{p_/}(\bm v([\ell]),\emptyset,[n])\right]
=-\sum_{i\in[n]}I\left(\bm x_{\bm v([\ell])},\bm x_i|\bm\chi_{[n]\setminus\{i\}},\bm v\right)\\
&=-nI\left(\bm x_{\bm v([\ell])},\bm x_{\bm v_{\ell+1}}|\bm\chi_{[n]\setminus\{\bm v_{\ell+1}\}},\bm v\right)
=-nI\left(\bm x_{\bm v_{[\ell]}},\bm x_{\bm v_{\ell+1}}|\bm\chi_{[n]\setminus\{\bm v_{\ell+1}\}},\bm v\right).
\end{align*}
For given $v$ and $\chi$ with $i=v_{\ell+1}$ we have
\begin{align*}
\iota_v(\chi)
&=I(\bm x_{v_{[\ell]}},\bm x_{i}|\bm\chi=\chi)
=(1-r_i)I(\bm x_{v_{[\ell]}},\bm x_{i}|\bm x_{\mclr\setminus\{i\}}=x_{\mclr\setminus\{i\}})\\
&=(1-r_i)I(\bm x_{v_{[\ell]}},\bm x_{i}|\bm\chi_{[n]\setminus\{i\}}=\chi_{[n]\setminus\{i\}}).
\end{align*}
Due to independence this gives $I(\bm x_{v_{[\ell]}},\bm x_i|\bm\chi)=\expe[\iota_v(\bm\chi)]=(1-p)I(\bm x_{v_{[\ell]}},\bm x_i|\bm\chi_{[n]\setminus\{i\}})$. This completes the proof by taking the expectation over $\bm v$ and using the first part.
\end{proof}
An immediate consequence of Lemma \ref{lemma_pin_pinder} is a uniform bound for the integral over $\iota^*$.
\begin{corollary}\label{cor_pin_pinint}
For $\ell\in\ints_{>0}$ we have $\int_0^1\iota^*(p)\mathrm{d}p\le\binom{\ell}{2}\ln(q)/n$.
\end{corollary}
\begin{proof}
With $\iota^*\equiv 0$ for $\ell=1$,
a telescoping sum and Lemma \ref{lemma_pin_pinder}\ref{lemma_pin_pinder_der} we have
\begin{align*}
\int_0^1\iota^*(p)\mathrm{d}p
&=\int_0^1\sum_{\ell'=1}^{\ell-1}\delta^*(p)\mathrm{d}p
\le\int_0^1\sum_{\ell'=1}^{\ell-1}\frac{\delta^*(p)}{1-p}\mathrm{d}p
=\frac{1}{n}\sum_{\ell'=1}^{\ell-1}(\eta^*_{\ell'}(0)-\eta^*_{\ell'}(1))\\
&=\frac{1}{n}\sum_{\ell'=1}^{\ell-1}H(\bm x_{\bm v_{\ell'}}|\bm v_{\ell'})
\le\frac{1}{n}\sum_{\ell'=1}^{\ell-1}\ln(q^{\ell'})
=\frac{\ln(q)}{n}\binom{\ell}{2}.
\end{align*}
\end{proof}
With $\bm\mu^*=[\mu]^\darr_{\setPR,\sigmaR}$ as defined in Lemma \ref{lemma_pinning} we have $\bm\mu^*\dequal\bm\mu_{\bm p_/}$ with $\bm p_/=(\bm p)_{i\in[n]}$ and $\bm p\dequal\unif([0,P])$, $P=\ThetaP/n$, so with Lemma \ref{lemma_pin_iotapinmu}, Corollary \ref{cor_pin_pinint} and $(\bm\mu^*,\bm v_\ell)\dequal\bm\mu^*\otimes\bm v_\ell$ we get
\begin{align*}
\expe[\expe[I(\bm x_{\bm\mu^*,\bm v_{\ell}})|\bm\mu^*,\bm v_\ell]]
&=\int_0^{P}\frac{n}{\ThetaP}\expe[\expe[I(\bm x_{\bm\mu,\bm v_\ell})|\bm\mu_{p_/},\bm v_\ell]]\mathrm{d}p
=\frac{n}{\ThetaP}\int_0^P\expe[\overline\iota_{p_/}(\bm v,\emptyset,[n])]\mathrm{d}p\\
&=\frac{n}{\ThetaP}\int_0^P\iota^*(p)\mathrm{d}p
\le\frac{n}{\ThetaP}\int_0^1\iota^*(p)\mathrm{d}p
\le\frac{\binom{\ell}{2}\ln(q)}{\ThetaP}.
\end{align*}
\subsubsection{Asymptotic Independence of Gibbs Spins}\label{pinning_gibbs}
Before we turn to the main result of this section, we briefly discuss a few basic results.
\begin{observation}\label{obs_pin_basic}
Let $\ThetaP\le n$ and $\thetaP\le\ThetaP$.
We have $|\indPRT^{-1}(1)|\dequal\Bin(n,\thetaP/n)$,
$\expe[\thetaPR]=\ThetaP/2$, $\expe[\indPRTa]=\ThetaP/(2n)$ and $\expe[|\setPR|]=\ThetaP/2$.
\end{observation}
\begin{proof}
The proof is left as an exercise to the reader.
\end{proof}
The following result is one of the main reasons to work with the Nishimori ground truth.
Recall the notions from Section \ref{ps_pinning}, where we consider $\iota_\circ$, $\iota$ and $\vR$ for $\ell\in\ints_{\ge 0}$.
However, notice that $\iota_\circ\equiv 0$ for $\ell\le 1$.
\begin{proposition}\label{proposition_pinG}
Let $\GTSM(\sigma)=\GTSM[\setPR](\sigma)$, $\hat{\bm\mu}=\lawG[,\GTSM(\sigmaNIS)]$, $\bm\mu^*=\lawG[,\GTSM(\sigmaIID)]$ and $\ell\in\ints_{\ge 0}$.
\begin{alphaenumerate}
\item\label{proposition_pinG_NISDKL}
We have $\expe[\iota(\hat{\bm\mu})],\expe[\iota(\hat{\bm\mu}_{\mR,\mIR})]\le\binom{\ell}{2}\ln(q)/\ThetaP$.
\item\label{proposition_pinG_IIDDKL}
There exists $C_\mfkg\in(0,1)\times\reals_{>0}$ such that for $c\in(0,C_1]$ and $m\le\mbu$ we have
\begin{align*}
\expe[\iota(\bm\mu^*)],
\expe[\iota(\bm\mu^*_{\mR,\mIR})]\le 
C_2(\ell-1)\left(\frac{\ell}{\ThetaP}\right)^{c}.
\end{align*}
\end{alphaenumerate}
\end{proposition}
\begin{proof}
Notice that for any $G$ we have $\lawG[,[G]^\darr_{\setP,\sigma}]=[\lawG[,G]]^\darr_{\setP,\sigma}$ and that $\setPR$ as defined in Section \ref{random_decorated_graphs} coincides with $\setPR$ from Lemma \ref{lemma_pinning}, so with $(\setPR,\sigmaRG[,G])\dequal\setPR\otimes\sigmaRG[,G]$, $\bm G(G)=[G]^\darr_{\setPR,\sigmaRG[,G]}$ we have $\lawG[,\bm G(G)]=[\lawG[,G]]^\darr_{\setPR,\sigmaRG[,G]}$ and hence Lemma \ref{lemma_pinning} yields
$\expe[\iota(\lawG[,\bm G(G)])]\le\binom{\ell}{2}\ln(q)/\ThetaP$.
Since this holds for any $G$, the expectation for the unpinned graph $\bm G^\circ(\sigma)=[\wTSM(\sigma)]^{\Gamma\lrarr}_{\mI,\psiITSM}$ is also bounded by
$\expe[\iota(\lawG[,\bm G(\bm G^\circ(\sigma))])]\le\binom{\ell}{2}\ln(q)/\ThetaP$.
Notice that by Observation \ref{obs_TSM_iid} the graphs $\GTSM(\sigma)$ and $\bm G(\bm G^\circ(\sigma))$ differ exactly in the choice of the pinning assignment, i.e.~$\sigma$ for the former and $\sigmaRG[,\bm G^\circ(\sigma)]$ for the latter.
Since this bound holds for any $\sigma$, it holds for $\sigmaNIS$.
But with Observation \ref{obs_nishicond}\ref{obs_nishicond_dec} (for $\setP=\emptyset$) we have
$(\sigmaRG[,\bm G^\circ(\sigmaNIS)],\bm G^\circ(\sigmaNIS))\dequal(\sigmaNIS,\bm G^\circ(\sigmaNIS))$ and using Observation \ref{obs_TSM_iid} further
\begin{align*}
\bm G(\bm G^\circ(\sigmaNIS))
=[\bm G^\circ(\sigmaNIS)]^\darr_{\setPR,\sigmaRG[,\bm G^\circ(\sigmaNIS)]}
\dequal[\bm G^\circ(\sigmaNIS)]^\darr_{\setPR,\sigmaNIS}
\dequal\GTSM(\sigmaNIS).
\end{align*}
This shows that $\expe[\iota(\hat{\bm\mu})]\le\binom{\ell}{2}\ln(q)/\ThetaP$,
and $\expe[\iota(\hat{\bm\mu}_{\mR,\mIR})]\le\binom{\ell}{2}\ln(q)/\ThetaP$ follows by taking expectations.

Now, let $r\in\reals_{>0}$, $c_\mrmm\in\reals_{>0}^2$ from Corollary \ref{cor_dega}, $c^*\in\reals_{>0}^2$ from Observation \ref{obs_gtiid}\ref{obs_gtiid_prob} and $\hat c\in\reals_{>0}$ from Corollary \ref{cor_mutcont}\ref{cor_mutcont_rnbl}.
Then for $\ell>0$ we have
\begin{align*}
\iota_\circ(\mu,v)=\sum_hH(\mu|_{v(h)})-H(\mu|v)=\sum_{h>1}H(\mu|_{v(h)})-H(\mu|_v|\mu|_{v(1)})\le(\ell-1)\ln(q),
\end{align*}
so we have $\iota(\mu)\le(\ell-1)\ln(q)$.
Hence, for $E_1=\expe[\iota(\bm\mu^*)]$, $E_2=\expe[\iota(\bm\mu^*_{\mR,\mIR})]$ and $\ell\ge\ThetaP$ we have $E_1,E_2\le c_2(\ell-1)(\ell/\ThetaP)^{c_1}$ for $c\in\reals_{>0}\times\reals_{\ge\ln(q)}$.
Otherwise, we have
\begin{align*}
E_1&\le
e^{\hat cr^2}\expe\left[\bmone\{\|\gammaNIS-\gamma^*\|_\mrmtv<r/\sqrt{n}\}\iota(\hat{\bm\mu})\right]+\delta,\\
E_2&\le 
e^{\hat cr^2}\expe\left[\bmone\{\|\gammaNIS-\gamma^*\|_\mrmtv<r/\sqrt{n},\mR\le\mbu\}\iota(\hat{\bm\mu}_{\mR,\mIR})\right]+\delta,
\end{align*}
with $\delta=(\ell-1)\ln(q)(c^*_2e^{-c^*_1r^2}+c_{\mrmm,2}e^{-\tilde cn})$ and $\tilde c=c_{\mrmm,1}\degabu^2/(1+\degabu)$ obtained by choosing $r=\degabu$ in Corollary \ref{cor_dega} to enforce
$\degaR\le\tI\degae+\degabu\le 2\degabu$.
With Part \ref{proposition_pinG}\ref{proposition_pinG_NISDKL} and $i\in[2]$ we have
\begin{align*}
E_i&\le 
e^{\hat cr^2}\frac{\ln(q)}{\ThetaP}\binom{\ell}{2}+\delta
=(\ell-1)\ln(q)\left(\frac{\ell e^{\hat cr^2}}{2\ThetaP}+c^*_2e^{-c^*_1r^2}+c_{\mrmm,2}e^{-\tilde cn}\right).
\end{align*}
In order to compensate the last contribution notice that $\ell/\ThetaP\ge 1/n$, $e^{\hat cr^2}\ge 1$ and hence
\begin{align*}
E_i&\le 
(\ell-1)\ln(q)\left(c'\frac{\ell e^{\hat cr^2}}{\ThetaP}+c^*_2e^{-c^*_1r^2}\right),\,
c'=\frac{1}{2}+c_{\mrmm,2}\max_{n>0}ne^{-\tilde cn}.
\end{align*}
So, with $C_2=2\ln(q)\max(c',c^*_2)$,
$r=\sqrt{\frac{1}{\hat c+c^*_1}\ln(\frac{\ThetaP}{\ell})}>0$, $C_1=\frac{c_1^*}{\hat c+c_1^*}$ and for $\ell>0$, $c\in(0,C_1]$ we have $E_i\le c_2(\ell-1)(\ell/\ThetaP)^{c}$, and
Part \ref{proposition_pinG}\ref{proposition_pinG_IIDDKL} holds since the assertion is trivial for $\ell=0$.
\end{proof}
\begin{remark}\label{remark_epssym}
The expectation over the relative entropy in Section \ref{ps_pinning} can be defined for general $f$-divergences, in particular for the total variation.
Analogously to $\iota_\circ(\mu,v)$ let
$\nu_\circ(\mu,v)=\|\mu|_v-\bigotimes_h\mu|_{v(h)}\|_\mrmtv$ and further
$\nu_\ell(\mu)=\expe[\nu_\circ(\mu,\vR)]$.
For $\eps\in\reals_{\ge 0}$ let
$\setES[,n,\eps,\ell]=\{\mu\in\mclp([q]^n):\nu(\mu)\le\eps\}$ be the $(\eps,\ell)$-symmetric measures.
Pinsker's inequality \ref{obs_tv}\ref{obs_tv_pinsker} yields $\nu_\circ(\mu,v)\le\sqrt{\iota_\circ(\mu,v)/2}$ and hence $\nu(\mu)\le\sqrt{\iota(\mu)/2}$ using Jensen's inequality.
\end{remark}
\subsubsection{Pinning Impact on the Quenched Free Entropy}\label{pinning_qfed}
We bound the distance of the pinned and the unpinned quenched free entropy.
\begin{proposition}\label{proposition_pin_qfed}
There exists $c_\mfkg\in\reals_{>0}$ such that
\begin{align*}
0\le\expe[\phiG(\GTSM[\mR,\mIR,\emptyset](\sigmaIID))]-\expe[\phiG(\GTSM[\mR,\mIR,\setPR](\sigmaIID))]\le\frac{c\ThetaP}{n}
\end{align*}
and the same result holds for $\sigmaIID$ replaced by $\sigmaNIS_{\mR}$.
\end{proposition}
\begin{proof}
Let $\bm G^*=\GTSM[\mR,\mIR,\emptyset](\sigmaIID)$ and $\bm G^\darr=\GTSM[\mR,\mIR,\setPR](\sigmaIID)$.
Recall that $\bm G^\darr\dequal[\bm G^*]^\darr_{\setPR,\sigmaIID}$ from Observation \ref{obs_TSM_iid}, so given $(\sigmaIID,\bm G^*$ we can obtain $\bm G^\darr$ by choosing $\setPR$, which means that $\bm G^\darr$ and $\bm G^*$ then exactly differ in the pins.
Hence, using this coupling we have $\phiG(\bm G^\darr)\le\phiG(\bm G^*)$ and
using the notions from Section \ref{phi_lipschitz} with $\bm G^*=(\bm v^*,\bm\psi^*)$ further
\begin{align*}
\distG(\bm G^*,\bm G^\darr)=2\left(|\{a\in[\mR]:\vTSM[a]([k])\cap\setPR\neq\emptyset\}|+\sum_{i\in\setPR}\mIR_i\right).
\end{align*}
We further bound the distance using the sum over the (factor) degrees of $i\in\setPR$ and $c'$ from Corollary \ref{cor_degbounds}\ref{cor_degbounds_m} to obtain
\begin{align*}
\expe\left[\distG(\bm G^*,\bm G^\darr)\right]
\le 2\expe\left[\left(\frac{c'k\mR}{n}+(1-\tI)\degae\right)|\setPR|\right]
=\left(\tI c'+1-\tI\right)\degae\ThetaP,
\end{align*}
using Observation \ref{obs_pin_basic}.
With $c''$ from Observation \ref{obs_phi_lipschitz} we have
$|\expe[\phiG(\bm G^*)]-\expe[\phiG(\bm G^\darr)]|\le\frac{c\ThetaP}{n}$ for $c=c''(c'+1)\degabu$,
and the result for $\sigmaNIS$ follows analogously.
\end{proof}
\subsubsection{Reweighted Marginal Distributions}\label{pinning_marginal_distributions}
The empirical marginal distribution $\pi_\mu\in\mclp[][2]([q])$ of $\mu\in\mclp([q]^n)$ is given by
\begin{align*}
\pi(\mcle)=\frac{1}{n}|\{i\in[n]:\mu|_i\in\mcle\}|
\end{align*}
for an event $\mcle\setle\mclp([q])$.
Let $\sigmaR_\mu\dequal\mu$, $\gammaR_\pi\dequal\pi$ and $\gammaa_\pi=\expe[\gammaR]$. For $\tau\in[q]$ we have
\begin{align}\label{equ_expmarginal_isexpcolfreq}
\gammaa(\tau)=\sum_{i\in[n]}\frac{1}{n}\mu|_i(\tau)
=\mu|_*(\tau)=\expe\left[\gammaN[,\sigmaR](\tau)\right].
\end{align}
For $\sigma\in [q]^n$ and $\tau\in\sigma([n])$ let $\check\pi_{\mu,\sigma,\tau}\in\mclp[][2]([q])$ be given by
\begin{align*}
\check\pi(\mcle)=\frac{1}{|\sigma^{-1}(\tau)|}|\{i\in[n]:\sigma_i=\tau,\mu|_i\in\mcle\}.
\end{align*}
For $\tau\in\gammaa^{-1}(\reals_{>0})$ let $\hat\pi_{\mu,\tau}\in\mclp[][2]([q])$ be given by the $(\hat\pi,\pi)$-derivative $\gamma\mapsto\gamma(\tau)/\gammaa(\tau)$.
Recall the couplings $\Gamma(\check\pi,\hat\pi)$ from Section \ref{notions_notation} and for $\sigma\in[q]^n$, $\tau\in\sigma([n])\cap\gammaa^{-1}(\reals_{>0})$ let
\begin{align*}
\distW(\check\pi,\hat\pi)=\inf_{\rho\in\Gamma(\check\pi,\hat\pi)}\expe[\|\gammaR_{\rho,1}-\gammaR_{\rho,2}\|_\mrmtv],\,\gammaR_\rho\dequal\rho,
\end{align*}
be the Wasserstein distance of $\check\pi$ and $\hat\pi$.
Recall $\setES[,\eps,\ell]$ from Remark \ref{remark_epssym}.
\begin{proposition}\label{proposition_piCR}
Let $\delta,\eps,\epsES\in\reals_{>0}$ and $\mu\in\mclp([q]^n)$ be such that $\gammaa\ge\psibl/2$, $\mu\in\setES[,\epsES,2]$ and $\prob[\|\gammaN[,\sigmaR]-\gammaa\|_\mrmtv>\delta]\le\eps$.
Then there exists $c_\mfkg\in\reals_{>0}$ such that
\begin{align*}
\expe\left[\sum_{\tau\in\sigmaR([n])}\distW(\check\pi_{\sigmaR,\tau},\hat\pi_\tau)\right]
\le c(\delta+\eps+\epsES^{1/(2q+1)}).
\end{align*}
\end{proposition}
\begin{proof}
If $\max(\delta,\eps,\epsES)\ge 1$, then the assertion holds with $c\ge q$ since the left hand side is at most $q$, so let $\delta,\eps,\epsES\in(0,1]$.
Notice that $\gammaa^{-1}(\reals_{>0})=\bigcup_{\sigma\in\mu^{-1}(\reals_{>0})}\sigma([n])$,
let $\sigma\in\mu^{-1}(\reals_{>0})$ and $\tau\in\sigma([n])$.
Let $a\in\ints_{>0}$, $\setb^*=(a^{-1}\ints)^{q-1}$, $\setq_\circ=[-1/(2a),1/(2a))^{q-1}$ and $\setq_b=b+\setq_\circ$ for $b\in\setb^*$, then $(\setq_b)_{b\in\setb^*}$ is a partition of $\reals^{q-1}$. This induces a partition $(\setq_b)_{b\in\setb}$ of $\setp([q])$, where
\begin{flalign*}
\setb&=\{(b^*,1-\|b^*\|_1):b^*\in\setb^*\}\cap\setp([q]),\\
\setq_b&=\{(b^*,1-\|b^*\|_1):b^*\in\setq^*_{b_{[q-1]}}\},\,b\in\setb.
\end{flalign*}
Notice that $|\setb|\le(a+1)^{q-1}=|\setb^*\cap[0,1]^{q-1}|$.
For $\gamma\in\setq_b$ we have $\tilde\gamma\in\setq_{\tilde b}$, where $\tilde\gamma=\gamma_{[q-1]}$ and $\tilde b=b_{[q-1]}$, hence
\begin{flalign*}
\|\gamma-b\|_\mrmtv=\frac{1}{2}\|\tilde\gamma-\tilde b\|_1+\frac{1}{2}|1-\|\tilde\gamma\|_1-(1-\|\tilde b\|_1)|
\le\|\tilde\gamma-\tilde b\|_1\le\frac{q-1}{2a}.
\end{flalign*}
Next, let $\seti_b=\{i\in[n]:\mu|_i\in\setq_b\}$ and $\check{\seti}_{\sigma,\tau,b}=\{i\in[n]:\sigma_i=\tau,\mu|_i\in\setq_b\}$.
Then we have
\begin{flalign*}
\pi(\setq_b)=\frac{|\seti_b|}{n},\,
\check\pi(\setq_b)=\frac{|\check{\seti}_{\sigma,\tau,b}|}{|\sigma^{-1}(\tau)|}.
\end{flalign*}
The expectation $\bar{I}_{\tau,b}=\expe[|\check{\seti}_{\sigmaR,\tau,b}|]$ is given by
\begin{flalign*}
\bar{I}_{\tau,b}=\sum_i\bmone\{\mu|_i\in\setq_b\}\prob[\sigmaR_{i}=\tau]
=\sum_i\bmone\{\mu|_i\in\setq_b\}\mu|_i(\tau)
=n\gammaa(\tau)\hat\pi(\setq_b).
\end{flalign*}
The variances, using $\mu\in\setES[,n,\epsES,2]$, are given by
\begin{flalign*}
V&=\sum_{b,\tau}\Var(|\check{\seti}_{\sigmaR,\tau,b}|)
=\sum_{b,\tau}\left(\expe\left[|\check{\seti}_{\sigmaR,\tau,b}|^2\right]-\expe\left[|\check{\seti}_{\sigmaR,\tau,b}|\right]^2\right)\\
&=\sum_{b,\tau}\sum_{v\in[n]^2}\bmone\left\{\mu|_{v(1)},\mu|_{v(2)}\in\setq_b\right\}\left(\prob\left[\sigmaR|_v=(\tau,\tau)\right]-\prob\left[\sigmaR|_{v(1)}=\tau\right]\prob\left[\sigmaR|_{v(2)}=\tau\right]\right)\\
&\le\sum_{v\in[n]^2}\sum_{\tau\in[q]^2}\left|\prob\left[\sigmaR|_v=\tau\right]-\prob\left[\sigmaR|_{v(1)}=\tau_1\right]\prob\left[\sigmaR|_{v(2)}=\tau_2\right]\right|\\
&=2\sum_{v\in[n]^2}\left\|\mu|_v-\mu|_{v(1)}\otimes\mu|_{v(2)}\right\|_\mrmtv
\le 2n^2\epsES,
\end{flalign*}
where in the extension of the summation region $b$ is determined by $\mu|_{v(1)}$, i.e.~the unique point $b$ with $\mu|_{v(1)}\in\setq_b$, while we drop the restriction $\mu|_{v(2)}\in\setq_b$.
With the union bound and Chebyshev's inequality we have
\begin{flalign*}
\prob\left[\sum_{b,\tau}\left||\check{\seti}_{\sigmaR,\tau,b}|-\bar I_{\tau,b}\right|\ge rn\right]
&\le\sum_{b,\tau}\prob\left[\left||\check{\seti}_{\sigmaR,\tau,b}|-\bar I_{\tau,b}\right|\ge\frac{rn}{q|\setb|}\right]
\le\frac{Vq^2|\setb|^2}{r^2n^2}
\le\frac{2q^2|\setb|^2\epsES}{r^2}
\end{flalign*}
for $r=\epsES^{1/(1+2q)}$. Next, we recall that the color frequencies concentrate and let
\begin{flalign*}
\mcls=\left\{\sigma\in[q]^n:\left\|\gammaN[,\sigma]-\gammaa\right\|_\mrmtv\le\delta,\,\sum_{b,\tau}\left||\check{\seti}_{\sigma,\tau,b}|-\bar I_{\tau,b}\right|< rn\right\}.
\end{flalign*}
Further, notice that $\distW(\check\pi,\hat\pi)\in[0,1]$ since $\|\cdot\|_\mrmtv\in[0,1]$, so
\begin{flalign*}
\expe\left[\sum_{\tau\in\sigmaR([n])}\distW\left(\check\pi_{\sigmaR,\tau},\hat\pi_{\tau}\right)\right]
\le\expe\left[\bmone\{\sigmaR\in\mcls\}\sum_{\tau\in\sigmaR([n])}\distW\left(\check\pi_{\sigmaR,\tau},\hat\pi_{\tau}\right)\right]+q\eps+\frac{2q^3|\setb|^2\epsES}{r^2}.
\end{flalign*}
For $\sigma\in\mcls$, $\tau\in\sigma([n])$ and $b\in\setb$ we have
\begin{flalign*}
\Delta_{\sigma,\tau}(b)&=\left|\check\pi(\setq_b)-\hat\pi(\setq_b)\right|
\le\left|\frac{|\check{\seti}|}{|\sigma^{-1}(\tau)|}-\frac{|\check{\seti}|}{n\gammaa(\tau)}\right|
+\left|\frac{|\check{\seti}|}{n\gammaa(\tau)}-\frac{\bar I}{n\gammaa(\tau)}\right|\\
&=\frac{1}{\gammaa(\tau)}\left(\frac{|\check{\seti}|}{|\sigma^{-1}(\tau)|}|\gammaa(\tau)-\gammaN[,\sigma](\tau)|
+\frac{1}{n}\left||\check{\seti}|-\bar I\right|\right).
\end{flalign*}
With $\sum_b\check{\seti}_b=|\sigma^{-1}(\tau)|$, $\gammaa\ge\psibl/2$ and
$\|\gammaN[,\sigma]-\gammaa\|_1=2\|\gammaN[,\sigma]-\gammaa\|_\mrmtv\le 2\delta$ we have
\begin{flalign*}
\sum_{b,\tau}\Delta_{\sigma,\tau}(b)<\frac{4\delta+2r}{\psibl}.
\end{flalign*}
Now, let $\check{\bm\gamma}\dequal\check\pi$ and $\hat{\bm\gamma}\dequal\hat\pi$.
For $\gamma\in\setp([q])$ let $b(\gamma)\in\setb$ be the unique index with $\gamma\in\setq_b$.
Then we have $\sum_b\Delta_{\sigma,\tau}(b)=2\|b(\check{\bm\gamma})-b(\hat{\bm\gamma})\|_\mrmtv$.
With the coupling lemma we obtain a coupling of $b(\check{\bm\gamma})$ and $b(\hat{\bm\gamma})$ that extends to a coupling of $\check{\bm\gamma}$ and $\hat{\bm\gamma}$ via $(\check{\bm\gamma}|b(\check{\bm\gamma})=\check b)\otimes(\hat{\bm\gamma}|b(\hat{\bm\gamma})=\hat b)$ given $(b(\check{\bm\gamma}),b(\hat{\bm\gamma}))=(\check b,\hat b)$ (by an abuse of notation in that we use the same notation for the coupling).
The triangle inequality yields $\|\check{\bm\gamma}-\hat{\bm\gamma}\|_\mrmtv\le\frac{q-1}{a}+\|b(\check{\bm\gamma})-b(\hat{\bm\gamma})\|_\mrmtv$ and hence
\begin{flalign*}
\prob\left[\left\|\check{\bm\gamma}-\hat{\bm\gamma}\right\|_\mrmtv>\frac{q-1}{a}\right]\le\prob[b(\check{\bm\gamma})\neq b(\hat{\bm\gamma})]
=\|b(\check{\bm\gamma})-b(\hat{\bm\gamma})\|_\mrmtv
=\frac{1}{2}\sum_b\Delta_{\sigma,\tau}(b)
<\frac{2\delta+r}{\psibl}.
\end{flalign*}
Using that $\|\cdot\|_\mrmtv\le 1$ this gives
\begin{flalign*}
\distW(\check\pi,\hat\pi)
\le\expe\left[\left\|\check{\bm\gamma}-\hat{\bm\gamma}\right\|_\mrmtv\right]
\le\frac{q-1}{a}+\frac{2\delta+r}{\psibl}.
\end{flalign*}
Hence, combining the results for $\sigmaR\in\mcls$ and $\sigmaR\not\in\mcls$ yields
\begin{align*}
\expe\left[\sum_{\tau\in\sigmaR([n])}\distW\left(\check\pi_{\sigmaR,\tau},\hat\pi_{\tau}\right)\right]
&\le q\left(\frac{q-1}{a}+\frac{2\delta+r}{\psibl}\right)+q\eps+\frac{2q^3(a+1)^{2(q-1)}\epsES}{r^2}.
\end{align*}
Now, let $a=\lfloor r^{-1}\rfloor$.
With $r=\epsES^{1/(2q+1)}\le 1$ we have $a\in\ints_{\ge 1}$, hence $a+1\le 2a$, $a\ge\frac{1}{2}(a+1)\ge\frac{1}{2r}$ and further
\begin{align*}
\expe\left[\sum_{\tau\in\sigmaR([n])}\distW\left(\check\pi_{\sigmaR,\tau},\hat\pi_{\tau}\right)\right]
&\le q\left(2(q-1)+\frac{1}{\psibl}+2^{2q-1}q^2\right)r+\frac{2q}{\psibl}\delta+q\eps.
\end{align*}
\end{proof}
\subsubsection{Reweighted Gibbs Marginal Distribution}\label{pinning_gibbs_marginal_distributions}
Using the notions from Section \ref{pinning_marginal_distributions} and for a decorated factor graph $G$ with Gibbs measure $\mu=\lawG[,G]$ let $\piG[,G]=\pi_{\mu}$, $\gammaaG[,G]=\gammaa_\mu$, $\piGC[,G,\sigma,\tau]=\check\pi_{\mu,\sigma,\tau}$ and $\piGR[,G,\tau]=\hat\pi_{\mu,\tau}$.
First, we focus on the expected Gibbs marginal $\gammaaG$ under various versions of the teacher-student model, and start with $\GTSM(\sigmaNIS)$.
\begin{lemma}\label{lemma_gammaaR}
Let $m\le\mbu$, $\gammaaR=\gammaa_{\GTSM}$ and $\gammaR=\gammaN[,\sigmaRG]$ with $\sigmaRG=\sigmaRG[,\GTSM]$, $\GTSM=\GTSM(\sigmaR)$ and $\sigmaR=\sigmaIID$ or $\sigmaR=\sigmaNIS$.
\begin{alphaenumerate}
\item\label{lemma_gammaaR_prob}
There exists $c\in\reals_{>0}^2$ such that $\prob[\|\gammaaR-\gamma^*\|_\mrmtv\ge r]\le c_2e^{-c_1r^2n}$.
\item\label{lemma_gammaaR_conc}
There exists $c\in\reals_{>0}^2$ such that $\prob[\|\gammaR-\gammaaR\|_\mrmtv\ge r]\le c_2e^{-c_1r^2n}$.
\item\label{lemma_gammaaR_expe}
There exists $c\in\reals_{>0}$ such that $\expe[\|\gammaaR-\gamma^*\|_\mrmtv]\le c/\sqrt{n}$.
\item\label{lemma_gammaaR_var}
There exists $c\in\reals_{>0}$ such that $\expe[\|\gammaaR-\gamma^*\|_\mrmtv^2]\le c/n$.
\end{alphaenumerate}
\end{lemma}
\begin{proof}
Let $\overline D(G)=\|\gammaaG[,G]-\gamma^*\|_\mrmtv$ and $\bm D(G)=\|\gammaN[,\sigmaRG[,G]]-\gamma^*\|_\mrmtv$.
With Equation (\ref{equ_expmarginal_isexpcolfreq}) and Jensen's inequality we have
\begin{align*}
\overline D(G)^x&=\|\expe[\gammaN[,\sigmaRG[,G]]]-\gamma^*\|_\mrmtv^x\le\expe[\bm D(G)^x],
e^{y\overline D(G)^2}\le \expe\left[e^{y\bm D(G)^2}\right]
\end{align*}
for $x\in\reals_{\ge 1}$ and $y\in\reals_{\ge 0}$, so
$\expe[\overline D(\GTSM))^x]\le\expe[\bm D(\GTSM)^x]$ and $\expe[e^{y\overline D(\GTSM)^2}]\le\expe[e^{y\bm D(\GTSM)^2}]$.
With Observation \ref{obs_nishicond}\ref{obs_nishicond_dec} we have $\bm D(\GTSM(\sigmaNIS))\dequal\|\gammaNIS-\gamma^*\|_\mrmtv$, so Corollary \ref{cor_mutcont} applies, on the other hand we have $\bm D(\GTSM(\sigmaIID))=\|\gammaR-\gamma^*\|_\mrmtv$ with $\gammaR$ from Observation \ref{obs_gibbsTSIID}.
Hence, in both cases there exists $c\in\reals_{>0}^2$ such that
\begin{align*}
\prob[\bm D(\GTSM)\ge r]\le c_2e^{-c_1r^2n},\,
\expe[\bm D(\GTSM)]\le c_2/\sqrt{n},\,
\expe[\bm D(\GTSM)^2]\le c_2/n.
\end{align*}
Notice that $\overline D(\GTSM)=\|\gammaaR-\gamma^*\|_\mrmtv$
and recall that $\expe[\overline D(\GTSM)^x]\le\expe[\bm D(\GTSM)^x]$,
so Part \ref{lemma_gammaaR}\ref{lemma_gammaaR_expe} follows for $x=1$ and Part \ref{lemma_gammaaR}\ref{lemma_gammaaR_var} follows for $x=2$.
Further, notice that
\begin{align*}
\expe\left[\exp\left(\frac{c_1}{2}\bm D(\GTSM)^2n\right)\right]
&=\int_0^\infty\prob\left[\exp\left(\frac{c_1}{2}\bm D(\GTSM)^2n\right)>r\right]\mathrm{d}r\\
&=1+\int_1^\infty\prob\left[\bm D(\GTSM)>\sqrt{\frac{2\ln(r)}{c_1n}}\right]\mathrm{d}r\\
&\le 1+c_2\int_{1}^\infty e^{-2\ln(r)}\mathrm{d}r=1+c_2\int_1^\infty\frac{1}{r^2}\mathrm{d}r
=1+c_2.
\end{align*}
So, the bound above for $y=c_1n/2$ and Markov's inequality yield
\begin{align*}
\prob\left[\overline D(\GTSM)\ge r\right]=\prob\left[\exp\left(\frac{c_1}{2}\overline D(\GTSM)^2n\right)\ge\exp\left(\frac{c_1}{2}r^2n\right)\right]
\le(1+c_2)\exp\left(-\frac{c_1}{2}r^2n\right).
\end{align*}
This establishes Part \ref{lemma_gammaaR}\ref{lemma_gammaaR_prob}.
Let $c'\in\reals_{>0}^2$ be the constants from Part \ref{lemma_gammaaR}\ref{lemma_gammaaR_prob}. Then the triangle inequality and the union bound yields
\begin{align*}
\prob[\|\gammaR-\gammaaR\|_\mrmtv\ge r]\le\prob[\bm D(\GTSM)\ge r/2]+\prob[\overline D(\GTSM)\ge r/2]\le c_2e^{-c_1r^2n/4}+c'_2e^{-c'_1r^2n/4},
\end{align*}  
so Part \ref{lemma_gammaaR}\ref{lemma_gammaaR_conc} holds with $c'_1/4$ and $c_2+c'_2$.
\end{proof}
\begin{remark}\label{remark_mucolfreqconc}
Similar to $\iota(\mu)$ and $\nu(\mu)$ in Remark \ref{remark_epssym}, quantifying the $\ell$-wise dependencies of $\mu$, we may consider $P_\mu:[0,1]\rarr[0,1]$, $r\mapsto\prob[\|\gammaN[,\sigmaR_\mu]-\gammaa_\mu\|_\mrmtv\ge r]$ to quantify the concentration of the color frequencies.
As for Proposition \ref{proposition_pinG}, Lemma \ref{lemma_gammaaR}\ref{lemma_gammaaR_conc} suggests that $\expe[P(\bm\mu,r)]\le c_2e^{-c_1r^2n}$ for both $\bm\mu=\bm\mu^*$ and $\bm\mu=\hat{\bm\mu}$.
\end{remark}
Lemma \ref{lemma_gammaaR} facilitates the application of Proposition \ref{proposition_piCR}.
For this purpose let
\begin{align*}
D(\sigma,\mu)=\sum_{\tau\in\sigma([n])}\distW(\check\pi_{\mu,\sigma,\tau},\hat\pi_{\mu,\tau})
\end{align*}
for $\mu\in\mclp([q]^n)$ using the notions from Proposition \ref{proposition_piCR}.
\begin{corollary}\label{cor_piCR}
Let $m\le\mbu$, $\hat{\bm\mu}=\lawG[,\GTSM(\sigmaNIS)]$, $\bm\mu^*=\lawG[,\GTSM(\sigmaIID)]$ and $\sigmaR_\mu\dequal\mu$.
Further, let $C_\mfkg\in(0,1)\times\reals_{>0}$ be the constants from Proposition \ref{proposition_pinG}\ref{proposition_pinG_IIDDKL}.
\begin{alphaenumerate}
\item\label{cor_piCR_NIS}
There exists $c_\mfkg\in\reals_{>0}$ such that
$\expe[D(\sigmaNIS,\hat{\bm\mu})],\expe[D(\sigmaR_{\hat{\bm\mu}},\hat{\bm\mu})]\le c/\ThetaP$.
\item\label{cor_piCR_IID}
There exists $c_\mfkg\in\reals_{>0}$ with
$\expe[D(\sigmaIID,\bm\mu^*)]\le c/\ThetaP[c']$ for $c'\in(0,C_1]$, and
$\expe[D(\sigmaR_{\bm\mu^*},\bm\mu^*)]\le c/\ThetaP[c']$ for $c'\in(0,C_1/3]$.
\end{alphaenumerate}
\end{corollary}
\begin{proof}
For $\ThetaP\le 1$ and $c\in\reals_{\ge 0}\times\reals_{\ge q}$ we have
$\expe[D(\cdot)]\le q\le c_2\ThetaP[-c_1]$.
Otherwise, notice that $n>1$,
let $\delta=\eps=\ln(n)/\sqrt{n}$, $\epsES\in\reals_{>0}$ and
\begin{align*}
\mclm=\left\{\mu\in\mclp([q]^n):\|\gammaa_\mu-\gamma^*\|_\mrmtv\le\frac{\psibl}{4},\iota_2(\mu)\le\epsES,P(\mu,\delta)\le\eps\right\}
\end{align*}
with $P$ from Remark \ref{remark_mucolfreqconc}.
For $\mu\in\mclm$ we have $\|\gammaa_\mu-\gamma^*\|_\infty\le\psibl/2$, hence $\gammaa_\mu\ge\psibl/2$, and with $c_\mrmd\in\reals_{>0}$ from Proposition \ref{proposition_piCR} further
$\expe[D(\sigmaR_\mu,\mu)]\le c_\mrmd(\delta+\eps+\sqrt{\epsES/2})$ using Remark \ref{remark_epssym}.
With $c_\mrmc\in\reals_{>0}^2$ for both Lemma \ref{lemma_gammaaR}\ref{lemma_gammaaR_prob} and Lemma \ref{lemma_gammaaR}\ref{lemma_gammaaR_conc}, Proposition \ref{proposition_pinG}\ref{proposition_pinG_NISDKL} and Markov's inequality we have
\begin{align*}
\prob[\hat{\bm\mu}\not\in\mclm]
&\le c_{\mrmc,2}\exp\left(-\frac{c_{\mrmc,1}\psibl^2}{16}n\right)+\frac{\ln(q)}{\epsES\ThetaP}+\frac{c_{\mrmc,2}}{\eps}e^{-c_{\mrmc,1}\delta^2n}\\
&\le\tilde ce^{-n/\tilde c}+\frac{\tilde c}{\ThetaP[1/3]}+\frac{\tilde c\sqrt{n}}{\ln(n)}e^{-\ln(n)^2/\tilde c}\le\frac{c'}{\ThetaP[1/3]},
\end{align*}
where $\tilde c$ is the implied maximum, $\epsES=\ThetaP[-2/3]$, $c'=\tilde cc''_1+\tilde c+\tilde cc''_2$, $c''_1=\max_{n>1}n^{1/3}e^{-n/\tilde c}$ and $c''_2=\max_{n>1}n^{5/6}e^{-\ln(n)^2/\tilde c}/\ln(n)$ using $\ThetaP\le n$.
Using $D(\cdot)\le q$ gives
\begin{align*}
\expe\left[D\left(\sigmaR_{\hat{\bm\mu}},\hat{\bm\mu}\right)\right]
\le\frac{2c_\mrmd\ln(n)}{\sqrt{n}}+\frac{c_{\mrmd}}{\sqrt{2}\ThetaP[1/3]}+\frac{qc'}{\ThetaP[1/3]}\le\frac{c}{\ThetaP[1/3]}
\end{align*}
with $c=\max(c_\mrmd+qc'+2c_\mrmd\max_{n>1}\ln(n)/n^{1/6},q)$, so 
$\expe[D(\sigmaR_{\hat{\bm\mu}},\hat{\bm\mu})]\le c/\ThetaP[1/3]$ holds for all $\ThetaP\ge 0$.
The Nishimori condition \ref{obs_nishicond}\ref{obs_nishicond_dec} completes the proof of Part \ref{cor_piCR}\ref{cor_piCR_NIS}.
With $C$ from Proposition \ref{proposition_pinG}\ref{proposition_pinG_IIDDKL},
branching off in the discussion above yields
\begin{align*}
\prob[\bm\mu^*\not\in\mclm]
&\le c_{\mrmc,2}\exp\left(-\frac{c_{\mrmc,1}\psibl^2}{16}n\right)+\frac{C_22^{C_1}}{\epsES\ThetaP[C_1]}+\frac{c_{\mrmc,2}}{\eps}e^{-c_{\mrmc,1}\delta^2n}
\le\frac{c'}{\ThetaP[C_1/3]}
\end{align*}
with $\epsES=\ThetaP[-2C_1/3]$ and $c'$ obtained analogously to the above.
Repeating the remaining steps yields $c_\mfkg\in\reals_{>0}$ with $\expe[D(\sigmaR_{\bm\mu^*},\bm\mu^*)]\le c/\ThetaP[C_1/3]$.
For the second part of Corollary \ref{cor_piCR}\ref{cor_piCR_IID} and $\ThetaP>1$
let $\hat c$ from Corollary \ref{cor_mutcont}\ref{cor_mutcont_rnbl}, $c^*$ from Observation \ref{obs_gtiid}\ref{obs_gtiid_prob}, $c'$ from Part \ref{cor_piCR}\ref{cor_piCR_NIS} and $r=\sqrt{\ln(\ThetaP)/((c_1^*+\hat c)n)}$, then we have
\begin{align*}
\expe[D(\sigmaIID,\bm\mu^*)]\le e^{\hat c r^2n}\frac{c'}{\ThetaP}+qc_2^*e^{-c^*_1r^2n}
=\frac{c}{\ThetaP[\rho]}
\end{align*}
with $c=c'+qc^*_2$ and $\rho=c^*_1/(c^*_1+\hat c)$.
Finally, recall from the proof of Proposition \ref{proposition_pinG} that $\rho=C_1$.
\end{proof}
\subsubsection{Gibbs Marginal Distribution Projection}\label{pinning_limiting_marginal_distributions}
Let $\GTSM=\GTSM(\sigmaR)$ with $\sigmaR=\sigmaIID$ or $\sigmaR=\sigmaNIS$.
Recall that $\gammaaG[,\GTSM]$ is defined as the expected law under the empirical marginal distribution $\bm\pi=\piG[,\GTSM]$, given $\bm\pi$.
Lemma \ref{lemma_gammaaR}\ref{lemma_gammaaR_prob} ensures that $\gammaa_{\GTSM}$ is asymptotically close to $\gamma^*$ with very high probability, but this is not sufficient for $\nablaIbl$ and $\bethebu$, being extremal only on $\mclp[*][2]([q])$, meaning that the expectation has to be exactly $\gamma^*$.

Hence, we map $\bm\pi\in\mclp[][2]([q])$ to some $\bm\pi^\circ\in\mclp[*][2]([q])$ such that the Wasserstein distance $\distW(\bm\pi,\bm\pi^\circ)$ vanishes, which is sufficient because both $\nablaI$ and $\bethe$ will turn out to be Lipschitz continuous.
First, we identify a suitable counterweight to $\gammaa$.

Let $\alpha_\mfkg:\mclp([q])\rarr[0,1]$ and $f_\mfkg:\mclp([q])\rarr\mclp([q])$, $\gamma\mapsto[\gamma]_\mrmc$, be given as follows.
Let $\ell_\circ=\ell_\mrmc^{-1}(\psibl)$ with $\ell_\mrmc:[0,1)\rarr\reals_{\ge 0}$, $\ell\mapsto-(1+\xlnx(\ell))/\ln(\ell)$ and $\ell(\gamma)=\|\gamma-\gamma^*\|_2$ for $\gamma\in\mclp([q])$.
For $\ell(\gamma)=0$ let $[\gamma]_\mrmc=\gamma^*$ and $\alpha(\gamma)=0$.
For $\ell(\gamma)\in(0,\ell_\circ]$ let $[\gamma]_\mrmc=\gamma^*+\frac{\ell_\mrmc(\ell(\gamma))}{\ell(\gamma)}(\gamma^*-\gamma)$ and $\alpha(\gamma)=-\xlnx(\ell(\gamma))$.
For $\ell(\gamma)\ge\ell_\circ$ let $[\gamma]_\mrmc=\gamma^*+\frac{\psibl}{\ell(\gamma)}(\gamma^*-\gamma)$ and $\alpha(\gamma)=\ell(\gamma)/(\ell(\gamma)+\psibl)$.
\begin{observation}\label{obs_gammaacounterweight}
The maps $\alpha$ and $f$ are continuous with $\alpha(\gamma)[\gamma]_\mrmc+(1-\alpha(\gamma))\gamma=\gamma^*$ for $\gamma\in\mclp([q])$.
With $c=(e-1)/e$ we have $\ell_\circ\in[e^{-\psibu},e^{-c\psibu}]$ and $\alpha$ is increasing in $\ell(\gamma)$.
\end{observation}
\begin{proof}
Clearly, the maps $\ell$ and $\ell_\mrmc$ are continuous. Further, $\ell_\mrmc$ is strictly increasing with $\ell_\mrmc(0)=0$ and $\ell_\mrmc(1)=\infty$, so $\ell_\circ\in(0,1)$ is well-defined.
Hence, we have $[\gamma]_\mrmc=\gamma^*+s(\gamma)(\gamma^*-\gamma)$ with $s(\gamma)=\min(\psibl,\ell_\mrmc(\ell(\gamma)))/\ell(\gamma)$ and thereby
\begin{align*}
\|[\gamma]_\mrmc-\gamma^*\|_\infty\le\|[\gamma]_\mrmc-\gamma^*\|_2
=\min(\psibl,\ell_\mrmc(\ell(\gamma)))\le\psibl,
\end{align*}
so $[\gamma]_\mrmc\ge 0$ and thereby $[\gamma]_\mrmc\in\mclp([q])$.
The map $s$ is clearly continuous for $\gamma\neq\gamma^*$. For $\ell(\gamma)\le\ell_\circ$ we further have
$\|[\gamma]_\mrmc-\gamma^*\|_2=\ell_\mrmc(\ell(\gamma))$ and hence $f$ is continuous.
Notice that $\ell_\mrmc(\ell_\circ)=\psibl$ implies $-\xlnx(\ell_\circ)=\ell_\circ/(\ell_\circ+\psibl)$ and hence $\alpha$ is clearly continuous for $\gamma\neq\gamma^*$, while continuity for $\gamma=\gamma^*$ follows from $-\xlnx(\ell(\gamma^*))=0$.
We have $\alpha(\gamma)[\gamma]_\mrmc+(1-\alpha(\gamma))\gamma=\gamma^*$ by construction.
With $\ell_\mrmc(\ell)\le -1/\ln(\ell)$ we have $\ell_\circ\ge e^{-\psibu}$, while the upper bound follows with $\ell_\mrmc(\ell)\ge -c/\ln(\ell)$.
With $\psibl\le 1/q\le 1/2$ we have $c\psibu>1$, so $\alpha$ is increasing since $\xlnx$ takes its unique minimum at $e^{-1}$.
\end{proof}
So, with the notation from Section \ref{pinning_marginal_distributions} for the general case let
\begin{align*}
\pi^\circ_\mu=(1-\alpha_\mu)\pi_\mu+\alpha_\mu\pi_\bullet
\end{align*}
with $\pi_\bullet=\opm[,\mclp([q]),\gamma]$, $\gamma=[\gammaa_\mu]_\mrmc$, and $\alpha_\mu=\alpha(\|\gammaa_\mu-\gamma^*\|_2)$.
For a decorated graph $G$ let $\piG[,G]^\circ=\pi^\circ_{\lawG[,G]}$ be the projection of $\piG[,G]$ onto $\mclp[*][2]([q])$.
\begin{lemma}\label{lemma_piapproxpstartwo}
Let $m\le\mbu$ and $\GTSM=\GTSM(\sigmaR)$ with $\sigmaR=\sigmaIID$ or $\sigmaR=\sigmaNIS$.
\begin{alphaenumerate}
\item\label{lemma_piapproxpstartwo_general}
For $\mu\in\mclp([q]^n)$ we have $\pi^\circ_\mu\in\mclp[*][2]([q])$ and $\distW(\pi_\mu,\pi^\circ_\mu)\le\alpha_\mu$.
\item\label{lemma_piapproxpstartwo_graph}
There exists $c_\mfkg\in\reals_{>0}$ such that
$\expe[\distW(\piG[,\GTSM],\piG[,\GTSM]^\circ)]\le c\sqrt{\ln(n+1)^3/n}$.
\end{alphaenumerate}
\end{lemma}
\begin{proof}
With $(\bm b,\bm\gamma_0,\bm\gamma_1)\dequal\Bin(1,\alpha_\mu)\otimes\pi_\mu\otimes\pi_\bullet$ we have $\bm\gamma_{\bm b}\dequal\pi_\mu^\circ$ and
\begin{align*}
\expe[\bm\gamma_{\bm b}]
=\expe[\bmone\{\bm b=0\}\bm\gamma_0]+\expe[\bmone\{\bm b=1\}\bm\gamma_1]
=(1-\alpha_\mu)\gammaa_\mu+\alpha_\mu[\gammaa_\mu]_\mrmc=\gamma^*,
\end{align*}
so $\pi_\mu^\circ\in\mclp[*][2]([q])$.
Further, we have
\begin{align*}
\distW(\pi_\mu,\pi_\mu^\circ)
\le\expe[\|\gammaR_{\bm b}-\gammaR_0\|_\mrmtv]
=\alpha_\mu\expe[\|\gammaR_{1}-\gammaR_0\|_\mrmtv]
\le\alpha_\mu.
\end{align*}
With $c'\in(0,1]\times\reals_{\ge 1}$ from Lemma \ref{lemma_gammaaR}\ref{lemma_gammaaR_prob} let
$r=\sqrt{\ln(n)/(2c'_1n)}$ and let $n_{\circ,\mfkg}\in\ints_{\ge 3}$ be such that $2r\le\ell_\circ$ if $n\ge n_\circ$.
For $n\le n_0$ we have
\begin{align*}
\expe[\distW(\piG[,\GTSM],\piG[,\GTSM]^\circ)]\le q\le \sqrt{\frac{q^2n_\circ}{\ln(2)^3}}\sqrt{\frac{\ln(n+1)^3}{n}}.
\end{align*}
Otherwise, with $\distW(\cdot)\le 1$, the first part and $\gammaaR=\gammaaG[,\GTSM]$ we have
\begin{align*}
\expe[\distW(\piG[,\GTSM],\piG[,\GTSM]^\circ)]
\le\expe\left[\bmone\{\|\gammaaR-\gamma^*\|_2<2r\}\alpha(\|\gammaaR-\gamma^*\|_2)\right]+\prob[\|\gammaaR-\gamma^*\|_2\ge 2r].
\end{align*}
With $\|\gammaaR-\gamma^*\|_2\le 2\|\gammaaR-\gamma^*\|_\mrmtv$ and Observation \ref{obs_gammaacounterweight} we have
\begin{align*}
\expe[\distW(\piG[,\GTSM],\piG[,\GTSM]^\circ)]
\le\alpha(2r)+c'_2e^{-c'_1r^2n}
=-\xlnx(2r)+\frac{c'_2}{\sqrt{n}}.
\end{align*}
Hence, for any sufficiently large $\tilde c$ we have
\begin{align*}
\expe[\distW(\piG[,\GTSM],\piG[,\GTSM]^\circ)]
\le\tilde c\sqrt{\frac{\ln(n)}{n}}\ln\left(\frac{n}{\ln(n)}\right)+\tilde c\sqrt{\frac{\ln(n+1)^3}{n}}
\le 2\tilde c\sqrt{\frac{\ln(n+1)^3}{n}}.
\end{align*}
\end{proof}
\subsection{The Interpolation Method}\label{interpolation}
Let $\ThetaP_\mfkg:\ints_{>0}\rarr\reals_{>0}$ be such that $\ThetaP=\omega(1)$ and $\ThetaP\le n$.
This resolves all dependencies on $\ThetaP$ into dependencies on $n$.
Unless mentioned otherwise we assume that $\degae>0$.
\subsubsection{Overview}\label{interpolation_overview}
Recall $\ZF$ from Equation (\ref{ZF_def}) and the second contribution to the Bethe free entropy in Section \ref{bethe_main}, in particular $(\psiR_\circ,\gammaR_\circ)$. The interpolation method relies on the derivative of the function
\begin{align*}
\phiI_{\lawpsi,\gamma^*,\degae,\pi,n}(\tI)=\expe\left[\phiG(\GTSNIS)\right]+\tI\phi_\circ,\,
\phi_\circ=\frac{\degae(k-1)}{\ZFabu k}\expe\left[\xlnx\left(\ZF(\psiR_\circ,\gammaR_\circ\right)\right],
\end{align*}
using the shorthand $\GTSNIS=\GTSM[\mR,\mIR,\setPR](\sigmaNIS_{\mR})$.
Now, if the derivative is (asymptotically) non-negative, then we have $\phiI(0)\le\phiI(1)$, and realignment yields Proposition \ref{proposition_int}.
Hence, we determine the asymptotics of the derivative.
Recall $\piG[,G]$ from Section \ref{pinning_gibbs_marginal_distributions} and its projection $\piG[,G]^\circ\in\mclp[*][2]([q])$ from Section \ref{pinning_limiting_marginal_distributions}.
\begin{proposition}\label{proposition_int_derivpos}
We have $\frac{\partial}{\partial\tI}\phiI(\tI)=\frac{\degae}{\ZFabu k}\expe[\nablaI(\piG[,\GTSNIS]^\circ,\pi)]+\mclo(\ThetaP[-1/3])$.
\end{proposition}
In order to establish Proposition \ref{proposition_int_derivpos},
we first compute the derivative of $\phiI$.
\begin{lemma}\label{lemma_intderiv}
We have $\frac{\partial}{\partial\tI}\phiI(\tI)=\frac{\degae}{k}\Delta^\circ+\phi_\circ-\degae\Delta^\lrarr$, where
\begin{align*}
\Delta^\circ&=\expe[n\phiG(\GTSM[\mR+1,\mIR,\setPR](\sigmaNIS_{\mR+1}))]-\expe[n\phiG(\GTSM[\mR,\mIR,\setPR](\sigmaNIS_{\mR}))],\\
\Delta^\lrarr&=\sum_{i\in[n]}\frac{1}{n}\left(\expe[n\phiG(\GTSM[\mR,\mIR+\opm[,[n],i],\setPR](\sigmaNIS_{\mR}))]-\expe[n\phiG(\GTSM[\mR,\mIR,\setPR](\sigmaNIS_{\mR}))]\right).
\end{align*}
\end{lemma}
\begin{proof}
We consider $n+1$ independent Poisson variables $\mR$, $\mIR$ depending on $\tI$, while the remainder does not. Hence, we use the product rule, which amounts to taking each derivative individually using Observation \ref{obs_phi_lipschitz} and Corollary \ref{cor_dega}. But for $\bm x\dequal\Po(at+b)$ we have
\begin{align*}
\frac{\partial}{\partial t}\expe[f(\bm x)]
=\sum_xf(x)\frac{\partial}{\partial t}\prob[\bm x=x]
=-a\expe[f(\bm x)]+a\expe[f(\bm x+1)].
\end{align*}
\end{proof}
The second contribution $\phi_\circ$ in Lemma \ref{lemma_intderiv} is exactly what we need.
For the other contributions recall $(\psiR,\hR,\gammaR_{\pi})$ from Equation (\ref{nablaI_def}) and let $\bm\pi=(\piG[,\GTSNIS],\pi)$.
\begin{lemma}\label{lemma_intderiv_zfm}
We have $\ZFabu\Delta^\lrarr=\expe[\xlnx(\ZFM(\psiR,\hR,\gammaR_{\piR}))]$.
\end{lemma}
The proof of Lemma \ref{lemma_intderiv_zfm} is presented in Section \ref{int_interpolator}.
The first contribution is demanding, because the joint Gibbs law is not a product measure. This is where Proposition \ref{proposition_pinG} comes into play.
\begin{lemma}\label{lemma_intderiv_zf}
We have $\ZFabu\Delta^\circ=\expe[\xlnx(\ZF(\psiR,\gammaR_{\piR,1}))]+\mclo(\ThetaP[-1/3])$.
\end{lemma}
The proof of Lemma \ref{lemma_intderiv_zf} is presented in Section \ref{int_factor}.
Proposition \ref{proposition_int_derivpos} now follows by establishing Lipschitz continuity of $\nablaI$ and thereby justifying the transition to the projection $\piG[,\GTSNIS]^\circ$.
The proof is presented in Section \ref{proof_proposition_int_derivpos}.
The proof of Proposition \ref{proposition_int} and the respective version for graphs with external fields over random factor counts $\mR^*$ is presented in Section \ref{proof_proposition_int}.
\subsubsection{Adding an Interpolator}\label{int_interpolator}
Fix the variable $i\in[n]$ with the additional interpolator and let
\begin{align*}
\Delta^\lrarr_{\mathrm{v}}(i)&=\expe[n\phiG(\GTSM[\mR,\mIR+\opm[,[n],i],\setPR](\sigmaNIS_{\mR}))]-\expe[n\phiG(\GTSM[\mR,\mIR,\setPR](\sigmaNIS_{\mR}))].
\end{align*}
\begin{lemma}\label{lemma_int_interpolator_i}
With $(\GR',\psiIR)\dequal\GTSNIS\otimes\psiIRa$ we have
\begin{align*}
\Delta^\lrarr_{\mathrm{v}}(i)
=\frac{1}{\ZFabu}\expe\left[\xlnx\left(\sum_{\tau\in[q]}\lawG[,\GR']|_i(\tau)\psiIR(\tau)\right)\right].
\end{align*}
\end{lemma}
\begin{proof}
Using independence due to Observation \ref{obs_TSM_iid}, we have a coupling 
$(\bm G_-,\bm G_+)$ of 
$\GTSM[\mR,\mIR,\setPR](\sigmaNIS_{\mR})$ and 
$\GTSM[\mR,\mIR+\opm[,[n],i],\setPR](\sigmaNIS_{\mR})$,
i.e.~given $\sigmaNIS_{\mR}$ and $\bm G_-$ we attach the factor $(\mI_i+1)$ to $i$
equipped with a weight given by $\psiITSa[,\sigma(i)]$ to obtain $\bm G_+$.
Explicitly introducing the conditional expectation gives
\begin{align*}
\Delta^\lrarr_{\mathrm{v}}(i)&=\expe\left[\Delta^\lrarr_{\mathrm{vms},i}(\mR,\mIR,\setPR,\sigmaNIS_{\mR})\right],\,
\Delta^\lrarr_{\mathrm{vms},i}(m,\mI,\setP,\sigma)
=\expe[n\phiG(\bm G_+)-n\phiG(\bm G_-)].
\end{align*}
With $\delta(G,G')=n\phiG(G')-n\phiG(G)$ we have $\delta(G,G')=\ln(\ZG(G')/\ZG(G))$, and further $\ZG(G')=\sum_\sigma\psiG[,G'](\sigma)$. If $G'$ is an extension of $G$ as above, i.e.~obtained by adding factors $\mcla[+]$ with wire-weight pairs $w=(v_a,\psi_a)_{a\in\mcla[+]}$, then we have $\psiG[,G'](\sigma)=\psiG[,G](\sigma)\prod_{a\in\mcla[+]}\psi_a(\sigma_{v_a})$.
This gives $\delta(G,G')=\ln(\psiWgG[,G](w))$ with
\begin{align}\label{equ_psiWgG}
\psiWgG[,G](w)=\sum_\sigma\lawG[,G](\sigma)\prod_{a\in\mcla[+]}\psi_a(\sigma_{v(a)})
=\expe\left[\prod_{a\in\mcla[+]}\psi_a\left(\sigmaRG[,G,v(a)]\right)\right],
\end{align}
so the difference of the free entropies is the logarithm of the expected weight of the additional factors of $G'$, under the Gibbs measure of the (smaller) base graph $G$.
So, using $(\GR''(\sigma),\psiITS,\psiIR)\dequal\GTSM(\sigma)\otimes\psiITSa[,\sigma(i)]\otimes\psiIRa$ for brevity, we have
\begin{align*}
\Delta^\lrarr_{\mathrm{vms}}
=\expe\left[\ln\left(\psiWgG[,\GR''(\sigma)]\left(i,\psiITS\right)\right)\right]
=\expe\left[\frac{\psiIR(\sigma_i)}{\ZFabu}\ln\left(\psiWgG[,\GR''(\sigma)](i,\psiIR)\right)\right].
\end{align*}
Taking the expectation over $\sigmaNIS$, using $\GR=\GR''(\sigmaNIS)$ and
the Nishimori condition \ref{obs_nishicond}\ref{obs_nishicond_dec} yields
\begin{align*}
\expe[\Delta^\lrarr_{\mathrm{vms}}(\sigmaNIS)]
=\expe\left[\frac{\psiIR(\sigmaNIS_i)}{\ZFabu}\ln\left(\psiWgG[,\GR](i,\psiIR)\right)\right]
=\expe\left[\frac{\psiIR(\sigmaRG[,\GR,i])}{\ZFabu}\ln\left(\psiWgG[,\GR](i,\psiIR)\right)\right].
\end{align*}
For the leading coefficient we take the conditional expectation given $\GR$ and $\psiIR$, i.e.~the expectation over the Gibbs spins $\sigmaRG[,\GR]$ only, which exactly matches the definition of $\psiWgG$ and hence
\begin{align*}
\expe[\Delta^\lrarr_{\mathrm{vms}}(\sigmaNIS)]
&=\frac{1}{\ZFabu}\expe\left[\xlnx\left(\psiWgG[,\GR](i,\psiIR)\right)\right],
\psiWgG[,\GR](i,\psiIR)=\sum_{\tau\in[q]}\lawG[,\GR]|_i(\tau)\psiIR(\tau).
\end{align*}
\end{proof}
Now, recall that we have $\Delta^\lrarr=\expe[\Delta^\lrarr_{\mathrm{v}}(\bm i)]$ for $\bm i\dequal\unif([n])$ and that $\lawG[,G]|_{\bm i}\dequal\piG[,G]$.
Further, recall $(\psiRa,\hR,\gammaR)$ and the definition of $\psiIR\dequal\psiIRa$ from Equation (\ref{psiIR_def}), which gives
\begin{align*}
\sum_{\tau\in[q]}\lawG[,G]|_{\bm i}(\tau)\psiIR(\tau)
\dequal\sum_{\tau\in[q]}\lawG[,G]|_{\bm i}(\tau)\sum_{\tau'}\bmone\{\tau'_{\hR}=\tau\}\psiRa(\tau')\prod_{h\neq\hR}\gammaR_{h}(\tau'_h)
\dequal\ZFM(\psiR,\hR,\gammaR_{\pi'})
\end{align*}
with $\pi'=(\piG[,G],\pi)$, $(\bm i,\psiIR)\dequal\bm i\otimes\psiIR$ and $(\psiR,\hR,\gammaR_{\pi'})$ from the assertion in Lemma \ref{lemma_intderiv_zfm}.
This completes the proof by considering the conditional expectation given $\GR'$ in $\Delta^\lrarr$.
\subsubsection{Adding a Factor}\label{int_factor}
Using the shorthand $\phiTSM[m](\sigma)=\expe[\phiG(\GTSM[m,\mIR,\setPR](\sigma))]$ we may rewrite
$\Delta^\circ=\expe\left[n\phiTSM[\mR+1](\sigmaNIS_{\mR+1})]-\expe[n\phiTSM[\mR](\sigmaNIS_{\mR})\right]$.
In the first step we align the ground truths, i.e.~we replace $\sigmaNIS_{\mR+1}$ by $\sigmaNIS_{\mR}$, and introduce the following typical event.
With $r(n)=\ln(n)/\sqrt{n}$ and $\mclb[][\circ]$ from Corollary \ref{cor_dega} for $r$, let
$\mclb[][\Gamma]=\{\sigma\in[q]^n:\|\gammaN[,\sigma]-\gamma^*\|_\mrmtv<r\}$
and further $\mcle=\{\degaR\in\mclb[][\circ],\sigmaNIS_{\mR}\in\mclb[][\Gamma]\}$.
\begin{lemma}\label{lemma_int_factor_nishieq}
We have $\Delta^\circ=n\expe\left[\bmone\mcle(\phiTSM[\mR+1](\sigmaNIS_{\mR})-\phiTSM[\mR](\sigmaNIS_{\mR}))\right]+\mclo(r(n))$.
\end{lemma}
\begin{proof}
Let $(\sigmaNIS^-_m,\sigmaNIS^+_m)$ be a coupling of $\sigmaNIS_m$, $\sigmaNIS_{m+1}$ from the coupling lemma \ref{obs_tv}\ref{obs_tv_coupling}
and further $\mcle[][\prime]=\{\degaR\in\mclb[][\circ],\sigmaNIS_{\mR}^-\in\mclb[][\Gamma],\sigmaNIS_{\mR}^+\in\mclb[][\Gamma]\}$.
Let $\bm\Phi^+=n\phiTSM[\mR+1](\sigmaNIS^+_{\mR})$ and $\bm\Phi^-=n\phiTSM[\mR](\sigmaNIS^-_{\mR})$.
With $c_\Phi$ from Observation \ref{obs_phi_lipschitz},
$c^\circ$ from Corollary \ref{cor_dega},
Observation \ref{obs_pin_basic}, $\ln(n)/\sqrt{n}\le 2/e$, $n\ge 1$ and $\ThetaP\le n$ we obtain
\begin{align*}
|\expe[\bmone\{\degaR\not\in\mclb[][\circ]\}\bm\Phi^+]|
&\le\expe\left[\bmone\left\{\degaR\not\in\mclb[][\circ]\right\}c_\Phi\left(\frac{\degaR n}{k}+1+(1-\tI)\degae n+\frac{\ThetaP}{2}\right)\right]
\le c'_2ne^{-c'_1\ln(n)^2},\\
c'_1&=\frac{c_1^\circ}{2},\,
c'_2=c_\Phi\left(\frac{1}{k}+1+\degabu+\frac{1}{2}\right)c^\circ_2,
\end{align*}
and we obtain the same bound for $\bm\Phi^-$.
On the event $\degaR\in\mclb[][\circ]$ we have $\mR+1\le\mbu$ since $\degabu$ is large.
But then with $\hat c$ from Corollary \ref{cor_mutcont}\ref{cor_mutcont_prob} we obtain
\begin{align*}
|\expe[\bmone\{\degaR\in\mclb[][\circ],\sigmaNIS^+_{\mR}\not\in\mclb[][\Gamma]\}\bm\Phi^+]|
&\le c_\Phi \hat c_2\left(\frac{\degabu n}{k}+1+\degabu n+\frac{\ThetaP}{2}\right)e^{-\hat c_1\ln(n)^2}
\le c''_2ne^{-c''_1\ln(n)^2},\\
c''_1&=\hat c_1,\,
c''_2=c_\Phi\hat c_2\left(\frac{\degabu}{k}+1+\degabu+\frac{1}{2}\right).
\end{align*}
The bound for $\sigmaNIS^-_{\mR}$ is the same, and the same bounds also follow for $\bm\Phi^-$.
This shows that
\begin{align*}
\Delta^\circ=n\expe\left[\bmone\mcle[][\prime](\phiTSM[\mR+1](\sigmaNIS^+_{\mR})-\phiTSM[\mR](\sigmaNIS^-_{\mR}))\right]+\mclo\left(ne^{-\tilde c\ln(n)^2}\right)
\end{align*}
for $\tilde c=\min(c'_1,c''_1)$.
Since $(\sigmaNIS^-,\sigmaNIS^+)$ is a coupling from the coupling lemma \ref{obs_tv}\ref{obs_tv_coupling}, we can use $c$ from Corollary \ref{obs_niscoupling}\ref{obs_niscoupling_tv} on $\mcle[][\prime]$.
Further, let $n_{\circ,\mfkg}$ be such that $\ln(n)/\sqrt{n}\le\psibl/4$ is $n\ge n_\circ$, then we have $\gammaN[,\sigmaNIS^-_{\mR}]\ge\psibl/2$ on $\mcle[][\prime]$ if $n\ge n_\circ$,
so using $c'$ from Corollary \ref{cor_contphiTSM} and $\|\gammaN[,\sigmaNIS^+_{\mR}]-\gammaN[,\sigmaNIS^-_{\mR}]\|_\mrmtv\le 2r(n)$ on $\mcle[][\prime]$ we obtain
\begin{align*}
|n\expe\left[\bmone\mcle[][\prime]\bmone\{\sigmaNIS^+_{\mR}\neq\sigmaNIS^-_{\mR}\}(\phiTSM[\mR+1](\sigmaNIS^+_{\mR})-\phiTSM[\mR](\sigmaNIS^-_{\mR}))\right]|
\le cc'\left(2r(n)+\frac{k}{n}\right)=\mclo\left(\frac{\ln(n)}{\sqrt{n}}\right).
\end{align*}
Now, we substitute $\sigmaNIS^+_{\mR}$ and then drop $\bmone\{\sigmaNIS^+_{\mR}=\sigmaNIS^-_{\mR},\sigmaNIS^+_{\mR}\in\mclb[][\Gamma]\}$ at expense $\mathcal O(1/n)$.
\end{proof}
With Lemma \ref{lemma_int_factor_nishieq} we obtain $\GTSM[\mR+1,\mIR,\setPR](\sigmaNIS_{\mR})$ from $\GTSNIS$ given $(\sigmaNIS_{\mR},\GTSNIS)$ by attaching a single additional standard factor $\mR+1$, since the ground truths coincide, as do the decorations.
So, we consider $\GTSNIS\otimes\wTSa[,\sigmaNIS_{\mR}]$, follow the steps in Section \ref{int_interpolator} to reduce this to $(\GTSNIS,\vR,\psiR)\dequal\GTSNIS\otimes\wRa$
and thereby, using $\gammaNIS_m=\gammaN[,\sigmaNIS]$, obtain
\begin{align*}
\Delta^\circ=\expe\left[\frac{\bmone\mcle}{\ZFa(\gammaNIS_{\mR})}\psiR(\sigmaNIS_{\mR,\vR})\ln\left(\psiWgG[,\GTSNIS](\vR,\psiR)\right)\right]+\mclo(r(n)).
\end{align*}
Now, recall that $\mcle$ covers $\sigmaNIS_{\mR}\in\mclb[][\Gamma]$, so with $c$ from Observation \ref{obs_fad}\ref{obs_fad_maxbound} and Observation \ref{obs_fad}\ref{obs_fad_bounds} we have
$1-\psibu cr(n)^2\le\ZFa(\gammaNIS_{\mR})/\ZFabu\le 1$, which yields
$\ZFa(\gammaNIS_{\mR})/\ZFabu=1+\mclo(r(n)^2)$ and
\begin{align*}
\Delta^\circ=(1+\mclo(r(n)^2))\expe\left[\frac{\bmone\mcle}{\ZFabu}\psiR(\sigmaNIS_{\mR,\vR})\ln\left(\psiWgG[,\GTSNIS](\vR,\psiR)\right)\right]+\mclo(r(n)).
\end{align*}
With $|\psiR(\sigmaNIS_{\mR,\vR})\ln(\psiWgG[,\GTSNIS](\vR,\psiR))/\ZFabu|\le\psibu^2\ln(\psibu)$
and the Nishimori condition \ref{obs_nishicond}\ref{obs_nishicond_dec}, analogously to Section \ref{int_interpolator}, we obtain
\begin{align*}
\Delta^\circ=\expe\left[\frac{\bmone\mcle}{\ZFabu}\xlnx\left(\psiWgG[,\GTSNIS](\vR,\psiR)\right)\right]+\mclo(r(n)).
\end{align*}
Notice that we can drop the restriction to $\mcle$ due to the uniform bound $\psibu^2\ln(\psibu)$ on the argument of the expectation at expense $\mclo(e^{-\tilde c\ln(n)^2})$ with $\tilde c$ from the proof of Lemma \ref{lemma_int_factor_nishieq}.
Finally, we turn to the application of Proposition \ref{proposition_pinG}.
With $\iota_\circ$ from Section \ref{pinning_gibbs} and $\hat{\bm\mu}$ from Proposition \ref{proposition_pinG} notice that $\expe[\iota_\circ(\hat{\bm\mu}_{\mR,\mIR},\vR)]=\expe[\iota_\ell(\hat{\bm\mu}_{\mR,\mIR})]$ for $\ell=k$ and hence we can use Proposition \ref{proposition_pinG}\ref{proposition_pinG_NISDKL}, Markov's inequality and the bound $\psibu^2\ln(\psibu)$ on the argument of the expectation with $\delta=\ThetaP[-2/3]$ to obtain
\begin{align*}
\Delta^\circ=\expe\left[\frac{\bmone\mcle}{\ZFabu}\xlnx\left(\psiWgG[,\GTSNIS](\vR,\psiR)\right)\right]+\mclo\left(r(n)+\ThetaP[-1/3]\right),\,
\mcle=\left\{\iota_\circ(\hat{\bm\mu}_{\mR,\mIR},\vR)<\delta\right\}.
\end{align*}
With $\nu_\circ$ from Remark \ref{remark_epssym} and by Observation \ref{obs_tv}\ref{obs_tv_coupling} there exists a coupling $(\bm\tau,\bm\tau')$ of $\mu|_v$ and $\bigotimes_h\mu|_{v(h)}$ such that $\prob[\bm\tau\neq\bm\tau^*]\le\nu_\circ(\mu,v)$ and hence
$|\expe[\psi(\bm\tau)-\psi(\bm\tau')]|
\le2\psibu\nu_\circ(\mu,v)$, $\psi\in\domPsi$.
So, with $\bm\zeta=\sum_\tau\bm\psi(\tau)\prod_h\lawG[,\GTSNIS]|_{\vR(h)}(\tau_h)$ we have
$|\psiWgG[,\GTSNIS](\vR,\psiR)-\bm\zeta|\le 2\psibu\nu_\circ(\hat{\bm\mu}_{\mR,\mIR},\vR)$.
Now, we can use Lipschitz continuity of $\xlnx$ on $[\psibl,\psibu]$ since both arguments live in this interval, i.e.~we obtain $L_\mfkg$ such that
$|\xlnx(\psiWgG[,\GTSNIS](\vR,\psiR))-\xlnx(\bm\zeta)|\le 2L\psibu\nu_\circ(\hat{\bm\mu}_{\mR,\mIR},\vR)$.
With Remark \ref{remark_epssym} we have
$|\xlnx(\psiWgG[,\GTSNIS](\vR,\psiR))-\xlnx(\bm\zeta)|
\le\sqrt{2}L\psibu\sqrt{\iota_\circ(\hat{\bm\mu}_{\mR,\mIR},\vR)}\le\sqrt{2}L\psibu\sqrt{\delta}$ on $\mcle$.
Then we drop the restriction to $\mcle$
and notice that $(\lawG[,\GTSNIS]|_{\vR(h)})_h\dequal\piG[,\GTSNIS]^{\otimes k}$ since $\vR$ is uniform, so
\begin{align*}
\ZFabu\Delta^\circ=\expe[\xlnx(\ZF(\psiR,\gammaR_{\piR,1}))]+\mclo\left(r(n)+\ThetaP[-1/3]\right).
\end{align*}
This completes the proof since $\ThetaP\le n$ and hence $\ThetaP[-1/3]=\omega(r(n))$.
\subsubsection{Proof of Proposition \ref{proposition_int_derivpos}}\label{proof_proposition_int_derivpos}
Combining Lemma \ref{lemma_intderiv}, Lemma \ref{lemma_intderiv_zfm} and Lemma \ref{lemma_intderiv_zf} gives
\begin{align*}
\frac{\partial}{\partial\tI}\phiI(\tI)
&=\frac{\degae}{k\ZFabu}\expe[\xlnx(\ZF(\psiR,\gammaR_{\piR,1}))]+\mclo\left(\ThetaP[-1/3]\right)
+\phi_\circ-\frac{\degae}{\ZFabu}\expe[\xlnx(\ZFM(\psiR,\hR,\gammaR_{\piR}))]\\
&=\frac{\degae}{\ZFabu k}\expe[\nablaI(\piR)]+\mclo\left(\ThetaP[-1/3]\right)
\end{align*}
using $\degae\in[0,\degabu]$ and $\psibl\le\ZFabu\le\psibu$ from Observation \ref{obs_fad}\ref{obs_fad_bounds}.
\begin{lemma}\label{lemma_nablaI_lipschitz}
There exists $L_\mfkg\in\reals_{>0}$ such that $|\nablaI(\pi_1,\pi_3)-\nablaI(\pi_2,\pi_3)|\le L\distW(\pi_1,\pi_2)$ for all $\pi\in\mclp[][2]([q])^3$.
\end{lemma}
\begin{proof}
Let $\ell_\mfkg\in\reals_{>0}$ be such that $|\xlnx(t_1)-\xlnx(t_2)|\le\ell|t_1-t_2|$ for $t\in[\psibl,\psibu]^2$.
For a coupling $\rho\in\Gamma(\pi_1,\pi_2)$ let $(\psiR,\hR,\gammaR)\dequal\lawpsi\otimes\unif([k])\otimes(\rho^{\otimes k}\otimes\pi_3^{\otimes k})$ with $\gammaR\in(\mclp([q])^k)^3$, so with Jensen's inequality and the triangle inequality we have
\begin{align*}
\delta&=|\nablaI(\pi_1,\pi_3)-\nablaI(\pi_2,\pi_3)|\\
&\le\ell\expe[|\ZF(\psiR,\gammaR_1)-\ZF(\psiR,\gammaR_2)|]
+k\ell\expe[|\ZFM(\psiR,\hR,(\gammaR_1,\gammaR_3))-\ZFM(\psiR,\hR,(\gammaR_2,\gammaR_3))|].
\end{align*}
Expanding the definitions, using the triangle inequality, $\psiR\le\psibu$ and Observation \ref{obs_tv}\ref{obs_tv_prod} further yields
\begin{align*}
\delta\le\ell\psibu\expe\left[\left\|\bigotimes_h\gammaR_{1,h}-\bigotimes_h\gammaR_{1,h}\right\|_\mrmtv\right]+k\psibu\expe\left[\|\gammaR_{1,\hR}-\gammaR_{2,\hR}\|_\mrmtv\right]
\le L\expe\left[\|\gammaR_{1,1}-\gammaR_{2,1}\|_\mrmtv\right]
\end{align*}
with $L=2k\ell\psibu$, which completes the proof since this holds uniformly for all couplings $\rho$.
\end{proof}
Lemma \ref{lemma_nablaI_lipschitz} with Lemma \ref{lemma_piapproxpstartwo}\ref{lemma_piapproxpstartwo_graph}, $|\nablaI|\le 2k\xlnx(\psibu)$, Corollary \ref{cor_dega}
and $1/\ThetaP=\Omega(1/n)$ completes the proof.
\subsubsection{Proof of Proposition \ref{proposition_int}}\label{proof_proposition_int}
First, we derive the result for graphs with pins and external fields, but without interpolators.
\begin{lemma}\label{lemma_int_bethebuwithpins}
For $\tI=1$ we have
$\expe[\phiG(\GTSNIS)]\ge\bethebu+\mclo(\frac{1}{\ThetaP[1/3]}+\frac{\ThetaP}{n})$.
\end{lemma}
\begin{proof}
From Proposition \ref{proposition_int_derivpos} we obtain $c_\mfkg\in\reals_{>0}$ such that
\begin{align*}
\frac{\partial}{\partial\tI}\phiI(\tI)\ge\nablaIbl-\frac{c}{\ThetaP[1/3]}\ge-\frac{c}{\ThetaP[1/3]}
\end{align*}
since $\nablaIbl\ge 0$ by assumption, so integration yields $\phi^{\lrarr}(1)-\phi^{\lrarr}(0)\ge-c/\ThetaP[1/3]$.
But for $\tI=0$ with $\check{\bm\psi}_i$ denoting the pin, i.e.~$\check{\bm\psi}_i\equiv 1$ for $i\not\in\setPR$ and $\check{\bm\psi}_i(\tau)=\bmone\{\tau=\sigmaNIS_{\mR}\}$ otherwise,
and using the notions from Observation \ref{obs_TSM_iid} we have
\begin{align*}
\phiI(0)&=\frac{1}{n}\expe\left[\ln\left(\prod_i\sum_\tau\gamma^*(\tau)\check{\bm\psi}_i(\tau)\prod_{h\in[\mIR_i]}\psiITSM[i,h](\tau)\right)\right]\\
&=\sum_i\frac{1}{n}\expe\left[\ln\left(\sum_\tau\gamma^*(\tau)\check{\bm\psi}_i(\tau)\prod_{h\in[\mIR_i]}\psiITSM[i,h](\tau)\right)\right]
\end{align*}
Notice that $\sigmaNIS_0\dequal\sigmaIID$ and $\mIR_i\dequal\Po(\degae)$, so
\begin{align*}
\phiI(0)&=\expe\left[\ln\left(\sum_\tau\gamma^*(\tau)\check{\bm\psi}_1(\tau)\prod_{h\in[\mIR_1]}\psiITSM[1,h](\tau)\right)\right].
\end{align*}
Notice that the argument of the logarithm is in $[\psibl^{\mIR_1+1},\psibu^{\mIR}]$ and the probability of $\check{\psiR}_1\equiv 1$ is $1-\expe[\thetaPR/n]=1-\ThetaP/(2n)$, so we have
\begin{align*}
\phiI(0)&=\left(1-\frac{\ThetaP}{2n}\right)\expe\left[\ln\left(\sum_\tau\gamma^*(\tau)\prod_{h\in[\mIR_1]}\psiITSM[1,h](\tau)\right)\right]+(\degae+1)\ln(\psibu)\mathcal O\left(\frac{\ThetaP}{n}\right)\\
&=\expe\left[\ln\left(\sum_\tau\gamma^*(\tau)\prod_{h\in[\mIR_1]}\psiITSM[1,h](\tau)\right)\right]+\mathcal O\left(\frac{\ThetaP}{n}\right).
\end{align*}
Next, we use the $(\psiITSM[1,h],\psiIRa)$-derivatives to recover the first contribution to the Bethe functional and hence
\begin{align*}
\expe\left[\phiG(\GTSNIS)\right]\ge\bethe(\pi)+\mclo\left(\frac{1}{\ThetaP[1/3]}+\frac{\ThetaP}{n}\right).
\end{align*}
\end{proof}
Based on Lemma \ref{lemma_int_bethebuwithpins} we restrict to $\tI=1$ and $\mI\equiv 0$ in the remainder, where we also discuss all $\degae\in[0,\degabu]$.
Next, we derive the result for graphs with external fields only.
Recall $m^\circ$ and $\sigma^\circ$ from Proposition \ref{proposition_phi_concon_cont}.
\begin{proposition}\label{proposition_int_external}
Let $\mI\equiv 0$ and $\setP=\emptyset$.
\begin{alphaenumerate}
\item\label{proposition_int_external_po}
We have $\expe[\phiG(\GTSM[\mR](\sigmaIID))]\ge\bethebu+\mclo(n^{-1/4})$.
\item\label{proposition_int_external_m}
For $d=km/n\le\degabu$ we have $\expe[\phiG(\GTSM(\sigmaIID))]\ge\bethebu(d)+\mclo(n^{-1/4})$.
\item\label{proposition_int_external_general}
We have $\expe[\phiG(\GTSM[\mR^*](\sigmaIID))]\ge\bethebu+\mclo(n^{-1/4}+\deltam+\epsm)$.
\end{alphaenumerate}
\end{proposition}
\begin{proof}
With Lemma \ref{lemma_int_bethebuwithpins} and Proposition \ref{proposition_pin_qfed} we have
\begin{align*}
\expe\left[\phiG(\GTSM[\mR](\sigmaNIS_{\mR}))\right]\ge\bethebu+\mclo\left(\frac{1}{\ThetaP[1/3]}+\frac{\ThetaP}{n}\right).
\end{align*}
This yields $\expe[\phiG(\GTSM[\mR](\sigmaNIS_{\mR}))]\ge\bethebu+\mclo(n^{-1/4})$ for $\ThetaP(n)=n^{3/4}\in(0,n]$ and $\degae\in(0,\degabu]$.
For $\degae=0$ notice that $\phiG(\GTSM[\mR](\sigmaNIS_{\mR}))=0$ and $\bethe\equiv 0$.
Without loss of generality let $\deltam\ge\deltam^\circ$ with $\deltam^\circ=\ln(n)/\sqrt{n}\le 1$ and $\epsm\ge\epsm^\circ$ with $\epsm^\circ=c_2e^{-c_1\ln(n)^2/2}$ and $c$, $c_2$ large, from Corollary \ref{cor_dega} since this does not affect the assertions.
Hence, we may take $\mR^*=\mR$, and then Corollary \ref{cor_phiTSIIDNIS}\ref{cor_phiTSIIDNIS_p} applied to $\mR^*$ and to $\mR$ yields
\begin{align*}
\expe[\phiG(\GTSM[\mR^*](\sigmaIID))]
&=\expe[\phiG(\GTSM[m^\circ](\sigma^\circ))]+\mclo(\epsm+\deltam+n^{-1/2})\\
&=\expe[\phiG(\GTSM[\mR](\sigmaNIS_{\mR}))]+\mclo(\epsm+\deltam+n^{-1/2})\\
&\ge\bethebu+\mclo(\epsm+\deltam+n^{-1/4}),
\end{align*}
which establishes Part \ref{proposition_int_external}\ref{proposition_int_external_general}.
Now, consider the special case $\deltam=\deltam^\circ=o(n^{-1/4})$ and $\epsm=\epsm^\circ=o(n^{-1/4})$, then  Part \ref{proposition_int_external}\ref{proposition_int_external_po} follows as a special case from  Part \ref{proposition_int_external}\ref{proposition_int_external_general}.
But also Part \ref{proposition_int_external}\ref{proposition_int_external_m} now follows as a special case from  Part \ref{proposition_int_external}\ref{proposition_int_external_general} by further considering $\degae=d$ and $\mR^*_{n'}=m^\circ_{n'}$ for $n'\in\ints_{>0}$, which in particular gives $\mR^*=m$.
\end{proof}
Observation \ref{obs_standard_graphs} yields the corresponding results for graphs without external fields and thereby completes the proof of Proposition \ref{proposition_int}.
\subsection{The Aizenman-Sims-Starr Scheme}\label{ass}
This section is dedicated to the proof of Proposition \ref{proposition_ass}, and hence Theorem \ref{thm_bethe}.
For the remainder of this contribution we fix $\tI=1$ and $\mI\equiv 0$, which also resolves any dependencies on $\pi$, $\psiI$ and $\psiIR$.
With $C_1\in(0,1)$ from Proposition \ref{proposition_pinG}\ref{proposition_pinG_IIDDKL}
let $c=C_1/3$, $\rho=c/(1+c)$ and $\ThetaP(n)=n^{1-\rho}$.
Notice that $\rho\in(0,1/4)$ and $\ThetaP\in[0,n]$.
Assume that $\degae>0$ unless mentioned otherwise.
\subsubsection{Overview}\label{ass_telescoping_sum}
We avoided the introduction of the projected Gibbs marginal distribution $\piG[,G]^\circ$ from Section \ref{pinning_limiting_marginal_distributions} in Section \ref{ps_bethe}, but now we can state the stronger version.
\begin{proposition}\label{proposition_ass_pinned}
We have $\expe[\phiG(\GTSM)]=\expe[\bethe(\piG[,\GTSM]^\circ)]+\mclo(n^{-\rho})$ with $\GTSM=\GTSM[\mR,\setPR](\sigmaIID)$.
\end{proposition}
We establish Proposition \ref{proposition_ass_pinned} using the Aizenman-Sims-Starr scheme, which is based on the representation of the quenched free entropy density as the average change of the quenched free entropies, meaning
\begin{align*}
\expe\left[\phiG(\GTSM[n,\mR,\setPR](\sigmaIID))\right]
&=\sum_{n'=0}^{n-1}\frac{1}{n}\Phi_{\Delta,n'},\,
\Phi_{\Delta,n'}=\expe\left[(n'+1)\phiG\left(\GR_{+,n'}\right)\right]-\expe\left[n'\phiG\left(\GR_{-,n'}\right)\right],\\
\GR_{+,n}&=\GTSM[n+1,\mR_{n+1},\setPR_{n+1}](\sigmaIID_{n+1}),\,
\GR_{-,n}=\GTSM[n,\mR_{n},\setPR_{n}](\sigmaIID_{n}),
\end{align*}
using $\expe[n'\phiG(\GR_{-,n'})]=0$ for $n'=0$.
Intuitively, we observe that if $\Phi_{\Delta,n}$ converges, so does the quenched free entropy density, with the same limit.
Hence, the main focus of this section is to establish the following result.
\begin{lemma}\label{lemma_ass_PhiDelta}
We have $\Phi_{\Delta,n}=\expe[\bethe(\piG[,\GTSM]^\circ)]+\mclo(n^{-\rho})$ with $\GTSM=\GTSM[\mR,\setPR](\sigmaIID)$.
\end{lemma}
Similar to Section \ref{interpolation} and Section \ref{phi_concon}
we will control the difference $\Phi_\Delta$ of the expectations by introducing a coupling of $\GR_-$ and $\GR_+$, say $(\GR^-,\GR^+)$.
However, as opposed to the previous sections we now have to deal with an additional variable.
Since the average degree is $\degae$, i.e.~we expect the new variable to wire to $\degae$ factors, but the expected difference in the number of factors is only $\degae/k$, we will have to rewire factors - like in Section \ref{phi_concon}. But as opposed to Section \ref{phi_concon} we cannot afford rough estimates, and have to control the behavior on a very granular level instead.

We can partially recover the convenient situation in Section \ref{interpolation} by taking the intersection graph, or base graph, as a starting point and then enrich this graph to obtain $\GR^-$ and $\GR^+$ each, say a triplet $(\GR_\cap,\GR^-,\GR^+)$.
The expectations give $\degae n/k$ factors for $\GR^-$,
$\degae(n+1)/k$ factors for $\GR^+$, with roughly $\degae$ wired to $i=n+1$.
So, we can hope for $\frac{\degae(n+1)}{k}-\degae=\frac{\degae n}{k}-\frac{\degae(k+1)}{k}$ factors in $\GR_\cap$ and attaching the remaining factors to obtain $\GR^-$ and $\GR^+$ respectively. This coupling allows to rewrite
\begin{align*}
\Phi_\Delta(n)=\expe\left[\ln\left(\frac{\ZG(\GR^+)}{\ZG(\GR_\cap)}\right)\right]-\expe\left[\ln\left(\frac{\ZG(\GR^-)}{\ZG(\GR_\cap)}\right)\right].
\end{align*}
Since the coupling is fairly involved, we present it in three parts.
In Section \ref{coupling_graphs} we use the discussion in Section \ref{known_neighborhoods} to couple the standard factor graphs.
Then we couple the factor counts using Observation \ref{obs_poisson}.
Finally, we turn to the pins in Section \ref{coupling_pins} and combine the three parts.
In Section \ref{ass_base_graph} we show that the law of $\GR_\cap$ is close to $\GR^-$, which allows to recycle our results for the teacher-student model.
Then, in Section \ref{ass_factor_count_asymptotics} we show that our rough estimates for the expectations are asymptotically correct.

Next, we discuss the asymptotics of the two contributions to $\Phi_\Delta(n)$ separately.
We start with the easier $\GR^-$-contribution, since only factors are added, which is covered by Sections \ref{assfactor_typical} to \ref{assfactor_final}.
Then we discuss the $\GR^+$-contribution in Sections \ref{assvar_typical} to \ref{assvar_final}.

While the discussion of the $\GR^-$-contribution is conceptually similar to the discussion in Section \ref{int_factor}, there are several additional obstacles.
In Section \ref{assfactor_typical} we discuss the restriction to typical instances.
In Section \ref{assfactor_normalization} we introduce an approximation of the joint distribution to resolve dependencies.
In Section \ref{assfactor_gibbsproduct} we use Proposition \ref{proposition_pinG} to transition to independent Gibbs marginals.
In Section \ref{assfactor_marginaldist} we use Proposition \ref{cor_piCR} to resolve the dependencies of the Gibbs marginals on the ground truth.
Finally, in Section \ref{assfactor_final} we discuss the remaining asymptotics, followed by the Lipschitz continuity of the factor contribution to the Bethe free entropy, which allows to transition to the projected Gibbs marginal distributions from Section \ref{pinning_limiting_marginal_distributions}.

Sections \ref{assvar_typical} to \ref{assvar_final} are devoted to the respective steps for the $\GR^+$-contribution, approaching the variable contribution to the Bethe free entropy.
Finally, in Section \ref{proof_proposition_ass} we establish Lemma \ref{lemma_ass_PhiDelta}, Proposition \ref{proposition_ass_pinned}, Proposition \ref{proposition_ass} and the respective version for graphs with (normalized) external fields over general factor counts $\mR^*$.
In Section \ref{proof_thm_bethe} we derive Theorem \ref{thm_bethe} for graphs with (normalized) external fields over general factor counts $\mR^*$, and thereby complete the proof of Theorem \ref{thm_bethe}.
\subsubsection{Coupling Standard Graphs}\label{coupling_graphs}
For the sake of readability we suppress dependencies in the following sections unless required.
Fix a ground truth $\sigma^+\in[q]^{n+1}$ with $\sigma^-=\sigma^+_{[n]}$ and
$m_{\mathrm{a}}=(m_\cap,m^-_\Delta,m^+_{\Delta-},m^+_{\Delta+})\in\ints_{\ge 0}^4$,
meaning $m_\cap$ factors in the base graph with variables $[n]$, additional $m^-_\Delta$ factors in $\GR^-$, additional $m^+_{\Delta-}$ factors in $\GR^+$ that do not wire to the variable $i=n+1$ and $m^+_{\Delta+}$ factors that do wire to $i$.
From these atoms we obtain the derived factor counts, namely
\begin{align*}
m^-=m_\cap+m^-_\Delta,\,m^+_-=m_\cap+m^+_{\Delta-},\,m^+_\Delta=m^+_{\Delta-}+m^+_{\Delta+},
m^+=m_\cap+m^+_\Delta.
\end{align*}
Recall the discussion in Section \ref{known_neighborhoods}.
We consider the wires-weight pairs
\begin{align*}
(\wR_{\cap},\wR^-_\Delta,\wR^+_{\Delta-},\wR^+_{\Delta+})
\dequal\wTS[\otimes m_\cap]_{-\circ,n+1,i,\sigma^-}
\otimes\wTS[\otimes m^-_\Delta]_{-\circ,n+1,i,\sigma^-}
\otimes\wTS[\otimes m^+_{\Delta-}]_{-\circ,n+1,i,\sigma^-}
\otimes\wTS[\otimes m^+_{\Delta+}]_{+\circ,n+1,i,\sigma^+}.
\end{align*}
This yields the graph $\wR^-=(\wR_\cap,\wR^-_\Delta)$ on $n$ variables.
Mimicking Section \ref{known_neighborhoods} for $n+1$ variables, $m^+$ factors, $i=n+1$ and $d=m^+_\Delta$, let $\wR^+_-=(\wR_\cap,\wR^+_{\Delta-})$ be the pairs not connected to $i$ and $\wR^+_+=\wR^+_{\Delta+}$ the pairs connected to $i$.
For $\mcla\in\binom{[m^+]}{d}$ let $\wR_{\mrma,\mcla}^+$ be the corresponding relabeling of $(\wR^+_-,\wR^+_+)$ and let $\wR^+_{\mrmd}=\wR_{\mrma,\bmcla}^+$ with $\bmcla\dequal\unif(\binom{[m^+]}{d})$.
Further, recall the degree $\dR^+_{m^+}\dequal\Bin(m^+,\sucD^+)$ with success probability $\sucD^+=\sucD[,n+1,\sigma^+](\sigma^+_i)$ from Section \ref{var_degrees},
$\wTSM$ from Observation \ref{obs_TSM_iid} and $\wTSM[\mrmd]$ from Observation \ref{obs_known}.
\begin{lemma}\label{lemma_ass_graph_coupling}
Let $i=n+1$ and $d=m^+_{\Delta+}$.
\begin{alphaenumerate}
\item\label{lemma_ass_graph_coupling_capminus}
We have $\wR_{\cap,n,m_\cap}(\sigma^-)\dequal\wTSM[n,m_\cap](\sigma^-)$ and
$\wR^-_{n,m^-}(\sigma^-)\dequal\wTSM[n,m^-](\sigma^-)$.
\item\label{lemma_ass_graph_coupling_pluscond}
We have
$\wR^+_{\mrmd,n,m^+,d}(\sigma^+)\dequal\wTSM[\mrmd,n+1,m^+,i,d](\sigma^+)$.
\item\label{lemma_ass_graph_coupling_plusprob}
Let $(\mR_\cap,\mR^-_{\Delta-},\mR^+_{\Delta-},\bm d)\in\ints_{\ge 0}^4$, $\mR^+=\mR_\cap+\mR^+_{\Delta-}+\bm d$ and $\wR^+_n(\sigma^+)=\wR^+_{\mrmd,n,\mR^+,\bm d}(\sigma^+)$.
Then we have $\wR^+_{n}(\sigma^+)\dequal\wTSM[n+1,\mR^+](\sigma^+)$ if $(\mR^+,\bm d)\dequal(\mR^+,\bm d^+_{\bm m^+})$.
\end{alphaenumerate}
\end{lemma}
\begin{proof}
Recall from the proof of Observation \ref{obs_known} that the
$(\wR^*_{-\circ},\wRa)$-derivative is given by
$(v,\psi)\mapsto\bmone\{i\not\in v([k])\}\psi(\sigma_v)/\expe[\bmone\{i\not\in\vRa([k])\}\psiRa(\sigma_{\vRa})]$, which yields $\wTSM[-\circ,n+1,i,\sigma^-]\dequal\wTSM[\circ,n,\sigma^-]$ since clearly $\vR_{\circ,n+1}|i\not\in\vR_{\circ,n+1}([k])\dequal\vR_{\circ,n}$, and thereby completes the proof of Part \ref{lemma_ass_graph_coupling}\ref{lemma_ass_graph_coupling_capminus}.
Part \ref{lemma_ass_graph_coupling}\ref{lemma_ass_graph_coupling_pluscond} holds by construction since we explicitly mimicked the construction in Section \ref{known_neighborhoods}.
Part \ref{lemma_ass_graph_coupling}\ref{lemma_ass_graph_coupling_plusprob} follows directly from Part \ref{lemma_ass_graph_coupling}\ref{lemma_ass_graph_coupling_pluscond} and Observation \ref{obs_known}.
\end{proof}
\begin{remark}\label{remark_asscoupfixedcounts}
Notice that the coupling of the graphs does not require Poisson counts, but they are very convenient to avoid case distinctions, as mentioned below.
In particular, we could take $\mR^+=m^+$, $\mR^-=m^-$ and let $\mR_\cap=\min(m^-,m^+-\bm d^+_{m^+})$.
\end{remark}
\subsubsection{Coupling Factor Counts}\label{coupling_counts}
In this section we introduce a coupling of $\mR_{n}$ and $\mR_{n+1}$ that meets the requirements of Lemma \ref{lemma_ass_graph_coupling}\ref{lemma_ass_graph_coupling_plusprob}.
For this purpose recall the Poisson parameter $\me_n=\degae n/k$ of $\mR_n$, let
$\me^-_n=\me_n$ and $\me^+_n=\me_{n+1}$.
Guided by Lemma \ref{lemma_ass_graph_coupling}\ref{lemma_ass_graph_coupling_plusprob} and Observation \ref{obs_poisson}\ref{obs_poisson_bin} let 
 $\me^+_{\Delta+}=\sucD^+\me^+$ and $\me^+_{-}=(1-\sucD^+)\me^+$.
Inspired by Remark \ref{remark_asscoupfixedcounts}, let $\me_\cap=\min(\me^-,\me^+_-)$, and denote the gaps by $\me^-_\Delta=\me^--\me_\cap$ and $\me^+_{\Delta-}=\me^+_--\me_\cap$.
So, the basic Poisson counts are
\begin{align*}
(\bm m_\cap,\bm m^-_\Delta,\bm m^+_{\Delta-},\bm m^+_{\Delta+})\dequal
\Po(\me_\cap)\otimes\Po(\me^-_\Delta)\otimes\Po(\me^+_{\Delta-})\otimes\Po(\me^+_{\Delta+}).
\end{align*}
Let $\mR^-=\mR_\cap+\mR^-_\Delta$, $\mR^+=\mR_\cap+\mR^+_{\Delta-}+\mR^+_{\Delta+}$ and $\wR_\cap$, $\wR^-$, $\wR^+$ from Lemma \ref{lemma_ass_graph_coupling}.
\begin{lemma}\label{lemma_ass_graph_couplingpo}
We have $\wR_{\cap,n,\mR_\cap(\sigma^+)}(\sigma^-)\dequal\wTSM[n,\mR_\cap(\sigma^+)](\sigma^-)$,
$\wR^-_{n,\mR^-}(\sigma^-)\dequal\wTSM[n,\mR^-](\sigma^-)$,
$\wR^+_{n}(\sigma^+)\dequal\wTSM[n+1,\mR^+](\sigma^+)$, further $\mR^-\dequal\mR_n$ and $\mR^+\dequal\mR_{n+1}$.
\end{lemma}
\begin{proof}
With Observation \ref{obs_poisson}\ref{obs_poisson_bin} we have $\mR^-\dequal\Po(\me^-)$ since $\me^-=\me_\cap+\me^-_{\Delta}$, further $\mR^+_-=\mR_\cap+\mR^+_{\Delta-}\dequal\Po(\me^+_-)$ since $\me^+_-=\me_\cap+\me^+_{\Delta-}$ and $\mR^+\dequal\Po(\me^+)$ since $\me^+=\me^+_{-}+\me^+_{\Delta+}$.
Hence, Observation \ref{obs_poisson}\ref{obs_poisson_bin} further yields that $(\mR^+,\mR^+_{\Delta+})\dequal(\mR^+,\bm d^+_{\mR^+})$, so Lemma \ref{lemma_ass_graph_coupling}\ref{lemma_ass_graph_coupling_plusprob} applies, and completes the proof with Lemma \ref{lemma_ass_graph_coupling}\ref{lemma_ass_graph_coupling_capminus}.
Notice that $\mR^+_-$ depends on $\sigma^+$ through $\sucD^+$ and hence $\mR_\cap=\mR_\cap(\sigma^+)$ depends on $\sigma^+$.
\end{proof}
\begin{remark}\label{remark_asscoupwnotation}
Notice that the notation for $\wR_\cap$, $\wR^-$ and $\wR^+$ is inconsistent, hence we change it as follows.
Let $\bm W_{\mrmp}(\sigma^+)=(\wR_{\cap,\mR_\cap(\sigma^+)}(\sigma^-),\wR^-_{\mR^-}(\sigma^-),\wR^+(\sigma^+))$ denote the pairs over random counts and
$\bm W_{\mrmm,M}(\sigma^-)=(\bm W_{\mrmp}|(\mR_\cap,\mR^-,\mR^+)=M)$ the pairs over given counts.
We let $(\wR_{\cap,m_\cap}(\sigma^-),\wR^-_{m^-}(\sigma^-),\wR^+_{m^+}(\sigma^+))=\bm W_{\mrmm,(m_\cap,m^-,m^+)}(\sigma^+)$ be the pairs for given counts and use
$\bm W_{\mrmp}(\sigma^+)\dequal(\wR_{\cap,\mR_\cap(\sigma^+)}(\sigma^-),\wR^-_{\mR^-}(\sigma^-),\wR^+_{\mR^+}(\sigma^+))$.

Also in the new notation we let $\mR^+_{\Delta+}\dequal\Po(\me^+_{\Delta+})$ be the factor degree of $i=n+1$ in $\wR^+_{\mR^+}(\sigma^+)$, so $\mR^+_-=\mR^+-\mR^+_{\Delta-}\dequal\Po(\me^+_-)$ factors are not wired to $i$ in $\wR^+_{\mR^+}(\sigma^+)$.
Similarly, we e.g.~still have $\wR^-_{\mR^-,[\mR_\cap(\sigma^+)]}(\sigma^-)=\wR_{\cap,\mR_\cap(\sigma^+)}(\sigma^-)$.
\end{remark}
\subsubsection{Coupling Pins}\label{coupling_pins}
Notice that $\ThetaP_{n+1}-\ThetaP_n>0$.
Let $\thetaPR^+_n\dequal\unif([0,\ThetaP_{n+1}])$ and let $\thetaPR^-_n=\thetaPR^-_{n,\thetaPR^+}$ be given by $\thetaPR^-_\theta=\theta$ for $\theta\in[0,\ThetaP_n]$ and $\thetaPR^-_\theta\dequal\unif([0,\ThetaP_n])$ otherwise.
Recall $\thetaPR$ from Section \ref{random_decorated_graphs}.
\begin{lemma}\label{lemma_ass_couptheta}
We have $\thetaPR^-\le\thetaPR^+$, $\thetaPR^-_n\dequal\thetaPR_n$ and $\thetaPR^+_n\dequal\thetaPR_{n+1}$.
\end{lemma}
\begin{proof}
By construction we have $\thetaPR^+_n\dequal\thetaPR_{n+1}$ and $\thetaPR^-\le\thetaPR^+$.
Further, for an event $\mcle\setle[0,\ThetaP_n]$ and with $(\bm u,\thetaPR^+)\dequal\bm u\otimes\thetaPR^+$, $\bm u\dequal\unif([0,\ThetaP_{n}])$, $I=\int\bmone\{t\in\mcle\}\mrmd t$ we have
\begin{align*}
\prob[\thetaPR^-\in\mcle]=\prob[\thetaPR^+\in\mcle]+\prob[\thetaPR^+>\ThetaP_n,,\bm u\in\mcle]
=\frac{I}{\ThetaP_{n+1}}+\frac{(\ThetaP_{n+1}-\ThetaP_n)I}{\ThetaP_{n+1}\ThetaP_n}
=\prob[\thetaPR_n\in\mcle].
\end{align*}
\end{proof}
For given $\theta^-\in[0,n]$, $\theta^+\in[0,n+1]$ with $\theta^-\le\theta^+$ we consider the success probabilities
\begin{align*}
p_\cap=\frac{\theta^-}{n+1},\,
p^-_\Delta=\frac{\theta^-}{n(n+1)-n\theta^-},\,
p^+_\Delta=\frac{\theta^+-\theta^-}{n+1-\theta^-}.
\end{align*}
Let $\indPR_{\cap\circ},\indPR_{\Delta\circ}^-,\indPR_{\Delta\circ}^+\in\{0,1\}$ be given by the success probabilities $p_\cap,p^-_\Delta,p^+_\Delta$ respectively,
$\indPR_{\times\mrmt,\theta^-,\theta^+}=(\indPR_{\cap\circ}\otimes\indPR_{\Delta\circ}^-\otimes\indPR_{\Delta\circ}^+)^{\otimes(n+1)}$,
$\indPR_{\times,n}=(\indPR_{\cap},\indPR_{\Delta}^-,\indPR_{\Delta}^+)=\indPR_{\times\mrmt,\thetaPR^-,\thetaPR^+}$ and
\begin{align*}
\setPR_\cap&=\{i\in[n]:\indPR_\cap(i)>0\},\,
\setPR^-=\{i\in[n]:\indPR_\cap(i)+\indPR^-_\Delta(i)>0\},\\
\setPR^+&=\{i\in[n+1]:\indPR_\cap(i)+\indPR^+_\Delta(i)>0\}.
\end{align*}
Recall the pinning set $\setPR_{n}$ from Section \ref{random_decorated_graphs}.
\begin{lemma}\label{lemma_ass_coupsetp}
We have $\setPR^-\dequal\setPR_n$ and $\setPR^+\dequal\setPR_{n+1}$.
\end{lemma}
\begin{proof}
Recall $\indPR_{\mrmt\circ}$ with success probability $\theta/n$ from Section \ref{random_decorated_graphs}.
For given $\theta^-\in[0,n]$, $\theta^+\in[0,n+1]$ with $\theta^-\le\theta^+$
let $(\indPR_{\cap\circ},\indPR_{\Delta\circ}^-,\indPR_{\Delta\circ}^+)=\indPR_{\cap\circ}\otimes\indPR_{\Delta\circ}^-\otimes\indPR_{\Delta\circ}^+$ and notice that
\begin{align*}
\prob[\indPR_{\cap\circ}+\indPR_{\Delta\circ}^->0]&=p_\cap+(1-p_\cap)p^-_\Delta=\frac{\theta^-}{n}=\prob[\indPR_{\mrmt\circ,\theta^-,n}=1],\\
\prob[\indPR_{\cap\circ}+\indPR_{\Delta\circ}^+>0]&=p_\cap+(1-p_\cap)p^+_\Delta=\frac{\theta^+}{n+1}=\prob[\indPR_{\mrmt\circ,\theta^+,n+1}=1].
\end{align*}
Now, let $n^-=n$, $n^+=n+1$ and $\indPR_{\times\mrmt}
=(\indPR_{\cap\mrmt},\indPR_{\Delta\mrmt}^-,\indPR_{\Delta\mrmt}^+)$.
Then the above shows that $\indPR^\pm_{\theta^\pm}=(\min(1,\indPR_{\cap\mrmt,i}+\indPR^\pm_{\Delta\mrmt,i}))_{i\in[n^\pm]}\dequal\indPR_{\mrmt\circ,\thetaP^\pm,n^\pm}^{\otimes n^\pm}$, so Lemma \ref{lemma_ass_couptheta} yields $\setPR^\pm=\indPR^{\pm-1}_{\thetaPR^\pm}(1)\dequal\setPR_{n^\pm}$.
\end{proof}
Finally, we complete the coupling with $\sigmaR^+\dequal\sigmaIID_{n+1}$ and $\sigmaR^-=\sigmaR^+_{[n]}$.
In order to clarify the dependency structure recall $\bm W_{\mrmp,n}(\sigma^+)$ from Remark \ref{remark_asscoupwnotation} and that it determines the factor counts.
On the other hand we have $\bm U_n=(\thetaPR^-,\thetaPR^+,\indPR_{\times})$ which determines $\setPR_\cap,\setPR^-,\setPR^+$.
The joint distribution is now given by $(\sigmaR^+,\bm W_{\mrmp}(\sigmaR^+),\bm U)\dequal(\sigmaR^+,\bm W_{\mrmp}(\sigmaR^+))\otimes\bm U$.
Now, the graphs are
$\GR^\pm(\sigmaR^\pm)
=[\wR^\pm_{n,\mR^\pm}(\sigmaR^\pm_n)]^{\Gamma\darr}_{\setPR^\pm_n,\sigma^\pm_n}$ and
$\GR_\cap=[\wR_{\cap,\mR_\cap(\sigmaR^+)}(\sigmaR^-)]^{\Gamma\darr}_{\setPR_{\cap},\sigmaR^-}$.
\begin{proposition}\label{proposition_ass_full_coupling}
We have
$\GR_\cap\dequal\GTSM[\mR_{\cap}(\sigmaR^+),\setPR_{\cap}](\sigmaR^-)$, further
$(\sigmaR^-,\GR^-(\sigmaR^-))\dequal(\sigmaIID,\GTSM[\mR,\setPR](\sigmaIID))$ and
$(\sigmaR^+,\GR^+(\sigmaR^+))
\dequal(\sigmaIID_{n+1},\GTSM[n+1,\mR_{n+1},\setPR_{n+1}](\sigmaIID_{n+1}))$.
\end{proposition}
\begin{proof}
The result is immediate from Lemma \ref{lemma_ass_graph_couplingpo}, Lemma \ref{lemma_ass_coupsetp} and Observation \ref{obs_TSM_iid}.
\end{proof}
Notice that $\GR^-(\sigmaR^-)$ and $\GR^+(\sigmaR^+)$ are conditionally independent given $(\thetaPR^-,\thetaPR^+,\sigmaR^+,\GR_\cap)$ and obtained as follows.
For $\GR^-(\sigmaR^-)$ we choose $\mR^-_{\Delta}(\sigmaR^+)$ and the additional standard factors i.i.d.~from $\wTSa[,n,\sigmaR^-]$. Further, we perform a second sweep of pinning with probability $p^-_{\Delta,\thetaPR^-}$, i.e.~pinning each (unpinned) variable $i\in[n]$ to $\sigmaR^-(i)$ independently with probability $p^-_{\Delta,\thetaPR^-}$.

For $\GR^+(\sigmaR^+)$ we choose $\mR^+_{\Delta-}(\sigmaR^+)$ and the corresponding additional standard factors i.i.d.~from $\wTSa[,n,\sigmaR^-]$.
Further, we choose $\mR^+_{\Delta^+}(\sigmaR^+)$ and the corresponding additional standard factors independently from $\wTSM[+\circ,n+1,i,\sigmaR^+]$.
Formally, we also have to randomly relabel all factors.
Finally, we perform a second sweep of pinning with probability $p^+_{\Delta,\thetaPR^-,\thetaPR^+}$ for the variables $[n]$ and pin $i=n+1$ with probability $\thetaPR^+/(n+1)$.

Since both $\GR^-(\sigmaR^-)$ and $\GR^+(\sigmaR^+)$ are obtained from $\GR_\cap$ exclusively by adding factors, the ratios in
$\Phi_\Delta(n)=\Phi_{\mrmv}(n)-\Phi_{\mrmf}(n)$ with
\begin{align}\label{equ_PhiDeltContributions}
\Phi_{\mrmv}(n)=\expe\left[\ln\left(\frac{\ZG(\GR^+(\sigmaR^+))}{\ZG(\GR_\cap)}\right)\right],\,
\Phi_{\mrmf}(n)=\expe\left[\ln\left(\frac{\ZG(\GR^-(\sigmaR^+))}{\ZG(\GR_\cap)}\right)\right],
\end{align}
can be understood as the expected additional weight caused by the new factors under the Gibbs spins $\sigmaRG[,\GR_\cap]$, as in Section \ref{int_interpolator} and Section \ref{int_factor}.
\subsubsection{The Base Graph}\label{ass_base_graph}
We define a coupling for the pairs $(\sigmaR^-,\GR_\cap)$ and $(\sigmaIID_n,\GTSM[n,\mR_n,\setPR_n](\sigmaIID_n))$.
For this purpose we start with a coupling of $(\sigmaR^-,\mR_\cap(\sigmaR^+),\setPR_\cap)$ and 
$(\sigmaIID_n,\mR_n,\setPR_n)$.
For $\sigma\in[q]^{n+1}$ let $(\mR'_{\cap,\sigma},\mR'_\sigma)$ be a coupling of $\mR_\cap(\sigma)$ and $\mR^-_n$ from the coupling lemma \ref{obs_tv}\ref{obs_tv_coupling}.
This conditional law and $(\sigmaR'_+,\setPR'_\cap,\setPR')\dequal(\sigmaR^+,\setPR_\cap,\setPR^-)$ induce
$\bm a=(\sigmaR'_+,\mR'_\cap(\sigmaR'_+),\mR'(\sigmaR'_+),\setPR'_\cap,\setPR')$, which further determines $\sigmaR'_\cap=\sigmaR'=\sigmaR'_{+,[n]}$.
For given $a=(\sigma_+,m_\cap,m,\setP_\cap,\setP)$ with $\sigma=\sigma_{+,[n]}$ we obtain the graphs as follows.
For $m_\cap=m$ and $\setP_\cap=\setP$ let $\GR'_\cap(a)=\GR'(a)\dequal\GTSM[m,\setP](\sigma)$, otherwise let $(\GR'_\cap(a),\GR'(a))\dequal\GTSM[m_\cap,\setP_\cap](\sigma)\otimes\GTSM[m,\setP](\sigma)$.
\begin{lemma}\label{lemma_ass_base_graph}
We have $(\sigmaR'_\cap,\GR'_\cap(\bm a))\dequal(\sigmaR^-,\GR_\cap)$ and
$(\sigmaR',\GR'(\bm a))\dequal(\sigmaIID,\GTSM[\mR,\setPR](\sigmaIID))$.
Further, we have $\prob[(\sigmaR'_\cap,\GR'_\cap(\bm a))\neq(\sigmaR',\GR'(\bm a))]=\mclo(n^{-\rho})$.
\end{lemma}
\begin{proof}
We have $(\sigmaR'_+,\sigmaR'_\cap,\setPR'_\cap,\GR'_\cap(\bm a))\dequal(\sigmaR^+,\sigmaR^-,\setPR_\cap,\GTSM[\mR_\cap(\sigmaR^+),\setPR_\cap](\sigmaR^-))$ by definition, and further
$(\sigmaR',\setPR',\GR'(\bm a))\dequal(\sigmaR^-,\setPR^-,\GTSM[\mR^-,\setPR^-](\sigmaR^-))$, so the first two assertions hold by Proposition \ref{proposition_ass_full_coupling}.
Since the graphs coincide if the counts do, we have
$\prob[(\sigmaR'_\cap,\GR'_\cap(\bm a))\neq(\sigmaR',\GR'(\bm a))]
\le\prob[(\mR'_{\cap,\sigmaR^+},\setPR'_\cap)\neq(\mR'_{\sigmaR^+},\setPR')]
\le\prob[\mR'_{\cap,\sigmaR^+}\neq\mR'_{\sigmaR^+}]+\prob[\setPR'_\cap\neq\setPR']$.
For the latter we have $\prob[\setPR'_\cap\neq\setPR']\le 
\prob[|\indPR_\Delta^{--1}(1)|>0]$, and further
$\prob[|\indPR_\Delta^{--1}(1)|>0]\le\expe[|\indPR_\Delta^{--1}(1)|]$ by Markov's inequality, where
\begin{align}
\expe[|\indPR_\Delta^{--1}(1)|]
=n\expe[p^-_\Delta(\thetaPR^-)]\le\frac{n\ThetaP_n}{n(n+1)-n\ThetaP_n}=(1+o(1))\frac{\ThetaP_n}{n}.
\end{align}
For the factor counts we
use the definition, i.e.~the coupling lemma \ref{obs_tv}\ref{obs_tv_coupling}, Pinsker's inequality \ref{obs_tv}\ref{obs_tv_pinsker} and Observation \ref{obs_poisson}\ref{obs_poisson_dkl} to obtain
\begin{align*}
\prob[\mR'_{\cap,\sigmaR^+}\neq\mR'_{\sigmaR^+}]
&=\expe\left[\|\mR_\cap(\sigmaR^+)-\mR_n\|_\mrmtv\right]
\le\expe\left[\expe\left[\sqrt{\frac{1}{2}\DKL\left(\mR_\cap(\sigmaR^+)\middle\|\mR_n\right)}\middle|\sigmaR^+\right]\right]\\
&=\frac{1}{\sqrt{2}}\expe\left[\sqrt{\me_n-\me_\cap(\sigmaR^+)+\me_\cap(\sigmaR^+)\ln\left(\frac{\me_\cap(\sigmaR^+)}{\me_n}\right)}\right].
\end{align*}
The argument of the expectation vanishes for $\me_\cap(\sigmaR^+)=\me^-=\me_n$.
Otherwise, we have $\me_\cap(\sigmaR^+)=\me^+-\me^+_{\Delta+}(\sigmaR^+)<\me^-$, or equivalently $\bm\Delta>0$ with $\bm\Delta=\me^+_{\Delta+}(\sigmaR^+)-\degae/k$, and using $\ln(1-t)\le-t$ further
\begin{align*}
\prob[\mR_\cap(\sigmaR^+)\neq\mR_n]
&\le\frac{1}{\sqrt{2}}\expe\left[\bmone\left\{\bm\Delta>0\right\}\sqrt{\bm\Delta-\me_\cap(\sigmaR^+)\frac{\bm\Delta}{\me_n}}\right]\\
&=\frac{1}{\sqrt{2}}\expe\left[\bmone\left\{\bm\Delta>0\right\}\sqrt{\bm\Delta-\left(1-\frac{\bm\Delta}{\me_n}\right)\bm\Delta}\right]
=\frac{\expe[\bmone\{\bm\Delta>0\}\bm\Delta]}{\sqrt{2\me_n}}.
\end{align*}
With $\tilde c$ from Corollary \ref{cor_degbounds}\ref{cor_degbounds_po} we have $\bm\Delta\le\tilde c\degae-\frac{\degae}{k}$, hence
\begin{align*}
\prob[\mR_\cap(\sigmaR^+)\neq\mR_n]
\le\left(\tilde c-\frac{1}{k}\right)\frac{\sqrt{k}\degae}{\sqrt{2\degae n}}
\le\left(\tilde c-\frac{1}{k}\right)\frac{\sqrt{k\degabu}}{\sqrt{2n}}.
\end{align*}
This completes the proof since $\ThetaP/n=n^{-\rho}=\omega(n^{-1/4})$.
\end{proof}
\subsubsection{Factor Count Asymptotics}\label{ass_factor_count_asymptotics}
Let $c_{\mrmr,\mfkg}\in\reals_{>0}$ be large, $r(n)=c_{\mrmr}\sqrt{\ln(n)/n}$ and $\mclb[+][\Gamma]=\{\sigma^+\in[q]^{n+1}:\|\gammaN[,\sigma^+]-\gamma^*\|_\mrmtv\le r(n)\}$.
In this section we show that for typical spins $\mclb[+][\Gamma]$ and for sufficiently large $n$ the coupling of the graphs simplifies.
\begin{lemma}\label{lemma_ass_factor_count_asymptotics}
Let $\sigma\in\mclb[+][\Gamma]$.
There exists $c_\mfkg\in\reals_{>0}$
such that $\me^\pm_\Delta(\sigma)\le c$.
Further, there exists $n_{\circ,\mfkg}\in\ints_{>0}$ such that
for $n\ge n_\circ$ we have
$\me^+_{\Delta-}(\sigma)=0$ and
\begin{align*}
\left|\me^+_{\Delta+}(\sigma)-\degae\right|,\,
\left|\me^-_{\Delta}(\sigma)-\frac{\degae(k-1)}{k}\right|
\le cr(n).
\end{align*}
\end{lemma}
\begin{proof}
With $\tilde c$ from Corollary \ref{cor_degbounds}\ref{cor_degbounds_conc} and with $\bm d^*_{\mrmf,\mR}\dequal\Po(\sucD\me)$ from Corollary \ref{cor_degbounds}\ref{cor_degbounds_po} we have $|\me^+_{\Delta+}-\degae|\le\tilde c\degae(\|\gammaN[,\sigma^+]-\gamma^*\|_\mrmtv+n^{-1})\le\tilde c\degae(r(n)+n^{-1})$. Using $k\ge 2$ fix
\begin{align*}
n_\circ=\min\left\{n_0\in\ints_{\ge 3}:\sup_{n\ge n_0}\tilde c\left(r(n)+\frac{1}{n}\right)<\frac{k-1}{k}\right\}.
\end{align*}
For $n\le n^\circ$ we have $\me^\pm_\Delta\le\me^+\le\degabu(n_\circ+1)/k$.
For $n\ge n_\circ$ we have
$|\me^+_{\Delta+}-\degae|<\degae(k-1)/k$ and $|\me^+_{\Delta+}-\degae|\le 2\tilde c\degabu r(n)\le 2\tilde c\degabu c_{\mrmr}/\sqrt{e}$.
The former yields $\me^+_-=\me^+-\me^+_{\Delta+}<\me^+-\frac{\degae}{k}=\me^-$, so $\me^+_{\Delta-}=0$. Finally, notice that
\begin{align*}
\left|\me^-_{\Delta}-\frac{\degae(k-1)}{k}\right|
=\left|\me^--\me^+_--\frac{\degae(k-1)}{k}\right|
=\left|\me^+_{\Delta+}-\degae\right|.
\end{align*}
\end{proof}
\subsubsection{Typical Events for the Factor Contribution}\label{assfactor_typical}
Analogously to $\mclb[+][\Gamma]$ let $\mclb[-][\Gamma]=\{\sigma\in[q]^n:\|\gammaN[,\sigma]-\gamma^*\|_\mrmtv\le r(n)\}$, and $\mclb[][\circ]=(\degae-r(n),\degae+r(n))$.
Further, let $\Phi_{\mrmf}(n)=\expe[\bm\Phi]$ with $\bm\Phi=\ln(\ZG(\GR^-(\sigmaR^-))/\ZG(\GR_\cap))$ and $\degaR^-=k\mR^-/n$.
\begin{lemma}\label{lemma_assfactor_typ}
We have $\Phi_{\mrmf}(n)=\expe[\bmone\{\sigmaR^+\in\mclb[+][\Gamma],\sigmaR^-\in\mclb[-][\Gamma],\degaR^-\in\mclb[][\circ]\}\bm\Phi]+o(n^{-1})$.
\end{lemma}
\begin{proof}
With $\mcle=\{\sigmaR^+\in\mclb[+][\Gamma],\sigmaR^-\in\mclb[-][\Gamma],\degaR^-\in\mclb[][\circ]\}$ we have
$\Phi_\mrmf(n)=\expe[\bmone\mcle\bm\Phi]+\eps$ with
\begin{align*}
\eps(n)=\expe\left[\bmone\lnot\mcle\overline{\bm\Phi}\right],\,
\overline{\bm\Phi}=\expe\left[\bm\Phi\middle|\mR_\cap(\sigmaR^+),\sigmaR^+,\sigmaR^-,\degaR^-\right].
\end{align*}
With Jensen's inequality we can consider the atypical events separately, i.e.
\begin{align*}
|\eps(n)|
\le\expe[\bmone\{\degaR^-\not\in\mclb[][\circ]\}|\overline{\bm\Phi}|]
+\expe[\bmone\{\sigmaR^+\not\in\mclb[+][\Gamma]\}|\overline{\bm\Phi}|]
+\expe[\bmone\{\sigmaR^-\not\in\mclb[-][\Gamma]\}|\overline{\bm\Phi}|].
\end{align*}
With $\bm\Phi=n\phiG(\ZG(\GR^-(\sigmaR^-)))-n\phiG(\ZG(\GR_\cap))$, Jensen's inequality, the triangle inequality, 
$\tilde c$ from Observation \ref{obs_phi_lipschitz}, Lemma \ref{lemma_ass_coupsetp}, Lemma \ref{lemma_ass_couptheta} and Observation \ref{obs_pin_basic} we have
\begin{align*}
|\overline{\bm\Phi}|
\le\tilde c\left(\mR^-+\frac{\ThetaP_n}{2}+\mR_\cap(\sigmaR^+)+\frac{n\ThetaP_n}{(n+1)2}\right)
\le\tilde c(2\mR^-+n).
\end{align*}
So, with Lemma \ref{lemma_ass_graph_couplingpo} and $c$ from Corollary \ref{cor_dega} we have
\begin{align*}
\expe[\bmone\{\degaR^-\not\in\mclb[][\circ]\}|\overline{\bm\Phi}|]
\le\frac{2\tilde cn}{k}\expe[\bmone\{\degaR\not\in\mclb[][\circ]\}\degaR]
+\tilde cn\prob[\degaR\not\in\mclb[][\circ]]
\le c'n\exp\left(-\frac{c_1r^2n}{1+r}\right)
\end{align*}
with $c'=\tilde cc_2(2+k)/k$. With $r=o(1)$ and $c_\mrmr^2>2/c_1$ we get
\begin{align*}
\expe[\bmone\{\degaR^-\not\in\mclb[][\circ]\}|\overline{\bm\Phi}|]
\le c'n\exp\left(-(1+o(1))c_1c_{\mrmr}^2\ln(n)\right)
=c'n^{-(1+o(1))c_1c_{\mrmr}^2+1}
=o(n^{-1}).
\end{align*}
With $c$ from Observation \ref{obs_gtiid}\ref{obs_gtiid_prob}, independence and $c_{\mrmr}^2>2/c_1$ we have
\begin{align*}
\expe[\bmone\{\sigmaR^+\not\in\mclb[+][\Gamma]\}|\overline{\bm\Phi}|]
\le\tilde c\left(\frac{2\degabu}{k}+1\right)c_2ne^{-c_1r^2(n+1)}
=\Theta\left(n^{1-c_1c_{\mrmr}^2}\right)=o(n^{-1}),
\end{align*}
and $\expe[\bmone\{\sigmaR^-\not\in\mclb[-][\Gamma]\}|\overline{\bm\Phi}|]=o(n^{-1})$ follows analogously.
\end{proof}
The following result further restricts the very typical event in Lemma \ref{lemma_assfactor_typ} to the typical event that no variables are pinned in the second sweep.
\begin{lemma}\label{lemma_assfactor_pin}
We have
\begin{align*}
\Phi_{\mrmf}(n)=\expe[\bmone\{\sigmaR^+\in\mclb[+][\Gamma],\sigmaR^-\in\mclb[-][\Gamma],\degaR^-\in\mclb[][\circ],\setPR^-=\setPR_\cap\}\bm\Phi]+\mclo(n^{-\rho}).
\end{align*}
\end{lemma}
\begin{proof}
With Lemma \ref{lemma_assfactor_typ} it is sufficient to consider
$\expe[\bmone\mcle|\bm\Phi|]$, where
\begin{align*}
\mcle=\left\{\sigmaR^+\in\mclb[+][\Gamma],\sigmaR^-\in\mclb[-][\Gamma],\degaR^-\in\mclb[][\circ],\setPR^-\neq\setPR_\cap\right\}.
\end{align*}
With $\tilde c$ from Observation \ref{obs_phi_lipschitz} we have 
$\expe[\bmone\mcle|\bm\Phi|]
\le\tilde c\expe[\bmone\mcle\distG(\GR_\cap,\GR^-(\sigmaR^-))]$.
With the notions in Section \ref{phi_lipschitz} we have
$\mclv[1][\darr]=[n]\setminus\setPR^-$, $\mclv[2][\darr]=\setPR_\cap$, $\mclv[][\darr]=[n]\setminus\setPR_\Delta$ with $\setPR_\Delta=\setPR^-\setminus\setPR_\cap$, further $m_\cap=\mR_\cap(\sigmaR^+)$, $\mcla[=]=[\mR_\cap(\sigmaR^+)]\setminus\bmcla_{\neq}$,
\begin{align*}
\bmcla_{\neq}=\{a\in[\mR_\cap(\sigmaR^+)]:\vR_{\cap,a}([k])\cap\setPR_\Delta\neq\emptyset\},
\end{align*}
where $\vR_\cap$ are the neighborhoods of $\GR_\cap$, so $D=0$,
$\tilde D=\mR^--\mR_\cap(\sigmaR^+)=\mR^-_\Delta(\sigmaR^+)$, 
$D_\cap=|\bmcla_{\neq}|$
and hence $\distG(\GR_\cap,\GR^-(\sigmaR^-))
=\mR^-_{\Delta}(\sigmaR^+)+2|\bmcla_{\neq}|+|\setPR_\Delta|$.
Recall that
\begin{align*}
|\bmcla_{\neq}|\le\sum_{i\in\setPR_\Delta}d_{\mrmf,\wR_\cap}(i)
\le\sum_{i\in\setPR_\Delta}d_{\mrmf,\wR^-}(i)
\end{align*}
with $\wR_\cap=\wR_{\cap,\mR_\cap(\sigmaR^+)}(\sigmaR^-)$, 
$\wR^-=\wR^-_{\mR^-}(\sigmaR^-)$ and $d_{\mrmf,G}(i)$ from Section \ref{var_degrees}.
With $c$ from Corollary \ref{cor_degbounds}\ref{cor_degbounds_m} this gives
\begin{align*}
\expe\left[\distG(\GR_\cap,\GR^-(\sigmaR^-))\middle|\sigmaR^+,\mR_\cap(\sigmaR^+),\mR^-,\setPR_\cap,\setPR^-\right]
\le\mR^-_\Delta(\sigmaR^+)+2|\setPR_\Delta|c\degaR^-+|\setPR_\Delta|.
\end{align*}
On $\mcle$ we further have $\degaR\le\degabu+r(n)$.
With this bound, standard bounds, and taking conditional expectations we obtain
$\expe[\bmone\mcle|\bm\Phi|]\le\tilde cE_1+\tilde c(2c(\degabu+r(n))+1)E_2$, where
\begin{align*}
E_1=\expe[\bmone\{\sigmaR^+\in\mclb[+][\Gamma],\setPR^-\neq\setPR_\cap\}\me^-_\Delta(\sigmaR^+)],\,
E_2=\expe[|\setPR_\Delta|].
\end{align*}
With $c'$ from Lemma \ref{lemma_ass_factor_count_asymptotics} we have $\me^-_\Delta(\sigmaR^+)\le c'$, so
as in the proof of Lemma \ref{lemma_ass_base_graph} we have
\begin{align*}
E_1\le c'\prob[\setPR^-\neq\setPR_\cap]\le (1+o(1))c'\frac{\ThetaP_n}{n},\,
E_2\le(1+o(1))\frac{\ThetaP_n}{n}.
\end{align*}
This completes the proof since $\ThetaP/n=n^{-\rho}=\omega(n^{-1})$.
\end{proof}
\subsubsection{Normalization Step for the Factor Contribution}\label{assfactor_normalization}
In Section \ref{assfactor_typical} we restricted the expectation over the coupled graphs to typical events, now we change the underlying law.
In particular, we replace $\GR_\cap$ by $\GTSM(\sigma)=\GTSM[\mR,\setPR](\sigma)$ and $\mR^-_\Delta(\sigmaR^+)$ by a Poisson variable $\mR_\Delta\dequal\Po(\degae(k-1)/k)$.
Clearly, we obtain the additional wires-weight pairs given $\sigmaIID$ from Observation \ref{obs_TSM_iid}, i.e.~we consider
\begin{align}\label{equ_assfactor_jointtsm}
(\GTSM(\sigma),\mR_\Delta,\wTSM(\sigma))\dequal\GTSM(\sigma)\otimes\mR_\Delta\otimes\wTSa[,\sigma][\otimes\ints_{>0}].
\end{align}
Further, let $\bm\Phi=\ln(\psiWgG[,\GTSM(\sigmaIID)](\wTSM[\sigmaIID,[\mR_\Delta]]))$ with $\psiWgG$ from Equation (\ref{equ_psiWgG}).
\begin{lemma}\label{lemma_assfactor_normalized}
We have
$\Phi_{\mrmf}(n)=\expe[\bmone\{\sigmaIID\in\mclb[-][\Gamma]\}\bm\Phi]+\mclo(n^{-\rho})$.
\end{lemma}
\begin{proof}
Let $\bm\Phi'=\ln(\ZG(\GR^-(\sigmaR^-))/\ZG(\GR_\cap))$, and let $(\GR_\cap,\mR_\Delta,\wTSM(\sigmaR^-))$ be conditionally independent given $\sigmaR^+$.
As explained in Section \ref{coupling_pins} and analogously to Section \ref{int_factor} on
\begin{align*}
\mcle=\left\{\sigmaR^+\in\mclb[+][\Gamma],\sigmaR^-\in\mclb[-][\Gamma],\degaR^-\in\mclb[][\circ],\setPR^-=\setPR_\cap\right\}
\end{align*}
we have $\bm\Phi'\dequal\ln(\psiWgG[,\GR_\cap](\wTSM[\sigmaR^-,[\mR^-_{\Delta}(\sigmaR^+)]]))$, i.e.~there are no additional pins, the additional factors are independent of the remainder and i.i.d.~from the teacher-student model for the given ground truth.
For given $\sigma^+\in[q]^{n+1}$ let $\bm\delta(\sigma^+)\dequal\Po(|\me^-_\Delta(\sigma^+)-\me_\Delta|)$ with $\me_\Delta=\frac{\degae(k-1)}{k}$.
For $\me_\Delta^-(\sigma^+)\ge\me_\Delta$ and using Observation \ref{obs_poisson}\ref{obs_poisson_bin} we consider the coupling $\mR^-_\Delta(\sigma^+)=\mR_\Delta+\bm\delta(\sigma^+)$, and $\mR_\Delta=\mR^-_\Delta(\sigma^+)+\bm\delta(\sigma^+)$ otherwise.
This gives
\begin{align*}
\left|\ln\left(\psiWgG[,\GR_\cap]\left(\wTSM[\sigmaR^-,[\mR^-_{\Delta}(\sigmaR^+)]]\right)\right)-\ln\left(\psiWgG[,\GR_\cap]\left(\wTSM[\sigmaR^-,[\mR_{\Delta}]]\right)\right)\right|
\le\bm\delta(\bm\sigma^+)\ln(\psibu).
\end{align*}
With Lemma \ref{lemma_ass_factor_count_asymptotics} we can bound $\expe[\bm\delta(\sigma^+)]$ on $\mcle$, so with Lemma \ref{lemma_assfactor_pin} we have
\begin{align*}
\Phi_{\mrmf}(n)=\expe\left[\bmone\mcle\ln\left(\psiWgG[,\GR_\cap]\left(\wTSM[\sigmaR^-,[\mR_{\Delta}]]\right)\right)\right]+\mclo\left(r(n)+n^{-\rho}\right).
\end{align*}
Due to the independence of $\mR_\Delta$ from the remainder we can use the upper bound $\mR_\Delta\ln(\psi)$ on the argument of the expectation and then take the expectation with respect to $\mR_\Delta$ to obtain the upper bound
$c=\me_\Delta\ln(\psibu)\le\degabu\ln(\psibu)(k-1)/k$ given the rest.
This shows that reducing $\mcle$ to $\{\sigmaR^-\in\mclb[-][\Gamma]\}$ causes an error of $\mclo(n^{-\rho})$.
Also, with the coupling from Lemma \ref{lemma_ass_base_graph} we then get
\begin{align*}
\Phi_{\mrmf}(n)=\expe\left[\bmone\{\sigmaIID\in\mclb[-][\Gamma]\}\ln\left(\psiWgG[,\GTSM(\sigmaIID)]\left(\wTSM[\sigmaIID,[\mR_{\Delta}]]\right)\right)\right]+\mclo\left(n^{-\rho}\right)
\end{align*}
since $n^{-\rho}=\omega(r(n))$, which completes the proof.
\end{proof}
\subsubsection{Gibbs Marginal Product for the Factor Contribution}\label{assfactor_gibbsproduct}
Now, it is time to apply Proposition \ref{proposition_pinG}.
Using the distribution (\ref{equ_assfactor_jointtsm}) let
\begin{align*}
\bm\Phi=\ln\left(\sum_\tau\left(\bigotimes_{(a,h)\in[\mR_\Delta]\times[k]}\gammaR_{a,h}\right)(\tau)\prod_{a\in[\mR_\Delta]}\psiTSM[a](\tau_a)\right)
=\sum_{a\in[\mR_\Delta]}\ln\left(\ZF(\psiTSM[a],\gammaR_a)\right),
\end{align*}
where $\gammaR=(\lawG[,\GTSM(\sigmaIID)]|_{\vTSM(a,h)})_{a,h}$, $(\vTSM,\psiTSM)=\wTSM(\sigmaIID)$ and $\ZF$ is from Equation (\ref{ZF_def}).
\begin{lemma}\label{lemma_assfactor_gibbsproduct}
We have $\Phi_{\mrmf}(n)=\expe[\bmone\{\sigmaIID\in\mclb[-][\Gamma]\}\bm\Phi]+\mclo(n^{-\rho})$.
\end{lemma}
\begin{proof}
Resolving the Radon-Nikodym derivative of the additional pairs in Lemma \ref{lemma_assfactor_normalized} yields
$\Phi_{\mrmf}(n)=\expe[\bmone\{\sigmaIID\in\mclb[-][\Gamma]\}\bm\Phi^*]+\mclo(n^{-\rho})$ with $\wR=(\vR,\psiR)\dequal(\unif([n]^k)\otimes\lawpsi)^{\otimes\ints_{>0}}$ and
\begin{align*}
\bm\Phi^*=\prod_{a\in[\mR_\Delta]}\frac{\psiR_a(\sigmaIID_{\vR(a)})}{\ZFa(\gammaIID)}\ln\left(\psiWgG[,\GTSM(\sigmaIID)](\wR_{[\mR_\Delta]})\right).
\end{align*}
Recall that for $\vR'=\vR_{[\mR_\Delta]}$ and $\alphaR^*=\bm\mu^*|_{\vR'}$ with $\bm\mu^*=\lawG[,\GTSM(\sigmaIID)]$ we have
\begin{align*}
\psiWgG[,\GTSM(\sigmaIID)](\wR_{[\mR_\Delta]})
&=\sum_{\tau\in([q]^k)^{\mR_\Delta}}\alphaR^*(\tau)\prod_{a\in[\mR_\Delta]}\psiR_a(\tau_a).
\end{align*}
Let $C$ from Proposition \ref{proposition_pinG}\ref{proposition_pinG_IIDDKL}, $\eps=C_1/3$ and $\delta=\ThetaP[-2\eps]$.
Using $\iota_\circ$ from Section \ref{pinning_gibbs} let
$\mcle=\{\sigmaIID\in\mclb[-][\Gamma],\iota_\circ(\bm\mu^*,\vR')\le\delta\}$ and notice that
$\iota_\circ(\bm\mu^*,\vR')=0<\delta$ on $\mR_\Delta=0$.
Hence, the bound
$|\bm\Phi^*|\le\ln(\psibu^{\mR_\Delta})\psibu^{2\mR_\Delta}\le\psibu^{3\mR_\Delta}$
and Markov's inequality conditional to $\mR_\Delta$ give
$\Delta\le\eps'+\mclo(n^{-\rho})$, where
$\Delta=|\Phi_{\mrmf}-\expe[\bmone\mcle\bm\Phi^*]|$ and
\begin{align*}
\eps'&=\expe\left[\bmone\{\mR_\Delta>0\}\frac{C_2(k\mR_\Delta-1)}{\delta}\left(\frac{k\mR_\Delta}{\ThetaP}\right)^{C_1}\psibu^{3\mR_\Delta}\right].
\end{align*}
Standard bounds imply
$\eps'\le\tilde c\expe[\exp(\tilde c\mR_{\Delta})]/\ThetaP[\eps]$ for some $\tilde c\in\reals_{>0}$.
The canonical coupling of $\mR_\Delta\dequal\Po(\degae(k-1)/k)$ and $\mR_{\Delta\uarr}\dequal\Po(\degabu(k-1)/k)$ from Observation \ref{obs_poisson}\ref{obs_poisson_bin} gives $\mR_\Delta\le\mR_{\Delta\uarr}$ and hence
$\eps'\le\tilde c\expe[\exp(\tilde c\mR_{\Delta\uarr})]/\ThetaP[\eps]$.
Finally, the moment generating function of the Poisson distribution gives $\eps'=\mclo(\ThetaP[-\eps])$.
But $\rho=\eps/(1+\eps)$ and $\ThetaP=n^{1-\rho}$ yields $\ThetaP[\eps]=n^{\rho}$,
thereby $\delta=n^{-2\rho}$ and $\eps',\Delta=\mclo(n^{-\rho})$.
Now, with $\alphaR=\bigotimes_{(a,h)\in[\mR_\Delta]\times[k]}\bm\mu^*|_{\vR'(a,h)}$ and
\begin{align*}
\bm\Phi=\prod_{a\in[\mR_\Delta]}\frac{\psiR_a(\sigmaIID_{\vR(a)})}{\ZFa(\gammaIID)}\ln\left(\sum_{\tau\in([q]^k)^{\mR_\Delta}}\alphaR(\tau)\prod_{a\in[\mR_\Delta]}\psiR_a(\tau_a)\right),
\end{align*}
notice that the arguments of the logarithm for both $\bm\Phi^*$ and $\bm\Phi$ are in $[\psibl^{\mR_\Delta},\psibu^{\mR_\Delta}]$ and that the logarithm is $\psibu^{\mR_\Delta}$-Lipschitz on this domain, so
\begin{align*}
|\bm\Phi^*-\bm\Phi|\le\psibu^{3\mR_\Delta}\left|
\sum_{\tau\in([q]^k)^{\mR_\Delta}}\alphaR^*(\tau)\prod_{a\in[\mR_\Delta]}\psiR_a(\tau_a)
-\sum_{\tau\in([q]^k)^{\mR_\Delta}}\alphaR(\tau)\prod_{a\in[\mR_\Delta]}\psiR_a(\tau_a)
\right|.
\end{align*}
This yields $|\bm\Phi^*-\bm\Phi|\le 2\psibu^{4\mR_\Delta}\|\alphaR^*-\alphaR\|_\mrmtv=2\psibu^{4\mR_\Delta}\nu_\circ(\bm\mu^*,\vR')\le\sqrt{2}\psibu^{4\mR_\Delta}\sqrt{\iota_\circ(\bm\mu^*,\vR')}$ with standard bounds and Remark \ref{remark_epssym}.
Since we have the same bound $|\bm\Phi|\le\psibu^{3\mR_\Delta}$, we can spare another $\eps'$ from above to obtain
\begin{align*}
\Phi_{\mrmf}(n)
=\expe[\bmone\mcle\bm\Phi]+\mclo\left(\sqrt{\delta}+n^{-\rho}\right)
=\expe[\bmone\{\sigmaIID\in\mclb[-][\Gamma]\}\bm\Phi]+\mclo\left(n^{-\rho}\right).
\end{align*}
The assertion follows by reintroducing $\wTSM$ using the Radon-Nikodym derivative in $\bm\Phi$.
\end{proof}
With $(\psiTSM[a],\gammaR_a)_a$ being i.i.d.~given $\sigmaIID$
and $\expe[\me_\Delta]=\degae(k-1)/k$
Lemma \ref{lemma_assfactor_gibbsproduct} yields
\begin{align*}
\Phi_{\mrmf}(n)=\frac{\degae(k-1)}{k}\expe\left[\bmone\{\sigmaIID\in\mclb[-][\Gamma]\}\ln(\ZF(\psiTSM,\gammaR))\right]+\mclo\left(n^{-\rho}\right),
\end{align*}
with $\wTSM=(\vTSM,\psiTSM)\dequal\wTSa[,\sigmaIID]$, $\gammaR=(\lawG[,\GTSM(\sigmaIID)]|_{\vTSM(h)})_{h\in[k]}$ and $(\wTSM,\GTSM(\sigmaIID))$ conditionally independent given $\sigmaIID$.
\subsubsection{Marginal Distribution for the Factor Contribution}\label{assfactor_marginaldist}
Now, we work towards the discussion in Section \ref{pinning_gibbs_marginal_distributions}.
Let $\gamma=\gammaN[,\sigma]$, $\GTSM(\sigma)=\GTSM[\mR,\setPR](\sigma)$ and
$\wTSYa$, $\tauTSa$, $\domC$ from Section \ref{GTSYM}.
Let
$(\vTSYM[\sigma,\tau],\psiTSYM[\tau],\gammaR_{\sigma,\tau})\dequal\wTSYa[,\sigma,\tau]\otimes\bigotimes_{h\in[k]}\piGC[,\GTSM(\sigma),\sigma,\tau(h)]$
for $\tau\in\domC[,\gamma]^k$, further
$(\tauTSM(\sigma),\GTSM(\sigma))\dequal\tauTSa[,\sigma]\otimes\GTSM(\sigma)$ and 
$\tauTSM=\tauTSM(\sigmaIID)$.
\begin{lemma}\label{lemma_assfactor_piGC}
We have $\Phi_{\mrmf}(n)=\frac{\degae(k-1)}{k}\expe[\bmone\{\sigmaIID\in\mclb[-][\Gamma]\}\ln(\ZF(\psiTSYM[\tauTSM],\gammaR_{\sigmaIID,\tauTSM}))]+\mclo(n^{-\rho})$.
\end{lemma}
\begin{proof}
With Lemma \ref{lemma_assfactor_gibbsproduct}, Observation \ref{obs_TSYM}, independence and $\bm\mu^*_\sigma=\lawG[,\GTSM(\sigma)]$ we have
\begin{align*}
\Phi_{\mrmf}(n)&=\frac{\degae(k-1)}{k}\expe\left[\bmone\{\sigmaIID\in\mclb[-][\Gamma]\}E(\sigmaIID,\tauTSM)\right]+\mclo(n^{-\rho}),\\
E(\sigma,\tau)
&=\expe\left[\ln\left(\ZF\left(\psiTSYM[\tau],\left(\bm\mu^*_\sigma|_{\vTSYM[\sigma,\tau](h)}\right)_{h\in[k]}\right)\right)\right],\,\tau\in\domC[,\gamma]^k,\,\gamma=\gammaN[,\sigma].
\end{align*}
Next, we use independence, expand the definition of $\vTSYM[\sigma,\tau]$ and obtain
\begin{align*}
E(\sigma,\tau)
=\expe\left[\sum_v\prod_{h\in[k]}\frac{\bmone\{\sigma_{v(h)}=\tau_h\}}{|\sigma^{-1}(\tau_h)|}\ln\left(\ZF\left(\psiTSYM[\tau],\left(\bm\mu^*_\sigma|_{v(h)}\right)_{h\in[k]}\right)\right)\right].
\end{align*}
The definition of $\piGC[,G,\sigma,\tau(h)]$ completes the proof.
\end{proof}
Now, we can combine Lemma \ref{lemma_assfactor_piGC} with Corollary \ref{cor_piCR}.
Hence, we introduce the reweighted marginals $(\vTSYM[\sigma,\tau],\psiTSYM[\tau],\hat\gammaR_{\sigma,\tau})\dequal\wTSYa[,\sigma,\tau]\otimes\bigotimes_{h\in[k]}\piGR[,\GTSM(\sigma),\tau(h)]$.
\begin{lemma}\label{lemma_assfactor_piGR}
We have $\Phi_{\mrmf}(n)=\frac{\degae(k-1)}{k}\expe[\bmone\{\sigmaIID\in\mclb[-][\Gamma]\}\ln(\ZF(\psiTSYM[\tauTSM],\hat\gammaR_{\sigmaIID,\tauTSM}))]+\mclo(n^{-\rho})$.
\end{lemma}
\begin{proof}
Fix $\sigma$, $\tau$, $G=[w]^{\Gamma\darr}_{\setP,\sigma}$ with $w\in\domG$ and
$\psi\in\domPsi$. Let $\check E(\sigma,\tau,G,\psi)=\expe[\ln(\ZF(\psi,\gammaR))]$ with
$\gammaR\dequal\bigotimes_h\piGC[,G,\sigma,\tau(h)]$ and 
$\hat E(\sigma,\tau,G,\psi)=\expe[\ln(\ZF(\psi,\gammaR))]$ with
$\gammaR\dequal\bigotimes_h\piGR[,G,\tau(h)]$.
Let $\pi_h\in\Gamma(\piGC[,G,\sigma,\tau(h)],\piGR[,G,\tau(h)])$ be a coupling for $h\in[k]$ and $(\check\gammaR,\hat\gammaR)\dequal\bigotimes_h\pi_h$ with $\check\gammaR,\hat\gammaR\in\mclp([q])^k$.
We have
$\ZF(\psi,\check\gammaR),\ZF(\psi,\hat\gammaR)\in[\psibl,\psibu]$, so the logarithm is $\psibu$-Lipschitz on this domain, and thereby using Observation \ref{obs_tv}\ref{obs_tv_prod} we obtain
\begin{align*}
\Delta(\sigma,\tau,G,\psi)
&=\left|\check E-\hat E\right|
\le\psibu\expe\left[\sum_{\tau'}\psi(\tau')\left|\prod_h\check\gammaR_h(\tau'_h)-\prod_h\hat\gammaR_h(\tau'_h)\right|\right]\\
&\le 2\psibu^2\sum_h\|\check\gammaR_h-\hat\gammaR_h\|_\mrmtv.
\end{align*}
Since this holds for any choice of coupling we have
$\Delta\le 2\psibu^2\sum_h\distW(\piGC[,G,\sigma,\tau(h)],\piGR[,G,\tau(h)])$.
With $\tau\in\sigma([n])^k$
and $D(\sigma,\mu)$ from Corollary \ref{cor_piCR} this yields
$\Delta\le 2k\psibu^2 D(\sigma,\lawG[,G])$.
Hence, taking the expectation and using $c$, $C_1$ from Corollary \ref{cor_piCR}\ref{cor_piCR_IID} with $\gammaR_{\sigma,\tau}$ from Lemma \ref{lemma_assfactor_piGC} gives
\begin{align*}
\Delta&=\left|\expe[\bmone\{\sigmaIID\in\mclb[-][\Gamma]\}\ln(\ZF(\psiTSYM[\tauTSM],\gammaR_{\sigmaIID,\tauTSM}))]
-\expe[\bmone\{\sigmaIID\in\mclb[-][\Gamma]\}\ln(\ZF(\psiTSYM[\tauTSM],\hat\gammaR_{\sigmaIID,\tauTSM}))]\right|\\
&\le 2k\psibu^2\expe\left[D(\sigmaIID,\GTSM(\sigmaIID))\right]
\le 2k\psibu\left(\frac{c}{\ThetaP[C_1]}+q\prob[\mR>\mbu]\right).
\end{align*}
Recall that $\prob[\mR>\mbu]=o(1/n)$, $\ThetaP=n^{1-\rho}$ with $\rho=c/(1+c)$, $c=C_1/3$, and notice that $(1-\rho)C_1=3\rho>\rho$, so $\Delta=o(n^{-\rho})$ and hence the assertion holds with Lemma \ref{lemma_assfactor_piGC}.
\end{proof}
\subsubsection{The Factor Contribution}\label{assfactor_final}
In this section we complete the discussion of $\Phi_\mrmf$.
First, we resolve the reweighting, then we turn to the projection onto $\mclp[*][2]([q])$.
Let $(\psiR,\gammaR)\dequal\lawpsi\otimes\piG[,\GTSM(\sigmaIID)]^{\otimes k}$
with $\GTSM(\sigmaIID)=\GTSM[\mR,\setPR](\sigmaIID)$.
\begin{lemma}\label{lemma_assfactor_final}
We have $\Phi_{\mrmf}(n)=\frac{\degae(k-1)}{k\ZFabu}\expe[\xlnx(\ZF(\psiR,\gammaR))]+\mclo(n^{-\rho})$.
\end{lemma}
\begin{proof}
Let $\gammaIID=\gammaN[,\sigmaIID]$ and $\gammaaR=\gammaaG[,\GTSM(\sigmaIID)]$.
With $c$ from Lemma \ref{lemma_gammaaR}\ref{lemma_gammaaR_prob} we have
\begin{align*}
\prob[\|\gammaaR-\gamma^*\|_\mrmtv\ge r]\le c_2e^{-c_1c_\mrmr\ln(n)}+\prob[\mR>\mbu]=o(1/n)
\end{align*}
since $c_\mrmr>1/c_1$ is large.
Lemma \ref{lemma_assfactor_piGR}, using that the argument to the logarithm is in $[\psibl,\psibu]$ (and the leading coefficient in $[0,\degabu]$), with $\mcle=\{\sigmaIID\in\mclb[-][\Gamma],\|\gammaaR-\gamma^*\|_\mrmtv\le r\}$ yields
\begin{align*}
\Phi_\mrmf(n)
=\frac{\degae(k-1)}{k}\expe[\bmone\mcle\ln(\ZF(\psiTSYM[\tauTSM],\hat\gammaR_{\sigmaIID,\tauTSM}))]+\mclo(n^{-\rho}).
\end{align*}
Resolving the Radon-Nikodym derivatives gives
$\Phi_\mrmf(n)=\frac{\degae(k-1)}{k}\expe[\bmone\mcle\bm\Phi]+\mclo(n^{-\rho})$ with
\begin{align*}
\bm\Phi&=\sum_\tau\frac{\psiae(\tau)\prod_h\gammaIID(\tau_h)}{\ZFa(\gammaIID)}\cdot\frac{\psiR(\tau)}{\psiae(\tau)}
\cdot\prod_h\frac{\gammaR_h(\tau_h)}{\gammaaR(\tau_h)}\ln(\ZF(\psiR,\gammaR))\\
&=\sum_\tau\frac{\psiR(\tau)\prod_h\gammaR_h(\tau_h)\prod_h\gammaIID(\tau_h)}{\ZFa(\gammaIID)\prod_h\gammaaR(\tau_h)}\ln(\ZF(\psiR,\gammaR)).
\end{align*}
On $\mcle$ we have $\gammaIID(\tau_h)/\gamma^*(\tau_h)\le 1+\psibu\|\gammaIID-\gamma^*\|_\infty$, which with the corresponding lower bound yields
$\gammaIID(\tau_h)/\gamma^*(\tau_h)=1+\mclo(r)$.
With Observation \ref{obs_fad}\ref{obs_fad_maxbound} and Observation \ref{obs_fad}\ref{obs_fad_bounds} we further have
$\ZFa(\gammaIID)/\ZFabu=1+\mclo(r^2)$.
Analogously to $\gammaIID$ we get $\gammaaR(\tau_h)/\gamma^*(\tau_h)=1+\mclo(r)$, so
\begin{align*}
\bm\Phi=(1+\mclo(r))\sum_\tau\frac{\psiR(\tau)\prod_h\gammaR_h(\tau_h)\prod_h\gamma^*(\tau_h)}{\ZFabu\prod_h\gamma^*(\tau_h)}\ln(\ZF(\psiR,\gammaR))
=(1+\mclo(r))\frac{\xlnx(\ZF(\psiR,\gammaR))}{\ZFabu}.
\end{align*}
With $|\bm\Phi|\le\psibu^2\ln(\psibu)$
and $r=o(n^{-\rho})$ we have
$\Phi_\mrmf(n)=\frac{\degae(k-1)}{k\ZFabu}\expe[\bmone\mcle\xlnx(\ZF(\psiR,\gammaR))]+\mclo(n^{-\rho})$.
Since the argument is still uniformly bounded, resolving $\mcle$ comes at a cost $o(1/n)$.
\end{proof}
Next, we show that we can replace $\piG[,\GTSM(\sigmaIID)]$ by its projection $\piG[,\GTSM(\sigmaIID)]^\circ$ by using Lemma \ref{lemma_piapproxpstartwo}. For this purpose we show that the factor contribution
\begin{align*}
\bethe_{\mrmf}:\mclp[][2]([q])\rarr\reals,\,
\pi\mapsto\frac{\degae(k-1)}{k\ZFabu}\expe\left[\xlnx\left(\ZF(\psiR,\gammaR_\pi)\right)\right],\,
(\psiR,\gammaR_\pi)\dequal\lawpsi\otimes\pi^{\otimes k},
\end{align*}
to the Bethe functional is Lipschitz in $\pi$ with respect to $\distW$.
\begin{lemma}\label{lemma_bethelipschitzfactor}
There exists $L_\mfkg$ such that $\bethe_{\mrmf}$ is $L$-Lipschitz.
\end{lemma}
\begin{proof}
Let $\pi_\circ\in\Gamma(\pi_1,\pi_2)$ be a coupling of $\pi\in\mclp[][2]([q])^2$ and
$(\psiR,\gammaR_1,\gammaR_2)\dequal\lawpsi\otimes\pi_\circ^{\otimes k}$ with $\gammaR_1,\gammaR_2\in\mclp([q])^k$.
Using that $\xlnx$ is $L'$-Lipschitz on $[\psibl,\psibu]$ with $L'=\ln(\psibu)+1$, we have
\begin{align*}
\Delta&=|\bethe_{\mrmf}(\pi_1)-\bethe_{\mrmf}(\pi_2)|
\le\frac{\degabu L'(k-1)}{k\psibl}\expe\left[\left|\ZF(\psiR,\gammaR_1)-\ZF(\psiR,\gammaR_2)\right|\right].
\end{align*}
With the triangle inequality, $\psiR\le\psibu$ and Observation \ref{obs_tv}\ref{obs_tv_prod} this gives 
\begin{align*}\Delta\le 2\degabu L'\psibu^2(k-1)\sum_h\frac{1}{k}\expe[\|\gammaR_{1,h}-\gammaR_{2,h}\|_\mrmtv]=2\degabu L'\psibu^2(k-1)\expe[\|\gammaR_{1,1}-\gammaR_{2,1}\|_\mrmtv].
\end{align*}
This completes the proof since $\pi_\circ\in\Gamma(\pi_1,\pi_2)$ was arbitrary.
\end{proof}
Now, we are finally ready to establish the easier part of Lemma \ref{lemma_ass_PhiDelta}.
\begin{lemma}\label{lemma_ass_PhiDeltaF}
We have $\Phi_{\mrmf}(n)=\expe[\bethe_{\mrmf}(\piG[,\GTSM]^\circ)]+\mclo(n^{-\rho})$ with $\GTSM=\GTSM[\mR,\setPR](\sigmaIID)$.
\end{lemma}
\begin{proof}
With Lemma \ref{lemma_assfactor_final} we have $\Phi_{\mrmf}(n)=\expe[\bethe_\mrmf(\piG[,\GTSM])]+\mclo(n^{-\rho})$.
Lemma \ref{lemma_bethelipschitzfactor} and Lemma \ref{lemma_piapproxpstartwo}\ref{lemma_piapproxpstartwo_graph} complete the proof, since $\distW\le q$ and $\prob[\mR>\mbu]=o(1/n)$.
\end{proof}
\subsubsection{Typical Events for the Variable Contribution}\label{assvar_typical}
Now, we turn to the variable contribution of the Bethe functional, respectively the contribution $\Phi_{\mrmv}(n)=\expe[\bm\Phi]$ from Equation (\ref{equ_PhiDeltContributions}) with $\bm\Phi=\ln(\ZG(\GR^+(\sigmaR^+))/\ZG(\GR_\cap))$.
Recall $r(n)$, $\mclb[+][\Gamma]$ from Section \ref{ass_factor_count_asymptotics},
$\mclb[][\circ]$ from Section \ref{assfactor_typical} and let $\degaR^+=k\mR^+/n$.
\begin{lemma}\label{lemma_assvar_typ}
We have $\Phi_{\mrmv}(n)=\expe[\bmone\{\sigmaR^+\in\mclb[+][\Gamma],\sigmaR^-\in\mclb[-][\Gamma],\degaR^+\in\mclb[][\circ]\}\bm\Phi]+o(n^{-1})$.
\end{lemma}
\begin{proof}
With $\mcle=\{\sigmaR^+\in\mclb[+][\Gamma],\sigmaR^-\in\mclb[-][\Gamma],\degaR^+\in\mclb[][\circ]\}$ we have
$\Phi_\mrmv(n)=\expe[\bmone\mcle\bm\Phi]+\eps$ with
\begin{align*}
\eps(n)=\expe\left[\bmone\lnot\mcle\overline{\bm\Phi}\right],\,
\overline{\bm\Phi}=\expe\left[\bm\Phi\middle|\mR_\cap(\sigmaR^+),\sigmaR^+,\sigmaR^-,\degaR^+\right].
\end{align*}
With Jensen's inequality we can consider the atypical events separately, i.e.
\begin{align*}
|\eps(n)|
\le\expe[\bmone\{\degaR^+\not\in\mclb[][\circ]\}|\overline{\bm\Phi}|]
+\expe[\bmone\{\sigmaR^+\not\in\mclb[+][\Gamma]\}|\overline{\bm\Phi}|]
+\expe[\bmone\{\sigmaR^-\not\in\mclb[-][\Gamma]\}|\overline{\bm\Phi}|].
\end{align*}
With $\bm\Phi=(n+1)\phiG(\ZG(\GR^+(\sigmaR^+)))-n\phiG(\ZG(\GR_\cap))$, Jensen's inequality, the triangle inequality, 
$\tilde c$ from Observation \ref{obs_phi_lipschitz},
Lemma \ref{lemma_ass_coupsetp}, Lemma \ref{lemma_ass_couptheta} and Observation \ref{obs_pin_basic} we have
\begin{align*}
|\overline{\bm\Phi}|
\le\tilde c\left(\mR^++\frac{\ThetaP_{n+1}}{2}+\mR_\cap(\sigmaR^+)+\frac{n\ThetaP_n}{(n+1)2}\right)
\le\tilde c(2\mR^++n+1).
\end{align*}
So, with Lemma \ref{lemma_ass_graph_couplingpo}, $c$ from Corollary \ref{cor_dega} and $n'=n+1$ we have
\begin{align*}
\expe[\bmone\{\degaR^+\not\in\mclb[][\circ]\}|\overline{\bm\Phi}|]
\le\frac{2\tilde cn'}{k}\expe[\bmone\{\degaR_{n'}\not\in\mclb[][\circ]\}\degaR_{n'}]
+\tilde cn'\prob[\degaR_{n'}\not\in\mclb[][\circ]]
\le c'n'\exp\left(-\frac{c_1r^2n'}{1+r}\right)
\end{align*}
with $c'=\tilde cc_2(2+k)/k$, so $\expe[\bmone\{\degaR^+\not\in\mclb[][\circ]\}|\overline{\bm\Phi}|]=o(1/n)$.
With $c$ from Observation \ref{obs_gtiid}\ref{obs_gtiid_prob}, independence and $c_{\mrmr}^2>2/c_1$ we have
\begin{align*}
\expe[\bmone\{\sigmaR^+\not\in\mclb[+][\Gamma]\}|\overline{\bm\Phi}|]
\le\tilde c\left(\frac{2\degabu}{k}+1\right)c_2n'e^{-c_1r^2n'}
=o(n^{-1}),
\end{align*}
and $\expe[\bmone\{\sigmaR^-\not\in\mclb[-][\Gamma]\}|\overline{\bm\Phi}|]=o(n^{-1})$ follows analogously.
\end{proof}
The following result further restricts the very typical event in Lemma \ref{lemma_assvar_typ} to the typical event that no variables are pinned in the second sweep.
\begin{lemma}\label{lemma_assvar_pin}
We have
\begin{align*}
\Phi_{\mrmv}(n)=\expe[\bmone\{\sigmaR^+\in\mclb[+][\Gamma],\sigmaR^-\in\mclb[-][\Gamma],\degaR^+\in\mclb[][\circ],\setPR^+=\setPR_\cap\}\bm\Phi]+\mclo(n^{-\rho}).
\end{align*}
\end{lemma}
\begin{proof}
With Lemma \ref{lemma_assvar_typ} it is sufficient to consider
$\expe[\bmone\mcle|\bm\Phi|]$, where
\begin{align*}
\mcle=\left\{\sigmaR^+\in\mclb[+][\Gamma],\sigmaR^-\in\mclb[-][\Gamma],\degaR^+\in\mclb[][\circ],\setPR^+\neq\setPR_\cap\right\}.
\end{align*}
As opposed to the proof of Lemma \ref{lemma_assfactor_pin} we cannot use Observation \ref{obs_phi_lipschitz} since now the numbers of variables do not coincide.
Let $\bmcla_\cap\dotcup\bmcla_+=[\mR^+]$ be the partition of the standard factors of $\GR^+(\sigmaR^+)$ such that $\bmcla_\cap$ is the relabeling of the standard factors $[\mR_\cap]$ in $\GR_\cap$. Further, let $\bmclv^\darr=\setPR^+\setminus\setPR_\cap$ be the additional pins and $\bmcla^\darr=\{a\in\bmcla_\cap:\vR^+_a([k])\cap\bmclv^\darr\neq\emptyset\}$ with $\vR^+$ being the neighborhoods in $\GR^+(\sigmaR^+)$.
The bounds from the proof of Observation \ref{obs_phi_lipschitz} and normalization of the external field for the last variable give
\begin{align*}
\ZG(\GR^+(\sigmaR^+))\le\psibu^{|\bmcla_+|}\ZG(\GR_\cap),\,
\ZG(\GR^+(\sigmaR^+))\ge\psibl^{|\bmcla_+|+2|\bmcla^\darr|+|\bmclv^\darr|}\ZG(\GR_\cap).
\end{align*}
This shows that $|\bm\Phi|\le\ln(\psibu)(\mR^+_\Delta+2|\bmcla^\darr|+|\bmclv^\darr|)$.
Bounding $|\bmcla^\darr|$ by the sum of the degrees of $i\in\bmclv^\darr\cap[n]$ in $\GR_\cap$  and taking the conditional expectation as in the proof of Lemma \ref{lemma_assfactor_pin} gives
the bound $\ln(\psibu)(\mR^+_\Delta+2c\frac{k\mR_\cap}{n}|\bmclv^\darr\cap[n]|+|\bmclv^\darr|)$ with $c$ from Corollary \ref{cor_degbounds}\ref{cor_degbounds_m}.
With $c'$ from Lemma \ref{lemma_ass_factor_count_asymptotics} we obtain the bound
$\ln(\psibu)(c'+2c\frac{n+1}{n}(\degae+r)|\bmclv^\darr\cap[n]|+|\bmclv^\darr|)$ on $\mcle$.
Hence, we obtain $c_\mfkg\in\reals_{>0}$ such that
\begin{align*}
\expe[\bmone\mcle|\bm\Phi|]\le c\prob[\setPR^+\neq\setPR]+c\expe[\bmone\{\setPR^+\neq\setPR\}|\setPR^+\setminus\setPR|]\le 2c\expe[|\setPR^+\setminus\setPR|].
\end{align*}
As in the proof of Lemma \ref{lemma_assfactor_pin}, we trace this back to the indicators $\indPR^+_\Delta$ for the variables $[n]$, and pinning probability $\thetaPR_{n+1}/(n+1)$ for $i=n+1$. This yields
\begin{align*}
\expe[\bmone\mcle|\bm\Phi|]
\le 2c\expe\left[np^+_\Delta(\thetaPR^-,\thetaPR^+)+\frac{\thetaPR^+}{n+1}\right]
\le c\left(\frac{n(\ThetaP_{n+1}-\ThetaP_n)}{n+1-\ThetaP_n}+\frac{\ThetaP_{n+1}}{n+1}\right).
\end{align*}
With $\ThetaP(n)=n^{1-\rho}$ we have $\ThetaP_{n+1}-\ThetaP_n=(1-\rho)\int_n^{n+1}t^{-\rho}\mrmd t\le n^{-\rho}$, and hence the assertion follows since $\expe[\bmone\mcle|\bm\Phi|]=\mclo(n^{-\rho})$.
\end{proof}
\subsubsection{Normalization Step for the Variable Contribution}\label{assvar_normalization}
Now, we simplify the underlying law using the typical behavior.
As before, we replace $\GR_\cap$ by $\GTSM(\sigma^-)=\GTSM[\mR,\setPR](\sigma^-)$,
and $\mR^+_\Delta(\sigmaR^+)$ by a Poisson variable $\dR\dequal\Po(\degae)$ reflecting the degree of $i=n+1$.
Recalling Lemma \ref{lemma_ass_factor_count_asymptotics}, we obtain the additional wires-weight pairs given $\sigmaR^+$ from Observation \ref{obs_known}, i.e.~we consider
\begin{align*}
(\GTSM(\sigma^-),\dR,\wTSM(\sigma^+))\dequal\GTSM(\sigma^-)\otimes\dR\otimes\wTSM[+\circ,n+1,i,\sigma^+][\otimes\ints_{>0}].
\end{align*}
Since we have an additional variable, we have to adjust the definition
\begin{align*}
\psiWgG[,G](v,\psi)=\expe\left[\prod_{a\in[d]}\psi_a(\sigmaR_{v(a)})\right],\,
\sigmaR\dequal\lawG[,G]\otimes\gamma^*,\,
(v,\psi)\in([n+1]^k\times\domPsi)^d,
\end{align*}
from Equation (\ref{equ_psiWgG}), where $G$ is still a decorated graph on $n$ variables.
\begin{lemma}\label{lemma_assvar_normalized}
We have
$\Phi_{\mrmv}(n)=\expe[\bmone\{\sigmaR^-\in\mclb[-][\Gamma]\}\ln(\psiWgG[,\GTSM(\sigmaR^-)](\wTSM[\sigmaR^+,[\dR]]))]+\mclo(n^{-\rho})$.
\end{lemma}
\begin{proof}
Let $\dR$ and $\wTSM(\sigma^+)$ be independent of anything else and
\begin{align*}
\mcle=\left\{\sigmaR^+\in\mclb[+][\Gamma],\sigmaR^-\in\mclb[-][\Gamma],\degaR^+\in\mclb[][\circ],\setPR^+=\setPR_\cap\right\}.
\end{align*}
For $n\ge n_\circ$ with $n_\circ$ from Lemma \ref{lemma_ass_factor_count_asymptotics} and on $\mcle$, as explained in Section \ref{coupling_pins} and conditional to $\sigmaR^+$, $\GR_\cap$, $\setPR^+$ and $\dR^+=\mR^+_{\Delta+}(\sigmaR^+)$, we obtain $\GR^+(\sigmaR^+)$ from $\GR_\cap$ by adding the variable $i=n+1$ with external field $\gamma^*$ and $\dR^+$ standard factors with 
wires-weight pairs from $(\vR^*,\psiR^*)=\wR^*=\wR^*_{\sigmaR^+,[\dR^+]}$. Hence, on $\mcle$ we have
\begin{align*}
\bm r=\frac{\ZG(\GR^+(\sigmaR^+))}{\GR_\cap}
\dequal\sum_{\sigma^+}\frac{\psiG[,\GR_\cap](\sigma^+_{[n]})}{\ZG(\GR_\cap)}\gamma^*(\sigma^+_i)\prod_{a\in[\dR^+]}\psiR^*_a(\sigma^+_{\vR^*(a)})
=\psiWgG[,\GR_\cap](\wTSM).
\end{align*}
Next, we couple $\dR^+$ and $\dR$ using $\bm\delta(\sigmaR^+)\dequal\Po(|\me^+_{\Delta+}(\sigmaR^+)-\degae|)$ as in the proof of Lemma \ref{lemma_assfactor_normalized} to obtain
\begin{align*}
\left|\ln\left(\psiWgG[,\GR_\cap]\left(\wTSM[\sigmaR^+,[\dR^+]]\right)\right)-\ln\left(\psiWgG[,\GR_\cap]\left(\wTSM[\sigmaR^+,[\dR]]\right)\right)\right|
\le\bm\delta(\bm\sigma^+)\ln(\psibu).
\end{align*}
With Lemma \ref{lemma_ass_factor_count_asymptotics} we can bound $\expe[\bm\delta(\sigma^+)]$ on $\mcle$, so with Lemma \ref{lemma_assvar_pin} we have
\begin{align*}
\Phi_{\mrmv}(n)=\expe\left[\bmone\mcle\ln\left(\psiWgG[,\GR_\cap]\left(\wTSM[\sigmaR^+,[\dR]]\right)\right)\right]+\mclo\left(r(n)+n^{-\rho}\right),
\end{align*}
since the expectations can be bounded by $c'\ln(\psibu)$ with $c'$ from Lemma \ref{lemma_ass_factor_count_asymptotics} and $\ln(\psibu)\degae$ respectively for $n\le n_\circ$.
This also shows that reducing $\mcle$ to $\{\sigmaR^-\in\mclb[-][\Gamma]\}$ causes an error of $\mclo(n^{-\rho})$, and that with the coupling from Lemma \ref{lemma_ass_base_graph} we get
\begin{align*}
\Phi_{\mrmv}(n)=\expe\left[\bmone\{\sigmaR^-\in\mclb[-][\Gamma]\}\ln\left(\psiWgG[,\GTSM(\sigmaR^-)]\left(\wTSM[\sigmaR^+,[\dR]]\right)\right)\right]+\mclo\left(n^{-\rho}\right).
\end{align*}
\end{proof}
In a second normalization step we simplify $\wTSM[\sigmaR^+,[\dR]]$ by establishing that $i=n+1$ typically does not wire more than once to the same factor.
As seen in Section \ref{assfactor_marginaldist}, it is reasonable to explicitly control the factor assignments.
With $\sigma^+\in[q]^{n+1}$, $\sigma^-=\sigma^+_{[n]}$, $\gamma^-=\gammaN[,\sigma^-]$ and $\sigma^\circ=\sigma^+_i$ let $(\tauR^+_{\circ,n,\sigma^+},\hR^+_{\circ,n,\sigma^+})\in[q]^k\times[k]$ be given by
\begin{align*}
\prob[\tauR^+_{\circ}=\tau,\hR^+_{\circ}=h]&=\frac{W(\tau,h)}{\ZFa^+(\sigma^\circ,\gamma^-)},\,\ZFa^+(\sigma^\circ,\gamma^-)=\sum_{\tau,h}W(\tau,h),\\
W(\tau,h)&=\bmone\{\tau_h=\sigma^\circ\}\frac{1}{k}\psiae(\tau)\prod_{h'\in[k]\setminus\{h\}}\gamma^-(\tau_{h'}).
\end{align*}
So, with
$\mclt[\sigma^+][+]=\{(\tau,h)\in[q]^k\times[k]:\tau_h=\sigma^\circ,\tau([k]\setminus\{h\})\setle\sigma^-([n])\}$
we have $(\tauR^+_{\circ},\hR^+_\circ)\in\mclt[][+]$ almost surely.
Further, for $\sigma^\circ\in\sigma^-([n])$ we have $\ZFa^+(\sigma^\circ,\gamma^-)=\ZFa(\gamma^-)\mu|_*(\sigma^\circ)/\gamma^-(\sigma^\circ)$ with $\mu=\lawYgC[,\gamma^-]$.
For $(\tau,h)\in\mclt[][+]$ let $(\vR^+_{\circ,\sigma^-,\tau,h},\psiR^+_{\circ,\tau})\dequal\unif(\mclv[][+])\otimes\psiTSYa[,\tau]$ with $\psiTSYa$ from Section \ref{GTSYM} and $\mclv[\sigma^-,\tau,h][+]=\{v\in[n+1]^k:v_h=n+1,\forall h'\in[k]\setminus\{h\}\,v(h')\in\sigma^{--1}(\tau_{h'})\}$.

For $d\in\ints_{\ge 0}$ and $(\tau,h)\in\mclt[\sigma^+][d]$
let $\wR^+_{\sigma^-,d,\tau,h}\dequal\bigotimes_{a\in[d]}(\vR^+_{\circ,\sigma^-,\tau(a),h(a)},\psiR^+_{\circ,\tau(a)})$ and
\begin{align}\label{equ_assvar_jointtsm}
\left(\GTSM[\mR,\setPR](\sigma^-),\dR,\tauR^+_{\sigma^+},\hR^+_{\sigma^+}\right)
\dequal\GTSM[\mR,\setPR](\sigma^-)\otimes\dR\otimes(\tauR^+_{\circ,\sigma^+},\hR^+_{\circ,\sigma^+})^{\otimes\ints_{>0}}.
\end{align}
Finally, let $\GTSM(\sigma^-)=\GTSM[\mR,\setPR](\sigma^-)$, $\tauR^+=\tauR^+_{\sigmaR^+,[\dR]}$, $\hR^+=\hR^+_{\sigmaR^+,[\dR]}$ and $\wR^+=\wR^+_{\sigmaR^-,\dR,\tauR^+,\hR^+}$.
\begin{lemma}\label{lemma_assvar_normalizedoneedge}
We have
$\Phi_{\mrmv}(n)=\expe[\bmone\{\sigmaR^-\in\mclb[-][\Gamma]\}\ln(\psiWgG[,\GTSM(\sigmaR^-)](\wR^+))]+\mclo(n^{-\rho})$.
\end{lemma}
\begin{proof}
Let $n_+=n+1$ and $i=n+1$. Further, let $\mclv[0]=\mclv[1]\dotcup\mclv[2]$ with
\begin{align*}
\mclv[1]=\{v\in[n_+]^k:|v^{-1}(i)|=1\},\,
\mclv[2]=\{v\in[n_+]^k:|v^{-1}(i)|>1\}.
\end{align*}
For $a\in\{0,1\}$ let $(\vR_{\circ,a},\psiR_{\circ,a})\dequal\unif(\mclv[a])\otimes\lawpsi$ and let $(\vR^*_{\circ,a},\psiR^*_{\circ,a})$ be given by the Radon-Nikodym derivative $r_a(v,\psi)=\psi(\sigma^+_v)/z_a(\sigma^+)$ with $z_a(\sigma^+)=\expe[\psiR_{\circ,a}(\sigma^+_{\vR_{\circ,a}})]$.
With
$(\vR,\psiR)=(\vR_{\circ,0},\psiR_{\circ,0})$,
$P_\circ=\prob[\vR\in\mclv[2]]$,
$P=\prob[\vR^*_{\circ,0}\in\mclv[2]]$,
$\vR_\mrmu\dequal\unif([n_+]^k)$ and using
$r_a\in[\psibl^2,\psibu^2]$ we have
\begin{align*}
P
\le\psibu^2P_\circ
=\frac{\psibu^2\prob[\vR_\mrmu\in\mclv[2]]}{\prob[\vR_\mrmu\in\mclv[0]]}
\le\frac{\psibu^2\prob[\vR_\mrmu\in\mclv[2]]}{\prob[\vR_\mrmu\in\mclv[1]]}
=\frac{\psibu^2(n_+^k-n^k-kn^{k-1})}{kn^{k-1}}
\le\frac{c}{n},\,
c=\frac{\psibu^2 2^k}{k},
\end{align*}
similar to the proof of Observation \ref{obs_degsucprob}\ref{obs_degsucprob_bounds}.
Further, since the $(\vR_{\circ,1},\vR)$-derivative is $r_\mrmv(v)=\bmone\{v\in\mclv[1]\}/\prob[\vR\in\mclv[1]]$, the Radon-Nikodym derivative of $(\vR^*_{\circ,1},\psiR^*_{\circ,1})$ with respect to $(\vR,\psiR)$ is $r(v,\psi)=r_1(v,\psi)r_\mrmv(v)=\bmone\{v\in\mclv[1]\}\psi(\sigma^+_v)/z^\circ_1(\sigma^+)$ with $z^\circ_1(\sigma^+)=\expe[\bmone\{\vR\in\mclv[1]\}\psiR(\sigma^+_{\vR})]$.
Clearly, we have $z^\circ_1(\sigma^+)\le z_0(\sigma^+)$, and on the other hand
\begin{align*}
R(\sigma^+)=\frac{z_0(\sigma^+)}{z^\circ_1(\sigma^+)}=
1+\frac{\expe[\bmone\{\vR\in\mclv[2]\}\psiR_0(\sigma^+_{\vR})]}{z^\circ_1(\sigma^+)}
\le 1+\frac{\psibu P_\circ}{\psibl\prob[\vR\in\mclv[1]]}
=1+\frac{\psibu^2\prob[\vR_\mrmu\in\mclv[2]]}{\prob[\vR_\mrmu\in\mclv[1]]},
\end{align*}
so the bound for $P$ from above yields
$1\le R(\sigma^+)\le 1+\frac{c}{n}$.

Now, we turn back to $\Phi_\mrmv$. With Lemma \ref{lemma_assvar_normalized} we have
$\Phi_\mrmv(n)=\expe[\bmone\mcle[][\circ]\bm\Phi]+\mclo(n^{-\delta})$, where
\begin{align*}
\bm\Phi=f_{\GR}(\wTSM),\,
f_G(w)=\ln(\psiWgG[,G](w)),\,
\GR=\GTSM(\sigmaR^-),\,
\wTSM=\wTSM[\sigmaR^+,[\dR]].
\end{align*}
Notice that $|\bm\Phi|\le\dR\ln(\psibu)$, and that $\wTSM$ given $\sigmaR^+$, $\dR$ are $\dR$ i.i.d.~copies of $(\vR^*_{\circ,0},\psiR^*_{\circ,0})$ from above.
Hence, the bound on $\bm\Phi$ with the union bound yield
$\Delta=|\expe[\bmone\mcle[][\circ]\bmone\lnot\mcle\bm\Phi]|
\le\ln(\psibu)\expe[\dR^2]P$.
Recall that $\dR\dequal\Po(\degae)$, hence $\expe[\dR^2]=\degae+\degae^2\le\degabu(\degabu+1)$, and that $P\le c/n$, so $\Delta=\mclo(1/n)$ and
$\Phi_\mrmv(n)=\Phi^\circ_\mrmv(n)+\mclo(n^{-\delta})$
with $\Phi^\circ_\mrmv(n)=\expe[\bmone\mcle[][\circ]\bmone\mcle\bm\Phi]$.
Now, let $\wR=(\vR,\psiR)\dequal(\vR_{\circ,0},\psiR_{\circ,0})^{\otimes\ints_{>0}}$ and
$\wTSM[\sigma^+]\dequal(\vTSM[\circ,1],\psiTSM[\circ,1])^{\otimes\ints_{>0}}$
be independent of anything else.
Then, with the shorthand $\wTSM=\wTSM[\sigmaR^+]$ we have
\begin{align*}
\Phi^\circ_\mrmv(n)
&=\expe\left[\bmone\mcle[][\circ]\prod_{a\in[\dR]}(\bmone\{\vR_a\in\mclv[1]\}r_0(\wR_a))f_{\GR}(\wR_{[\dR]})\right]
=\expe\left[\bmone\mcle[][\circ]f_{\GR}(\wTSM[[\dR]])R(\sigmaR^+)^{-\dR}\right].
\end{align*}
Using $|f_{\GR}(\wTSM[[\dR]])|\le\dR c'$, $c'=\ln(\psibu)$, and $1\le\bm R\le 1+\frac{c}{n}$ with $\bm R=R(\sigmaR^+)$ further gives
\begin{align*}
\Delta
&=\left|\Phi^\circ_\mrmv(n)-\expe\left[\bmone\mcle[][\circ]f_{\GR}(\wTSM[[\dR]])\right]\right|
\le c'\expe\left[\dR\left(1-\bm R^{-\dR}\right)\right]
\le c'\expe\left[\dR\left(\left(1+\frac{c}{n}\right)^{\dR}-1\right)\right].
\end{align*}
With $\dR_{\uarr}\dequal\Po(\degabu)$, the standard coupling of $\dR$ and $\dR_\uarr$, Lipschitz continuity (for $\dR_\uarr>0$) and the moment generating function of $\Po(\degabu)$ we have
\begin{align*}
\Delta
\le\frac{c'c}{n}\expe\left[\dR_\uarr^2\left(1+\frac{c}{n}\right)^{\dR_\uarr-1}\right]
\le\frac{c'c}{n}\expe\left[e^{\dR_\uarr\lambda}\right]=\mclo(n^{-1}),\,
\lambda=2+\ln\left(1+\frac{c}{n}\right)\le 2+\frac{c}{n}.
\end{align*}
This shows that $\Phi_\mrmv(n)=\expe[\bmone\mcle[][\circ]f_{\GR}(\wTSM[[\dR]])]+\mclo(n^{-\rho})$.
Now, due to the conditional independence given $\sigmaR^+$ it suffices to show that $\wR^+$ and $\wTSM[[\dR]]$ given $\sigmaR^+$, $\dR$ have the same law.
Hence, for fixed $\sigma^+$ we have to show that
$(\vR^+,\psiR^+)\dequal(\vTSM,\psiTSM)$, where
\begin{align*}
(\vR^+,\psiR^+)=\left(\vR^+_{\circ,\tauR^+_\circ,\hR^+_\circ},\psiR^+_{\circ,\tauR^+_\circ}\right),\,
(\vTSM,\psiTSM)=\left(\vTSM[\circ,1],\psiTSM[\circ,1]\right).
\end{align*}
First, notice that the normalization constants coincide, i.e.
\begin{align*}
z_1&=\sum_v\frac{\bmone\{v\in\mclv[1]\}}{kn^{k-1}}\psiae(\sigma^+_v)
=\sum_{(\tau,h)\in\mclt[][+]}\frac{1}{k}\psiae(\tau)\sum_v\frac{\bmone\{v\in\mclv[][+]\}}{n^{k-1}}
=\ZFa^+,
\end{align*}
similar to the discussion in Section \ref{GTSYM}.
For $v\in\mclv[1]$ let $\tau(v)=\sigma^+_v$ and $h(v)\in[k]$ uniquely determined by $v(h(v))=i$.
Notice that we have $\tauR^+_\circ=\tau(\vR^+_\circ)$ and $\hR^+_\circ=h(\vR^+_\circ)$ by definition, and $\mclt[][+]=\{(\tau(v),h(v)):v\in\mclv[1]\}$.
So, for an event $\mcle$ and with $\psiR\dequal\lawpsi$ we have
\begin{align*}
\prob[(\vR^+,\psiR^+)\in\mcle]
&=\expe\left[\sum_{(\tau,h)\in\mclt[][+]}\sum_{v\in\mclv[\tau,h][+]}
\frac{\psiae(\tau)\prod_{h'\neq h}\gamma^-(\tau_{h'})\psiR(\tau)}
{k\ZFa^+\prod_{h'\neq h}(n\gamma^-(\tau_{h'}))\psiae(\tau)}
\bmone\{(v,\psiR)\in\mcle\}\right]\\
&=\expe\left[\sum_{(\tau,h)\in\mclt[][+]}\sum_{v\in\mclv[\tau,h][+]}
\frac{\psiR(\tau)}
{kn^{k-1}z_1}
\bmone\{(v,\psiR)\in\mcle\}\right]\\
&=\expe\left[r_1(\vR_{\circ,1},\psiR_{\circ,1})\bmone\{(\vR_{\circ,1},\psiR_{\circ,1})\in\mcle\}\right]
=\prob[(\vTSM,\psiTSM)\in\mcle].
\end{align*}
\end{proof}
\subsubsection{Gibbs Marginal Product for the Variable Contribution}\label{assvar_gibbsproduct}
We are ready to apply Proposition \ref{proposition_pinG}.
Using the distribution (\ref{equ_assvar_jointtsm}) and the corresponding shorthands let $(\vR^+,\psiR^+)=\wR^+$, $\gammaR^+=(\lawG[,\GTSM(\sigmaR^-)]|_{\vR^+(a,h)})_{a\in[\dR],h\neq\hR^+(a)}$, recall $\ZV$ from Section \ref{bethe_main} and let
$\bm\Phi=\ln\left(\ZV\left(\dR,\psiR^+,\hR^+,\gammaR^+\right)\right)$, where we dropped the redundant dependencies on $\gamma_{a,h(a)}$ in the definition of $\ZV$.
\begin{lemma}\label{lemma_assvar_gibbsproduct}
We have $\Phi_{\mrmv}(n)=\expe[\bmone\{\sigmaR^-\in\mclb[-][\Gamma]\}\bm\Phi]+\mclo(n^{-\rho})$.
\end{lemma}
\begin{proof}
Recall $n_+$, $i$, $\mclv[1]$, $(\vR_{\circ,1},\psiR_{\circ,1})$, $(\vR^*_{\circ,1},\psiR^*_{\circ,1})$, $r_1$, $\mcle[][\circ]$ from the proof of Lemma \ref{lemma_assvar_normalizedoneedge}, and that $(\vR^+_{\circ,\tauR^+_\circ,\hR^+_\circ},\psiR^+_{\circ,\tauR^+_\circ})\dequal(\vR^*_{\circ,1},\psiR^*_{\circ,1})$, all for given $\sigma^+$.
First, we resolve the reweighting, i.e.~we consider $(\sigmaR^+,\GTSM(\sigmaR^-),\dR,\wR)\dequal(\sigmaR^+,\GTSM(\sigmaR^-))\otimes\dR\otimes(\vR_{\circ,1},\psiR_{\circ,1})^{\otimes\ints_{>0}}$, and use Lemma \ref{lemma_assvar_normalizedoneedge} to obtain
$\Phi_{\mrmv}=\Phi^\circ_\mrmv+\mclo(n^{-\rho})$ with
$\Phi^\circ_\mrmv=\expe[\bmone\mcle[][\circ]\prod_{a\in[\bm d]}r_1\left(\wR_a\right)\ln(\psiWgG[,\GTSM(\sigmaR^-)](\wR_{[\dR]}))]$.
Next, notice that $V:[k]\times[n]^{k-1}\rarr\mclv[1]$ is a bijection, where $v'=V(h,v)$ is given by $v'_h=n+1$ and $v'\circ\eta=v$, with $\eta:[k-1]\rarr[k]\setminus\{h\}$ denoting the enumeration.
Further, notice that $\vR_{\circ,1}\dequal V(\hR,\vR)$ with $(\hR,\vR)\dequal\unif([k])\otimes\unif([n]^{k-1})$. So, with
\begin{align*}
\left(\sigmaR^+,\GTSM(\sigmaR^-),\dR,\hR,\vR,\psiR\right)
\dequal(\sigmaR^+,\GTSM(\sigmaR^-))\otimes\dR\otimes
\left(\unif([k])\otimes\unif([n]^{k-1})\otimes\lawpsi\right)^{\otimes\ints_{>0}}
\end{align*}
we have $\Phi_{\mrmv}^\circ=\expe[\bmone\mcle[][\circ]\bm\Phi^*]$, where
$\bm\Phi^*=\prod_{a\in[\bm d]}r_1\left(\wR_a\right)\ln(\bm Z^*)$,
$\wR=(V(\hR_a,\vR_a),\psiR_a)_a$,
\begin{align*}
\bm Z^*
=\psiWgG[,\GTSM(\sigmaR^-)](\wR_{[\dR]})
=\sum_{\sigma^\circ}\gamma^*(\sigma^\circ)\sum_{\tau}
\bm\alpha^*(\tau)
\prod_{a\in[\dR]}\psiR_a(t(\tau_a,\sigma^\circ,\hR_a)),
\end{align*}
further $\bm\alpha^*=\bm\mu^*|_{\vR'}$, $\vR'=\vR_{[\dR]}\in([n]^{k-1})^{\dR}$, $\bm\mu^*=\lawG[,\GTSM(\sigmaR^-)]$, and $\tau'=t(\tau,\sigma^\circ,h)\in[q]^k$ given by $\tau'_h=\sigma^\circ$ and $\tau'\circ\eta=\tau$ using the enumeration $\eta:[k-1]\rarr[k]\setminus\{h\}$.
Now, regarding $\bm\alpha^*$, the situation is very similar to the proof of Lemma \ref{lemma_assfactor_gibbsproduct}, in particular given $\dR$ we have $(\bm\mu^*,\vR')\dequal\bm\mu^*\otimes\unif([n])^{\otimes(k-1)\dR}$.
Hence, let $C$ from Proposition \ref{proposition_pinG}\ref{proposition_pinG_IIDDKL}, $\eps=C_1/3$ and $\delta=\ThetaP[-2\eps]$.
Using $\iota_\circ$ from Section \ref{pinning_gibbs} let
$\mcle=\{\sigmaR^-\in\mclb[-][\Gamma],\iota_\circ(\bm\mu^*,\vR')\le\delta\}$.
Hence, the bound
$|\bm\Phi^*|\le\ln(\psibu^{\dR})\psibu^{2\dR}\le\psibu^{3\dR}$
and Markov's inequality conditional to $\dR$ give
$\Delta\le\eps'+\mclo(n^{-\rho})$, where
$\Delta=|\Phi_{\mrmv}-\expe[\bmone\mcle\bm\Phi]|$ and
\begin{align*}
\eps'&=\expe\left[\bmone\{\dR>0\}\frac{C_2((k-1)\dR-1)}{\delta}\left(\frac{(k-1)\dR}{\ThetaP}\right)^{C_1}\psibu^{3\dR}\right].
\end{align*}
Standard bounds imply
$\eps'\le\tilde c\expe[\exp(\tilde c\dR)]/\ThetaP[\eps]$ for some $\tilde c\in\reals_{>0}$.
The canonical coupling of $\dR\dequal\Po(\degae)$ and $\Po(\degabu)$ gives
$\eps'=\mclo(\ThetaP[-\eps])$.
Recall that $\delta=n^{-2\rho}$ and $\eps',\Delta=\mclo(n^{-\rho})$ as in the proof of Lemma \ref{lemma_assfactor_gibbsproduct}.
Now, with $\alphaR=\bigotimes_{(a,h)\in[\dR]\times[k-1]}\bm\mu^*|_{\vR'(a,h)}$, further
\begin{align*}
\bm Z
=\sum_{\sigma^\circ}\gamma^*(\sigma^\circ)\sum_{\tau}
\bm\alpha(\tau)
\prod_{a\in[\dR]}\psiR_a(t(\tau_a,\sigma^\circ,\hR_a)),
\end{align*}
and $\bm\Phi=\prod_{a\in[\bm d]}r_1\left(\wR_a\right)\ln(\bm Z)$,
notice that $\bm Z^*,\bm Z\in[\psibl^{\dR},\psibu^{\dR}]$, so Lipschitz continuity of the logarithm gives
$|\bm\Phi^*-\bm\Phi|\le\psibu^{3\dR}|\bm Z^*-\bm Z|\le 2\psibu^{4\dR}\|\bm\alpha^*-\bm\alpha\|_\mrmtv$.
Remark \ref{remark_epssym} yields $|\bm\Phi^*-\bm\Phi|\le\sqrt{2}\psibu^{4\dR}\sqrt{\iota_\circ(\bm\mu^*,\vR')}\le\sqrt{2}\psibu^{4\dR}\sqrt{\delta}$ on $\mcle$ and hence
\begin{align*}
\Phi_{\mrmv}
=\expe[\bmone\mcle\bm\Phi]+\mclo\left(\sqrt{\delta}+n^{-\rho}\right)
=\expe[\bmone\{\sigmaR^-\in\mclb[-][\Gamma]\}\bm\Phi]+\mclo\left(n^{-\rho}\right).
\end{align*}
The assertion follows by reintroducing $\wR^+$ using the Radon-Nikodym derivative in $\bm\Phi$.
\end{proof}
\subsubsection{Marginal Distribution for the Variable Contribution}\label{assvar_marginaldist}
Now, we work towards the discussion in Section \ref{pinning_gibbs_marginal_distributions}.
Using the distribution (\ref{equ_assvar_jointtsm})
and for $\sigma^+$, $d$, $\tau$, $h$ let $(\wR^+,\gammaR_{\sigma^-,d,\tau,h})\dequal\wR^+\otimes\bigotimes_{a\in[d],h'\in[k]\setminus\{h(a)\}}\piGC[,\GTSM(\sigma^-),\sigma^-,\tau(a,h')]$, recall the shorthands, let $\gammaR=\gammaR_{\sigmaR^-,\dR,\tauR^+,\hR^+}$ and $(\vR^+,\psiR^+)=\wR^+$.
\begin{lemma}\label{lemma_assvar_piGC}
We have $\Phi_{\mrmv}(n)=\expe[\bmone\{\sigmaR^-\in\mclb[-][\Gamma]\}\ln(\ZV(\dR,\psiR^+,\hR^+,\gammaR))]+\mclo(n^{-\rho})$.
\end{lemma}
\begin{proof}
As for Lemma \ref{lemma_assfactor_piGC}, the assertion is immediate using the definition of $\vR^+$ and $\piGC$.
\end{proof}
Now, let $(\wR^+,\hat\gammaR_{\sigma^-,d,\tau,h})\dequal\wR^+\otimes
\bigotimes_{a\in[d],h'\in[k]\setminus\{h(a)\}}\piGR[,\GTSM(\sigma^-),\tau(a,h')]$.
As before, we use the shorthand $\hat\gammaR=\hat\gammaR_{\sigmaR^-,\dR,\tauR^+,\hR^+}$.
\begin{lemma}\label{lemma_assvar_piGR}
We have $\Phi_{\mrmv}(n)=\expe[\bmone\{\sigmaR^-\in\mclb[-][\Gamma]\}\ln(\ZV(\dR,\psiR^+,\hR^+,\hat\gammaR))]+\mclo(n^{-\rho})$.
\end{lemma}
\begin{proof}
As for Lemma \ref{lemma_assfactor_piGR} we consider the difference $\Delta=|\check E-\hat E|$ of the expectations $\check E=\expe[\ln(\ZV(d,\psi,h,\gammaR))]$ with $\gammaR\dequal\bigotimes_{a,h'\neq h(a)}\piGC[,G,\sigma^-,\tau(a,h')]$ and
$\hat E=\expe[\ln(\ZV(d,\psi,h,\gammaR))]$ with $\gammaR\dequal\bigotimes_{a,h'\neq h(a)}\piGR[,G,\tau(a,h')]$.
For couplings
$\pi_{a,h'}\in\Gamma(\piGC[,G,\sigma^-,\tau(a,h')],\piGR[,G,\tau(a,h')])$ with $h'\neq h(a)$ we define 
$(\check\gammaR,\hat\gammaR)\dequal\bigotimes_{a,h'\neq h(a)}\pi_{a,h'}$ analogously.
With $\psibl^d\le\ZV(d,\psi,h,\cdot)\le\psibu^d$ and Observation \ref{obs_tv}\ref{obs_tv_prod} we get
\begin{align*}
\Delta
&\le\psibu^{2d}\sum_{\sigma^\circ}\gamma^*(\sigma^\circ)\expe\left[\sum_{\tau}\left|\prod_{a,h'\neq h(a)}\check\gammaR_{a,h'}(\tau_{a,h'})-\prod_{a,h'\neq h(a)}\hat\gammaR_{a,h'}(\tau_{a,h'})\right|\right]\\
&\le 2\psibu^{2d}\sum_{a,h'\neq h(a)}\expe\left[\|\check\gammaR_{a,h'}-\hat\gammaR_{a,h'}\|_\mrmtv\right].
\end{align*}
Hence, we have
$\Delta\le 2\psibu^{2d}\sum_{a,h'\neq h(a)}\distW(\piGC[,G,\sigma^-,\tau(a,h')],\piGR[,G,\tau(a,h')])$, so with $D(\sigma,\mu)$ from Corollary \ref{cor_piCR} this yields
$\Delta\le 2d(k-1)\psibu^{2d}D(\sigma^-,\lawG[,G])$.
Taking the expectation and using Corollary \ref{cor_piCR}\ref{cor_piCR_IID} with $\gammaR$ from Lemma \ref{lemma_assvar_piGC} gives
\begin{align*}
\Delta&=\left|\expe[\bmone\{\sigmaR^-\in\mclb[-][\Gamma]\}\ln(\ZV(\dR,\psiR^+,\hR^+,\gammaR))]
-\expe[\bmone\{\sigmaR^-\in\mclb[-][\Gamma]\}\ln(\ZV(\dR,\psiR^+,\hR^+,\hat\gammaR))]\right|\\
&\le\expe[2\dR(k-1)\psibu^{2\dR}]\expe\left[D(\sigmaR^-,\GTSM(\sigmaR^-))\right]
=o(n^{-\rho})
\end{align*}
analogously to the proof of Lemma \ref{lemma_assfactor_piGR} with the standard coupling of $\dR$ and $\Po(\degae)$, so the assertion holds with Lemma \ref{lemma_assvar_piGC}.
\end{proof}
\subsubsection{The Variable Contribution}\label{assvar_final}
In this section we complete the discussion of $\Phi_\mrmv$.
First, we resolve the reweighting, then we turn to the projection onto $\mclp[*][2]([q])$.
Similar to Section \ref{bethe_main} let
\begin{align*}
(\dR,\psiR,\hR,\gammaR)\dequal\Po(\degae)\otimes(\lawpsi\otimes\unif([k])\otimes\piG[,\GTSM(\sigmaIID)]^{\otimes k})^{\otimes\ints_{>0}}
\end{align*}
with $\GTSM(\sigmaIID)=\GTSM[\mR,\setPR](\sigmaIID)$,
$\psiR=\psiR_{[\dR]}$, $\hR=\hR_{[\dR]}$ and $\gammaR=\gammaR_{[\dR]}$ by an abuse of notation.
\begin{lemma}\label{lemma_assvar_final}
We have $\Phi_{\mrmv}(n)=\expe[\ZFabu^{-\dR}\xlnx(\ZV(\dR,\psiR,\hR,\gammaR))]+\mclo(n^{-\rho})$.
\end{lemma}
\begin{proof}
Let $i=n+1$, $\gammaR^-=\gammaN[,\sigmaR^-]$, $\sigmaR^\circ=\sigmaR^+_i$ and $\gammaaR=\gammaaG[,\GTSM(\sigmaR^-)]$.
With Lemma \ref{lemma_gammaaR}\ref{lemma_gammaaR_prob} we have
$\prob[\|\gammaaR-\gamma^*\|_\mrmtv\ge r]=o(1/n)$.
Lemma \ref{lemma_assvar_piGR} with $\mcle=\{\sigmaR^-\in\mclb[-][\Gamma],\|\gammaaR-\gamma^*\|_\mrmtv\le r\}$ yields
\begin{align*}
\Phi_\mrmv(n)=\expe[\bmone\mcle\bm\Phi]+\mclo(n^{-\rho}),\,
\bm\Phi=\ln(\ZV(\dR,\psiR^+,\hR^+,\hat\gammaR)),
\end{align*}
using $|\bm\Phi|\le\dR\ln(\psibu)$ and independence.
Using $(\dR,\psiR,\hR,\gammaR)$ with $\sigmaIID$ replaced by $\sigmaR^-$,
resolving the Radon-Nikodym derivatives and reusing the terms $1/k$ to introduce $\hR$ gives
$\Phi_\mrmv(n)=\expe[\bmone\mcle\bm\Phi]+\mclo(n^{-\rho})$, where
$\bm\Phi=\bm r\ln(\ZV(\dR,\psiR,\hR,\gammaR))$ and
\begin{align*}
\bm r&=\prod_{a\in[\dR]}\sum_{\tau}\frac{\bmone\{\tau_{\hR(a)}=\sigmaR^\circ\}\psiae(\tau)\prod_{h\neq\hR(a)}\gammaR^-(\tau_{h})}{\ZFa^+(\sigmaR^\circ,\gammaR^-)}\cdot\frac{\psiR_a(\tau)}{\psiae(\tau)}
\cdot\prod_{h\neq\hR(a)}\frac{\gammaR_{a,h}(\tau_{h})}{\gammaaR(\tau_{h})}\\
&=\prod_{a\in[\dR]}\sum_{\tau}
\frac{\bmone\{\tau_{\hR(a)}=\sigmaR^\circ\}\psiR_a(\tau)\prod_{h\neq\hR(a)}\gammaR_{a,h}(\tau_{h})\prod_{h\neq\hR(a)}\gammaR^-(\tau_{h})}
{\ZFa^+(\sigmaR^\circ,\gammaR^-)\prod_{h\neq\hR(a)}\gammaaR(\tau_{h})}.
\end{align*}
As in the proof of Lemma \ref{lemma_assfactor_final} we have
$\gammaR^-(\tau_h)/\gamma^*(\tau_h)=1+\mclo(r)$ and
$\gammaaR(\tau_h)/\gamma^*(\tau_h)=1+\mclo(r)$ on $\mcle$.
For $r(n)<\psibl/2$ we have $\gammaR^->0$ and hence
$\ZFa^+(\sigmaR^\circ,\gammaR^-)=\ZFa(\gammaR^-)\bm\mu|_*(\sigmaR^\circ)/\gammaR^-(\sigmaR^\circ)$ with $\bm\mu=\lawYgC[,\gammaR^-]$ as pointed out after the definition of $\ZFa^+$, above Equation (\ref{equ_assvar_jointtsm}).
As in the proof of Lemma \ref{lemma_assfactor_final} with Observation \ref{obs_fad}\ref{obs_fad_rnamlipschitz} and Observation \ref{obs_fad}\ref{obs_fad_amstar} this yields
$\ZFa^+(\sigmaR^\circ,\gammaR^-)/\ZFabu=(1+\mclo(r^2))(1+\mclo(r))=1+\mclo(r)$.
Hence, there exists $c_\mfkg\in\reals_{>0}$ such that
$\bm r\le(1+cr)^{\dR}\bm r^\circ$ and $\bm r^\circ\le(1+cr)^{\dR}\bm r$, where
\begin{align*}
\bm r^\circ
=\prod_{a\in[\dR]}\sum_{\tau}
\frac{\bmone\{\tau_{\hR(a)}=\sigmaR^\circ\}\psiR_a(\tau)\prod_{h\neq\hR(a)}\gammaR_{a,h}(\tau_{h})}{\ZFabu}.
\end{align*}
Now, with $\bm\Phi^\circ=\bm r^\circ\ln(\ZV(\dR,\psiR,\hR,\gammaR))$
and $|\bm\Phi^\circ|\le\ln(\psibu)\dR\psibu^{2\dR}$
we get
\begin{align*}
|\expe[\bmone\mcle\bm\Phi]-\expe[\bmone\mcle\bm\Phi^\circ]|\le\expe\left[\ln(\psibu)\dR\psibu^{2\dR}\left((1+cr)^{\dR}-(1+cr)^{-\dR}\right)\right]=\mclo(r).
\end{align*}
Notice that $\ln(\ZV(\dR,\psiR,\hR,\gammaR))$ does not depend on $\sigmaR^\circ$, so summing over $\sigmaR^\circ$ explicitly and using $\sigmaR^-\dequal\sigmaIID$ yields the assertion.
\end{proof}
Next, we show that we can replace $\piG[,\GTSM(\sigmaIID)]$ by its projection $\piG[,\GTSM(\sigmaIID)]^\circ$ by using Lemma \ref{lemma_piapproxpstartwo}. For this purpose we show that the variable contribution
\begin{align*}
\bethe_{\mrmv}:\mclp[][2]([q])\rarr\reals&,\,
\pi\mapsto\expe\left[\ZFabu^{-\dR}\xlnx\left(\ZV(\dR,\psiR,\hR,\gammaR_\pi)\right)\right],\\
(\dR,\psiR,\hR,\gammaR_\pi)
&\dequal\Po(\degae)\otimes(\lawpsi\otimes\unif([k])\otimes\pi^{\otimes k})^{\otimes\ints_{>0}},
\end{align*}
to the Bethe functional is Lipschitz in $\pi$ with respect to $\distW$.
\begin{lemma}\label{lemma_bethelipschitzvar}
There exists $L_\mfkg$ such that $\bethe_{\mrmv}$ is $L$-Lipschitz.
\end{lemma}
\begin{proof}
Let $\pi_\circ\in\Gamma(\pi_1,\pi_2)$ be a coupling of $\pi\in\mclp[][2]([q])^2$.
Further, let $(\gammaR_1,\gammaR_2)\dequal(\pi_\circ^{\otimes k})^{\otimes\ints_{>0}}$ with $\gammaR_1,\gammaR_2\in(\mclp([q])^k)^{\ints_{>0}}$.
With $(\dR,\psiR,\hR)$ from the definition of $\bethe_\mrmv$ let $(\dR,\psiR,\hR,\gammaR_1,\gammaR_2)\dequal\dR\otimes\psiR\otimes\hR\otimes(\gammaR_1,\gammaR_2)$.
With the Lipschitz continuity of $\xlnx$ yields
\begin{align*}
\Delta&=|\bethe_{\mrmv}(\pi_1)-\bethe_{\mrmv}(\pi_2)|
\le\expe\left[\psibu^{\dR}(\dR\ln(\psibu)+1)\left|\ZV(\dR,\psiR,\hR,\gammaR_1)-\ZV(\dR,\psiR,\hR,\gammaR_2)\right|\right].
\end{align*}
With the triangle inequality, $\psiR\le\psibu$ and Observation \ref{obs_tv}\ref{obs_tv_prod} this gives 
\begin{align*}
\Delta
&\le\expe\left[\psibu^{2\dR}(\dR\ln(\psibu)+1)\sum_{a\in[\dR]}\sum_{h'\neq\hR(a)}\|\gammaR_{1,a,h'}-\gammaR_{2,a,h'}\|_\mrmtv\right]\\
&=\expe\left[\psibu^{2\dR}(\dR\ln(\psibu)+1)\dR(k-1)]\expe[\|\gammaR_{1,1,1}-\gammaR_{2,1,1}\|_\mrmtv\right].
\end{align*}
This completes the proof since $\pi_\circ\in\Gamma(\pi_1,\pi_2)$ was arbitrary.
\end{proof}
Now, we finally obtain the asymptotics of $\Phi_\mrmv$.
\begin{lemma}\label{lemma_ass_PhiDeltaV}
We have $\Phi_{\mrmv}(n)=\expe[\bethe_{\mrmv}(\piG[,\GTSM]^\circ)]+\mclo(n^{-\rho})$ with $\GTSM=\GTSM[\mR,\setPR](\sigmaIID)$.
\end{lemma}
\begin{proof}
With Lemma \ref{lemma_assvar_final} we have $\Phi_{\mrmv}(n)=\expe[\bethe_\mrmv(\piG[,\GTSM])]+\mclo(n^{-\rho})$.
Lemma \ref{lemma_bethelipschitzvar} and Lemma \ref{lemma_piapproxpstartwo}\ref{lemma_piapproxpstartwo_graph} complete the proof, since $\distW\le q$ and $\prob[\mR>\mbu]=o(1/n)$.
\end{proof}
\subsubsection{Proof of Proposition \ref{proposition_ass}}\label{proof_proposition_ass}
First, we establish Lemma \ref{lemma_ass_PhiDelta} and Proposition \ref{proposition_ass_pinned}.
Then, we establish a stronger version of Proposition \ref{proposition_ass} for graphs with external fields.
\begin{proof}[Proof of Lemma \ref{lemma_ass_PhiDelta}]
Lemma \ref{lemma_ass_PhiDelta} follows from Lemma \ref{lemma_ass_PhiDeltaF} and Lemma \ref{lemma_ass_PhiDeltaV} with Equation (\ref{equ_PhiDeltContributions}).
\end{proof}
\begin{proof}[Proof of Proposition \ref{proposition_ass_pinned}]
With $\GTSM=\GTSM[\mR,\setPR](\sigmaIID)$ let $c_\mfkg\in\reals_{>0}$ be such that $|\Phi_{\Delta,n}-\expe[\bethe(\piG[,\GTSM]^\circ)]|\le cn^{-\rho}$.
Recall that $|\expe[n\phiG(\GTSM[n])]|\le c'(\frac{\degabu n}{k}+\frac{1}{2}n^{1-\rho})$
with $c'$ from Observation \ref{obs_phi_lipschitz} using Observation \ref{obs_pin_basic}, and
notice that $\Phi_{\Delta,0}=\expe[n\phiG(\GTSM[n])]$ for $n=1$.
With $|\bethe|\le\expe[\psibu^{\dR}\xlnx(\psibu^{\dR})]+\degae\psibu\xlnx(\psibu)\le c'$ for some $c'_{\mfkg}\in\reals_{>0}$ the telescoping sum with the triangle inequality yields
\begin{align*}
|\expe[\phiG(\GTSM)]-\expe[\bethe(\piG[,\GTSM]^\circ)]|
&\le\mclo(n^{-1})+\frac{c}{n}\sum_{n'=2}^{n-1}n'^{-\rho}
=\mclo(n^{-1})+\frac{c}{n}\int_{1}^{n-1}\lceil t^{-\rho}\rceil\mrmd t\\
&\le\mclo(n^{-1})+\frac{c}{n}\int_1^{n-1}t^{-\rho}\mrmd t
=\mclo(n^{-1})+\frac{c((n-1)^{1-\rho}-1)}{(1-\rho)n},
\end{align*}
which shows that $\expe[\phiG(\GTSM)]=\expe[\bethe(\piG[,\GTSM]^\circ)]+\mclo(n^{-\rho})$ and thereby completes the proof.
\end{proof}
In the remainder we let $\bm\pi=\piG[,\GTSM[\mR,\setPR](\sigmaIID)]^\circ\in\mclp[*][2]([q])$ be the projected marginal distributions \emph{including pins}.
On the other hand, we let $\setP=\emptyset$ in the remainder, where we also cover the case $\degae=0$.
Now, we turn to Proposition \ref{proposition_ass} for graphs with external fields.
\begin{proposition}\label{proposition_ass_external}
Notice that the following holds.
\begin{alphaenumerate}
\item\label{proposition_ass_external_po}
We have $\expe[\phiG(\GTSM[\mR](\sigmaIID))]=\expe[\bethe(\bm\pi)]+\mclo(n^{-\rho})$.
\item\label{proposition_ass_external_m}
For $d=km/n\le\degabu$ we have $\expe[\phiG(\GTSM[m](\sigmaIID))]=\expe[\bethe_d(\bm\pi)]+\mclo(n^{-\rho})$.
\item\label{proposition_ass_external_general}
We have $\expe[\phiG(\GTSM[\mR^*](\sigmaIID))]=\expe[\bethe(\bm\pi)]+\mclo(\deltam+\epsm+n^{-\rho})$.
\end{alphaenumerate}
\end{proposition}
\begin{proof}
Proposition \ref{proposition_ass_pinned}, Proposition \ref{proposition_pin_qfed} and $\ThetaP=n^{1-\rho}$ yield 
$\expe[\phiG(\GTSM[\mR](\sigmaIID))]=\expe[\bethe(\bm\pi)]+\mclo(n^{-\rho})$ for $\degae>0$.
Recall from the proof of Proposition \ref{proposition_int_external} that 
$\phiG(\GTSM[\mR](\sigmaIID))=\phiG(\GTSM[\mR](\sigmaNIS_{\mR}))=0$ and $\bethe\equiv 0$ for $\degae=0$, so Part \ref{proposition_ass_external}\ref{proposition_ass_external_po} holds.
The remainder follows similar to the proof of Proposition \ref{proposition_int_external}, but easier since the transition from $\sigmaNIS$ to $\sigmaIID$ is not required.
\end{proof}
Observation \ref{obs_standard_graphs} yields the corresponding results for graphs without external fields and thereby completes the proof of Proposition \ref{proposition_ass}.
\subsubsection{Proof of Theorem \ref{thm_bethe}}\label{proof_thm_bethe}
The following result for graphs with external fields implies Theorem \ref{thm_bethe}.
Recall $\rho$ from Section \ref{ass} and that $\mI\equiv 0$, $\setP=\emptyset$, $\tI=1$ and $\ThetaP=0$, i.e.~we consider standard graphs with external fields only.
\begin{theorem}\label{thm_bethe_external}
Notice that the following holds.
\begin{alphaenumerate}
\item\label{thm_bethe_external_po}
We have $\expe[\phiG(\GTSM[\mR](\sigmaIID))]=\bethebu(\degae)+\mclo(n^{-\rho})$.
\item\label{thm_bethe_external_m}
For $d=km/n\le\degabu$ we have $\expe[\phiG(\GTSM[m](\sigmaIID))]=\bethebu(d)+\mclo(n^{-\rho})$.
\item\label{thm_bethe_external_general}
We have $\expe[\phiG(\GTSM[\mR^*](\sigmaIID))]=\bethebu(\degae)+\mclo(\deltam+\epsm+n^{-\rho})$.
\end{alphaenumerate}
\end{theorem}
\begin{proof}
The assertion follows from Proposition \ref{proposition_int_external}, Proposition \ref{proposition_ass_external}, $\bethe_d\le\bethebu(d)$, and $\rho\in(0,1/4)$ as discussed in the introduction of Section \ref{ass}.
\end{proof}
Observation \ref{obs_standard_graphs} yields the corresponding results for graphs without external fields and thereby completes the proof of Theorem \ref{thm_bethe}.
\section{Relative Entropy, Condensation and Mutual Information}\label{main_proofs}
In this section we derive Theorem \ref{thm_infth}, Theorem \ref{thm_cond} and Theorem \ref{thm_mutinf} from Theorem \ref{thm_bethe_external}, for both graphs with and without external fields over more general factor counts $\mR^*$.
We also establish Lipschitz continuity in the average degree for all key quantities, i.e.~the corresponding versions of Proposition \ref{proposition_phi_concon}\ref{proposition_phi_concon_cont}.
Let $\mI\equiv 0$, $\setP=\emptyset$ and $\rho$ from Section \ref{ass}.
\subsection{The Relative Entropy}\label{main_proofs_dkl}
\subsubsection{The Annealed Free Entropy}\label{phiA}
In this section we briefly discuss the properties of the annealed free entropy.
For this purpose recall $\phia(d)=\frac{d}{k}\ln(\ZFabu)$ from Section \ref{information_theoretic_threshold}.
\begin{observation}\label{obs_phia}
Let $\phi(m)=\frac{1}{n}\ln(\expe[\ZG(\GR)])$.
\begin{alphaenumerate}
\item\label{obs_phia_lipschitz}
There exists $c_\mfkg\in\reals_{>0}$ such that
$\phi$ is $kc/n$-Lipschitz and $|\phi(m)|\le ckm/n$.
\item\label{obs_phia_asymptotics_m}
We have $\phi(m)=\phia(km/n)+\mclo(1/n)$ for $m\le\mbu$.
\item\label{obs_phia_asymptotics_po}
We have $\expe[\phi(\mR^*)]=\phia(\degae)+\mclo(\epsm+\deltam+n^{-1})$, so $\expe[\phi(\mR)]=\phia(\degae)+\mclo(\sqrt{\ln(n)/n})$.
\end{alphaenumerate}
\end{observation}
\begin{proof}
With the proofs of Observation \ref{obs_phi_lipschitz} and Lemma \ref{lemma_contphiM} we get $|\phi(m)|\le\frac{cm}{n}$ and $|\phi(m'_1)-\phi(m'_2)|\le\frac{c}{n}|m'_1-m'_2|$ for $m'\in\ints_{\ge 0}^2$ and $c=\ln(\psibu)$.
With Lemma \ref{lemma_firstmom}\ref{lemma_firstmomZ} we have $\phi(m)=\phia(km/n)+\mclo(1/n)$ for $m\le\mbu$. With Part \ref{obs_phia}\ref{obs_phia_lipschitz} and the expectation bound we have $\expe[\phi(\mR^*)]=\expe[\bmone\{|\degaR^*-\degae|\le\deltam\}\phi(\mR^*)]+\mclo(\epsm)$, so e.g.~with Part \ref{obs_phia}\ref{obs_phia_asymptotics_m}
and the probability bound we get
$\expe[\phi(\mR^*)]=\phia(\degae)+\mclo(\epsm+\deltam+n^{-1})$. The result for $\mR$ then follows with Corollary \ref{cor_dega} and $r=c'\sqrt{\ln(n)/n}$ for large $c'$.
\end{proof}
\subsubsection{Proof of Theorem \ref{thm_infth}}\label{proof_thm_infth}
The Nishimori ground truth establishes a finite size connection between the quenched free entropies, the annealed free entropy and the relative entropies.
\begin{observation}\label{obs_qfedinequ}
With $\phi(m)=\frac{1}{n}\ln(\expe[\ZG(\GR)])$ we have
\begin{align*}
\expe[\phiG(\GTSM(\sigmaNIS))]=\phi(m)+\DKL(\GTSM(\sigmaNIS)\|\GR)
\ge\phi(m)-\DKL(\GR\|\GTSM(\sigmaNIS))
=\expe[\phiG(\GR)].
\end{align*}
\end{observation}
\begin{proof}
Notice that Observation \ref{obs_nishicond}\ref{obs_nishicond_GNIS} yields both
$\DKL(\GTSM(\sigmaNIS)\|\GR)=\expe[\phiG(\GTSM(\sigmaNIS))]-\phi(m)$ and
$\DKL(\GR\|\GTSM(\sigmaNIS))=\phi(m)-\expe[\phiG(\GR)]$.
\end{proof}
The asymptotics from Theorem \ref{thm_bethe_external} using Corollary \ref{cor_phiTSIIDNIS} and from Observation \ref{obs_phia} with the first equality in Observation \ref{obs_qfedinequ} yield the asymptotics of $\DKL(\GTSM(\sigmaNIS)\|\GR)$.
Now, we obtain Theorem \ref{thm_infth} for graphs with external fields using the results of Section \ref{nishimori_ground_truth}.
\begin{theorem}\label{thm_infth_external}
Let $\delta(m)=\frac{1}{n}\DKL(\sigmaIID,\GTSM(\sigmaIID)\|\sigmaRG[,\GR],\GR)$ and
$\delta^*(d)=\bethebu(d)-\phia(d)$.
\begin{alphaenumerate}
\item\label{thm_infth_external_asymptotics_m}
We have $\delta(m)=\delta^*(km/n)+\mclo(n^{-\rho})$ for $km/n\le\degabu$.
\item\label{thm_infth_external_asymptotics_po}
We have $\expe[\delta(\mR^*)]=\delta^*(\degae)+\mclo(\epsm+\deltam+n^{-\rho})$, so $\expe[\delta(\mR)]=\delta^*(\degae)+\mclo(n^{-\rho})$.
\end{alphaenumerate}
\end{theorem}
\begin{proof}
The Radon-Nikodym derivative of
$(\sigmaIID,\GTSM(\sigmaIID))$ with respect to $(\sigmaRG[,\GR],\GR)$ is
\begin{align*}
(\sigma,G)\mapsto\frac{\gamma^{*\otimes n}(\sigma)\psiG[,G](\sigma)\ZG(G)}{\psiM(\sigma)\psiG[,G](\sigma)}=\frac{\ZG(G)}{\hat r(\sigma)\ZM},
\end{align*}
and thereby $\delta(m)=\phi^*(m)-\frac{m}{n}\expe[\ln(\ZFa(\gammaIID))]=\phi^*(m)-\phi(m)+\delta'(m)$ using Observation \ref{obs_GRM_expebounds}\ref{obs_GRM_expeboundsPsiM} and
with $\phi^*(m)=\expe[\phiG(\GTSM(\sigmaIID))]$,
$\phi(m)=\frac{1}{n}\ln(\ZM)$ and $\delta'(m)=\frac{1}{n}\DKL(\sigmaIID\|\sigmaNIS)$.
For Part \ref{thm_infth_external}\ref{thm_infth_external_asymptotics_m} we combine Theorem \ref{thm_bethe_external}\ref{thm_bethe_external_m} with Observation \ref{obs_phia}\ref{obs_phia_asymptotics_m} and Observation \ref{obs_DKLgt}\ref{obs_DKLgt_IIDNIS}.
For Part \ref{thm_infth_external}\ref{thm_infth_external_asymptotics_po}
we use Observation \ref{obs_DKLgt}\ref{obs_DKLgt_IIDNIS}, Observation \ref{obs_GRM_expebounds}\ref{obs_GRM_expeboundsPsiM} and \ref{obs_GRM_expebounds}\ref{obs_GRM_expeboundsZM} to obtain
\begin{align*}
0\le\expe[\delta'(\mR^*)]\le\frac{c}{n}+\expe\left[\bmone\{\mR^*>\mbu\}\frac{2\ln(\psibu)\mR^*}{n}\right]=\mclo\left(\frac{1}{n}+\epsm\right).
\end{align*}
Now, the assertion follows with Theorem \ref{thm_bethe_external}\ref{thm_bethe_external_general}, Observation \ref{obs_phia}\ref{obs_phia_asymptotics_po} and Corollary \ref{cor_dega}.
\end{proof}
Let $\GR_{\circ,m}$, $\GTSM[\circ,m](\sigmaIID)\in\domG$ be the graphs without external fields from Section \ref{random_factor_graphs}.
We use the shorthands $\GR=\GR_{\mR^*}$, $\GR_{\circ}=\GR_{\circ,\mR^*}$, $\GTSM(\sigmaIID)=\GTSM[\mR^*](\sigmaIID)$ and $\GTSM[\circ](\sigmaIID)=\GTSM[\circ,\mR^*](\sigmaIID)$.
Recall $Z_{\gamma^*}(G)$ from Section \ref{factor_graphs} and $\sigmaR_{\gamma^*,G}$ from Section \ref{information_theoretic_threshold} for $G\in\domG$.
Notice that the expectation in Theorem \ref{thm_infth_external}\ref{thm_infth_external_asymptotics_po} recovers
\begin{align*}
n\expe[\delta(\mR^*)]
=\DKL(\sigmaIID,\GTSM(\sigmaIID)\|\sigmaRG[,\GR],\GR|\mR^*)
=\DKL(\sigmaIID,\GTSM(\sigmaIID)\|\sigmaRG[,\GR],\GR).
\end{align*}
Let $r(\sigma,[G]^{\Gamma})=\ZG(G)/(\hat r(\sigma)\ZM)$ be the 
Radon-Nikodym derivative of $(\sigmaIID,\GTSM(\sigmaIID))$ with respect to $(\sigmaRG[,\GR],\GR)$ from the proof of Theorem \ref{thm_infth_external}. Further, let
\begin{align*}
r_\circ(\sigma,G)
=\frac{\gamma^{*\otimes n}(\sigma)\psiG[,G](\sigma)Z_{\gamma^*}(G)}{\expe[\psiG[,\GR_\circ(m)](\sigma)]\gamma^{*\otimes n}(\sigma)\psiG[,G](\sigma)}
=\frac{\gamma^{*\otimes n}(\sigma)\ZG([G]^{\Gamma})}{\psiM(\sigma)}
=r(\sigma,[G]^{\Gamma})
\end{align*}
be the Radon-Nikodym derivative of $(\sigmaIID,\GTSM[\circ](\sigmaIID))$ with respect to $(\sigmaG_{\gamma^*,\GR_\circ},\GR_{\circ})$.
Combining this with Observation \ref{obs_standard_graphs} completes the proof of Theorem \ref{thm_infth} since
\begin{align*}
\DKL(\sigmaIID,\GTSM(\sigmaIID)\|\sigmaRG[,\GR],\GR)
&=\expe\left[\ln\left(r\left(\sigmaIID,[\GTSM[\circ](\sigmaIID)]^{\Gamma}\right)\right)\right]
=\expe\left[\ln\left(r_\circ\left(\sigmaIID,\GTSM[\circ](\sigmaIID)\right)\right)\right]\\
&=\DKL(\sigmaIID,\GTSM[\circ](\sigmaIID)\|\sigmaR_{\gamma^*,\GR_\circ},\GR_\circ).
\end{align*}
\subsection{The Condensation Threshold}\label{main_proofs_condth}
In this section we establish Theorem \ref{thm_cond}.
First, we show Theorem \ref{thm_cond}\ref{thm_cond_r} in Section \ref{condth_rs}, followed by
the proof of Theorem \ref{thm_cond}\ref{thm_cond_c} in Section \ref{condth_cond}.
\subsubsection{The Replica Symmetric Regime}\label{condth_rs}
Recall that $(\lawpsi,\gamma^*,\degae)\in\mfkPr$ means that $\bethebu(\degae)=\phia(\degae)$.
\begin{lemma}\label{lemma_condth_rs}
Assume that $\bethebu(\degae)=\phia(\degae)$ and let $\phi(m)=\expe[\phiG(\GR)]$.
\begin{alphaenumerate}
\item\label{lemma_condth_rs_m}
We have $\phi(m)=\phia(\degae)+\mclo(n^{-\rho/2})$ if $\degae=km/n$.
\item\label{lemma_condth_rs_p}
We have $\expe[\phi(\mR^*)]=\phia(\degae)+\mclo(\deltam+\epsm+n^{-\rho/2})$, so
$\expe[\phi(\mR)]=\phia(\degae)+\mclo(n^{-\rho/2})$.
\end{alphaenumerate}
\end{lemma}
\begin{proof}
Using Theorem \ref{thm_bethe_external}\ref{thm_bethe_external_m}, Corollary \ref{cor_phiTSIIDNIS}\ref{cor_phiTSIIDNIS_m} and $\bethebu(\degae)=\phia(\degae)$ let $c_\mfkg\in\reals_{>0}$ be such that
$|\hat\phi(m)-\phia(\degae)|\le r$, where $\hat\phi(m)=\expe[\phiG(\GTSM(\sigmaNIS))]$ and $r=cn^{-\rho}$.
With $\hat c$ from Lemma \ref{lemma_concphiTSMIIDNIS} and $\hat{\mcle}=\{|\GTSM(\sigmaNIS)-\hat\phi(m)|<r\}$ we have $\prob[\lnot\hat{\mcle}]\le\hat c_2\exp(-\hat c_1n^{1-2\rho})$.
Further, with
\begin{align*}
n_{\circ,\mfkg}=\left(\frac{\ln(2\hat c_2)}{\hat c_1}\right)^{1/(1-2\rho)}
\end{align*}
we have $\prob[\hat{\mcle}]\ge 1/2$ for $n\ge n_\circ$
(for $n\le n_\circ$ we use $|\phi(m)-\phia(\degae)|\le\frac{1}{k}\ln(\psibu)\degabu n_\circ^{\rho/2}n^{-\rho/2}$).
Notice that $\hat{\mcle}_{\mrma}=\{|\GTSM(\sigmaNIS)-\phia(\degae)|<2r\}$ holds on $\hat{\mcle}$ by the triangle inequality,
so with
$\mcle[\mrma]=\{|\phiG(\GR)-\phia(\degae)|<2r\}$,
$\bm Z=\ZG(\GR)\bmone\mcle[\mrma]$ and $\overline Z=\expe[\bm Z]$ Observation \ref{obs_nishicond}\ref{obs_nishicond_GNIS} yields
\begin{align*}
\overline Z=\ZM\expe\left[\frac{\ZG(\GR)}{\ZM}\bmone\mcle[\mrma]\right]
=\ZM\prob[\hat{\mcle}_{\mrma}]\ge\frac{1}{2}\ZM.
\end{align*}
Further, we have
$\bm Z^2=\exp(2n\phiG(\GR))\bmone\mcle[\mrma]\le\exp(2n\phia(\degae)+2rn)=e^{2rn}\ZM^2$,
using $\mcle[\mrma]$ and the definition of $\phia$. Now, the Paley-Zygmund inequality yields
\begin{align*}
\prob\left[\bm Z\ge\frac{1}{2}\overline Z\right]
\ge\frac{\overline Z^2}{4\expe[\bm Z^2]}
\ge\frac{\ZM^2}{16\ZM^2}e^{-2rn}>0.
\end{align*}
Using $\bm Z\le\ZG(\GR)$ and $\overline Z\ge\frac{1}{2}\ZM$ gives $\ZG(\GR)\ge\frac{1}{4}\ZM$ on $\bm Z\ge\frac{1}{2}\overline Z$, so
\begin{align*}
P=\prob\left[\phiG(\GR)\ge\phia(\degae)-\frac{\ln(4)}{n}\right]
=\prob\left[\ZG(\GR)\ge\frac{1}{4}\ZM\right]\ge\frac{\ZM^2}{16\ZM^2}e^{-2rn}
\frac{1}{16}e^{-2rn}>0.
\end{align*}
Now, with $c^\circ$ from Lemma \ref{lemma_mcdiarmid} and $r_\circ=\sqrt{\frac{1}{c^\circ_1n}\ln(\frac{2c^\circ_2}{P})}$ we have
\begin{align*}
\prob\left[\phiG(\GR)\ge\phia(\degae)-\frac{\ln(4)}{n},|\phiG(\GR)-\phi(m)|<r_\circ\right]
\ge P-c^\circ_2e^{-c^\circ_1r_\circ^{2}n}=\frac{1}{2}P>0.
\end{align*}
On this event we have $\phi(m)\ge\phia(\degae)-r_\circ-\frac{\ln(4)}{n}$, which
establishes Part \ref{lemma_condth_rs}\ref{lemma_condth_rs_m} since $\phia(\degae)\ge\phi(m)$ by Observation \ref{obs_qfedinequ} and Lemma \ref{lemma_firstmom}\ref{lemma_firstmomZ}, and further
\begin{align*}
r_\circ+\frac{\ln(4)}{n}
=\sqrt{\frac{2}{c^\circ_1}r+\frac{\ln(32c^\circ_2)}{c^\circ_1n}}+\frac{\ln(4)}{n}
\le c'n^{-\rho/2},\,
c'=\sqrt{\frac{2c+\ln(32c^\circ_2)}{c^\circ_1}}+\ln(4).
\end{align*}
Observation \ref{obs_phi_lipschitz} and Lemma \ref{lemma_contphiM} give $\expe[\phi(\mR^*)]=\phi(\lfloor\degae n/k\rfloor)+\mclo(\deltam+\epsm+n^{-1})$, so Part \ref{lemma_condth_rs}\ref{lemma_condth_rs_m} completes the proof.
\end{proof}
Observation \ref{obs_standard_graphs} establishes Theorem \ref{thm_cond}\ref{thm_cond_r}.
\subsubsection{The Condensation Regime}\label{condth_cond}
Notice that as opposed to all other results, Theorem \ref{thm_cond}\ref{thm_cond_c} does not address the asymptotics, only the limits.
Hence, we do not discuss finite size approximations like Lemma \ref{lemma_condth_rs}\ref{lemma_condth_rs_m}. Let
\begin{align*}
\phiqbu(\degae)=\limsup_{n\rarr\infty}\expe[\phiG(\GR_{\mR^*})],\,
\phiqbl(\degae)=\liminf_{n\rarr\infty}\expe[\phiG(\GR_{\mR^*})].
\end{align*}
\begin{lemma}\label{lemma_condth_c}
There exists $c_\mfkg\in\reals_{>0}$ such that for $d\in[0,\degabu]$ we have
\begin{align*}
\phia(d)-\phiqbu(d)\ge c\sup_{d'\in[0,d]}(\bethebu(d')-\phiqbl(d'))^2.
\end{align*}
\end{lemma}
\begin{proof}
Observation \ref{obs_qfedinequ}, Observation \ref{obs_phia}\ref{obs_phia_asymptotics_po}, Theorem \ref{thm_bethe_external}\ref{thm_bethe_external_general} and Corollary \ref{cor_phiTSIIDNIS}\ref{cor_phiTSIIDNIS_p} yield
\begin{align*}
\bethebu(d)\ge\phia(d)\ge\phiqbu(d)\ge\phiqbl(d).
\end{align*}
For $d'\in[0,d]$ with $\delta^*(d')=0$, where $\delta^*(d)=\bethebu(d)-\phia(d)$,
we have $\phiqbl(d')=\phiqbu(d')=\phia(d')=\bethebu(d')$
by Lemma \ref{lemma_condth_rs}\ref{lemma_condth_rs_p}, and hence $\phia(d)-\phiqbu(d)\ge c(\bethebu(d')-\phiqbl(d'))^2$ for all $c\in\reals$.
Hence, assume that $\delta^*(d')>0$, let $m'_n=\lfloor d'n/k\rfloor$ and $m_n=\lfloor dn/k\rfloor$. Notice that $m'_n\le m_n\le\degabu n/k$.
Fix $\eps\in(0,1)$ with $\eps<\delta^*(d)/2$ and let $\delta'(n)=\phi^*(m')-\phi(m')$ with $\phi^*(m)=\expe[\phiG(\GTSM(\sigmaIID))]$ and $\phi(m)=\expe[\phiG(\GR)]$.
With $\tilde c_\mfkg$ satisfying both Theorem \ref{thm_bethe_external}\ref{thm_bethe_external_general} and Observation \ref{obs_phia}\ref{obs_phia_asymptotics_po} for any small $\deltam\ge k/n$, $\epsm\ge 0$ and using $n^{-\rho}$,
let $n_{\circ}(\eps)=(\eps/\tilde c)^{-\rho}$, so for $n\ge n_\circ(\eps)$ we have $|\phi^*(m')-\bethebu(d')|\le\eps$ and $|\bar\phi(m')-\phia(d')|\le\eps$, where $\bar\phi(m)=\frac{1}{n}\ln(\ZM)$.
This yields $\delta'(n)>0$ since
\begin{align*}
\phi^*(m')\ge\bethebu(d')-\eps>\phia(d')+\eps\ge\bar\phi(m')\ge\phi(m').
\end{align*}
With $c^\circ$ from Lemma \ref{lemma_mcdiarmid}, $c^*$ from Lemma \ref{lemma_concphiTSMIIDNIS},
$\hat c$ from Corollary \ref{cor_mutcont}\ref{cor_mutcont_rnbu} and the canonical coupling $(\GR,\GR')$ of $\GR_m$ and $\GR_{m'}$, meaning $\GR'=R(\GR)$ with $R([w]^\Gamma)=[w_{[m']}]^{\Gamma}$, we have
\begin{align*}
P(n)&=\prob[\phiG(\GR')\le\phi(m')+\eps\delta']\\
&\le c^\circ_2e^{-c^\circ_1\eps^2n}+
\prob\left[\phiG(\GR')\le\phi(m')+\eps\delta',|\phiG(\GR)-\phi(m)|<\eps\right]\\
&\le c^\circ_2e^{-c^\circ_1\eps^2n}+
\expe\left[\frac{\ZG(\GR)}{\exp(n(\phi(m)-\eps))}\bmone\{\phiG(\GR')\le\phi(m')+\eps\delta'\}\right]\\
&=c^\circ_2e^{-c^\circ_1\eps^2n}+
e^{n(\bar\phi(m)-\phi(m)+\eps)}\prob\left[\phiG(R(\GTSM(\sigmaNIS)))\le\phi(m')+\eps\delta'\right],
\end{align*}
where we used Observation \ref{obs_nishicond}\ref{obs_nishicond_GNIS}.
Observation \ref{obs_TSM_iid} yields $R(\GTSM(\sigmaIID))\dequal\GTSM[m'](\sigmaIID)$,
so Corollary \ref{cor_mutcont}\ref{cor_mutcont_rnbu} with Lemma \ref{lemma_concphiTSMIIDNIS} yields
\begin{align*}
P&\le c^\circ_2e^{-c^\circ_1\eps^2n}+\hat c
e^{n(\bar\phi(m)-\phi(m)+\eps)}\prob\left[\phiG(\GTSM[m'](\sigmaIID)))\le\phi(m')+\eps\delta'\right]\\
&\le c^\circ_2e^{-c^\circ_1\eps^2n}+\hat cc^*_2\exp(n\beta_\eps(n)),\,
\beta_\eps(n)=\bar\phi(m)-\phi(m)+\eps-c^*_1(1-\eps)^2\delta'^2,
\end{align*}
where we used that $\phi(m')+\eps\delta'=\phi^*(m')-(1-\eps)\delta'$.
For $\beta(\eps)=\liminf_{n\rarr\infty}\beta_\eps(n)$ taking the limits yields
$\beta(\eps)=\phia(d)-\phiqbu(d)-c^*_1(1-\eps)^2(\bethebu(d')-\phiqbl(d'))^2+\eps$.
On the other hand, Lemma \ref{lemma_mcdiarmid} yields $P\ge 1-c^\circ_2\exp(-n\beta'_\eps(n))$ with $\beta'_\eps(n)=c^\circ_1\eps^2\delta'^2$.
Since we assume $\delta^*(d)>0$, we have $\beta'(\eps)=\liminf_{n\rarr\infty}\beta'_\eps(n)=c^\circ_1\eps^2(\bethebu(d')-\phiqbu(d'))^2>0$.
This shows that $\lim_{n\rarr\infty}P(n)=1$, which in turn yields $\beta(\eps)\ge 0$.
Since $\beta$ is a quadratic polynomial in $\eps$, and in particular continuous, we have $\beta(0)\ge 0$, so the assertion holds with $c^*_1$.
\end{proof}
Observation \ref{obs_standard_graphs} establishes Theorem \ref{thm_cond}\ref{thm_cond_c}.
\subsection{The Mutual Information}\label{main_proofs_mutinf}
We turn to the proof of the last main result.
As before, we show that the mutual information for graphs with external fields converges to $\iota^*(d)=\frac{d}{k\ZFabu}\expe[\xlnx(\psiR(\sigmaR))]-\bethebu(d)$ from Theorem \ref{thm_mutinf}, and then obtain Theorem \ref{thm_mutinf} as a corollary.
\begin{theorem}\label{thm_mi_external}
Let $\iota(m)=\frac{1}{n}I(\sigmaIID,\GTSM(\sigmaIID))$.
\begin{alphaenumerate}
\item\label{thm_mi_external_asymptotics_m}
We have $\iota(m)=\iota^*(km/n)+\mclo(n^{-\rho})$ for $km/n\le\degabu$.
\item\label{thm_mi_external_asymptotics_po}
We have $\expe[\iota(\mR^*)]=\iota^*(\degae)+\mclo(\epsm+\deltam+n^{-\rho})$, so $\expe[\iota(\mR)]=\iota^*(\degae)+\mclo(n^{-\rho})$.
\end{alphaenumerate}
\end{theorem}
We prove Theorem \ref{thm_mi_external} in three parts.
For this purpose recall the notions from Section \ref{pinning_conditional_entropy} and $\sigmaIID_{\mrmg}$ from Section \ref{ground_truth_given_graph}.
First, we split $\iota$ into three contributions, the ground truth entropy $H(\gamma^*)$, the conditional cross entropy $\overline\eta(m)=\expe[\expe[H(\sigmaRG[,\GTSM(\sigmaIID)]^*\|\sigmaRG[,\GTSM(\sigmaIID)])|\GTSM(\sigmaIID)]]$ and the conditional relative entropy $\overline\delta(m)=\expe[\expe[\DKL(\sigmaRG[,\GTSM(\sigmaIID)]^*\|\sigmaRG[,\GTSM(\sigmaIID)])|\GTSM(\sigmaIID)]]$.
\begin{lemma}\label{lemma_mi_decomp}
We have $\iota(m)=H(\gamma^*)-\overline\eta(m)+\overline\delta(m)$.
\end{lemma}
The proof is presented in Section \ref{proof_lemma_mi_decomp}.
Then we determine the limit of $\overline\eta$.
\begin{lemma}\label{lemma_mi_ce}
Notice that the following holds.
\begin{alphaenumerate}
\item\label{lemma_mi_ce_m}
We have $\overline\eta(m)=H(\gamma^*)-\iota^*(km/n)+\mclo(n^{-\rho})$ for $km/n\le\degabu$.
\item\label{lemma_mi_ce_po}
We have $\expe[\overline\eta(\mR^*)]=H(\gamma^*)-\iota^*(\degae)+\mclo(\epsm+\deltam+n^{-\rho})$.
\end{alphaenumerate}
\end{lemma}
The proof is presented in Section \ref{proof_lemma_mi_ce}.
Finally, we complete the proof of Theorem \ref{thm_mi_external} in Section \ref{proof_thm_mutinf}, where we also establish Theorem \ref{thm_mutinf}.
\subsubsection{The Entropy Decomposition}\label{proof_lemma_mi_decomp}
Using $\GTSM=\GTSM(\sigmaIID)$, recall that $(\sigmaIID,\GTSM)\dequal(\sigmaRG[,\GTSM]^*,\GTSM)$ from Observation \ref{obs_condiid}\ref{obs_condiid_law}, so by the chain rule of the relative entropy we have
$n\iota(m)=\DKL(\sigmaRG[,\GTSM]^*\|\sigmaR|\GTSM)$, using $(\sigmaR,\GTSM)\dequal\sigmaIID\otimes\GTSM$.
The decomposition into the (conditional) cross entropy and the entropy gives
$n\iota(m)=H(\sigmaRG[,\GTSM]^*\|\sigmaR|\GTSM)-H(\sigmaRG[,\GTSM]^*|\GTSM)$.
Using linearity of the cross entropy in the first component and independence, we can take the expectation over $\GTSM$ to obtain $H(\sigmaRG[,\GTSM]^*\|\sigmaR|\GTSM)=H(\sigmaIID)$ since $\sigmaRG[,\GTSM]^*\dequal\sigmaIID$.
We split the latter entropy into the cross entropy and the relative entropy with respect to $\sigmaRG$, yielding
$H(\sigmaRG[,\GTSM]^*|\GTSM)=H(\sigmaRG[,\GTSM]^*\|\sigmaRG[,\GTSM]|\GTSM)-\DKL(\sigmaRG[,\GTSM]^*\|\sigmaRG[,\GTSM]|\GTSM)$, and hence $\iota(m)=H(\gamma^*)-\overline\eta(m)+\overline\delta(m)$.
\subsubsection{The Cross Entropy Contribution}\label{proof_lemma_mi_ce}
Recall that $H(\sigmaRG[,G]^*\|\sigmaRG[,G])=\expe[-\ln(\psiG[,G](\sigmaRG[,G]^*)/\ZG(G))]$, so Observation \ref{obs_condiid}\ref{obs_condiid_law} yields
\begin{align*}
\overline\eta(m)=\expe\left[\phiG(\GTSM)\right]-\expe\left[\frac{1}{n}\ln\left(\psiG[,\GTSM](\sigmaIID)\right)\right].
\end{align*}
Unlike the partition function $\ZG$, the weight $\psiG[,\GTSM(\sigma)](\sigma)\dequal\gamma^{*\otimes n}(\sigma)\prod_{a\in[m]}\psiR^*_a(\sigma_{\vR^*(a)})$ factorizes, where
$(\vR^*,\psiR^*)\dequal\wTSa[,\sigma][\otimes m]$, and hence $\overline\eta(m)=\expe[\phiG(\GTSM)]+H(\gamma^*)-\frac{m}{n}\expe[\ln(\psiR^*(\sigmaIID_{\vR^*}))]$, where 
$(\vR^*,\psiR^*)\dequal\wTSa[,\sigmaIID]$.
Resolving the Radon-Nikodym derivative yields
\begin{align*}
\overline\eta(m)
=\expe[\phiG(\GTSM)]+H(\gamma^*)-\frac{m}{n}\expe\left[\frac{\xlnx(\psiR(\sigmaIID_{\vR}))}{\ZFa(\gammaIID)}\right]
\end{align*}
with $(\sigmaIID,\vR,\psiR)\dequal\gamma^{*\otimes n}\otimes\unif([n]^k)\otimes\lawpsi$.
Hence, with Observation \ref{obs_fad}\ref{obs_fad_maxbound}, Observation \ref{obs_gtiid}\ref{obs_gtiid_prob} and Theorem \ref{thm_bethe_external}\ref{thm_bethe_external_m} we obtain Part \ref{lemma_mi_ce}\ref{lemma_mi_ce_m}, since $\sigmaIID_{\vR}\dequal\sigmaR$.
For Part \ref{lemma_mi_ce}\ref{lemma_mi_ce_po} we notice that $|\expe[\degaR^*]-\degae|\le\deltam+\degabu\epsm$, hence Observation \ref{obs_fad}\ref{obs_fad_maxbound}, Observation \ref{obs_gtiid}\ref{obs_gtiid_prob} and Theorem \ref{thm_bethe_external}\ref{thm_bethe_external_general}
complete the proof.
\subsubsection{Proof of Theorem \ref{thm_mutinf}}\label{proof_thm_mutinf}
Part \ref{thm_mi_external}\ref{thm_mi_external_asymptotics_m} is immediate from Lemma \ref{lemma_mi_decomp}, Lemma \ref{lemma_mi_ce}\ref{lemma_mi_ce_m} and Observation \ref{obs_DKLgt}\ref{obs_DKLgt_IIDCG}.
Part \ref{thm_mi_external}\ref{thm_mi_external_asymptotics_po} follows from Lemma \ref{lemma_mi_decomp}, Lemma \ref{lemma_mi_ce}\ref{lemma_mi_ce_po}, Observation \ref{obs_DKLgt}\ref{obs_DKLgt_IIDCG} and the expectation bound for the relative entropy and $\mR^*>\mbu$,
since the proof of Observation \ref{obs_DKLgt} reveals that the $(\sigmaRG[,G],\sigmaRG[,G]^*)$-derivative is $r_G(\sigma)=\psiG[,G](\sigma)/(\gamma^{*\otimes n}(\sigma)\ZG(G))$, thereby establishing $|\ln(r_G(\sigma))|\le 2m\ln(\psibu)$ and further $\DKL(\sigmaRG[,G]^*\|\sigmaRG[,G])\le 2m\ln(\psibu)$ for $G\in\domG$.
This completes the proof of Theorem \ref{thm_mi_external}.
Theorem \ref{thm_mutinf} follows with Observation \ref{obs_standard_graphs} and analogously to the derivation of Theorem \ref{thm_infth} from Theorem \ref{thm_infth_external}.
\section{Additional Discussion}\label{additional_discussion}
In Section \ref{constant_weights} we discuss constant weights and the special cases $q=1$, $k=0$.
In Section \ref{reweighting_relative_entropies} we formalize the discussion of the planted model in Section \ref{implications_extensions_related_work}.
In Section \ref{adddisc_lipschitz_bounds} we formalize the discussion of the modes of convergence in Section \ref{implications_extensions_related_work}.
Then, in Section \ref{unary_weights} we discuss the last remaining special case $k=1$.
\subsection{Constant Weights}\label{constant_weights}
We consider weights
$\mclc=\{\lawpsi\in\mclp(\domPsi):\max_\tau\psiR_{\lawpsi}(\tau)=\min_\tau\psiR_{\lawpsi}(\tau)\}$ with $\psiR_{\lawpsi}\dequal\lawpsi$.
This covers the special cases $q=1$ and $k=0$.
Recall $\mclp[-1]$ and $\mclp[1]$ from Section \ref{assumptions}.
\begin{lemma}\label{lemma_constant_weights}
Theorem \ref{thm_bethe_external}, Theorem \ref{thm_infth_external}, Lemma \ref{lemma_condth_rs}, Lemma \ref{lemma_condth_c} and Theorem \ref{thm_mi_external} also hold whenever $\lawpsi\in\mclc\setle\mclp[-1]\cap\mclp[1]$.
\end{lemma}
\begin{proof}
First, notice that $\mclc\setle\mclp[-1]\cap\mclp[1]$ holds by taking $\bm b_i=0$ and $\bm\Delta_i\equiv 0$ for $i\in\{-1,1\}$.
Next, for $\lawpsi\in\mclc$ we may assume without loss of generality that $\psiR_{\lawpsi}\equiv\bm c_\circ$ for some $\bm c_\circ\in[\psibl,\psibu]$.
Let $\bm c=\bm c_{\circ}^{\otimes\ints_{>0}}$, let $\bm c^*_\circ$ be given by the $(\bm c^*_\circ,\bm c_\circ)$-derivative $c\mapsto c/\overline c$ with $\overline c=\expe[\bm c_\circ]$, and let $\bm c^*\dequal\bm c_\circ^{*\otimes\ints_{>0}}$.
Then we have $\ZFabu=\overline c$, $\psiG[,\GR](\sigma)\dequal\gamma^{*\otimes n}(\sigma)\prod_{a\in[m]}\bm c_a$,
$\psiM(\sigma)=\gamma^*(\sigma)\overline c^m$ and
$\psiG[,\GTSM(\sigma)](\sigma')\dequal\gamma^{*\otimes n}(\sigma')\prod_{a\in[m]}\bm c^*_a$, $\sigma'\in[q]^n$, which gives
$\phi^*(m)=\expe\left[\phiG(\GTSM(\sigmaIID))\right]=\frac{m}{n}\expe[\ln(\bm c^*_\circ)]$.
On the other hand, due to normalization of $\gammaR$, $\gammaR_\circ$ we get 
\begin{align*}
\bethe_d\equiv\expe\left[\frac{\xlnx\left(\prod_{a\in[\dR]}\bm c_a\right)}{\overline c^{\dR}}\right]-\frac{d(k-1)}{k\overline c}\expe[\xlnx(\bm c_\circ)]
=\frac{d}{k\overline c}\expe[\xlnx(\bm c_\circ)]
\end{align*}
and thereby $\phi^*(m)=\bethebu(km/n)$, so Theorem \ref{thm_bethe_external} holds.
Notice that $\dcond\in\{0,\infty\}$ with $\dcond=\infty$ if and only if $\bm c_\circ=\overline c$ almost surely since $\phia(d)=d\ln(\overline c)/k$.
We further have $\lawG[,\GR]=\gamma^{*\otimes n}$ and thereby
\begin{align*}
\DKL(\sigmaIID,\GTSM(\sigmaIID)\|\sigmaRG[,\GR],\GR)
=\DKL(\GTSM(\sigma)\|\GR)=\expe\left[\ln\left(\frac{\prod_a\bm c^*_a}{\overline c^m}\right)\right]
=\bethebu\left(\frac{km}{n}\right)-\phia\left(\frac{km}{n}\right),
\end{align*}
which establishes Theorem \ref{thm_infth_external}.
Further, Theorem \ref{thm_mi_external} holds since  $(\sigmaIID,\GTSM(\sigmaIID))\dequal\sigmaIID\otimes\GTSM(\sigmaIID)$ and hence both sides vanish.
Notice that $\phi(m)=\expe[\phiG(\GR)]=\frac{m}{n}\expe[\ln(\bm c_\circ)]$, so Lemma \ref{lemma_condth_rs} and Lemma \ref{lemma_condth_c} hold for $\bm c_\circ=\overline c$ since then $\phi(m)=\phi^*(m)=\phia(km/n)$ for all $m$.
Otherwise, we have $\phiqbl(d)=\phiqbu(d)=\frac{d}{k}\expe[\ln(\bm c_\circ)]$ and hence
\begin{align*}
\delta(d)&=\phia(d)-\phiqbu(d)=\frac{d}{k}(\ln(\overline c)-\expe[\ln(\bm c_\circ)]),\\
\delta^*(d)&=\bethebu(d)-\phiqbl(d)=\frac{d}{k}(\expe[\ln(\bm c_\circ^*)]-\expe[\ln(\bm c_\circ)]).
\end{align*}
This yields $\delta(d)=r(d)\delta^*(d)^2$ with $r(d)=\delta(d)/\delta^*(d)^2\ge k\rho/\degabu$ and
\begin{align*}
\rho=\frac{\ln(\overline c)-\expe[\ln(\bm c_\circ)]}{(\expe[\ln(\bm c_\circ^*)]-\expe[\ln(\bm c_\circ)])^2}
=\frac{\DKL(\bm c_\circ\|\bm c^*_\circ)}{(\DKL(\bm c^*_\circ\|\bm c_\circ)+\DKL(\bm c_\circ\|\bm c^*_\circ))^2}.
\end{align*}
We follow \cite{sason2016} to bound
$\DKL(\bm c^*_\circ\|\bm c_\circ)$ in terms of 
$\DKL(\bm c_\circ\|\bm c^*_\circ)$ and
Let $f(t)=\xlnx(t)-(t-1)$, $g(t)=(t-1)-\ln(t)$.
Notice that $\DKL(\bm c^*_\circ\|\bm c_\circ)=\expe[f(r(\bm c_\circ))]$ and
$\DKL(\bm c_\circ\|\bm c^*_\circ)=\expe[g(r(\bm c_\circ))]$, where $r:[\psibl,\psibu]\rarr[\psibl^2,\psibu^2]$, $c\mapsto c/\overline c$.
Both $f$ and $g$ have their global minimum $0$ at $t=1$, so
$f''(t)=1/t=tg''(t)\le\psibu^2g''(t)$ for $t\in[\psibl^2,\psibu^2]$ yields
$\DKL(\bm c^*_\circ\|\bm c_\circ)\le\psibu^2\DKL(\bm c_\circ\|\bm c^*_\circ)$, and hence
\begin{align*}
\rho\ge\frac{1}{(\psibu^2+1)^2\DKL(\bm c_\circ\|\bm c^*_\circ)}
\ge\frac{1}{(\psibu^2+1)^22\ln(\psibu)}.
\end{align*}
\end{proof}
Lemma \ref{lemma_constant_weights} covers the case $q=1$ since then $[q]^k=\{(1)_h\}$ and hence $\mclc=\mclp(\domPsi)$. Clearly, the main results do \emph{not} hold for $k=0$, e.g.~since $k$ appears in the denominator of $\bethe$ and the Poisson parameter of $\mR$.
However, using the embedding $f:[\psibl,\psibu]^{k}\rarr[\psibl,\psibu]^{k+1}$ given by $\psi'=f(\psi)$ with $\psi'(\tau)=\psi(\tau_{[k]})$ for $\tau\in[q]^{k+1}$, we have $k\ge 1$ without loss of generality.
\subsection{Reweighting and Relative Entropies}\label{reweighting_relative_entropies}
In this section we build some context for $\nablaI$ from Section \ref{assumptions} and $\bethe$ from Section \ref{bethe_main}.
As opposed to the proofs, for the theory in this section we exclusively consider the restrictions to $\mclp[*][2]([q])$ with $\ZFa(\gamma^*)=\ZFabu$ for $\gamma^*\in\mclp([q])$, i.e.~we require $\gamma^*$ to be a maximizer of $\ZFa$.

Let $r_{\mrmf}:\domPsi\times\mclp([q])^k\rarr\reals_{>0}$, $(\psi,\gamma)\mapsto\ZF(\psi,\gamma)/\ZFabu$.
Further, for $\sigma\in[q]$ and $\pi\in\mclp[*][2]([q])$ with $\mclr=\{(\psi,h,\gamma):\psi\in\domPsi,h\in[k],\gamma\in\mclp([q])^{[k]\setminus\{h\}}\}$ let
\begin{align*}
r_{\mrmv,\sigma}:\mclr\rarr\reals_{>0},\,
(\psi,h,\gamma)\mapsto\frac{1}{\ZFabu}\sum_{\tau}\bmone\{\tau_h=\sigma\}\psi(\tau)\prod_{h'\neq h}\gamma_{h'}(\tau_{h'}).
\end{align*}
For $\pi\in\mclp[*][2]([q])^2$ let $\bm x_i\dequal\lawpsi\otimes\pi_i^{\otimes k}$, $i\in[2]$, further $\hR\dequal\unif([k])$ and for $h\in[k]$ let $\bm x_{3,h}\dequal\lawpsi\otimes\bigotimes_{h'\in[k]}\pi_{3,h'}$ with $\pi_{3,h}=\pi_1$ and $\pi_{3,h'}=\pi_2$ for $h'\in[k]\setminus\{h\}$.
Let $\bm x_1^*$, $\bm x_2^*$, $\bm x^*_{3,h}$ be given by the Radon-Nikodym derivative $r_\mrmf$ with respect to $\bm x_1$, $\bm x_2$, $\bm x_{3,h}$ respectively, and
\begin{align*}
\nablaI_2(\pi_1,\pi_2)
&=\ZFabu(\expe[\ln(\ZF(\bm x^*_1))]+(k-1)\expe[\ln(\ZF(\bm x^*_2))]-k\expe[\ln(\ZF(\bm x^*_{3,\hR}))]),\\
\nablaI_3(\pi_1,\pi_2)
&=\ZFabu(\DKL(\bm x^*_1\|\bm x_1)+(k-1)\DKL(\bm x^*_2\|\bm x_2)-k\DKL(\bm x^*_{3,\hR}\|\bm x_{3,\hR}|\hR)).
\end{align*}
For $\pi\in\mclp[*][2]([q])$ let $\bm x_\mrmf\dequal\lawpsi\otimes\pi^{\otimes k}$,
$\bm x_{\mrmv,\circ}=(\psiR,\hR,\gammaR_{[k]\setminus\{\bm h\}})$, where $(\psiR,\hR,\gammaR)\dequal\lawpsi\otimes\unif([k])\otimes\pi^{\otimes k}$, let $\bm x^*_{\mrmf}$ be given by the Radon-Nikodym derivative $r_{\mrmf}$, and let $\bm x^*_{\mrmv,\circ,\sigma}$ be given by the Radon-Nikodym derivative $r_{\mrmv,\sigma}$ for $\sigma\in[q]$.
Further, let $(\dR,\bm x_{\mrmv})\dequal\Po(d)\otimes\bm x_{\mrmv,\circ}^{\otimes\ints_{>0}}$,
$\bm x^*_{\mrmv,\sigma}\dequal\bm x_{\mrmv,\circ,\sigma}^{*\otimes\ints_{>0}}$ and $\sigmaIID\dequal\gamma^*$ with $(\dR,\sigmaIID,\bm x^*_{\mrmv,\sigmaIID})\dequal\dR\otimes(\sigmaIID,\bm x^*_{\mrmv,\sigmaIID})$.
Finally, let $\bm X_{\mrmv}=\bm x_{\mrmv,[\dR]}$, $\bm X^*_{\mrmv}=\bm x^*_{\mrmv,\sigmaIID,[\dR]}$,
$Z_{\mrmv}(\psi_{[d]},h_{[d]},(\gamma_{a,h'})_{a\in[d],h'\neq h(a)})=Z_{\mrmv}(d,\psi,h,\gamma)$ and
\begin{align*}
\bethe_{2,d}(\pi)&=\expe\left[\ln\left(\ZV(\bm X^*_{\mrmv})\right)\right]-\frac{d(k-1)}{k}\expe[\left[\ln\left(\ZF(\bm x^*_{\mrmf})\right)\right],\\
\bethe_{3,d}(\pi)&=\phia(d)+\DKL(\bm X^*_{\mrmv}\|\bm X_{\mrmv})-\frac{d(k-1)}{k}\DKL(\bm x^*_{\mrmf}\|\bm x_{\mrmf}).
\end{align*}
\begin{lemma}\label{lemma_reweighting}
We have $\nablaI=\nablaI_2=\nablaI_3$ and $\bethe=\bethe_2=\bethe_3$.
\end{lemma}
\begin{proof}
For $\pi\in\mclp[*][2]([q])^k$ and $(\psiR,\gammaR)\dequal\lawpsi\otimes\bigotimes_h\pi_h$ we have $\expe[\ZF(\psiR,\gammaR)]=\ZFa(\gamma^*)=\ZFabu$, which shows that
$\bm x^*_1$, $\bm x^*_2$, $\bm x^*_{3,h}$ for $\nablaI$ and $\bm x^*_{\mrmf}$ for $\bethe$ are well-defined.
Let $\mclc=\gamma^{*-1}(\reals_{>0})$ be the support of $\gamma^*$ and $\mu=\lawYgC[,\gamma^*]$ from Section \ref{factor_assignment_distribution}.
For $|\mclc|=1$ we have $\mu|_*=\gamma^*$ since both are necessarily one-point masses on the only element of $\mclc$, otherwise we have $\mu|_*=\gamma^*$ by Observation \ref{obs_fad}\ref{obs_fad_amstar} (since $\gamma^*$ is a fully supported stationary point of $\ZFa$ on $\mclp(\mclc)$).
Hence, for $\pi\in\mclp[*][2]([q])^k$ and $\sigma\in[q]$ with $(\psiR,\hR,\gammaR)\dequal\lawpsi\otimes\unif([k])\otimes\bigotimes_h\pi_h$ we have
\begin{align*}
\expe\left[r_{\mrmv,\sigma}\left(\psiR,\hR,\gammaR_{[k]\setminus\{\bm h\}}\right)\right]
=\frac{1}{\ZFabu}\sum_h\frac{1}{k}\sum_\tau\bmone\{\tau_h=\sigma\}\psiae(\tau)\prod_{h'\neq h}\gamma^*(\tau_{h'})=\frac{\mu|_*(\sigma)}{\gamma^*(\sigma)}=1.
\end{align*}
This shows that $\bm X^*_{\mrmv}$ is well-defined, and hence the assertion clearly holds.
\end{proof}
\subsection{Lipschitz Continuity and Boundedness}\label{adddisc_lipschitz_bounds}
In this section we stress the relevant properties that allow to extend the main results to $\mR^*$ and the equivalence of various modes of convergence.

For any $\ThetaP$ and $\tI$ let
$\bar\phi_n(m)=\expe[\phiG(\GR_{m,\mIR,\setPR})]$,
$\phi^*_{\sigma,n}(m)=\expe[\phiG(\GTSM[m,\mIR,\setPR](\sigma))]$,
$\phi^\star_{\sigma,\tau,n}(m)=\expe[\phiG(\GTSYM[m,\mIR,\setPR](\sigma,\tau))]$,
further $\bar\phi^*_n(m)=\expe[\phiG(\GTSM[m,\mIR,\setPR](\sigmaIID))]$ and
$\hat\phi^*_n(m)=\expe[\phiG(\GTSM[m,\mIR,\setPR](\sigmaNIS_m))]$.
First, we recall the properties for the free entropies.
\begin{lemma}\label{lemma_phiIIDNIScont}
Notice that the following holds.
\begin{alphaenumerate}
\item\label{lemma_phiIIDNIScont_bound}
There exists $c_\mfkg\in\reals_{>0}$ such that
$|\bar\phi(m)|\le c(\frac{km}{n}+(1-\tI)\degae+\frac{\ThetaP}{n})$.
The same holds for $\bar\phi$ replaced by
$\phi^*_{\sigma},\phi^\star_{\sigma,\tau},
\bar\phi^*_{\sigma},\hat\phi^*_{\sigma}$.
\item\label{lemma_phiIIDNIScont_lipschitz}
There exists $L_\mfkg\in\reals_{>0}$ such that
$|\bar\phi(m_1)-\bar\phi(m_2)|\le L\left|\frac{km_1}{n}-\frac{km_2}{n}\right|$
for $m\in\ints_{\ge 0}^2$.
The same holds for $\bar\phi$ replaced by
$\phi^*_{\sigma},\phi^\star_{\sigma,\tau},
\bar\phi^*_{\sigma}$.
This also holds for $\phi$ replaced by $\hat\phi^*_{\sigma}$ if $m\le\mbu$.
\end{alphaenumerate}
\end{lemma}
\begin{proof}
Part \ref{lemma_phiIIDNIScont}\ref{lemma_phiIIDNIScont_bound} follows from Observation \ref{obs_phi_lipschitz} and Observation \ref{obs_pin_basic}.
For Part \ref{lemma_phiIIDNIScont}\ref{lemma_phiIIDNIScont_lipschitz}
assume that $m_1\le m_2$ and let $\GR^\circ_m$ be any of $\GR_{m,\mIR,\setPR}$, $\GTSM[m,\mIR,\setPR](\sigma)$ or $\GTSYM[m,\mIR,\setPR](\sigma,\tau)$.
Under the canonical coupling (using Observation \ref{obs_TSM_iid})
we obtain $\GR^\circ(m_2)$ from $\GR^\circ(m_1)$ given $\GR^\circ(m_1)$ by adding $m_2-m_1$ factors with pairs drawn i.i.d.~from the underlying wires-weight pair distribution,
then with $c$ from Observation \ref{obs_phi_lipschitz} we have 
\begin{align*}
\left|\expe[\phiG(\GR^\circ_{m_2})]-\expe[\phiG(\GR^\circ_{m_1})]\right|
\le L\left|m_2-m_1\right|,\,
L=kc.
\end{align*}
The result for $\bar\phi^*$ now follows from $\bar\phi^*(m)=\expe[\phi^*_{\sigmaIID}(m)]$ and Jensen's inequality.
For $\hat\phi^*$ we first have to couple the ground truths using $c'$ from Observation \ref{obs_niscoupling}\ref{obs_niscoupling_tv} and the coupling lemma \ref{obs_tv}\ref{obs_tv_coupling}.
Hence, assume that $m_2=m_1+1$.
On the event that they coincide, the coupling from above applies. Otherwise, the left hand side is still at most $\frac{c''}{n}(2\mbu+2\degabu n+\ThetaP)$ with $c''$ from Observation \ref{obs_phi_lipschitz}, obtained by taking expectations. Recall that $\ThetaP\le n$, so with $c'''=c''(\frac{4\degabu}{k}+2\degabu+1)$ we have
$|\hat\phi(m_1)-\hat\phi(m_2)|\le\frac{\ln(\psibu)}{n}+\frac{c'c'''}{n}$.
Now, the assertion follows by the triangle inequality.
\end{proof}
For the remainder we restrict to $\ThetaP=0$, $\tI=1$, $\setP=\emptyset$ and $\mI\equiv 0$.
On the finite size side let $\degae_n(m)=km/n$.
Lemma \ref{lemma_phiIIDNIScont} suggests that e.g.~$|\bar\phi(m)|\le c\degae(m)$ and
$|\bar\phi(m_1)-\bar\phi(m_2)|\le L|\degae(m_1)-\degae(m_2)|$.
Next, we show that in general under these two properties convergence in probability, convergence of the expectation and pointwise convergence with respect to $m^\circ_n=\lfloor\degae n/k\rfloor$ coincide for $\mR^*$.
Let $\mclf[\circ]=\reals^{\ints_{>0}\times\ints_{\ge 0}}=\{f:\ints_{>0}\times\ints_{\ge 0}\rarr\reals\}$, and
for $c\in\reals_{>0}$ let $\mclf[c]=\mclf[\mrmb,c]\cap\mclf[\mrml,c]=\{f\in\mclf[\mrml,c]:\forall n\,f_n(0)=0\}$ with
\begin{align*}
\mclf[\mrmb,c]
&=\left\{f\in\mclf[\circ]:
\forall n\forall m\,
|f(n,m)|\le c\degae_n(m)\right\},\\
\mclf[\mrml,c]
&=\left\{f\in\mclf[\circ]:
\forall n\forall m\le\mbu\,
|f(n,m_1)-f(n,m_2)|\le c|\degae_n(m_1)-\degae_n(m_2)|\right\}.
\end{align*}
\begin{lemma}\label{lemma_convergenceequivalence}
For $c_\mfkg\in\reals_{>0}$, $f\in\mclf[c]$, $f^*:\reals_{\ge 0}\rarr\reals$ and $g_{\mfkg}:\ints_{>0}\rarr\reals_{\ge 0}$ with $g(n)=o(1)$ the following statements are equivalent.
Let $\Delta_\mfkg(n)=g(n)+\deltam(n)+\epsm(n)+n^{-1}$.
\begin{alphaenumerate}
\item\label{lemma_convergenceequivalence_pt}
We have $|f_n(m^\circ)-f^*(\degae)|=\mclo(\Delta(n))$.
\item\label{lemma_convergenceequivalence_prob}
There exists $C_\mfkg\in\reals_{>0}$ such that $\prob[|f_n(\mR^*)-f^*(\degae)|>C\Delta(n)]\le\epsm(n)$.
\item\label{lemma_convergenceequivalence_expe}
We have $|\expe[f_n(\mR^*)]-f^*(\degae)|=\mclo(\Delta(n))$.
\end{alphaenumerate}
\end{lemma}
\begin{proof}
Using $E(n)=\expe[f_n(\mR^*)]$ we have
\begin{align*}
E(n)&\le\expe[\bmone\{|\degaR^*-\degae|\le\deltam\}f(\mR^*)]+c\epsm
\le\expe[\bmone\{|\degaR^*-\degae|\le\deltam\}f(m^\circ)]+c\deltam+c\epsm\\
&\le f(m^\circ)+c\degae(m^\circ)\epsm+c\deltam+c\epsm
\le f(m^\circ)+2c\degabu\Delta
\end{align*}
and analogously $E(n)\ge f(m^\circ)-2c\degabu\Delta$, which shows that
the statements \ref{lemma_convergenceequivalence}\ref{lemma_convergenceequivalence_pt} and \ref{lemma_convergenceequivalence}\ref{lemma_convergenceequivalence_expe} are equivalent.
Now, assume that \ref{lemma_convergenceequivalence}\ref{lemma_convergenceequivalence_pt} holds and let $C'_\mfkg$ be such that $|f(m^\circ)-f^*(\degae)|\le C'\Delta$.
By the triangle inequality we have $|f(m)-f^*(\degae)|\le c'|\degae(m)-\degae(m^\circ)|+C'\Delta\le(ck+C')\Delta$ if $|\degae(m)-\degae|\le\deltam$, since then $|\degae(m)-\degae(m^\circ)|\le\deltam+\frac{k}{n}\le k\Delta$ by the triangle inequality, so 
\ref{lemma_convergenceequivalence}\ref{lemma_convergenceequivalence_prob} holds with $C= ck+C'$.
Conversely, let $C'$ be the constant from \ref{lemma_convergenceequivalence}\ref{lemma_convergenceequivalence_prob} and $n_{\circ,\mfkg}$ such that $\epsm<1/2$ for all $n\ge n_\circ$. In this case we have
\begin{align*}
\prob[|f_n(\mR^*)-f^*(\degae)|\le C'\Delta(n),|\degaR^*-\degae|\le\deltam]\ge 1-2\epsm>0,
\end{align*}
so there exists $m$ with $|f(m)-f^*(\degae)|\le C'\Delta(n)$ and $|\degae(m)-\degae|\le\deltam$.
As above, the triangle inequality and Lipschitz continuity give
$|f(m^\circ)-f^*(\degae)|\le C'\Delta+ck\Delta$.
By taking the limit this shows that $|f^*(\degae)|\le c\degae$, so for $n\le n_\circ$ we have
$|f(m^\circ)-f^*(\degae)|\le 2c\degae\le 2c\degabu n_\circ\Delta$ and thereby Part
\ref{lemma_convergenceequivalence}\ref{lemma_convergenceequivalence_pt} holds with $C=\max(C'+ck,2c\degabu n_\circ)$.
\end{proof}
For the sake of brevity, we only verify that Lemma \ref{lemma_convergenceequivalence} applies to the quantities appearing in the main results,
i.e.~$\bar\phi$, $\bar\phi^*$, $\bar\phi_{\mrma}(m)=\frac{1}{n}\ln(\ZM)$,
$\iota(m)=\frac{1}{n}I(\sigmaIID,\GTSM(\sigmaIID))$ and
$\delta(m)=\DKL(\sigmaIID,\GTSM(\sigmaIID)\|\sigmaRG[,\GR],\GR)$.
\begin{lemma}\label{lemma_keyboundlip}
There exists $c_\mfkg\in\reals_{>0}$ such that $\bar\phi,\bar\phi^*,\bar\phi_{\mrma},\iota,\delta\in\mclf[c]$.
\end{lemma}
\begin{proof}
The assertion for $\bar\phi,\bar\phi^*$ follows from Lemma \ref{lemma_phiIIDNIScont}, $\ThetaP=0$ and $\tI=1$.
The assertion for $\bar\phi_{\mrma}$ is Observation \ref{obs_phia}\ref{obs_phia_lipschitz}.
For $\iota$ we recall that $\iota=\frac{1}{n}\DKL(\sigmaRG[,\GTSM]\|\sigmaR|\GTSM)$ from Section \ref{proof_lemma_mi_decomp}
with $(\sigmaR,\GTSM)\dequal\sigmaIID\otimes\GTSM(\sigmaIID)$ and $\sigmaRG$ from Section \ref{ground_truth_given_graph} given by the $(\sigmaRG[,G],\sigmaR)$-derivative $r_{\mrms,G}$, so with Observation \ref{obs_condiid}\ref{obs_condiid_bound} we have $\psibl^{4m}\le r_{\mrms,G}\le\psibu^{4m}$ for $G\in\domG$ and thereby $|\iota|\le 4\ln(\psibu)m/n$.
Further, for $G\in\domG_m$ and an extension $G'\in\domG_{m+1}$ we have $\psibl^4r_{\mrms,G}\le r_{\mrms,G'}\le \psibu^4r_{\mrms,G}$ so $|\iota(m_1)-\iota(m_2)|\le 4k\ln(\psibu)|\degae(m_1)-\degae(m_2)|$ using the canonical coupling of $\GTSM[m_1](\sigmaIID)$ and $\GTSM[m_2](\sigmaIID)$ from the proof of Lemma \ref{lemma_phiIIDNIScont}.
For $\delta$ we recall the derivative $r(G)=\ZG(G)/(\hat r(\sigma)\ZM)=\gamma^{*\otimes n}(\sigma)\ZG(G)/\psiM(\sigma)$ from the proof of Theorem \ref{thm_infth_external}, and notice that $\psibl^2r(G)\le r(G')\le\psibu^2 r(G)$ for an extension $G'\in\domG_{m+1}$ of $G\in\domG_m$, so
$|\delta(m_1)-\delta(m_2)|\le 2k\ln(\psibu)|\degae(m_1)-\degae(m_2)|$, and $|\delta(m)|\le 2k\ln(\psibu)\degae(m)$.
\end{proof}
\begin{remark}\label{remark_convergenceprobability}
The combination of Lemma \ref{lemma_keyboundlip} and Lemma \ref{lemma_convergenceequivalence} yields equivalent definitions of the limiting quantities, including $\phiqbl$ and $\phiqbu$, and all main results for graphs with and without external fields for the modes of convergence in Lemma \ref{lemma_convergenceequivalence}.
\end{remark}
Now, we turn to the limiting quantities.
Clearly, due to uniform convergence both boundedness and Lipschitz continuity translate to the limit, however, only for $(\lawpsi,\gamma^*,d)\in\mfkp$.
Now, we consider the limiting quantities directly, only assuming $\ZFa(\gamma^*)=\ZFabu$.
Let
\begin{align*}
\mclf[c]=\{f:\reals_{\ge 0}\rarr\reals:\forall d |f(d_1)-f(d_2)|\le c|d_1-d_2|,f(0)=0\},
\end{align*}
further $\phia(d)=d\ln(\ZFabu)/k$, $\iota^*(d)$ from Section \ref{main_proofs_mutinf} and $\delta^*(d)=\bethebu(d)-\phia(d)$.
\begin{lemma}\label{lemma_keyboundlipinf}
There exists $c_\mfkg\in\reals_{>0}$ such that $\bethe_\pi,\bethebu,\phia,\iota^*,\delta^*\in\mclf[c]$.
\end{lemma}
\begin{proof}
With Lemma \ref{lemma_reweighting} we have 
\begin{align*}
|\bethe_d(\pi)|=|\bethe_{2,d}(\pi)|
\le\expe[\degaR\ln(\psibu)]+\frac{d(k-1)}{k}\ln(\psibu)
=d,\,
c=\frac{(2k-1)\ln(\psibu)}{k}.
\end{align*}
For $d_2\ge d_1$ we use the canoncial coupling of $\degaR_i\dequal\Po(d_i)$, $i\in[2]$, to obtain
\begin{align*}
|\bethe_{2,d_1}(\pi)-\bethe_{2,d_2}(\pi)|\le(d_2-d_1)\ln(\psibu)+(d_2-d_1)\frac{k-1}{k}\ln(\psibu)=c|d_2-d_1|.
\end{align*}
This also yields $|\bethebu(d)|\le cd$ and $\bethebu(d_1)-\bethebu(d_2)|\le c|d_1-d_2|$, where the former is obvious and the latter follows by considering maximizing sequences $(\pi_{1,n})_n$, $(\pi_{2,n})_n$ to obtain
\begin{align*}
\bethebu(d_1)=\lim_{n\rarr\infty}\bethe_{d_1}(\pi_{1,n})\le\lim_{n\rarr\infty}\bethe_{d_2}(\pi_{1,n})+c|d_2-d_1|\le\bethebu(d_2)+c|d_2-d_1|
\end{align*}
and the analogous result by switching $1$ and $2$ in the above.
The result for $\phia$ is immediate, which directly implies the result for $\delta^*$.
The result for $\iota^*$ follows from the result for $\bethebu$ and the immediate result $d\expe[\xlnx(\psiR(\sigmaR))]/(k\ZFabu)\in\mclf[c]$ for $c=\ln(\psibu)/k$ using that $(\sigma,\psi)\mapsto\psi(\sigma)/\ZFabu$ is a Radon-Nikodym derivative for $(\sigmaR,\psiR)$ since $\expe[\psiR(\sigmaR)]=\ZFa(\gamma^*)=\ZFabu$.
\end{proof}
Specifically for the Bethe functional we also recall the Lipschitz continuity in $\pi$.
\begin{lemma}\label{lemma_bethelipschitzpi}
There exists $L_\mfkg\in\reals_{>0}$ such that $\bethe_d:\mclp[][2]([q])\rarr\reals$ is $L$-Lipschitz if $d\le\degabu$. Hence, there exists $\pi\in\mclp[*][2]([q])$ such that $\bethe_d(\pi)=\bethebu(d)$.
\end{lemma}
\begin{proof}
Lipschitz continuity follows from Lemma \ref{lemma_bethelipschitzvar} and Lemma \ref{lemma_bethelipschitzfactor}.
Recall from \cite{coja2018} that $\mclp[][2]([q])$ is a compact Polish space (Corollary 2.2.5, Theorem 2.2.7 and Proposition 2.2.8 in \cite{panaretos2020}),
notice that $\mclp[*][2]([q])\setle\mclp[][2]([q])$ is closed
and hence compact, so by the extreme value theorem the maximum is attained.
\end{proof}
\subsection{Unary Weights}\label{unary_weights}
In this section we discuss the last remaining special case, namely $k=1$ and $q\ge 2$, which we assume throughout this section.
First, we notice that $\ZFa$ is trivial.
\begin{observation}
We have $\psiae,\ZFa\equiv\ZFabu$.
\end{observation}
\begin{proof}
Notice that we have $\ZFa(\gamma)=\sum_\tau\gamma(\tau)\psiae(\tau)\le\|\psiae\|_\infty$ with equality if and only if $\gamma^{-1}(\reals_{>0})\setle\psiae^{-1}(\|\psiae\|_\infty)$, so $\psiae\equiv\|\psiae\|_\infty$ since $\ZFa(\gamma^*)=\ZFabu=\|\psiae\|_\infty$ and $\gamma^*\ge\psibl$.
\end{proof}
Now, we verify the main results.
\begin{lemma}
Theorem \ref{thm_bethe_external}, Theorem \ref{thm_infth_external}, Lemma \ref{lemma_condth_rs}, Lemma \ref{lemma_condth_c} and Theorem \ref{thm_mi_external} hold for $k=1$, $q\ge 2$ and $(\lawpsi,\gamma^*,\degae)\in\mfkP$.
\end{lemma}
\begin{proof}
Using Observation \ref{obs_poisson}\ref{obs_poisson_bin}
we may take $\GR_{\mR}=[\psiR]^{\Gamma}$ without loss of generality, where
$\psiR\dequal\bigotimes_{i\in[n]}\lawpsi^{\otimes\dR(i)}$ and $\dR\dequal\Po(\degae)^{\otimes n}$.
By Observation \ref{obs_TSM_iid}, the wires-weight pairs $\wTSa[,\sigma]=(\vTSa,\psiTSa)$ for the teacher-student model are obtained from the derivative $(v,\psi)\mapsto\psi(\sigma_v)/\ZFa(\gammaN[,\sigma])=\psi(\sigma_v)/\ZFabu$, in particular $\vTSa$ from the derivative $v\mapsto\psiae(\sigma_v)/\ZFabu=1$, so $\vTSa\dequal\unif([n])$, and $\psiTSa$ given $\vTSa=i$ from the derivative $\psi\mapsto\psi(\sigma_i)/\ZFabu$.
Hence, we have $\GTSM[\mR](\sigma)=[\psiTSM]^{\Gamma}$ without loss of generality, where
$\psiTSM\dequal\bigotimes_{i\in[n]}\psiTSa[,\sigma(i)][\otimes\dR(i)]$ and $\psiTSa[,\tau]$ is given by the $(\psiTSa[,\tau],\psiRa)$-derivative $\psi\mapsto\psi(\tau)/\ZFabu$.
Due to independence this yields
\begin{align*}
\bar\phi(d)=\expe\left[\phiG(\GR_{\mR})\right]
&=\expe\left[\ln\left(\sum_\tau\gamma^*(\tau)\prod_{a\in[\dR]}\psiR_a(\tau)\right)\right],\,
\dR\dequal\Po(\degae),\,
\psiR\dequal\lawpsi^{\otimes\ints_{>0}},\\
\expe\left[\phiG(\GTSM[\mR](\sigmaIID))\right]
&=\expe\left[\ZFabu^{-\dR}\xlnx\left(\sum_\tau\gamma^*(\tau)\prod_{a\in[\dR]}\psiR_a(\tau)\right)\right]=\bethe_{\degae}(\pi)=\bethebu(\degae).
\end{align*}
Further, notice that $\expe[\frac{1}{n}\ln(\ZM[,\mR])]=\degae\ln(\ZFabu)=\phia(\degae)$ since $\psiM[,m](\sigma)=\gamma^{*\otimes n}(\sigma)\ZFabu^m$, which also gives $\sigmaNIS\dequal\sigmaIID$.
As in Observation \ref{obs_nishicond} this yields $(\sigmaIID,\GTSM(\sigmaIID))\dequal(\sigmaRG[,\GTSM(\sigmaIID)],\GTSM(\sigmaIID))$, further the $(\GTSM(\sigmaIID),\GR)$-derivative $G\mapsto\ZG(G)/\ZM(G)$, and using the chain rule of the relative entropy thereby
\begin{align*}
\delta^*(m)=\frac{1}{n}\DKL\left(\sigmaIID,\GTSM(\sigmaIID)\middle\|\sigmaRG[,\GR],\GR\right)
=\expe[\phiG(\GTSM(\sigmaIID)]-\frac{1}{n}\ln(\ZM).
\end{align*}
This shows that $\expe[\delta^*(\mR)]=\bethebu(\degae)-\phia(\degae)$.
Using the above, the derivative for the mutual information is $(\sigma,G)\mapsto\frac{\psiG[,G](\sigma)\ZM}{\psiM(\sigma)\ZG(G)}$
and hence
\begin{align*}
\frac{1}{n}I(\sigmaIID,\GTSM[\mR](\sigmaIID))
&=\frac{1}{n}\expe\left[\ln\left(\frac{\psiG[,\GTSM(\sigmaIID)](\sigmaIID)}{\gamma^{*\otimes n}(\sigmaIID)}\right)\right]
-\expe[\phiG(\GTSM[\mR](\sigmaIID))]\\
&=\expe\left[\sum_\tau\gamma^*(\tau)\prod_{a\in[\dR]}\frac{\psiR_a(\tau)}{\ZFabu}\ln\left(\prod_{a\in[\dR]}\psiR_a(\tau)\right)\right]-\expe[\phiG(\GTSM[\mR](\sigmaIID))]\\
&=\frac{\degae}{k\ZFabu}\expe[\xlnx(\psiRa(\sigmaR))]-\bethebu(\degae),
\end{align*}
with $\sigmaR\dequal\gamma^*$.
Using the discussion in Section \ref{reweighting_relative_entropies}, let $\psiR^*_\tau\dequal\psiR_{\circ,\tau}^{*\otimes\ints_{>0}}$, where $\psiR^*_{\circ,\tau}$ is given by the $(\psiR^*_{\circ,\tau},\psiRa)$-derivative $\psi\mapsto\psi(\tau)/\ZFabu$ with $\psiRa\dequal\lawpsi$. Let $\bm X^*=\psiR^*_{\tauR,[\dR]}$ with $\tauR\dequal\gamma^*$ and $\bm X=\psiR_{[\dR]}$ with $\psiR\dequal\psiRa^{\ints_{>0}}$, then we have $\bethebu(d)=\phia(d)+\DKL(\bm X^*\|\bm X)$ and analogously $\bar\phi(d)=\phia(d)-\DKL(\bm X\|\bm X^*)$.
Now, Lemma \ref{lemma_keyboundlip} and Lemma \ref{lemma_convergenceequivalence} establish all results but Lemma \ref{lemma_condth_c}.
For simplicity, we notice that the proof of Lemma \ref{lemma_condth_c} including the underlying concentration results and using $\sigmaIID\dequal\sigmaNIS$ from above also holds for $k=1$.
\end{proof}
\subsection{External Fields}\label{external_fields}
In this section we follow up on the discussion of graphs with external fields in Section \ref{implications_extensions_related_work}.
For $\eta:[q]\rarr\reals_{>0}$ and $G=(v,\psi)\in\domG$ let $[G]^\Gamma_{\eta}=G'=(v'_a,\psi'_a)_{a\in\mcla}$ with $\mcla=[m]\dotcup[n]$ be given by $G'_{[m]}=G$, $G'_a=(a,\eta)$ for $a\in[n]$, i.e.~a graph with fixed external fields.
Let $\GR_{\mrme}=[\wR]^\Gamma_\eta$ with $\wR$ from Section \ref{random_decorated_graphs} and let $\GTSM[\mrme](\sigma)$ be given by the $(\GTSM[\mrme](\sigma),\GR_{\mrme})$-derivative $G\mapsto\psiG[,G](\sigma)/\expe[\psiG[,\GR_{\mrme}](\sigma)]$.
Let $\mfkP_{\mrme}=\{(\lawpsi,\gamma^*,d,c\gamma^*):(\lawpsi,\gamma^*,d)\in\mfkP,c\in\reals_{>0}\}$ be the parameters including external fields.
We reduce the general case to normalized external fields via $\phiG^\circ([G]^\Gamma_\eta)=\phiG([G]^\Gamma_{\eta})-\ln(\|\eta\|_1)$.
\begin{corollary}
Theorem \ref{thm_bethe_external}, Theorem \ref{thm_infth_external}, Lemma \ref{lemma_condth_rs}, Lemma \ref{lemma_condth_c} and Theorem \ref{thm_mi_external} hold for $\GR$, $\GTSM$, $\phiG$ replaced by $\GR_{\mrme}$, $\GTSM[\mrme]$, $\phiG^\circ$ and $(\lawpsi,\gamma^*,d,\eta)\in\mfkP_{\mrme}$.
\end{corollary}
\begin{proof}
We have $\|c\gamma^*\|_1=c$ and hence $\phiG^\circ([G]^\Gamma_{c\gamma^*})=\frac{1}{n}\ln(c^n\ZG([G]^\Gamma_{\gamma^*}))-\ln(c)=\phiG([G]^\Gamma_{\gamma^*})$.
Similarly, notice that $\psiG[,[G]^\Gamma_{c\gamma^*}](\sigma)=c^n\psiG[,[G]^\Gamma_{\gamma^*}](\sigma)$, so the Radon-Nikodym derivatives coincide in this sense and analogously to Observation \ref{obs_standard_graphs}, and thereby $\GTSM[\mrme](\sigma)\dequal[\wTSM(\sigma)]^\Gamma_{c\gamma^*}$.
The remainder is analogous to the translation of the results to graphs without external fields.
Notice that the results are uniform over $c\in\reals_{>0}$.
\end{proof}
\subsection{Simple Hypergraphs}\label{simple_hypergraphs}
In this section we illustrate how the results can be translated to similar models.
To be specific, we eliminate parallel edges and duplicate neighborhoods, starting with parallel edges only.
Let $k\ge 2$ and $n\ge k$.

With $\wRa=(\vRa,\psiRa)\dequal\unif([n]_k)\otimes\lawpsi$ let $\wR\dequal\wRa^{\otimes m}$, let $\GR=[\wR]^\Gamma_{\gamma^*}$ be the null model, and $\GTSM(\sigma)$ be the teacher-student model given by $G\mapsto\psiG[,G](\sigma)/\expe[\psiG[,\GR](\sigma)]$.
\begin{corollary}\label{cor_paralleledges}
Theorem \ref{thm_bethe_external}, Theorem \ref{thm_infth_external}, Lemma \ref{lemma_condth_rs}, Lemma \ref{lemma_condth_c} and Theorem \ref{thm_mi_external} hold for $\GR$, $\GTSM$ as defined here.
\end{corollary}
\begin{proof}
With $\wR_{\mrmr,\circ}=(\vR_{\mrmr\circ},\psiR_{\mrmr\circ})\dequal\unif([n]^k)\otimes\lawpsi$ let $\wR_{\mrmr}\dequal\wR_{\mrmr\circ}^{\otimes m}$, let $\GR_{\mrmr}=[\wR_{\mrmr}]^\Gamma_{\gamma^*}$ and $\GTSM[\mrmr](\sigma)$ given by $G\mapsto\psiG[,G](\sigma)/\expe[\psiG[,\GR_{\mrmr}](\sigma)]$ be the decorated graphs from Section \ref{random_decorated_graphs}.
Let $\wR^*_{\mrmr\circ}(\sigma)$ given by $(v,\psi)\mapsto\psi(\sigma_v)/\ZFa(\gammaN[,\sigma])$ be the reweighted pair from Observation \ref{obs_TSM_iid}, and let $\wR^*_\circ(\sigma)$ be given by $(v,\psi)\mapsto\psi(\sigma_v)/z(\sigma)$ with $z(\sigma)=\expe[\psiae(\sigma_{\vR_\circ})]$.
With $\wR^*(\sigma)\dequal\wR_\circ^{*\otimes m}$ and $\wR^*_{\mrmr}(\sigma)\dequal\wR_{\mrmr\circ}^{*\otimes m}$ we have
$\expe[\psiG[,\GR](\sigma)]=\gamma^{*\otimes n}(\sigma)z(\sigma)^m$ and hence $\GR^*(\sigma)\dequal[\wR^*(\sigma)]^\Gamma$.
Further, since the $(\vR_\circ,\vR_{\mrmr\circ})$-derivative is $\bmone\{v\in[n]^k\}/P$ with $P=\prob[\vR_{\mrmr\circ}\in[n]_k]$, we have $\wR\dequal(\wR_{\mrmr}|\vR\in[n]_k^m)$ and $z(\sigma)=\expe[\psiR_{\mrmr\circ}(\sigma_{\vR_{\mrmr\circ}})|\vR_{\mrmr\circ}\in[n]_k]
=P^*(\sigma)\ZFa(\gammaN[,\sigma])/P$ with $P^*(\sigma)=\prob[\vR^*_{\mrmr\circ}(\sigma)\in[n]_k]$, which further yields $\wR^*(\sigma)\dequal(\wR^*_{\mrmr}(\sigma)|\vR^*(\sigma)\in[n]_k^m)$.

We couple $\wR$ with $\wR_{\mrmr}$, as well as $\wR^*(\sigma)$ with $\wR^*_{\mrmr}(\sigma)$ for $m\le\mbu$ as follows.
Recall that we have $r\le\psibu^2$ for the $(\wR^*_{\mrmr\circ},\wR_{\mrmr\circ})$-derivative $r$.
With $\wR^\circ=(\vR^\circ,\psiR^\circ)\dequal\wR^{*\otimes\ints_{>0}^2}_{\mrmr\circ}$ let $\wR'_{\mrmr}=(\wR^\circ_{a,1})_{a\in[m]}\dequal\wR^*_{\mrmr}$. For $a\in[m]$ let
$\bm b(a)=\min\{b\in\ints_{>0}:\vR^\circ_{a,b}\in[n]_k\}$
and $\wR'=(\wR^\circ_{a,\bm b(a)})_{a\in[m]}\dequal\wR^*$.
For $\bmcla=\{a\in[m]:\wR'(a)\neq\wR'_{\mrmr}(a)\}=\bm b(\ints_{>1})^{-1}$ we have
\begin{align*}
\prob[|\bmcla|\ge\sqrt{n}]\le\frac{\expe[|\bmcla|]}{\sqrt{n}}\le\frac{\mbu\prob[\vR_{\mrmr\circ}^*\not\in[n]_k]}{\sqrt{n}}\le\frac{c}{\sqrt{n}},\,
c=2\degabu\psibu^2k.
\end{align*}
Repeating the coupling with $\wR^*$ replaced by $\wR$ yields a coupling of $\wR$ and $\wR_{\mrmr}$ such that they differ on at least $\sqrt{n}$ factors with probability at most $c/\sqrt{n}$.
With $c$ from Observation \ref{obs_phi_lipschitz} we have
$|\phiG(G)|\le 2\degabu c/k$ since $m\le\mbu$ and $|\phiG(G)-\phiG(G')|\le 2c/n|\{a\in[m]:G_a\neq G'_a\}|$ for $G,G'\in\domG$, so
\begin{align*}
\left|\expe\left[\phiG(\GR^*(\sigmaIID))\right]-\expe\left[\phiG(\GR_{\mrmr}(\sigmaIID))\right]\right|
\le\frac{4\degabu c}{k}\prob[|\bmcla|\ge\sqrt{n}]+\frac{2c\sqrt{n}}{n}=\mclo(\sqrt{n}^{-1}),
\end{align*}
and analogously $|\expe[\phiG(\GR)]-\expe[\phiG(\GR_{\mrmr})]|=\mclo(\sqrt{n}^{-1})$.
Using $\gamma=\gammaN[,\sigma]$, notice that
\begin{align*}
\ZFa(\gammaN[,\sigma])
=\expe[\psiR_{\mrmr\circ}(\sigma_{\vR_{\mrmr\circ}})]
=Pz(\sigma)+\Delta,\,
\Delta=\expe[\bmone\{\vR_{\mrmr\circ}\not\in[n]_k\}\psiR_{\mrmr\circ}(\sigma_{\vR_{\mrmr\circ}})].
\end{align*}
Clearly, we have $\psibl\eps\le\Delta\le\psibu\eps$, $\eps=1-P$, and hence
$1-\eps\le\ZFa(\gamma)/z(\sigma)\le 1+\psibu^2\eps$.
With $c=\psibu^2 k^2$ and using $P=(n)_k/n^k$, $n\ge k$, where $(n)_k=\prod_{h=0}^{k-1}(n-h)$, we obtain
\begin{align*}
\frac{1}{1+\frac{c}{n}}\le\frac{\ZFa(\gamma)}{z(\sigma)}\le 1+\frac{c}{n}.
\end{align*}
With $\expe[\psiG[,\GR](\sigma)]=\gamma^{*\otimes n}(\sigma)z(\sigma)^m$, Observation \ref{obs_GRM_expebounds} and $C=\exp(2\degabu c/k)$ this yields
\begin{align*}
C^{-1}\le\left(1+\frac{c}{n}\right)^{-\mbu}
\le\frac{\expe[\psiG[,\GR](\sigma)]}{\expe[\psiG[,\GR_{\mrmr}](\sigma)]},\,
\frac{\expe[\ZG(\GR)]}{\expe[\ZG(\GR_{\mrmr})]}
\le\left(1+\frac{c}{n}\right)^{\mbu}\le C,
\end{align*}
and thereby $|\frac{1}{n}\ln(\expe[\ZG(\GR)])-\frac{1}{n}\ln(\expe[\ZG(\GR_{\mrmr})])|\le\ln(C)/n$.
Next, recall $\delta$ from Theorem \ref{thm_infth_external}, that
$\delta(m)=\expe[\phiG(\GTSM[\mrmr](\sigmaIID))]-\expe[\frac{m}{n}\ln(\ZFa(\gammaIID))]$ from the proof thereof, and let
\begin{align*}
\delta'(m)=\DKL\left(\sigmaIID,\GTSM(\sigmaIID)\middle\|\sigmaRG[,\GR],\GR\right)
=\expe[\phiG(\GTSM(\sigmaIID))]-\expe\left[\frac{m}{n}\ln\left(z(\sigmaIID)\right)\right].
\end{align*}
This yields $\delta(m)-\delta'(m)=\mclo(\sqrt{n}^{-1})$ using $m\le\mbu$.
Recall $\iota$ from Theorem \ref{thm_mi_external}, and from
Lemma \ref{lemma_mi_decomp}, Section \ref{ground_truth_given_graph} that
$\iota(m)=\frac{1}{n}\DKL(\sigmaRG[,\GTSM[\mrmr](\sigmaIID)]^*\|\sigmaR|\GTSM[\mrmr](\sigmaIID))
=\expe[\delta(\sigmaIID,\GTSM[\mrmr](\sigmaIID))]$, where
$\delta(\sigma,G)=\expe[\frac{1}{n}\ln(r_\sigma(G)/r^*(G))]$,
$r_\sigma(G)=\psiG[,G](\sigma)/\expe[\psiG[,\GR_{\mrmr}](\sigma)]$,
$r^*(G)=\expe[r_{\sigmaIID}(G)]$.

Let $\iota_\circ(m)=\frac{1}{n}I(\sigmaIID,\GTSM(\sigmaIID))
=\expe[\delta_\circ(\sigmaIID,\GTSM(\sigmaIID))]$,
$\delta_\circ(\sigma,G)=\expe[\frac{1}{n}\ln(r_{\circ,\sigma}(G)/r_\circ^*(G))]$, $r_{\circ,\sigma}(G)=\psiG[,G](\sigma)/\expe[\psiG[,\GR](\sigma)]$,
$r_\circ^*(G)=\expe[r_{\circ,\sigmaIID}(G)]$,
obtained analogously.
With the uniform bounds $|\delta|,|\delta_\circ|\le 4\ln(\psibu)\mbu/n=8\ln(\psibu)\degabu/k$ we have
\begin{align*}
|\iota(m)-\iota_\circ(m)|\le\expe[\bmone\{|\bmcla|<\sqrt{n}\}|\delta(\sigmaIID,[\wR'_{\mrmr}(\sigmaIID)]^\Gamma)-\delta_\circ(\sigmaIID,[\wR'(\sigmaIID)]^\Gamma)|]+\mclo(\sqrt{n}^{-1}).
\end{align*}
On the event $|\bmcla|<\sqrt{n}$ we have $\psibl^{2\sqrt{n}}\psiG[,\wR'](\sigma)\le\psiG[,\wR'_{\mrmr}](\sigma)\le\psibu^{2\sqrt{n}}\psiG[,\wR'](\sigma)$, and with the bounds above further $C^{-1}\psibl^{2\sqrt{n}}\le r_\sigma([\wR'_{\mrmr}]^\Gamma)/r_{\circ,\sigma}([\wR']^\Gamma)\le C\psibu^{2\sqrt{n}}$, which in turn yields the same bounds for $r^*$, hence
\begin{align*}
|\delta(\sigmaIID,[\wR'_{\mrmr}(\sigmaIID)]^\Gamma)-\delta_\circ(\sigmaIID,[\wR'(\sigmaIID)]^\Gamma)|\le\frac{1}{n}\ln\left(\left(C\psibu^{2\sqrt{n}}\right)^2\right)
=\mclo(\sqrt{n}^{-1})
\end{align*}
and thereby $|\iota(m)-\iota_\circ(m)|=\mclo(\sqrt{n}^{-1})$.
Lemma \ref{lemma_convergenceequivalence} with $\rho<1/4$ yields the assertion.
\end{proof}
Now, we also eliminate duplicate neighborhoods.
Hence, let $\wR=(\vR,\psiR)\dequal\unif(\mclv)\otimes\lawpsi^{\otimes m}$ with $\mclv=\{v\in[n]_k^m:\forall u\in[n]_k |v^{-1}(u)|\le 1\}$ denoting the injections (so $\mclv=([n]_k)_m$) for $m\le n!/(n-k)!$ and $n\ge k$.
Further, let $\GR=[\wR]^\Gamma_{\gamma^*}$ be the null model, and $\GTSM(\sigma)$ be the teacher-student model given by $G\mapsto\psiG[,G](\sigma)/\expe[\psiG[,\GR](\sigma)]$.
\begin{corollary}\label{cor_duplicatefactors}
Theorem \ref{thm_bethe_external}, Theorem \ref{thm_infth_external}, Lemma \ref{lemma_condth_rs}, Lemma \ref{lemma_condth_c} and Theorem \ref{thm_mi_external} hold for $\GR$, $\GTSM$ as defined here.
\end{corollary}
\begin{proof}
Let $\GR_{\mrmr}=[\wR_{\mrmr}]^\Gamma$, $\GTSM[\mrmr](\sigma)\dequal[\wTSM[\mrmr](\sigma)]^\Gamma$ be the graphs from Corollary \ref{cor_paralleledges} and let $\wR_{\mrmr}=(\vR_{\mrmr},\psiR_{\mrmr})$, $\wR^*_{\mrmr}(\sigma)=(\vR^*_{\mrmr}(\sigma),\psiR^*_{\mrmr}(\sigma))$, $\wR=(\vR,\psiR)$ and $\wR^*(\sigma)$ given by $w\mapsto\psiG[,w](\sigma)/\expe[\psiG[,\wR](\sigma)]$. Notice that the $(\vR,\vR_{\mrmr})$-derivative is $v\mapsto\bmone\{v\in\mclv\}/P$ with $P=\prob[\vR_{\mrmr}\in\mclv]$ and that $\GTSM(\sigma)\dequal[\wTSM(\sigma)]^\Gamma$.
This yields $\expe[\psiG[,\wR](\sigma)]=P^*(\sigma)\expe[\psiG[,\wR_{\mrmr}](\sigma)]/P$ with $P^*(\sigma)=\prob[\vR^*_{\mrmr}(\sigma)\in\mclv]$ and hence $\wR^*(\sigma)\dequal(\wR^*_{\mrmr}(\sigma)|\vR^*_{\mrmr}(\sigma)\in\mclv)$.
Recall $\wR_{\mrmr\circ}$, $\wR^*_{\mrmr\circ}$ with
$\wR_{\mrmr}\dequal\wR_{\mrmr\circ}^{\otimes m}$, $\wR^*_{\mrmr}(\sigma)\dequal\wR_{\mrmr\circ}^{*\otimes m}$ from the proof of Corollary \ref{cor_paralleledges}.
Let $\wR^\circ=(\vR^\circ,\psiR^\circ)\dequal\wR_{\mrmr\circ}^{*\otimes\ints_{>0}^2}$, $\wR'_{\mrmr}=(\vR'_{\mrmr},\psiR'_{\mrmr})=(\wR^\circ_{a,1})_{a\in[m]}\dequal\wR^*_{\mrmr}(\sigma)$, for $a\in[m]$ let
\begin{align*}
\bm b(a)=\min\left\{b\in\ints_{>0}:\forall a'\in[a-1]\,\vR^\circ_{a,b}\neq\vR^\circ_{a',\bm b(a')}\right\},
\end{align*}
and $\bm w'=(\vR',\psiR')=(\wR^\circ_{a,\bm b(a)})_{a\in[m]}\dequal\wR^*(\sigma)$.
Notice that $\wR$, $\wR_{\mrmr}$ can be coupled analogously.
Using $r\le\psibu^2$ for the $(\wR^*_{\mrmr\circ},\wR_{\mrmr\circ})$-derivative and the union bound yields
\begin{align*}
\prob[\bm b(a)>1]=\expe\left[\prob\left[\exists a'\in[a-1]\,\vR^*_{\mrmr\circ}=\vR'_{a'}\middle|(\vR'_{a'})_{a'\in[a-1]}\right]\right]\le\frac{(a-1)\psibu^2}{(n)_k}
\end{align*}
with $(n)_k=\prod_{h=0}^{k-1}(n-h)$.
This yields $\expe[|\bm b^{-1}(\ints_{>1})|]\le\frac{\psibu^2}{(n)_k}\binom{m}{2}$.
For $m\le\mbu$ we have
\begin{align*}
P&=\frac{((n)_k)_m}{(n)_k^m}=\prod_{a=0}^{m-1}\left(1-\frac{a}{(n)_k}\right)
=\exp\left(-\sum_{a=0}^{m-1}\left(1+\mclo\left(\frac{\mbu}{(n)_k}\right)\right)\frac{a}{(n)_k}\right)\\
&=\exp\left(-\left(1+\mclo\left(\frac{1}{n}\right)\right)\frac{1}{(n)_k}\binom{m}{2}\right)
=(1+\mclo(n^{-1}))P_\infty,\,
P_\infty=\exp\left(\frac{-1}{(n)_k}\binom{m}{2}\right).
\end{align*}
Now, we show that for all $m\le(n)_k$, using $p(a)=\frac{a-1}{(n)_k}$, $c=\psibu^2-1$ and induction, we have
\begin{align*}
\prod_{a=1}^m(1-P(a))\le\frac{\expe[\psiG[,\wR_{\mrmr}](\sigma)]}{\expe[\psiG[,\wR](\sigma)]}\le\prod_{a=1}^m(1+cP(a)).
\end{align*}
For $m=0$ we have $\expe[\psiG[,\wR_{\mrmr}](\sigma)]=\expe[\psiG[,\wR](\sigma)]=1$, which coincides with both bounds, so assume that the hypothesis holds for $m$.
Then we have
\begin{align*}
\expe[\psiG[,\wR_{\mrmr}(m+1)](\sigma)]
=\expe[\psiG[,\wR_{\mrmr}](\sigma)]z(\sigma)
\le\prod_{a=1}^m(1+cP(a))\expe[\psiG[,\wR](\sigma)z(\sigma)].
\end{align*}
Let $v\in[n]_k^m$ with $|v([m])|=m$, $a=m+1$, $\vR=\vR_{m+1}$, recall that $\vR_{a}|\vR_{[a-1]}=v$ is uniform on $[n]_k\setminus v([m])$,
as is $\vR_{\mrmr\circ}|\vR_{\mrmr\circ}\not\in v([m])$, so we have
$E(v)=\expe[\psiae(\sigma_{\vR(a)})|\vR_{[a-1]}=v]=\expe[\psiae(\sigma_{\vR_{\mrmr\circ}})|\vR_{\mrmr\circ}\not\in v([m])]$, and $P(a)=\frac{m}{(n)_k}=\prob[\vR_{\mrmr\circ}\in v([m])]$.
Hence, we have
\begin{align*}
z(\sigma)&=(1-P(a))E(v)+\expe[\bmone\{\vR_{\mrmr\circ}\in v([m])\}\psiae(\sigma_{\vR_{\mrmr\circ}})]
=(1+\delta(v))E(v),\\
\delta(v)&=\frac{\expe[\bmone\{\vR_{\mrmr\circ}\in v([m])\}\psiae(\sigma_{\vR_{\mrmr\circ}})]}{E(v)}-P(a)\le cP(a),
\end{align*}
and analogously $\delta(v)\ge-(1-\psibl^2)P(a)$. Hence, we have
\begin{align*}
\expe[\psiG[,\wR_{\mrmr}(m+1)](\sigma)]
&\le\prod_{a=1}^{m+1}(1+cP(a))\expe[\psiG[,\wR](\sigma)E(\vR_{[m]})]\\
&=\prod_{a=1}^{m+1}(1+cP(a))\expe\left[\prod_{a=1}^m\psiae(\sigma_{\vR(a)})\expe[\psiae(\sigma_{\vR(m+1)})|\vR_{[m]}]\right]\\
&=\prod_{a=1}^{m+1}(1+cP(a))\expe[\psiG[,\wR(m+1)](\sigma)].
\end{align*}
With $\expe[\psiG[,\wR_{\mrmr}](\sigma)]z(\sigma)\ge\prod_a(1-P(a))\expe[\psiG[,\wR](\sigma)z(\sigma)]$ and $\delta(v)\ge -P(m+1)$ the lower bound follows analogously.
Next, we have $\prod_a(1+cP(a))\le\exp(c\sum_aP(a))=\exp(c\binom{m}{2}/(n)_k)$, and using $m\le\mbu$, $k\ge 2$ further $P(a)\le\mbu/(n)_k=\mclo(1/n)$ and
\begin{align*}
\prod_a(1-P(a))
&=\exp\left(\sum_a\ln(1-P(a))\right)
=\exp\left(-(1+\mclo(n^{-1}))\sum_aP(a)\right)\\
&=(1+\mclo(n^{-1}))P_\infty.
\end{align*}
Since these bounds do not depend on $\sigma$ we have
\begin{align*}
(1+\mclo(n^{-1}))P_\infty
\le\frac{\expe[\psiG[,\wR_{\mrmr}](\sigma)]}{\expe[\psiG[,\wR](\sigma)]},
\frac{\expe[\ZG(\GR_{\mrmr})]}{\expe[\ZG(\GR)]}\le
P_\infty^{-c}.
\end{align*}
We obtain and define $\delta$, $\delta'$, $\iota$ and $\iota_\circ$ analogously to the proof of Corollary \ref{cor_paralleledges}, with the slight modification $\delta(m)=\expe[\phiG(\GTSM[\mrmr](\sigmaIID))]-\expe[\frac{1}{n}\ln(\expe[\psiG[,\wR_{\mrmr}](\sigmaIID)|\sigmaIID])]$, and analogously for $\delta'$.
Now, for $k>2$ we have $\expe[|\bm b^{-1}(\ints_{>1})|]=\mclo(1/n)$ and $P_\infty=1+\mclo(1/n)$, so the graphs differ with probability $\mclo(1/n)$ and hence
\begin{align*}
|\expe[\phiG(\GR)]-\expe[\phiG(\GR_{\mrmr})]|,\,
|\expe[\phiG(\GTSM(\sigmaIID))]-\expe[\phiG(\GTSM[\mrmr](\sigmaIID))]|&=\mclo(n^{-1}),\\
\left|\frac{1}{n}\ln(\expe[\ZG(\GR)])-\frac{1}{n}\ln(\expe[\ZG(\GR_{\mrmr})])\right|&=\mclo(n^{-2}),\\
|\delta(m)-\delta'(m)|,\,|\iota(m)-\iota_\circ(m)|&=\mclo(n^{-1}).
\end{align*}
For $k=2$ there exists $C_\mfkg\in\reals_{>1}$ with $\expe[|\bm b^{-1}(\ints_{>1})|]\le C$ and $P_\infty\ge C^{-1}$, so the graph difference is at least $\sqrt{n}$ with probability $\mclo(\sqrt{n}^{-1})$, and repeating the arguments gives
\begin{align*}
|\expe[\phiG(\GR)]-\expe[\phiG(\GR_{\mrmr})]|,\,
|\expe[\phiG(\GTSM(\sigmaIID))]-\expe[\phiG(\GTSM[\mrmr](\sigmaIID))]|&=\mclo(\sqrt{n}^{-1}),\\
\left|\frac{1}{n}\ln(\expe[\ZG(\GR)])-\frac{1}{n}\ln(\expe[\ZG(\GR_{\mrmr})])\right|&=\mclo(n^{-1}),\\
|\delta(m)-\delta'(m)|,\,|\iota(m)-\iota_\circ(m)|&=\mclo(\sqrt{n}^{-1}).
\end{align*}
Lemma \ref{lemma_convergenceequivalence} with $\rho<1/4$ yields the assertion.
\end{proof}
Now, let $\mcla=\{\alpha\in\ints_{\ge 0}^q:\|\alpha\|_1=k\}$ be absolute frequencies of factor assignments, $\psiR_{\mrma\circ}:\mcla\rarr[\psibl,\psibu]$, $\psiR_{\circ}:[q]^k\rarr[\psibl,\psibu]$, $\tau\mapsto\psiR_{\mrma\circ}((|\tau^{-1}(\sigma)|)_{\sigma})$ and let $\lawpsi$ be the law of $\psiR_\circ$, i.e.~$\psiR_\circ$ is invariant to permutations of the coordinates.
Now, we may consider the hypergraph $\vR\dequal\unif(\mclv)$ with $\mclv=\{v\setle\binom{[n]}{k}:|v|=m\}$, equipped with weights $\psiR_{\mrma}\dequal\psiR_{\mrma\circ}^{\otimes\vR}$.
The weight for the null model is $\psiG[,(\vR,\psiR_{\mrma})](\sigma)=\prod_{a\in\vR}\psiR_{\mrma,a}(\sigma_a)$, which induces the teacher-student model.
An immediate consequence of Corollary \ref{cor_duplicatefactors} is that the main results also hold for this pair of models, which are pushforwards of the models in Corollary \ref{cor_duplicatefactors}.
\subsection{The Pinning Lemma}\label{pinning_lemma_tv}
In this section we briefly derive the strengthened generalization of Lemma 3.5 in \cite{coja2018} from Lemma \ref{lemma_pinning}.
Recall the notions from Lemma \ref{lemma_pinning} and the $(\eps,\ell)$-symmetric measures $\setES$ from Remark \ref{remark_epssym}.
\begin{corollary}\label{cor_pinning_lemma}
For $n\in\ints_{>0}$, $\mu\in\mclp([q]^n)$, $\ell\in\ints_{>0}$, $\eps\in\reals_{>0}$ and $\ThetaP\in(0,n]$ we have
$\prob[[\mu]^\darr_{\setPR,\sigmaR}\not\in\setES[,n,\eps,\ell]]\le\binom{\ell}{2}\ln(q)/(2\eps^2\ThetaP)$.
\end{corollary}
\begin{proof}
Let $\bm\mu=[\mu]^\darr_{\setPR,\sigmaR}$, recall $\nu$ from Remark \ref{remark_epssym} and that $\nu(\mu)\le\sqrt{\iota(\mu)/2}$. Hence, on $\nu(\bm\mu)>\eps$ we have $\iota(\bm\mu)\ge 2\eps^2$ and thereby Lemma \ref{lemma_pinning} with Markov's inequality yields
$\prob[\bm\mu\not\in\setES]=\prob[\nu(\bm\mu)>\eps]\le\prob[\iota(\bm\mu)\ge 2\eps^2]\le\binom{\ell}{2}\ln(q)/(2\eps^2\ThetaP)$.
\end{proof}
Corollary \ref{cor_pinning_lemma} suggests that for $\ThetaP>\binom{\ell}{2}\ln(q)/(2\eps^3)$ and $n\ge\ThetaP$ we have $\prob[[\mu]^\darr_{\setPR,\sigmaR}\not\in\setES]<\eps$.
In order to recover Lemma 3.5 in \cite{coja2018} we consider $\ell=2$.
\bibliography{literature}

\begin{thebibliography}{10}

\bibitem{abbe2017}
Emmanuel Abbe.
\newblock Community detection and stochastic block models: recent developments.
\newblock {\em J. Mach. Learn. Res.}, 18:Paper No. 177, 86, 2017.

\bibitem{abbe2015}
Emmanuel Abbe and Andrea Montanari.
\newblock Conditional random fields, planted constraint satisfaction, and
  entropy concentration.
\newblock {\em Theory Comput.}, 11:413--443, 2015.
\newblock \href {https://doi.org/10.4086/toc.2015.v011a017}
  {\path{doi:10.4086/toc.2015.v011a017}}.

\bibitem{montanari2008a}
Montanari Andrea.
\newblock Estimating random variables from random sparse observations.
\newblock {\em European Transactions on Telecommunications}, 19(4):385--403,
  2008.
\newblock URL:
  \url{https://onlinelibrary-wiley-com.emedien.ub.uni-muenchen.de/doi/abs/10.1002/ett.1289},
  \href
  {http://arxiv.org/abs/https://onlinelibrary-wiley-com.emedien.ub.uni-muenchen.de/doi/pdf/10.1002/ett.1289}
  {\path{arXiv:https://onlinelibrary-wiley-com.emedien.ub.uni-muenchen.de/doi/pdf/10.1002/ett.1289}},
  \href
  {https://doi.org/https://doi-org.emedien.ub.uni-muenchen.de/10.1002/ett.1289}
  {\path{doi:https://doi-org.emedien.ub.uni-muenchen.de/10.1002/ett.1289}}.

\bibitem{coja2018b}
Amin Coja-Oghlan, Charilaos Efthymiou, Nor Jaafari, Mihyun Kang, and Tobias
  Kapetanopoulos.
\newblock Charting the replica symmetric phase.
\newblock {\em Comm. Math. Phys.}, 359(2):603--698, 2018.
\newblock \href {https://doi.org/10.1007/s00220-018-3096-x}
  {\path{doi:10.1007/s00220-018-3096-x}}.

\bibitem{coja2021}
Amin Coja-Oghlan, Max Hahn-Klimroth, Philipp Loick, Noela M\"{u}ller,
  Konstantinos Panagiotou, and Matija Pasch.
\newblock Inference and mutual information on random factor graphs.
\newblock In {\em 38th {I}nternational {S}ymposium on {T}heoretical {A}spects
  of {C}omputer {S}cience}, volume 187 of {\em LIPIcs. Leibniz Int. Proc.
  Inform.}, pages Art. No. 24, 15. Schloss Dagstuhl. Leibniz-Zent. Inform.,
  Wadern, 2021.

\bibitem{coja2020}
Amin Coja-Oghlan, Tobias Kapetanopoulos, and Noela M\"{u}ller.
\newblock The replica symmetric phase of random constraint satisfaction
  problems.
\newblock {\em Combin. Probab. Comput.}, 29(3):346--422, 2020.
\newblock \href {https://doi.org/10.1017/s0963548319000440}
  {\path{doi:10.1017/s0963548319000440}}.

\bibitem{coja2018}
Amin Coja-Oghlan, Florent Krzakala, Will Perkins, and Lenka Zdeborov\'{a}.
\newblock Information-theoretic thresholds from the cavity method.
\newblock {\em Adv. Math.}, 333:694--795, 2018.
\newblock \href {https://doi.org/10.1016/j.aim.2018.05.029}
  {\path{doi:10.1016/j.aim.2018.05.029}}.

\bibitem{franz2003}
Silvio Franz and Michele Leone.
\newblock Replica bounds for optimization problems and diluted spin systems.
\newblock {\em J. Statist. Phys.}, 111(3-4):535--564, 2003.
\newblock \href {https://doi.org/10.1023/A:1022885828956}
  {\path{doi:10.1023/A:1022885828956}}.

\bibitem{janson2000}
Svante Janson, Tomasz {\L}uczak, and Andrzej Rucinski.
\newblock {\em Random graphs}.
\newblock Wiley-Interscience Series in Discrete Mathematics and Optimization.
  Wiley-Interscience, New York, 2000.
\newblock \href {https://doi.org/10.1002/9781118032718}
  {\path{doi:10.1002/9781118032718}}.

\bibitem{mezard2009}
M.~M\'{e}zard and A.~Montanari.
\newblock {\em Information, physics, and computation}.
\newblock Oxford Graduate Texts. Oxford University Press, Oxford, 2009.
\newblock \href {https://doi.org/10.1093/acprof:oso/9780198570837.001.0001}
  {\path{doi:10.1093/acprof:oso/9780198570837.001.0001}}.

\bibitem{montanari2008}
Andrea Montanari, Federico Ricci-Tersenghi, and Guilhem Semerjian.
\newblock Clusters of solutions and replica symmetry breaking in
  randomk-satisfiability.
\newblock {\em Journal of Statistical Mechanics: Theory and Experiment},
  2008(04):P04004, apr 2008.
\newblock \href {https://doi.org/10.1088/1742-5468/2008/04/p04004}
  {\path{doi:10.1088/1742-5468/2008/04/p04004}}.

\bibitem{moore2017}
Cristopher Moore.
\newblock The computer science and physics of community detection: landscapes,
  phase transitions, and hardness.
\newblock {\em Bull. Eur. Assoc. Theor. Comput. Sci. EATCS}, (121):26--61,
  2017.

\bibitem{panaretos2020}
Victor~M. Panaretos and Yoav Zemel.
\newblock {\em An invitation to statistics in {W}asserstein space}.
\newblock SpringerBriefs in Probability and Mathematical Statistics. Springer,
  Cham, [2020] \copyright 2020.
\newblock URL:
  \url{https://doi-org.emedien.ub.uni-muenchen.de/10.1007/978-3-030-38438-8},
  \href {https://doi.org/10.1007/978-3-030-38438-8}
  {\path{doi:10.1007/978-3-030-38438-8}}.

\bibitem{panchenko2004}
Dmitry Panchenko and Michel Talagrand.
\newblock Bounds for diluted mean-fields spin glass models.
\newblock {\em Probab. Theory Related Fields}, 130(3):319--336, 2004.
\newblock \href {https://doi.org/10.1007/s00440-004-0342-2}
  {\path{doi:10.1007/s00440-004-0342-2}}.

\bibitem{sason2016}
Igal Sason and Sergio Verd\'{u}.
\newblock {$f$}-divergence inequalities.
\newblock {\em IEEE Trans. Inform. Theory}, 62(11):5973--6006, 2016.
\newblock URL:
  \url{https://doi-org.emedien.ub.uni-muenchen.de/10.1109/TIT.2016.2603151},
  \href {https://doi.org/10.1109/TIT.2016.2603151}
  {\path{doi:10.1109/TIT.2016.2603151}}.

\bibitem{talagrand2001}
Michel Talagrand.
\newblock The high temperature case for the random {$K$}-sat problem.
\newblock {\em Probab. Theory Related Fields}, 119(2):187--212, 2001.
\newblock \href {https://doi.org/10.1007/PL00008758}
  {\path{doi:10.1007/PL00008758}}.

\bibitem{villani2009}
C\'{e}dric Villani.
\newblock {\em Optimal transport}, volume 338 of {\em Grundlehren der
  mathematischen Wissenschaften [Fundamental Principles of Mathematical
  Sciences]}.
\newblock Springer-Verlag, Berlin, 2009.
\newblock Old and new.
\newblock \href {https://doi.org/10.1007/978-3-540-71050-9}
  {\path{doi:10.1007/978-3-540-71050-9}}.

\bibitem{zdeborova2016}
Lenka Zdeborová and Florent Krzakala.
\newblock Statistical physics of inference: thresholds and algorithms.
\newblock {\em Advances in Physics}, 65(5):453--552, 2016.
\newblock \href
  {http://arxiv.org/abs/https://doi.org/10.1080/00018732.2016.1211393}
  {\path{arXiv:https://doi.org/10.1080/00018732.2016.1211393}}, \href
  {https://doi.org/10.1080/00018732.2016.1211393}
  {\path{doi:10.1080/00018732.2016.1211393}}.

\end{thebibliography}
\end{document}